\documentclass[12pt]{amsart}

  \usepackage{a4wide}
  \usepackage{amssymb}
  \usepackage{amsmath}
  \usepackage{amsthm}
  \usepackage{amsfonts}
  \usepackage{amscd}
  \usepackage{xspace}
  \usepackage{boxedminipage}
  \usepackage{nomencl}

  \usepackage{booktabs}
  \usepackage{multirow}
  \usepackage[figuresright]{rotating}
  \usepackage[format=hang,font=small]{caption} 
  \usepackage[curve,graph]{xypic}
  \usepackage{lscape}
  \usepackage{subfigure}

  \usepackage{stmaryrd}

  \usepackage{tikz}
  \usetikzlibrary{calc}
  \newcommand{\tikztableau}[2][scale=0.6,every node/.style={font=\small}]{
      \def\newtableau{#2}
      \begin{array}{c}
      \begin{tikzpicture}[#1]
      \coordinate (x) at (-0.5,0.5);
      \coordinate (y) at (-0.5,0.5);
      \foreach \row in \newtableau {
          \coordinate (x) at ($(x)-(0,1)$);
          \coordinate (y) at (x);
          \foreach \entry in \row {
              \node (y) at ($(y) + (1,0)$) {\entry};
              \draw ($(y)-(0.5,0.5)$) rectangle +(1,1);
              }
          }
      \end{tikzpicture}
      \end{array}}

  \newcommand{\tikztableausmall}[1]{\tikztableau[scale=0.45,every node/.style={font=\small}]{#1}}

  \makeatletter
  \CheckCommand*\refstepcounter[1]{\stepcounter{#1}%
      \protected@edef\@currentlabel
       {\csname p@#1\endcsname\csname the#1\endcsname}%
  }
  \renewcommand*\refstepcounter[1]{\stepcounter{#1}%
    \protected@edef\@currentlabel
      {\csname p@#1\expandafter\endcsname\csname the#1\endcsname}%
  }
  \def\labelformat#1{\expandafter\def\csname p@#1\endcsname##1}
  \DeclareRobustCommand\Ref[1]{\protected@edef\@tempa{\ref{#1}}%
     \expandafter\MakeUppercase\@tempa
  }
  \makeatother

  \makeatletter
  \newcommand{\numberlike}[2]{%
     \expandafter\def\csname c@#1\endcsname{%
         \expandafter\csname c@#2\endcsname}%
  }
  \makeatother

  \def\DefaultNumberTheoremWithin{section}

  \theoremstyle{plain}
  \newtheorem{Lemma}{Lemma}
     \numberwithin{Lemma}{\DefaultNumberTheoremWithin}
     \labelformat{Lemma}{Lemma~#1}
  
     \numberwithin{Claim}{\DefaultNumberTheoremWithin}
     \numberlike{Claim}{Lemma}
     \labelformat{Claim}{Claim~#1}
  \newtheorem{Theorem}{Theorem}
     \numberwithin{Theorem}{\DefaultNumberTheoremWithin}
     \numberlike{Theorem}{Lemma}
     \labelformat{Theorem}{Theorem~#1}
  \newtheorem{Corollary}{Corollary}
     \numberwithin{Corollary}{\DefaultNumberTheoremWithin}
     \numberlike{Corollary}{Lemma}
     \labelformat{Corollary}{Corollary~#1}
  \newtheorem{Proposition}{Proposition}
     \numberwithin{Proposition}{\DefaultNumberTheoremWithin}
     \numberlike{Proposition}{Lemma}
     \labelformat{Proposition}{Proposition~#1}
  \newtheorem{Conjecture}{Conjecture}
     \numberwithin{Conjecture}{\DefaultNumberTheoremWithin}
     \numberlike{Conjecture}{Lemma}
     \labelformat{Conjecture}{Conjecture~#1}

  \theoremstyle{definition}
  \newtheorem{Definition}{Definition}
     \numberwithin{Definition}{\DefaultNumberTheoremWithin}
     \numberlike{Definition}{Lemma}
     \labelformat{Definition}{Definition~#1}

  \theoremstyle{definition}
  \newtheorem{Question}{Question}
     \numberwithin{Question}{\DefaultNumberTheoremWithin}
     \numberlike{Question}{Lemma}
     \labelformat{Question}{Question~#1}

  \theoremstyle{definition}
  \newtheorem{Problem}{Problem}
     \numberwithin{Problem}{\DefaultNumberTheoremWithin}
     \numberlike{Problem}{Lemma}
     \labelformat{Problem}{Problem~#1}

  \theoremstyle{remark}
  \newtheorem{Remark}{Remark}
     \numberwithin{Remark}{\DefaultNumberTheoremWithin}
     \numberlike{Remark}{Lemma}
     \labelformat{Remark}{Remark~#1}
  \theoremstyle{remark}

  \newtheorem{Example}{Example}
     \numberwithin{Example}{\DefaultNumberTheoremWithin}
     \numberlike{Example}{Lemma}
     \labelformat{Example}{Example~#1}
  
     \labelformat{Case}{Case~#1}
     \numberwithin{Case}{Lemma}
  
     \labelformat{Step}{Step~#1}
     \numberwithin{Step}{Lemma}

  \labelformat{equation}{(#1)}
  \labelformat{figure}{Figure~#1}
  \labelformat{section}{\S#1}
  \labelformat{subsection}{\S#1}
  \labelformat{subsubsection}{\S#1}

  \def\eqref{\ref}

  \allowdisplaybreaks

  \newcommand{\mb}{\mathbb}
  \newcommand{\mc}{\mathcal}
  \newcommand{\mf}{\mathfrak}

  \def\binomial(#1,#2){{#1\choose #2}}
\def\delplus(#1,#2,#3){\ensuremath{\pi}_{#1,#2,#3}}
\def\delminus(#1,#2){\ensuremath{\operatorname{\partial}}_{#1,#2}}
  
  \def\redword(#1){{\ensuremath{\tilde{#1}}}}

  \newcommand{\spi}{{\frac{\pi}{3}}}
  \newcommand{\npi}{{\hskip1pt{\mathbf\not}\hskip-0pt\frac{\pi}{3}}}
  \newcommand{\maj}{\ensuremath{\operatorname{maj}}}
  \newcommand{\inv}{\ensuremath{\operatorname{inv}}}
  \newcommand{\ninv}{\ensuremath{\operatorname{noninv}}}
  \newcommand{\oddcols}{\ensuremath{\operatorname{oddcols}}}
  \newcommand{\rev}{\ensuremath{\operatorname{rev}}}
  \newcommand{\sgn}{\ensuremath{\operatorname{sgn}}}
  \newcommand{\Mat}{\ensuremath{\operatorname{Mat}}}
  \newcommand{\Centr}{\ensuremath{\operatorname{Z}}}
  \newcommand{\Stab}{\ensuremath{\operatorname{Stab}}}
  \newcommand{\Norma}{\ensuremath{\operatorname{N}}}
  \newcommand{\chain}{{\ensuremath{\operatorname{\text{\sc c}}}}}
  \newcommand{\ReducedChainGroup}{{\ensuremath{\operatorname{\widetilde{\text{\sc C}}}}}}
  \newcommand{\Homology}{{\ensuremath{\operatorname{{\text{\sc H}}}}}}
  \newcommand{\ReducedHomology}{{\ensuremath{\operatorname{\widetilde{\text{\sc H}}}}}}
  \newcommand{\WH}{{\ensuremath{\operatorname{\text{\sc{WH}}}}}}
  \newcommand{\grid}{{\ensuremath{\operatorname{{\mathfrak 1}}}}}
  \newcommand{\semigrid}{{\ensuremath{\operatorname{{\mathfrak 1}}}}}
  \newcommand{\idem}{{\ensuremath{\operatorname{{\mathfrak e}}}}}
  \newcommand{\unitbase}{{\ensuremath{\operatorname{\text{\sc e}}}}}
  \newcommand{\matid}{{\ensuremath{\operatorname{I}}}}
  \newcommand{\matone}{{\ensuremath{\operatorname{{1 \hskip-4pt 1}}}}}
  \newcommand{\size}{\ensuremath{\operatorname{size}}}
  \newcommand{\trivial}{{\mathbf{1}}}
  
  \newcommand{\Fix}{\ensuremath{\operatorname{Fix}}}
  \newcommand{\Ind}{\ensuremath{\operatorname{Ind}}}
  \newcommand{\Res}{\ensuremath{\operatorname{Res}}}
  \newcommand{\Irr}{\ensuremath{\operatorname{Irr}}}
  \newcommand{\Trace}{\ensuremath{\operatorname{Trace}}}
  \newcommand{\ch}{\ensuremath{\operatorname{ch}}}
  \newcommand{\im}{\ensuremath{\operatorname{im}}}
  \newcommand{\Des}{\ensuremath{\operatorname{Des}}}
  \newcommand{\Sym}{\ensuremath{\operatorname{Sym}}}
  \newcommand{\Lie}{\ensuremath{\operatorname{Lie}}}
  \newcommand{\Lin}{\ensuremath{\operatorname{Lin}}}

  \newcommand{\shape}{\ensuremath{\operatorname{shape}}}
  \newcommand{\eig}{\ensuremath{\operatorname{eig}}}
  \newcommand{\SYT}{\ensuremath{\operatorname{SYT}}}

  \newcommand{\demote}{\ensuremath{\operatorname{demote}}}
  
  \newcommand{\rank}{\ensuremath{\operatorname{rank}}}
  \newcommand{\concat}{\ensuremath{\bullet}}
  \newcommand{\OS}{\ensuremath{\operatorname{OS}}}

  \newcommand{\AAA}{\mc{A}}
  
  \newcommand{\CCC}{\mc{C}}
  \newcommand{\FFF}{\mc{F}}
  \newcommand{\III}{\mc{I}}
  \newcommand{\LLL}{\mc{L}}
  \newcommand{\OOO}{\mc{O}}

  \newcommand{\ZZ}{\mb{Z}}
  \newcommand{\QQ}{\mb{Q}}
  \newcommand{\RR}{\mb{R}}
  \newcommand{\CC}{\mb{C}}
  \newcommand{\KK}{\mb{K}}

  \newcommand{\xx}{\mathbf{x}}
  \newcommand{\yy}{\mathbf{y}}
  \newcommand{\nicet}{\mathbf{t}}
  \newcommand{\pp}{\mathbf{p}}

  \newcommand{\frL}{\mathfrak{L}}
  \newcommand{\BHR}{{\sf BHR}\xspace}
  
  \newcommand{\End}{{\mathrm{End}}}
  
  \newcommand{\GL}{\mathrm{GL}}

  \newcommand{\oo}{\mf{o}}
  \newcommand{\Prob}{\mathcal{P}}
  \newcommand{\symm}{\mf{S}}
  \newcommand{\supp}{{\mathrm{supp}}}

  \newcommand{\deffont}{\it}
  \newcommand{\emphfont}{\sf}

  \newenvironment{note}[1][Note]
   {\bigskip\begin{center}\begin{boxedminipage}{4.5in}\setlength{\parindent}{1em}\noindent\textbf{#1. }}
   {\end{boxedminipage}\end{center}\bigskip}

  \makenomenclature

  \makeindex

\begin{document}


 \title[Spectra of symmetrized shuffling operators]{Spectra of symmetrized shuffling operators}


  \author{Victor Reiner}
     \address{School of Mathematics\\
              University of Minnesota\\
              Minneapolis MN 55455\\
              USA}
     \email{reiner@math.umn.edu}

  \author{Franco Saliola}
     \address{Laboratoire de Combinatoire et d'Informatique Math\'ematique (LaCIM) \\
              Universit\'e du Qu\'ebec \`a Montr\'eal \\
              CP 8888, Succ. Centre-ville \\
              Montr\'eal (Qu\'ebec) H3C 3P8 \\
              Canada}
    \email{saliola.franco@uqam.ca}

  \author{Volkmar Welker}
     \address{Fachbereich Mathematik und Informatik\\
              Philipps-Universit\"at Marburg\\
              35032 Marburg\\
              Germany}
     \email{welker@mathematik.uni-marburg.de}

  \thanks{Work of first author supported by NSF grant DMS-0601010. 
  Work of the second author was supported by
      Agence Nationale de la Recherche (France) grant ANR-06-BLAN-0380
  and 
      the Canada Research Chair of N. Bergeron.
  The work of the third author was supported by DFG}

  \begin{abstract}
     For a finite real reflection group $W$ and a $W$-orbit 
     $\OOO$ of flats in its reflection arrangement 
     -- or equivalently a conjugacy class of its parabolic subgroups -- 
     we introduce a statistic $\ninv_\OOO(w)$ on $w$ in $W$ that counts the number of ``$\OOO$-noninversions''
     of $w$.  This generalizes the classical (non-)inversion statistic for permutations $w$
     in the symmetric group $\symm_n$.  We then study the 
     operator $\nu_\OOO$ of right-multiplication
     within the group algebra $\CC W$ by the element 
     that has $\ninv_\OOO(w)$ as its coefficient on $w$.
     
     We reinterpret $\nu_\OOO$  geometrically in terms of the arrangement of 
     reflecting hyperplanes for $W$, and more generally, for any real arrangement of linear
     hyperplanes.  At this level of generality, one finds that, after appropriate scaling, 
     $\nu_\OOO$ corresponds to a Markov chain on the chambers of the arrangement.  
     We show that $\nu_\OOO$ is self-adjoint and positive semidefinite, via 
     two explicit factorizations into a symmetrized form $\pi^t \pi$.  
     In one such factorization, the matrix $\pi$ is a generalization of the 
     projection of a simplex onto the {\it linear ordering polytope} from the 
     theory of social choice.

     In the other factorization of $\nu_\OOO$ as $\pi^t \pi$, the matrix $\pi$ 
     is the transition matrix for one of the well-studied
     {\it Bidigare-Hanlon-Rockmore random walks} on the chambers of an arrangement.  
     We study closely the example of the family of 
     operators $\{ \nu_{(k,1^{n-k})} \}_{k=1,2,\ldots,n}$,
     corresponding to the case where $\OOO$ is 
     the conjugacy classes of Young subgroups in $W=\symm_n$ of type
     $(k,1^{n-k})$.  The $k=n-1$ special case within this family is the operator
     $\nu_{(n-1,1)}$ corresponding to {\it random-to-random shuffling}, 
     factoring as $\pi^t \pi$ where $\pi$  corresponds to {\it random-to-top shuffling}.
     We show in a purely enumerative fashion that this family of 
     operators $\{ \nu_{(k,1^{n-k})} \}$
     pairwise commute.  We furthermore conjecture that they have integer spectrum,
     generalizing a conjecture of Uyemura-Reyes for the case $k=n-1$.  Although we
     do not know their complete simultaneous eigenspace decomposition, we give a coarser 
     block-diagonalization of these operators, along with explicit descriptions 
     of the $\CC W$-module structure on each block.

     We further use representation theory to show that 
     if $\OOO$ is a conjugacy class of {\it rank one} parabolics in $W$, 
     multiplication by $\nu_\OOO$ has integer spectrum;
     as a very special case, this holds for the matrix 
     $(\inv(\sigma \tau^ {-1}))_{\sigma,\tau \in \symm_n}$.
     The proof uncovers a fact of independent interest.   
     Let $W$ be an irreducible finite reflection group and $s$ any reflection in $W$,
     with reflecting hyperplane $H$. Then the $\{\pm 1\}$-valued character $\chi$ of the centralizer subgroup $Z_W(s)$
     given by its action on the line $H^\perp$ has the property that $\chi$ is 
     multiplicity-free when induced up to $W$.  In other words, $(W, Z_W(s), \chi)$ forms a
     {\it twisted Gelfand pair}.

     We also closely study the example of the family of 
     operators $\{ \nu_{(2^k,1^{n-2k})} \}_{k=0,1,2,\ldots,\lfloor \frac{n}{2} \rfloor}$
      corresponding to the case where $\OOO$ is 
     the conjugacy classes of Young subgroups in $W=\symm_n$ of type
     $(2^k,1^{n-2k})$. Here the 
     construction of a {\it Gelfand model} for $\symm_n$ shows
     both that these operators pairwise pairwise commute,
     and that they have integer spectrum.
     
     We conjecture that, apart from these two commuting families 
     $\{ \nu_{(k,1^{n-k})} \}$ and $\{ \nu_{(2^k,1^{n-2k})} \}$
     and trivial cases, 
     no other pair of operators of the form $\nu_\OOO$ commutes for $W=\symm_n$.
  \end{abstract}

  \maketitle

\section{Introduction}
           \label{sec:introduction}

  This work grew from the desire to understand why a certain family of
  combinatorial matrices were pairwise-commuting and had only integer 
  eigenvalues. We start by describing them.

  \subsection{The original family of matrices}

    The matrices are constructed from certain statistics
    on the symmetric group 
    \nomenclature[al]{$\symm_n$}{symmetric group on n $n$ letters}%
    $W=\symm_n$ on $n$ letters.
    \index{symmetric group}%
    \index{group!symmetric}%
    Given a permutation $w$ in $W$,
    define the {\deffont $k$-noninversion number}
    \index{$k$-noninversion number}%
    \footnote{The terminology comes
    from the case $k=2$, where $\ninv_2(w)$ counts the pairs $(i,j)$ with 
    $1 \leq i < j \leq n$ that index a {\deffont noninversion} in a permutation
    \index{noninversion number}%
    $w$ in $W$, meaning that $w_i < w_j$.} 
    \nomenclature[co]{$\ninv_k(w)$}{$k$-noninversion number of $w$}%
    $\ninv_k(w)$ to be the number of
    $k$-element subsets $\{i_1,\ldots,i_k\}$ 
    with $1 \leq i_1 < \cdots < i_k \leq n$ for which 
    $w_{i_1} < \cdots < w_{i_k}$.
    In the literature on {\deffont permutation patterns}, one
    \index{permutation!pattern}%
    might call $\ninv_k(w)$ the number of occurrences of the
    permutation pattern $1 2 \cdots k$.  Alternately, $\ninv_k(w)$ is
    the number of {\deffont increasing subsequences} 
    \index{increasing subsequence}%
    of length $k$ occurring in the word $w=w_1w_2\cdots w_n$. 

    From this statistic $\ninv_k(-)$ on the group $W=\symm_n$,
    create a matrix 
    \nomenclature[co]{$\nu_{(k,1^{n-k})}$}{matrix of $k$-noninversion numbers}%
    $\nu_{(k,1^{n-k})}$ in $\ZZ^{|W| \times |W|}$,
    having rows and columns indexed by the 
    permutations $w$ in $W$, and whose $(u,v)$-entry is $\ninv_k(v^{-1}u)$.
    One of the original mysteries that began this project was the
    following result, now proven in \ref{sec:original-family}.

    \begin{Theorem}
      \label{thm:original-family-commutativity}
      The operators from the family $\{\nu_{(k,1^{n-k})} \}_{k=1,2,\ldots,n}$
      pairwise commute.
    \end{Theorem}

    It is not hard to see (and will be shown in 
    \ref{prop:rectangular-square-root}) 
    that one can factor each of these matrices
    $\nu_{(k,1^{n-k})} = \pi^T \pi$ for certain other integer (even $0/1$) 
    matrices $\pi$.
    Therefore, each $\nu_{(k,1^{n-k})}$ is symmetric positive semidefinite, 
    and hence diagonalizable with only real nonnegative eigenvalues.  
    \ref{thm:original-family-commutativity} asserts that they
    form a commuting family, and hence can be simultaneously diagonalized.
    The following conjecture also motivated this project, but has
    seen only partial progress here.

    \begin{Conjecture}
       \label{conj:original-family-integrality}
       The operators $\{\nu_{(k,1^{n-k})} \}_{k=1,2,\ldots,n}$ have only
       integer eigenvalues.
    \end{Conjecture}

    In the special case $k=n-1$, this matrix $\nu_{(n-1,1)}$ 
    was studied already in the
    Stanford University PhD thesis of Jay-Calvin Uyemura-Reyes \cite{UyemuraReyes2002}.
    Uyemura-Reyes examined a certain random walk on $W$ called the
    {\deffont random-to-random shuffling} operator, whose Markov matrix is a
    \index{random!-to-random shuffle}%
    \index{shuffle!random-to-random}%
    rescaling of $\nu_{(n-1,1)}$.  He was interested in its eigenvalues
    in order to investigate the rate of convergence of this random walk 
    to the uniform distribution on $W$. He was surprised to discover 
    empirically, and conjectured, that $\nu_{(n-1,1)}$ has only 
    {\emphfont integer} eigenvalues\footnote{In addition, the thesis \cite[p 152-153]{UyemuraReyes2002} mentions other
    shuffling operators that have ``{\it eigenvalues with surprising
    structure}''. We have been informed by Persi Diaconis, the advisor of
    Uyemura-Reyes, that among others this refers to computational experiments on 
    shuffling operators that are convex combinations with rational coefficients of the
    shuffling operators corresponding to $\nu_{(k,1^{n-k})}$.
    Uyemura-Reyes observed  
    integral spectrum for small $n$ after suitable scaling.  
    Clearly, using \ref{thm:original-family-commutativity}
    this fact for general $n$ is implied by \ref{conj:original-family-integrality}.}.
    This was one of many unexpected connections encountered during the
    work on this project, since a question from computer science 
    (see \ref{subsec:linear-ordering-polytopes}) independently led to our 
    \ref{thm:original-family-commutativity} and 
    \ref{conj:original-family-integrality}.

  \subsection{Using the $W$-action}

    One can readily check that the matrix $\nu_{(k,1^{n-k})}$ in 
    $\ZZ^{|W| \times |W|}$ represents multiplication on the {\emphfont right} 
    within the group algebra 
    \nomenclature[al]{$\ZZ W$}{group algebra of $W$ over the integers $\ZZ$}%
    \index{algebra!group}%
    \index{group!algebra}%
    $\ZZ W$ by the following element of $\ZZ W$ 
    (also denoted $\nu_{(k,1^{n-k})}$, by an abuse of notation):
    $$
      \nu_{(k,1^{n-k})} := \sum_{w \in W} \ninv_k(w) \cdot w.
    $$
    Consequently, the action of $\nu_{(k,1^{n-k})}$ commutes with the 
    {\emphfont left}-regular action of $\RR W$
    \nomenclature[al]{$\RR$}{real numbers}%
    \nomenclature[al]{$\RR W$}{group algebra of $W$ over $\RR$}%
    on itself, and the (simultaneous) eigenspaces of the
    matrices $\nu_{(k,1^{n-k})}$ are all representations of $W$.  
    This extra structure will prove to be extremely useful in the
    rest of the work.

    In fact, Uyemura-Reyes \cite{UyemuraReyes2002} conjectured 
    descriptions for the $\RR W$-irreducible decompositions of certain 
    of the eigenspaces of $\nu_{(n-1,1)}$, and was able to prove some of 
    these conjectures in special cases. 
    Furthermore, he reported \cite[\S5.2.3]{UyemuraReyes2002} 
    an observation of R.~Stong noting that one of the 
    factorizations of $\nu_{(n-1,1)} = \pi^T \pi$
    mentioned earlier can be obtained by letting $\pi$ be the 
    well-studied {\deffont random-to-top} shuffling operator on $W$.  
    \index{random!-to-top shuffle}%
    \index{shuffle!random-to-top}%
    These operators are one example from a family of 
    very well-behaved random walks on $W$ that were introduced 
    by Bidigare, Hanlon, and Rockmore, \BHR 
    \nomenclature[co]{\BHR}{Bidigare, Hanlon and Rockmore}%
    for short, in
     \cite{Bidigare1997} and \cite{BidigareHanlonRockmore1999}. 
    These authors showed that the 
    {\deffont \BHR random walks} have very simply predictable integer 
    \index{random!walk}%
    eigenvalues, and the $W$-action on their eigenspaces are also 
    well-described.

    We exploit this connection further, as follows.  First, we will show
    (in \ref{prop:second-square-root} and 
    \ref{cor:second-square-root-is-bhr}) 
    that more generally one has a factorization 
    $\nu_{(k,1^{n-k})} = \pi^T \pi$
    in which $\pi$ is another family of \BHR random walks.  
    Second, we will use the fact that this implies 
    $\ker\nu_{(k,1^{n-k})} =\ker \pi$,
    along with \ref{thm:original-family-commutativity}, to obtain a
    $W$-equivariant filtration of $\RR W$ that is preserved by each 
    $\nu_{(k,1^{n-k})}$,
    with a complete description of the $\RR W$-structure on each
    filtration factor.  This has consequences (see e.g.
    \ref{subsec:defining-representation})
    for the $\RR W$-module structure on the simultaneous 
    eigenspaces of the commuting family of $\nu_{(k,1^{n-k})}$.

  \subsection{An eigenvalue integrality principle}

    Another way in which we will exploit the $W$-action comes from
    a simple but powerful {\deffont eigenvalue integrality principle} 
    \index{eigenvalue integrality principle}%
    for combinatorial operators.  We record it here, as we will
    use it extensively later.

    To state it, recall that for a finite group $W$, when one considers
    representations of $W$ over fields $\KK$ 
    \nomenclature[al]{$\KK$}{generic field}%
    of characteristic zero, any
    finite-dimensional $\KK W$-module 
    \nomenclature[al]{$\KK W$}{group algebra of $W$ over $\KK$}%
    $U$ is semisimple, 
    that is, it can be decomposed as a direct sum of simple $\KK W$-modules.  
    When considering field extensions $\KK' \supset \KK$, the simple 
    $\KK W$-modules may or may not split further when extended to 
    $\KK' W$-modules;  one says that a simple $\KK W$-module is 
    {\deffont absolutely irreducible} if it remains irreducible as a
    \index{absolutely irreducible module}%
    \index{module!absolutely irreducible}%
    $\KK' W$-module for any extension $\KK'$ of $\KK$.  
    Given any finite group $W$, a {\deffont splitting field} 
    \index{splitting field}%
    (see \cite[Chapter X]{CurtisReiner1962}) 
    for $W$ over $\QQ$ is a field extension $\KK$ of $\QQ$ such that
    \nomenclature[al]{$\QQ$}{rational numbers}%
    every simple $\KK W$-module is absolutely irreducible. 
    Equivalently, $\KK$ is a splitting
    field of $W$ over $\QQ$ if and only if every irreducible matrix 
    representation of $W$ over $\QQ$ is realizable with entries in $\KK$
    \cite[Theorem 70.3]{CurtisReiner1962}.
    For such a field $\KK$, the simple
    $\KK W$-modules biject with the simple $\CC W$-modules, that is, 
    \nomenclature[al]{$\CC$}{complex numbers}%
    the set of simple $\KK W$-modules when extended to $\CC W$-modules gives
    exactly the set of simple $\CC W$-modules corresponding
    to the complex irreducible $W$-characters $\chi$.
    \index{character!irreducible}%
    \nomenclature[al]{$\chi$}{character of a group}%
    For finite $W$ the
    splitting field $\KK$ over $\QQ$ can always be chosen 
    to be a finite, and hence algebraic, extension of $\QQ$
    \cite[Theorem 70.23]{CurtisReiner1962}.
    If $W$ is a reflection group, then there is a unique 
    minimal
    \index{reflection!group}%
    \index{group!reflection}%
    extension $\KK$ of $\QQ$ such that $\KK$ is a splitting field for
    $W$ in characteristic $0$ (see \cite[Theorem 0.2]{Bessis1997}, \cite[Theorem 1]{Benard1976},
    and \cite[\S 1.7]{LehrerTaylor2009}).
    
    Denote by $\oo$ 
    \nomenclature[al]{$\oo$}{ring of integers within a fixed number field $\KK$}%
    the ring of integers within the unique minimal splitting 
    field $\KK$ for the reflection group $W$ in characteristic $0$;
    that is, the elements of $\KK$ that are roots of monic polynomials with 
    coefficients in $\ZZ$.  
    An important example occurs when $W$ is a crystallographic reflection group or
    equivalently a Weyl group.
    \index{Weyl group}%
    \index{group!Weyl}%
    Here it is known that one can take as a splitting field 
    $\KK=\QQ$ itself (see \cite[Corollary 1.15]{Springer1978}),
    and hence that $\oo=\ZZ$. 

    \begin{Proposition}[Eigenvalue integrality principle]
        \label{prop:integrality-principle}
        Let $W$ be a finite group acting in a $\ZZ$-linear fashion on $\ZZ^n$ and 
        let $\KK$ be a splitting field of $W$ in characteristic $0$. 
        Further let $A : \ZZ^n \rightarrow \ZZ^n$ be a $\ZZ$-linear operator 
        that commutes with the action $W$.  Extend the action of $A$ and of $W$ to $\KK^n$

        Then for any subspace $U \subseteq \KK^n$ which is stable under both $A$ and $W$,
        and on which $W$ acts without multiplicity (that is,
        each simple $\KK W$-module occurs at most once),
        all eigenvalues of the restriction of $A$ to $U$ lie in the ring of 
        integers $\oo$ of $\KK$.

        In particular, if $W$ is a Weyl group these eigenvalues of $A$ 
        lie in $\ZZ$.
    \end{Proposition}
    \begin{proof}
       An eigenvalue of $A$ is a root of its characteristic
       polynomial $\det(t \cdot \matid_{\KK^n} - A)$, a monic polynomial with $\ZZ$
       coefficients.  As usual $\matid_{\KK^n}$ denotes the identity matrix. 
       Hence, it is enough to show that the
       eigenvalues of $A$ acting on the $\KK$-subspace $U$ all lie in $\KK$.

       Because $\KK$ is a splitting field for $W$, one has an isotypic
       $\KK W$-module decomposition
       $ 
         U = \bigoplus_\chi U^{\chi}
       $
       \nomenclature{$U^\chi$}{$\chi$-isotypic component of $W$-module $U$}%
       in which the sum is over the irreducible characters $\chi$ of $W$.
       Since $A$ commutes with the $W$-action, it preserves this decomposition.
       The assumption that $U$ is multiplicity-free says each $U^{\chi}$ is
       a single copy of a simple $\KK W$-module.  Schur's Lemma 
       asserts that, on extending $\KK$ to its algebraic closure,
       $A$ must act on each $U^{\chi}$ by some scalar $\lambda_\chi$.
       However, $\lambda_\chi$ must lie in $\KK$ since $A$ acts $\KK$-linearly.
       Thus the isotypic decomposition diagonalizes the action of $A$ on $U$, 
       and all its eigenvalues lie in $\KK$ (and hence in $\oo$).
    \end{proof}

  \subsection{A broader context, with more surprises}

    Some of the initial surprises led us to consider a more general
    family of operators, in the context of {\deffont real reflection groups} 
    $W$,
    \index{real reflection group}%
    \index{reflection!group!real}%
    \index{group!real reflection}%
    leading to even more surprises.  We describe some of these
    briefly and informally here, indicating where they are
    discussed later.

    Let $W$ be a finite real reflection group, acting on an $\RR$-vector space
    \index{hyperplane!reflecting}%
    $V$, with set of reflecting hyperplanes 
    \nomenclature[ar]{$\AAA$}{arrangement of hyperplanes}%
    $\AAA$, and 
    \nomenclature[ar]{$\LLL(\AAA)$}{intersection lattice of the arrangement $\AAA$ of hyperplanes}%
    and let $\LLL$ be the (partially-ordered)
    \index{arrangement!hyperplanes}%
    \index{arrangement!intersection lattice}%
    \index{intersection lattice}%
    \index{lattice!intersection}%
    set of subspaces $X$ that arise as intersections of hyperplanes from some subset of $\AAA$.
    The hyperplanes in $\AAA$ dissect $V$ into connected
    components called {\deffont chambers}, and the set 
    \nomenclature[ar]{$\CCC(\AAA)$}{chambers of the arrangement $\AAA$ of hyperplanes}%
    \index{chamber}%
    $\CCC$ of all chambers carries a simply-transitive 
    action of $W$. Thus, if $\grid$ 
    \nomenclature[al]{$\grid$}{identity element of a group}%
    denotes the identity element of $W$, then one can choose 
    an identity chamber 
    \nomenclature[ar]{$c_\grid$}{chamber indexed by neutral element $\grid$ of group}%
    $c_\grid$ and an indexing of the chambers
    $\CCC=\{c_w:=w(c_\grid)\}_{w \in W}$.
    \index{identity chamber}%
    \index{chamber!identity}%

    Given a $W$-orbit $\OOO$ 
    \nomenclature[ar]{$\OOO$}{set of intersection subspaces of an arrangement. Often orbit or union or orbits under group action.}%
    of intersection subspaces, define
    $\ninv_\OOO(w)$
    \nomenclature[co]{$\ninv_\OOO(w)$}{number of subspaces in $\OOO$ for which the chambers indexed by $\grid$ and $w$ lie on the same side}%
    to be the number of subspaces $X$ in $\OOO$
    for which the two chambers $c_w$ and $c_\grid$ lie on the same side
    of every hyperplane $H \supseteq X$.  In the case where $W=\symm_n$
    acts on $V=\RR^n$ by permuting coordinates, if one takes $\OOO$ to
    be the $W$-orbit of intersection subspaces of the form
    $x_{i_1} = \cdots = x_{i_k}$, one finds that $\ninv_\OOO(w)=\ninv_k(w)$.

    Again consider the operator $\nu_\OOO$ representing multiplication by 
    $\sum_{w \in W} \ninv_\OOO(w) \cdot w$ within 
    $\ZZ W$ or $\RR W$.  As before, one can show that $\nu_\OOO=\pi^T \pi$ 
    for certain integer matrices $\pi$, and again one such
    choice of a matrix $\pi$ is the transition matrix for a \BHR random walk on $W$.
    In this general context, but when $\OOO$ is taken to be a $W$-orbit of
    {\emphfont codimension one} subspaces (that is, {\deffont hyperplanes})
    \index{hyperplane}%
    one encounters the following surprise, proven in \ref{subsec:rank-one-proof}.

    \begin{Theorem}
      \label{thm:rank-one}
      For any finite irreducible real reflection group $W$, and
      any (transitive) $W$-orbit $\OOO$ of hyperplanes,
      the matrix $\nu_\OOO$ has all its eigenvalues within
      the ring of integers of the unique minimal splitting field for $W$.
      In particular, when $W$ is crystallographic, these
      eigenvalues all lie in $\ZZ$.
    \end{Theorem}

    This result will follow from 
    applying the integrality principle (\ref{prop:integrality-principle})
    together with the discovery of the following (apparently) new 
    family of {\emphfont twisted Gelfand pairs}.  This is proven in \ref{subsec:new-twisted-Gelfand-pair},
    but only via a case-by-case proof.

    \begin{Theorem}
      \label{thm:Gelfand-triple}
      Let $W$ be a finite irreducible real reflection group and let
      $H$ be the reflecting hyperplane for a reflection $s \in W$. 

      Then the linear character $\chi$ of the $W$-centralizer $\Centr_W(s)$ 
      \nomenclature[al]{$\Centr_W(w)$}{centralizer of the element $w$ in the group $W$}%
      given by its action on the line $V/H$ or $H^\perp$ has a multiplicity-free
      induced $W$-representation $\Ind_{\Centr_W(s)}^W \chi$.
       \nomenclature[al]{$\Ind_H^G$}{induction of a representation from the subgroup $H$ to the group $G$}%
    \end{Theorem}

    We mention a further surprise proven via
    \ref{prop:integrality-principle}
    and some standard representation theory of the symmetric group.
    With $W=\symm_n$ acting on $V=\RR^n$ by permuting coordinates,
    for each $k=1,2,\ldots,\lfloor \frac{n}{2} \rfloor$ consider the $W$-orbit
    $\OOO$ of codimension $k$ intersection subspaces of the form
    $$
      \{ x_{i_1}=x_{i_2} \} \cap \{ x_{i_3}=x_{i_4} \} \cap  \cdots 
      \cap \{ x_{i_{2k-1}}=x_{i_{2k}} \},
    $$
    where $\{ i_1, i_2 \}, \ldots, \{ i_{2k-1}, i_{2k} \}$ are $k$
    pairwise disjoint sets of cardinality two.
    Let $\nu_{(2^k,1^{n-2k})}$ denote the operator $\nu_\OOO$ for this
    orbit $\OOO$.

    \begin{Theorem}
       \label{thm:second-family}
       The operators from the family 
       $\{\nu_{(2^k,1^{n-2k})}\}_{k=1,2\ldots,\lfloor \frac{n}{2}\rfloor}$
       pairwise commute, and have only integer eigenvalues.
    \end{Theorem}

    Interestingly, the proof of this given in \ref{sec:second-family} 
    tells us that the non-kernel eigenspaces $V_\lambda$
    in the simultaneous eigenspace decomposition for $\{\nu_{(2^k,1^{n-2k})}\}$
    should be indexed by all number partitions $\lambda$ of $n$, and that $V_\lambda$ carries
    \index{number partition} 
    \index{partition!number} 
    the irreducible $\RR \symm_n$-module indexed by $\lambda$,  but it tells us
    \index{module!irreducible}%
    \index{irreducible!module}%
    very little about the integer eigenvalue for each $\nu_{(2^k,1^{n-2k})}$
    acting on $V_\lambda$.  

    More generally, we can define an operator $\nu_\lambda$ for each
    partition $\lambda = (\lambda_1, \lambda_2, \ldots, \lambda_\ell)$ of
    $n$ by considering the $\symm_n$-orbit of the subspace
    \begin{align*}
    \{x_{1} = x_{2} = \cdots = x_{\lambda_1} \}
    \cap
    \{x_{\lambda_1+1} & = x_{\lambda_1+2} = \cdots = x_{\lambda_1+\lambda_2}\}
    \\
    &
    \cap
    \{x_{\lambda_1+\lambda_2+1} = x_{\lambda_1+\lambda_2+2} = \cdots =
    x_{\lambda_1+\lambda_2+\lambda_3}\}
    \cap
    \cdots.
    \end{align*}
    In light of \ref{thm:original-family-commutativity}
    and \ref{thm:second-family}, it is natural to ask whether
    these operators commute and have integer eigenvalues.
    Our computer explorations led us to conjecture the following,
    which we verified for $1 \leq n \leq 6$.
    \begin{Conjecture}
    Let $\lambda$ and $\gamma$ be distinct partitions of $n$, both different 
    from $(1^n)$ and $(n)$.
    The operators $\nu_\lambda$ and $\nu_\gamma$ commute if and only if 
    they \emph{both} belong to
        $\left\{ \nu_{( k, 1^{n-k} )} : 1 < k < n \right\}$ or
        $\left\{ \nu_{(2^k,1^{n-2k})} : 0 < k \leq \lfloor\frac{n}{2}\rfloor \right\}$.
    Furthermore, $\nu_\lambda$ has integer eigenvalues if and only if
    $\nu_\lambda$ belongs to one of these two families.
    \end{Conjecture}

  \subsection{Outline of the paper}

    We will define and study the operators $\nu_\OOO$ 
    at various levels of generality. 
    \begin{enumerate}
      \item[(H)] For hyperplane arrangements $\AAA$ (see \ref{subsec:arrangements}).
      \item[(L)] For hyperplane arrangements invariant under a (linear) action of
                 a finite group $W$ (see \ref{subsec:equivariant-setting}).
      \item[(R)] For reflection arrangements corresponding to a real reflection group $W$
                 (see \ref{subsec:reflection-group-setting}).
      \item[(W)] For crystallographic reflection groups or, equivalently, Weyl groups $W$.
      \item[(S)] For the symmetric group $\symm_n$ (see \ref{sec:second-family} and 
                 \ref{sec:original-family}). 
    \end{enumerate}
    Different properties of the operators $\nu_\OOO$ manifest themselves
    at different levels of generality.

    In \ref{sec:definitions} we define $\nu_\OOO$ as in (H) for all hyperplane 
    arrangements $\AAA$,
    and prove semidefiniteness by exhibiting a ``square root'' $\pi$
    for which  $\nu_\OOO = \pi^T \pi$. We also explain how $\nu_\OOO$ interacts with
    any finite group $W$ acting on $\AAA$ as in (L).
    We then particularize to case (R), and exhibit a second
    square root $\pi$ that will turn out to be the transition matrix
    for a certain \BHR random walk. 
    The rest of this chapter contains some general reductions
    and principles, such as the {\deffont Fourier transform} reduction
    \index{Fourier transform}%
    for the reflection group case, and an analysis of the Perron-Frobenius eigenspace.
    \index{Perron-Frobenius Theorem}%

    In \ref{sec:rank-one}, we discuss and prove \ref{thm:Gelfand-triple}
    and deduce from it \ref{thm:rank-one}.  We also discuss
    some interesting conjectures that it suggests, and a
    relation to {\deffont linear ordering polytopes}.
    \index{linear ordering polytope}%
    \index{polytope!linear ordering}%

    In \ref{sec:BHR} we review some of the theory of \BHR random walks, with
    features at different levels of generality.  In particular, some
    of the $W$-equivariant theory of the \BHR random walks presented here
    have neither been stated nor proven in the literature 
    in the generality required for the later results,
    so these are discussed in full detail here.  This equivariant theory
    extends to a commuting $\ZZ_2$-action coming from 
    the scalar multiplication operator $-1$.  Whenever $W$ does not already contain
    this scalar $-1$, the $W \times \ZZ_2$-equivariant picture provides extra structure
    in analyzing the eigenspaces of $\nu_\OOO$.  This chapter concludes
    with some useful reformulations of the representations that make up the
    eigenspaces, which are closely related to {\deffont Whitney cohomology}, 
    \index{Whitney cohomology}%
    {\deffont free Lie algebras} and {\deffont higher Lie characters}.
    \index{Lie!algebra}%
    \index{algebra!free Lie}%
    \index{Lie!algebra!free}%
    \index{Lie!character}%
    \index{character!Lie}%
 
    The remainder of the paper focuses on the case (S), that is, reflection
    arrangements of type $A_{n-1}$, where $W=\symm_n$.

    In \ref{sec:second-family} we discuss $\nu_{({2^k,1^{n-2k}})}$
    and prove \ref{thm:second-family}.  As mentioned earlier,
    although the proof predicts the $\RR W$-module structure on the
    simultaneous eigenspaces, it does {\emphfont not} predict the
    eigenvalues themselves.

    In \ref{sec:original-family} we discuss the original
    family of matrices $\{\nu_{(k,1^{n-k})} \}_{k=1,2,\ldots,n}$,
    starting with a proof of \ref{thm:original-family-commutativity}.
    We then proceed to examine their
    simultaneous eigenspaces.
    Here one can take advantage of a block-diagonalization that comes
    from a certain $W$-equivariant
    filtration respected by these operators.  One can also fully analyze the irreducible
    decomposition of the filtration factors using a close connection
    with {\deffont derangements}, {\deffont desarrangements} and the
    \index{derangement}%
    \index{desarrangement}%
    homology of the {\deffont complex of injective words}.  We review this material,
    \index{complex of injective words}%
    including some unpublished results \cite{ReinerWachs2002} of the first author 
    and M. Wachs, and extend this to the $W \times \ZZ_2$-equivariant
    picture mentioned earlier.  Some of this is used to
    piggyback on Uyemura-Reyes's construction of
    the eigenvectors of $\nu_{(n-1,1)}$ within a certain isotypic component;
    we show with no extra work that these are simultaneous eigenvectors 
    for {\emphfont all} of the $\{\nu_{(k,1^{n-k})} \}_{k=1,2,\ldots,n}$.

  \tableofcontents

\section{Defining the operators}
           \label{sec:definitions}

  \subsection{Hyperplane arrangements and definition of $\nu_\OOO$}
    \label{subsec:arrangements}

    We review here some standard notions for arrangements of hyperplanes;
    good references are \cite{OrlikTerao1992} and \cite{Stanley2007}.

    A ({\deffont central}) {\deffont hyperplane arrangement} $\AAA$ in a $d$-dimension real vector space $V$
    \index{hyperplane arrangement}%
    \index{central hyperplane arrangement}%
    \index{arrangement!central}%
    will here mean a finite collection $\{ H \}_{H \in \AAA}$ of codimension 
    one $\RR$-linear subspaces, that is, {\deffont hyperplanes} passing through
    \index{hyperplane}%
    the origin.  

    An intersection $X=H_{i_1} \cap \cdots \cap H_{i_m}$ of some subset of the
    hyperplanes will be called an {\deffont intersection subspace}. The collection of
    \index{intersection subspace}%
    all intersection subspaces, partially ordered by reverse inclusion, is called
    the {\deffont intersection lattice} $\LLL=\LLL(\AAA)$.
    \index{intersection lattice}%
    This turns out to be a {\deffont geometric lattice}
    \index{geometric lattice}%
    (= atomic, upper semimodular lattice), ranked by the rank function
    $r(x)=\dim{V/X}$ 
    \nomenclature[ar]{$r(x)$}{rank function on geometric lattice}%
    with bottom element $\hat{0}= V := \bigcap_{H \in \emptyset} H$, 
    \nomenclature[ar]{$\hat{0}$}{unique minimal element of poset}%
    and a top element $\hat{1}=\bigcap_{H \in \AAA} H$.
    \nomenclature[ar]{$\hat{1}$}{unique maximal element of a poset}%
    We will sometimes assume that $\AAA$ is {\deffont essential}, meaning
    \index{essential hyperplane arrangement}%
    \index{arrangement!essential}%
    that $\bigcap_{H \in \AAA} H =\{0\}$, so that $\LLL$ has rank $d=\dim(V)$.

     For each $X$ in $\LLL$, we will consider the
     {\deffont localized arrangement}
     \index{localized arrangement}%
     \index{arrangement!localized}%
     $$
       \AAA/X:=\{ H / X: H \in \AAA, H \supset X \}
     $$
     \nomenclature[ar]{$\AAA/X$}{localized arrangements of all $H/X$ for $H \in \AAA$}%
     inside the quotient space $V/X$, having intersection lattice 
     $\LLL(\AAA/X) \cong [V,X]$. Here for elements $U_1,U_2 \in \LLL$ we denote by
     $[U_1,U_2]$ the closed interval $\{ U \in \LLL~|~U_1 \leq U \leq U_2 \}$.  
     \nomenclature[ar]{$[X,Y]$}{closed interval between $X$ and $Y$ in a poset}%
     The {\deffont complement} $V \setminus \bigcup_{H \in \AAA} H$ decomposes
     \index{complement of an arrangement}%
     into connected components which are open polyhedral cones $c$, called {\deffont chambers};
     \index{chamber}%
     the set of all chambers will be denoted $\CCC=\CCC(\AAA)$.  

     Given any chamber $c$ in $\CCC$ and any intersection subspace $X$,
     there is a unique chamber $c/X$ in $V/X$ for the localized arrangement
     \nomenclature[ar]{$c/X$}{chamber in the localized arrangement $\AAA/X$ corresponding to the chamber $c$}%
     $\AAA/X$ for which the quotient map $q:V \twoheadrightarrow V/X$
     has $q^{-1}(c/X) \supseteq c$ (see \ref{fig:localized}).  

     \begin{figure}
       \setlength{\unitlength}{0.75cm}
       \begin{picture}(0,0)(5,2)
         \includegraphics[width=0.7\textwidth]{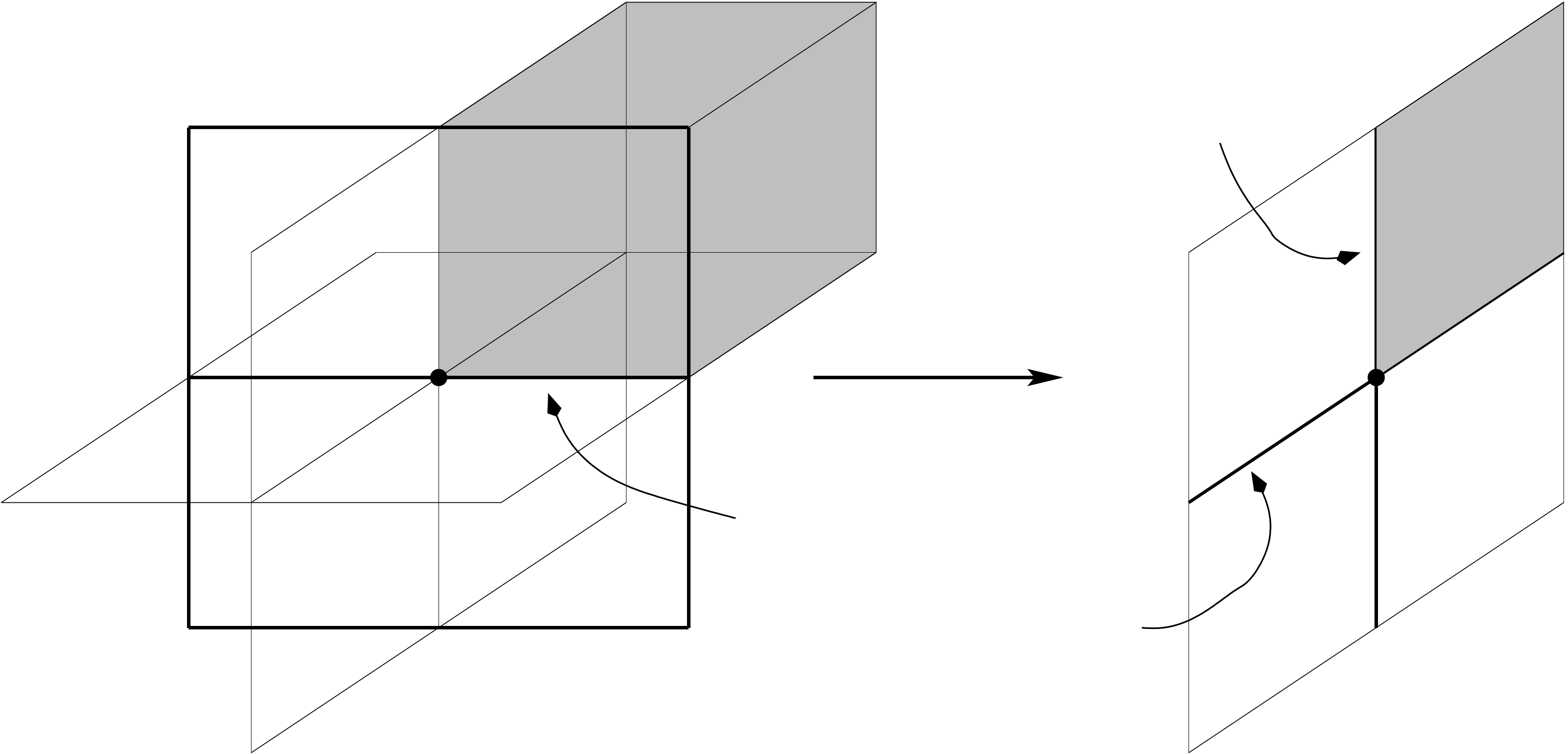}
       \end{picture}%
       \setlength{\unitlength}{4100sp}%
       \begin{picture}(2000,2000)(4000,0)
         \put(2950, 1150){\makebox(0,0)[lb]{\color[rgb]{0,0,0}$H_1$}}
         \put(3150,-0500){\makebox(0,0)[lb]{\color[rgb]{0,0,0}$H_2$}}
         \put(2600, 0180){\makebox(0,0)[lb]{\color[rgb]{0,0,0}$H_3$}}
         \put(4700, 0000){\makebox(0,0)[lb]{\color[rgb]{0,0,0}$X$}}
         \put(7000,-0400){\makebox(0,0)[lb]{\color[rgb]{0,0,0}$\RR^3/X$}}
         \put(6000, 1300){\makebox(0,0)[lb]{\color[rgb]{0,0,0}$H_1/X$}}
         \put(5530,-0350){\makebox(0,0)[lb]{\color[rgb]{0,0,0}$H_3/X$}}
         \put(4750, 1300){\makebox(0,0)[lb]{\color[rgb]{0,0,0}$c$}}
         \put(7000, 0950){\makebox(0,0)[lb]{\color[rgb]{0,0,0}$c/X$}}
         \put(5300, 0600){\makebox(0,0)[lb]{\color[rgb]{0,0,0}$q$}}
       \end{picture}%
       \bigskip
       \bigskip
       \bigskip
       \bigskip
       \caption{Arrangement and its localization}
         \label{fig:localized}
     \end{figure}

     We can now define our main object of study.

     \begin{Definition}
        \label{defn:nu-definition}
        Given two chambers $c, c'$ in $\CCC$, and an intersection subspace 
        $X$ in $\LLL$, say that $X$ is a {\deffont noninversion subspace} for $\{c,c'\}$
        \index{noninversion number}%
        if $c/X = c'/X$.

        Given any subset $\OOO \subseteq \LLL$, define a statistic on (unordered) 
        pairs $\{c,c'\}$ of chambers 
        $$
          \ninv_\OOO(c,c') := \ninv_\OOO(c',c) := \Big|\big\{ X \in \OOO: c/X=c'/X   \big\} \Big|.
          \nomenclature[co]{$\ninv_\OOO(c,c')$}{number of $X \in \OOO$ for which $c/X = c'/X$}%
        $$

        Define the matrix $\nu_\OOO$ in $\ZZ^{\CCC \times \CCC}$ whose 
        $(c,c')$-entry equals $\ninv_\OOO(c,c')$.
        \nomenclature[co]{$\nu_\OOO$}{matrix in $\ZZ^{\CCC \times \CCC}$ whose $(c,c')$-entry equals $\ninv_\OOO(c,c')$}%
        Alternatively, identify $\nu_\OOO$ with the following $\ZZ$-linear operator on the 
        free $\ZZ$-module $\ZZ\CCC$ that has basis indexed by the chambers
        \nomenclature[ar]{$\ZZ\CCC$}{free $\ZZ$-module with basis $\CCC$}%
        $\CCC$:
        \begin{equation}
          \label{eqn:nu-definition}
          \begin{array}{rcl}
            \ZZ\CCC &\overset{\nu_\OOO}\longrightarrow & \ZZ\CCC \\
            c'      &\longmapsto                       & \sum_{c \in \CCC} \ninv_\OOO(c,c')  \cdot c.\\
          \end{array}
        \end{equation}
      \end{Definition}

      Note that since by definition $\ninv_\OOO(c,c') = \ninv_\OOO(c',c)$ it follows that
      $\nu_\OOO$ is a symmetric matrix.

      \begin{Example}
        We consider the arrangement $\AAA = \{ H_1,H_2,H_3 \}$ of the coordinate hyperplanes in $\RR^3$ 
        from \ref{fig:localized}. Chambers $\CCC$ are in bijection with $\{ +1,-1\}^3$, where the 
        image of the chamber is the sign pattern $\epsilon = (\epsilon_1, \epsilon_2, \epsilon_3)$ 
        of any of its points.  

        If $X = H_1 \cap H_3$, then $\AAA/X = \{ H_1 / X, H_3/X \}$. For $c=(+1,+1,+1)$, the chamber $c / X$ in 
        $\RR^3/X \cong \RR^2$ can again be seen as the positive quadrant.
        The only other chamber $c' \in \CCC$ for which $c/ X = c'/X$ is the image $c'=(+1,-1,+1)$ of $c$ 
        reflection through $H_2$. 
      \end{Example}

      \begin{Example}
        \label{ex:first-type-A}
        Let $V=\RR^n$ and $\AAA$ the
        reflection arrangement of type $A_{n-1}$, whose
        hyperplanes are $H_{ij}=\{ x_i = x_j \}_{1 \leq i < j \leq n}$ 
        with the action of $W=\symm_n$ permuting coordinates.

        Intersection subspaces $X$, such as the subspace 
        $\{x_1 = x_3 = x_4, x_2 = x_7\}$ inside $V=\RR^7$,
        correspond to set partitions of the coordinates $[n]:=\{1,2,\ldots,n\}$
        \index{set partition}
        \index{partition!set}
        \nomenclature[co]{$[n]$}{set $\{1,\ldots, n\}$ of the first $n$ natural numbers}%
        into blocks $[n]=\bigsqcup_i B_i$ which indicate which coordinates are equal;
        \nomenclature[co]{$[n] = \bigsqcup_i B_i$}{set partition of $[n]$ with blocks $B_i$}%
        in this example, this set partition is
        $$
          [7]=\{1,3,4\} \sqcup \{2,7\} \sqcup \{5\} \sqcup\{6\}.
        $$
        The intersection lattice $\LLL$ is therefore isomorphic to
        the lattice of set-partitions
        of $[n]$, ordered by refinement, having the discrete (all singleton)
        partition as $\hat{0}$, and the trivial partition with one block
        as $\hat{1}$.

        Chambers in the reflection arrangement of type $A_{n-1}$ are the collections of vectors
        $(x_1, \ldots, x_n) \in \RR^n$ for which $x_{w_1} < x_{w_2} < \cdots < x_{w_n}$ 
        given a fixed $w \in \symm_n$, where $w_i = w(i)$ for $i \in [n]$. We will denote the chamber
        corresponding to a fixed $w$ by $c_w$.

        Given an intersection subspace $X$, corresponding
        to the partition $[n]=\bigsqcup_i B_i$, and a chamber $c_w$,
        the information contained in the chamber $c_w/X$ records for each $i$
        the linear ordering in which the letters of $B_i$ appear as a subsequence
        within $w=(w_1,w_2,\ldots,w_n)$.  Therefore, $c_u/X = c_v/X$ if and only if for each $i$ the
        letters of $B_i$ appear in the same order in both $u$ and $v$.
      \end{Example}

    \subsection{Semidefiniteness}
      \label{subsec:semidefiniteness}

      As explained after \ref{defn:nu-definition} the matrix $\nu_\OOO$ is symmetric and hence the
      corresponding linear operator is self-adjoint with respect to the usual pairing 
      $\langle -,- \rangle$ on $\ZZ\CCC$ that makes the basis vectors $c$
      orthonormal.  It is also positive semidefinite,
      as it has the following easily identified ``square root''.  

      \begin{Definition}
        \label{defn:chamber-localization-maps}
        Consider for each intersection
        subspace $X$ the $\ZZ$-linear map
        $$
          \begin{array}{rcl}
            \ZZ\CCC &\overset{\pi_X}\longrightarrow & \ZZ\CCC(\AAA/X) \\
             c      &\longmapsto                    & c/X \\
         \end{array}
       $$
       and having chosen a subset $\OOO \subseteq \LLL$, 
       consider the direct sum of maps $\pi_\OOO := \bigoplus_{X \in \OOO} \pi_X$
       \nomenclature[co]{$\pi_\OOO$}{rectangular ``square root'' of $\nu_\OOO$}%
       $$
         \ZZ\CCC \longrightarrow  \bigoplus_{X \in \OOO} \ZZ\CCC(\AAA/X)
       $$
    \end{Definition}

    \begin{Proposition}
       \label{prop:rectangular-square-root}
       One has the factorization
       $$
         \nu_\OOO = \pi_\OOO^T \circ \pi_\OOO.
       $$
       In particular, when scalars are extended from $\ZZ$ to $\RR$,
       one has
       $$
         \ker \nu_\OOO = \ker \pi_\OOO.
       $$
    \end{Proposition}
    \begin{proof}
       The $(c,c')$-entry of $\pi_\OOO^T \circ \pi_\OOO$ equals
     \begin{align*}
         &\sum_{X \in \OOO} \,\, \sum_{d \in \CCC(\AAA/X)} 
                 (\pi_{X})_{d,c} (\pi_{X})_{d,c'} \\
         & =\sum_{X \in \OOO} \Big|\{d \in \CCC(\AAA/X): c/X = d = c'/X\}\Big|  \\
         & = \Big|\{X \in \OOO: c/X = c'/X\}\Big|\\
         & = \ninv_\OOO(c,c').
     \qedhere
     \end{align*}
    \end{proof}

  \subsection{Equivariant setting}
    \label{subsec:equivariant-setting}

    Now assume that one has a finite subgroup $W$ of $\GL(V)$
    that preserves the arrangement $\AAA$ in the sense that for
    every $w$ in $W$ and every hyperplane $H$ of $\AAA$, the
    hyperplane $w(H)$ is also in $\AAA$.  Then $W$ permutes
    each of the sets $\AAA, \LLL, \CCC$, and hence acts
    $\ZZ$-linearly on $\ZZ\CCC$.  

    \begin{Proposition}
      If the subset $\OOO \subseteq \LLL$ 
      is also preserved by $W$, then the operator $\nu_\OOO$ on $\ZZ^\CCC$
      is $W$-equivariant.
    \end{Proposition}
    \begin{proof}
      This is straightforward from the observation that
      since $W$ preserves $\OOO$, one 
      has 
      \begin{gather*}
        \ninv_\OOO(c,c') = \ninv_\OOO(w(c),w(c')).
        \qedhere
      \end{gather*}
    \end{proof}

    \begin{Example}
      \label{ex:first-type-A-action}
      We resume \ref{ex:first-type-A} and let $V = \RR^n$ and $\AAA$ the
      reflection arrangement of the symmetric group $W = \symm_n$. Hence the intersection lattice
      $\LLL$ is the lattice of set partitions of $[n]$ ordered by refinement. 
      The group $\symm_n$ acts on $\RR^n$ be permuting coordinates. Thus $w \in \symm_n$ acts on 
      $\LLL$ by sending the set partition $[n]=\bigsqcup_i B_i$ to the set partition
      $[n] = \bigsqcup_i w(B_i)$, where $w(B_i) = \{ w(j)~|~j \in B_i\}$. 
      Therefore, the $\symm_n$-orbits on $\LLL$ are indexed by number partitions $\lambda \vdash n$.
      \nomenclature[co]{$\lambda \vdash n$}{number partition $\lambda$ of $n$}%
      The orbit $\OOO_\lambda$ consist of those intersection subspaces or equivalently set partitions
      of $[n]$ for which the block sizes ordered in decreasing order are the parts of $\lambda$. 
      We call such a set partition a set partition of {\deffont type} $\lambda$.
      \index{type of a partition}%
      \nomenclature[co]{$\OOO_\lambda$}{set partitions of type $\lambda$}%

      For example, for $\lambda = (k,1^{n-k})$ we obtain as $\OOO_\lambda$ the ${n \choose k}$ set 
      partitions of $[n]$ whose only non-singleton block is a block of size $k$.
    \end{Example}

   \subsection{$\ZZ_2$-action and inversions versus noninversions}
      \label{subsec:Z2-action}

      Let $\matid_V$ be matrix of the identity endomorphism of $V$. 
      \nomenclature[al]{$\matid_V$}{matrix of the identity endomorphism of $V$}%
      The scalar matrix $-\matid_V$ acting on $V$ preserves any
      arrangement $\AAA$,
      and hence gives rise to an action of $\ZZ_2=\{1,\tau\}$ in
      which $\tau$ acts by $-\matid_V$.  When one has a subgroup $W$ of
      $\GL(V)$ preserving $\AAA$, since $\tau$ acts by a scalar matrix,
      this $\ZZ_2$-action commutes with the action of $W$,
      giving rise to a $W \times \ZZ_2$-action.  Of course,
      if $W$ already contains the element $-\matid_V$, this provides
      no extra information beyond the $W$-action.  But when $-\matid_V$ is
      not an element of $W$ already, it is worthwhile to 
      consider this extra $\ZZ_2$-action.

      We wish to explain how carrying along this $\ZZ_2$-action
      naturally eliminates a certain choice we have made.  
      Instead of considering the matrix/operator $\nu_\OOO$,
      one might have considered the matrix/operator
      $\iota_\OOO = (\inv_\OOO(c,c'))_{c,c' \in \CCC}$ having entry $\inv_\OOO(c,c')$ 
      defined to be
      \nomenclature[co]{$\iota_\OOO$}{matrix in $\ZZ^{\CCC \times \CCC}$ whose $(c,c')$-entry equals $\inv_\OOO(c,c')$}%
      \nomenclature[co]{$\inv_\OOO(c,c')$}{number of subspace $X$ in $\OOO$ for which $c/X = - c'/X$}%
      the number of subspaces $X$ in $\OOO$ which are 
      {\deffont inversions} for $c,c'$ in the sense that
      \index{inversion}%
      $c/X=-c'/X$.  
      Taking into account the $\ZZ_2$-action
      eliminates the need to consider $\iota_\OOO$ separately:

      \begin{Proposition}
        The two operators $\nu_\OOO$ and $\iota_\OOO$ are
        sent to each other by the generator $\tau$ of the
        $\ZZ_2$-action:
        $$ 
          \iota_\OOO = \tau \circ \nu_\OOO = \nu_\OOO \circ \tau. \qed
        $$
      \end{Proposition}

      Thus if we want to consider the eigenvalues 
      and eigenspaces, it is equivalent to consider
      either $\nu_\OOO$ or $\iota_\OOO$, as long
      as we also keep track of the $\ZZ_2$-action
      on the eigenspaces.  In what follows, we prefer to consider 
      the positive semidefinite operator $\nu_\OOO$
      rather than the indefinite operator $\iota_\OOO$.

      \begin{Example}
         \label{ex:first-type-A-iota}
         We return to the setting of \ref{ex:first-type-A} and \ref{ex:first-type-A-action}. 
         For $\lambda = (k,1^{n-k})$ we had seen that $\OOO_\lambda$ consists of all 
         set partitions of $[n]$ whose unique non-singleton block is of size $k$.
         Thus $X \in \OOO_\lambda$ is uniquely defined by specifying a $k$-subset $B$ of
         $[n]$. 
         Let $u,v \in \symm_n$ with corresponding chambers $c_u$ and $c_v$. 
         From \ref{ex:first-type-A} we know $c_u / X = c_v / X$ if and only if 
         the linear orders defined by $u$ and $v$ coincide on $B$. 
         Since there are ${n \choose k}$ choices for $k$-subsets $B$ we have
         $$\inv_{\OOO_{(k,1^{n-k})}} (c_u,c_v) = {n \choose k} - \ninv_{\OOO_{(k,1^{n-k})}} (c_u,c_v).$$
         In particular, $\inv_{\OOO_{(2,1^{n-2})}} (c_u,c_v)$ is the number of inversions of
         $v^{-1}u$.
      \end{Example}

  \subsection{Real reflection groups}
    \label{subsec:reflection-group-setting}

    We review here some standard facts about real, Euclidean
    finite reflection groups; a good reference is \cite{Humphreys1992}.

    Here we will adopt the convention that an {\deffont (orthogonal) reflection}
    \index{reflection}%
    \index{reflection!orthogonal}%
    in $\GL(V)$ for an $\RR$-vector space $V$ is an
    orthogonal involution $s$ whose fixed subspace $V^s$ is some hyperplane $H$.
    Necessarily, such an element $s$ has $s^2=\matid_V$ and acts by multiplication
    by $-1$ on the line $H^\perp$.  A {\deffont (real) reflection group} $W$
    \index{real reflection group}%
    \index{reflection!group}%
    \index{group!reflection}%
    is a finite subgroup of $\GL(V)$ generated by
    {\deffont reflections}.

    To any reflection group $W$ there is naturally associated its
    {\deffont arrangement of reflecting hyperplanes} $\AAA$, consisting
    \index{arrangement!reflecting hyperplane}%
    \index{reflecting!hyperplane!arrangement}%
    of all hyperplanes $H$ arising as $V^s$ for reflections $s$ in $W$.
    In this situation it is known that the set of chambers $\CCC$ carries
    a {\deffont simply transitive} action of $W$.  Therefore, after making
    \index{simply transitive}%
    a choice of {\deffont fundamental/identity/base chamber} $c_\grid$,
    \index{fundamental!chamber}%
    \index{identity chamber}%
    \index{base chamber}%
    \index{chamber!base}%
    \index{chamber!fundamental}%
    \index{chamber!identity}%
    one can identify the $W$-action on $\ZZ\CCC$
    with the left-regular $W$-action on the group algebra $\ZZ W$:
    \begin{equation}
      \label{eqn:group-algebra-is-chambers}
      \begin{array}{rcl}
         \ZZ W &\longrightarrow &\ZZ\CCC \\
             w &\longmapsto     & c_w:=w(c_\grid).
      \end{array}
    \end{equation}
    Now assume one is given a $W$-stable subset $\OOO \subseteq \LLL$,
    and define the statistic 
    $$
      \begin{aligned}
         \ninv_\OOO(w)&:=\ninv_\OOO(c_\grid,c_w) \\
             &= \ninv_\OOO(c_w,c_\grid) \\
             &= \ninv_\OOO(w(c_\grid),w(c_{w^{-1}})) \\
             &= \ninv_\OOO(c_\grid,c_{w^{-1}}) \\
             &= \ninv_\OOO(w^{-1}).
      \end{aligned}
    $$

    \begin{Proposition}
      For any $W$-stable subset $\OOO \subseteq \LLL$,
      under the identification \ref{eqn:group-algebra-is-chambers}, the
      operator $\nu_\OOO$ acts on $\ZZ W$ as right-multiplication by the element
      $$
        \sum_{w \in W} \ninv_\OOO(w) \cdot w.
      $$
    \end{Proposition}

    \begin{proof}
      Within the group algebra, for any basis element $v$ in $W$, one has
      \begin{align*}
         v \cdot \left( \sum_{w \in W} \ninv_\OOO(w) \cdot w \right) 
            &=\sum_{w \in W} \ninv_\OOO(w) \cdot vw \\
            &=\sum_{u \in W} \ninv_\OOO(v^{-1}u) \cdot u \\
            &=\sum_{u \in W} \ninv_\OOO(c_{v^{-1}u}, c_\grid) \cdot u \\
            &=\sum_{u \in W} \ninv_\OOO(c_u, c_v) \cdot u.
            \qedhere
      \end{align*}
    \end{proof}

    By abuse of notation, we will use $\nu_\OOO$ also to denote the 
    the element $\sum_{w \in W} \ninv_\OOO(w) \cdot w$ of $\CC W$.
    
    When $W$ is a real reflection group, the 
    $\ZZ_2$-action corresponds to the 
    action of the {\deffont longest element} $w_0$
    \index{longest element}%
    in $W$, defined uniquely by the property 
    that 
    \begin{equation}
      \label{eqn:longest-element-definition}
         c_{w_0} = -c_\grid,
    \end{equation}
    where $c_{w_0} = w_0(c_\grid)$.
    Note that this forces $w_0$ to always be an 
    {\deffont involution}: $w_0^2=\grid$.
    \index{involution}%

    \begin{Proposition}
       For any $W$-stable subset $\OOO \subseteq \LLL$,
       under the identification \ref{eqn:group-algebra-is-chambers}, the
       scalar matrix $-\matid_V$ or the generator $\tau$ of the $\ZZ_2$-action
       on $\ZZ\CCC$ acts on $\ZZ W$ as right-multiplication by $w_0$.
    \end{Proposition}
    \begin{proof}
       Applying $w$ on the left of \ref{eqn:longest-element-definition} gives
       $$
         c_{ww_0} = ww_0(c_\grid) = -w(c_\grid) = -c_w
       $$
       for any $w$ in $W$.
    \end{proof}

    It is known that $-\matid_V$ is an element of a reflection group $W$
    acting on $V$ if and only if $W$ only has {\deffont even degrees} 
    \index{degree of reflection group}%
    $d_1,\ldots,d_n$ for any system of basic invariants $f_1,\ldots,f_n$ 
    that generate the $W$-invariant polynomials $\CC[V]^W=\CC[f_1,\ldots,f_n]$.
    If $-\matid_V$ is an element of $W$, then necessarily $-\matid_V=w_0$.

    For the {\deffont irreducible} real reflection groups, one
    \index{irreducible!real reflection group}%
    \index{real reflection group!irreducible}%
    has 
    \begin{enumerate}
      \item[$\bullet$] $-\matid_V=w_0$ in types $B_n=C_n$, 
         in type $D_n$ when $n$ is even, 
         in the dihedral types $I_2(m)$ for $m$ even, 
         as well as the exceptional types $F_4, E_7, E_8, H_3, H_4$.  
      \item[$\bullet$] $-\matid_V \not\in W$ in all other cases,
         that is, in type $A_{n-1}$, 
         in type $D_n$ for $n$ odd, 
         in dihedral types $I_2(m)$ for $m$ odd, 
         and in type $E_6$.
         Thus, in these cases there is an extra $\ZZ_2$-action
         to consider.
    \end{enumerate}

    \begin{Example}
      Again we return to \ref{ex:first-type-A}.
      The longest word $w_0 \in \symm_n$ corresponds to the permutation $n~n-1 \cdots 2~1$.
      Thus, multiplication by $w_0$ on the right sends a permutation $w(1) \cdots w(n) \in \symm_n$
      to the permutation $w(n) \cdots w(1)$.
    \end{Example}

    Given an intersection subspace $X$,
    denote by $\Norma_W(X)$ and $\Centr_W(X)$, respectively, its not-necessarily-pointwise
    stabilizer subgroup and pointwise stabilizer subgroup within $W$:
    \begin{align*}
    & \Norma_W(X) = \{w \in W : w(X) = X\}, \\
    & \Centr_W(X) = \{w \in W : w(x) = x \text{ for all } x \in X\}.
    \end{align*}
    \nomenclature[nwx]{$\Norma_W(X)$}{stabilizer subgroup (not necessarily pointwise) within $W$ of the subspace $X$}%
    \nomenclature[zwx]{$\Centr_W(X)$}{pointwise stabilizer subgroup within $W$ of the subspace $X$}%
    It is well-known (see for example \cite[Lemma 3.75]{AbramenkoBrown2008}) that $\Centr_W(X)$ is itself a finite real reflection
    group, called the {\deffont parabolic subgroup} associated to $X$,
    \index{parabolic subgroup}%
    which one can view as acting on the quotient space $V/X$,
    and having reflection arrangement equal to the
    localization $\AAA/X$.  Consequently, the chambers $\CCC(\AAA/X)$
    are in natural bijection with $\Centr_W(X)$.  This gives the following
    interpretation to the map $c \longmapsto c/X$ that we have been
    using.

    \begin{Proposition}
      \label{prop:parabolic-factorization}
      Let $W$ be a finite real reflection group $W$, and $X$ an intersection subspace 
      in $\LLL$.  Then every $w$ in $W$ factors uniquely as
      $w = z \cdot y$ where $z$ lies in $\Centr_W(X)$ and $y$ lies in 
      $$
        {}^XW:=\{ y \in W: c_y/X = c_\grid/X \}.
      $$
      In particular, the map $\pi_X: \CCC \rightarrow \CCC(\AAA/X)$ 
      sending $c \longmapsto c/X$ corresponds under \eqref{eqn:group-algebra-is-chambers}
      to the map sending $w \longmapsto z$.
    \end{Proposition}
    \begin{proof}
      Given $w$ in $W$, consider the chamber $c_w/X$ in the localized
      arrangement $\AAA/X$.  Since this localized arrangement is the reflection
      arrangement for $\Centr_W(X)$, there is a unique element $z$ in $\Centr_W(X)$
      for which $c_z/X=c_w/X$.  In particular, the element $z \in \Centr_W(X)$ acts on the chambers of
      $\AAA$ and on the chambers of $\AAA/X$. Thus
      $y:=z^{-1} \cdot w$ satisfies
      $$
        \begin{aligned}
          c_y/X &= c_{z^{-1} w}/X \\
                &= (z^{-1} c_{w})/(z^{-1}X) \\
                &= z^{-1} (c_{w}/ X) \\
                &= z^{-1} (c_{z}/ X) \\
                &= c_{\grid}/z^{-1} X \\
                &= c_{\grid}/X 
        \end{aligned}      
      $$
      that is, $y$ lies in ${}^XW$. 
    \end{proof}

    In the sequel, for a finite real reflection group $W$ and an intersection subspace
    $X$ in $\LLL$, we denote by ${}^XW$ the set $\{ y \in W : c_y/X = c_\grid/X \}$, which by the preceding 
    proposition is a set of right coset representatives of $\Centr_W(X)$ in $W$.
    \nomenclature[al]{${}^XW$}{set of right coset representative for $\Centr_W(X)$ in $W$}%
    We write $W^X=\{w^{-1}: w \in {}^XW\}$ for the corresponding set of left coset representatives.
    \nomenclature[al]{$W^X$}{$\{w^{-1}~|~w \in {}^XW \}$}%
    If $X$ intersects the identity chamber $c_\grid$ then $\Centr_W(X) = W_J = \langle J \rangle$ 
    for some subsets $J \subseteq S$.  Here $W_J$ is a (standard) 
    parabolic subgroup for the Coxeter system $(W,S)$
    that generates $W$ using the set $S$ of reflections 
    through the walls of $c_\grid$.
    \index{Coxeter system}%
    In this case we also write ${}^JW$ for ${}^XW$ and
    $W^J$ for $W^X$ respectively. The ${}^JW$ and $W^J$ are sets of minimal length
    \nomenclature[al]{$W^J$}{Minimal length right coset representatives for the parabolic subgroup $W_J$}%
    \nomenclature[al]{${}^JW$}{Minimal length right coset representatives for the parabolic subgroup $W_J$}%
    left and right coset representatives for $W_J$.

    \begin{Example}
       Returning to \ref{ex:first-type-A} and \ref{ex:first-type-A-action}, 
       where $W=\symm_n$ acts on $V=\RR^n$ by
       permuting coordinates and $X$ is the intersection subspace corresponding
       to the partition $[n]=\bigsqcup_i B_i$, the centralizer $\Centr_W(X)$
       is the {\deffont Young subgroup} $\prod_i \symm_{B_i}$ that permutes
       \index{Young subgroup}%
       \nomenclature[al]{$\symm_B$}{symmetric group of all permutations of the set $B$}%
       each block $B_i$ of coordinates separately.  The map $W \longmapsto \Centr_W(X)$
       that sends $w \mapsto z$ corresponding to $c \longmapsto c/X$
       remembers only the ordering of the coordinates {\emphfont within each block} $B_i$.
    \end{Example}

  \subsection{The case where $\OOO$ is a single $W$-orbit}

    When $W$ is a real reflection group, and 
    $\OOO= X^W := \{ w\cdot X : w \in W\}$ is the $W$-orbit of some intersection subspace $X$,
    there are two extra features that will help us to analyze the
    eigenspaces of $\nu_\OOO$.

    \subsubsection{A second square root}
      \label{subsec:second-square-root}
      First, there is another ``square root'' for
      $\nu_\OOO$  when $W$ is the
      orbit $X_0^W$ of a single subspace $X_0$.  This will 
      connect $\nu_\OOO$ with the \BHR random walks in \ref{sec:BHR}.
      Given an intersection subspace $X$, with the
      associated subgroups $\Centr_W(X) \subseteq \Norma_W(X)$ we have introduced in
      \ref{prop:parabolic-factorization} and subsequent definitions
      the parabolic factorizations and coset representatives
      $$
        \begin{aligned}
          W &=\Centr_W(X) \cdot {}^XW \\
          W &=W^X \cdot \Centr_W(X).
        \end{aligned}
      $$
      Define
      $$
        \begin{aligned}
          n_{X} &:= [\Norma_W(X):\Centr_W(X)]\\
           {}^XR&:=\sum_{u \in {}^XW} u\\
             R^X&:=\sum_{u \in W^X} u.
        \end{aligned}
      $$
      \nomenclature[al]{$n_X$}{$[\Norma_W(X):\Centr_W(X)]$}%
      \nomenclature[al]{${}^XR$}{$\sum_{u \in {}^XW} u \in \CC W$}%
      \nomenclature[al]{$R^X$}{$\sum_{u \in W^X} u \in \CC W$}%

      For later use and analogous to our previous convention we write ${}^JR$ and
      $R^J$ in case $X$ lies in the boundary of the identity chamber
      $c_\grid$ and $\Centr_W(X) = W_J$ is a (standard) parabolic subgroup.

      \begin{Proposition}
         \label{prop:second-square-root}
         Let $W$ be a real reflection group and $\OOO= X_0^W \subset \LLL$
         the $W$-orbit of the intersection subspace $X_0$. Then
         $$
           \ninv_\OOO(w) = \frac{1}{n_{X_0}} \Big| {}^{X_0}W \cap {}^{X_0}Ww \Big|
         $$
         and
         $$
           \nu_\OOO = \frac{1}{n_{X_0}}  R^{X_0} \cdot  {}^{X_0}R.
         $$
      \end{Proposition}
      \begin{proof}
         Since $\OOO$ is the $W$-orbit of $X_0$, and
         $\Norma_W(X_0)$ the $W$-stabilizer of $X_0$, the elements
         $u \cdot X_0$ as $u$ runs over coset representatives for $W/\Norma_W(X_0)$
         give each $X$ in $\OOO$ exactly once.  Therefore, 
         the elements $uX_0$ as $u$ runs over the coset representatives $W^{X_0}$
         of $W/\Centr_W(X_0)$ give each $X$ in $\OOO$ exactly 
         $n_{X_0}=[\Norma_W(X_0):Z_W(X_0)]$ times.
         Since 
         $$
           \ninv_\OOO(w) = \left| \{X \in \OOO: w \in {}^XW\} \right|
         $$
         this implies that 
         $$
           n_{X_0} \cdot \ninv_\OOO(w) = \left| \{u \in W^{X_0}: w \in {}^{uX_0}W \} \right|.
         $$
         We wish to rewrite the set appearing on the right side of
         this equation. Note that $u$ lies in $W^{X_0}$ if and only if 
         $u^{-1}$ lies in ${}^{X_0}W$ if and only if $c_{u^{-1}}/X_0 = c_\grid/X_0$.
         Similarly, $w$ lies in ${}^{uX_0}W$ if and only if
         $c_{w}/uX_0 = c_\grid/uX_0$ if and only if
         $c_{u^{-1}w}/X_0 = c_{u^{-1}}/X_0 = c_\grid/X_0$ if and only if
         $u^{-1}w$ lies in ${}^{X_0}W$.  Letting $v=u^{-1} w$, one
         concludes that $v$ lies in both ${}^{X_0}W$ and in ${}^{X_0}Ww$,
         so that 
         $$
           n_{X_0} \cdot \ninv_\OOO(w) = \Big| {}^{X_0}W \cap {}^{X_0}Ww \Big|.
         $$
         This proves the first assertion.  For the second assertion,
         compare with the calculation
           \begin{align*}
              R^{X_0} \cdot  {}^{X_0}R 
                 &= \left( \sum_{u \in {}^{X_0}W} u^{-1} \right)
                    \left( \sum_{v \in {}^{X_0}W} v \right) \\
                 &= \sum_{w \in W} w \cdot 
               \Big|\left\{ (u,v) \in {}^{X_0}W \times {}^{X_0}W : u^{-1}v=w \right\}\Big| \\
                 &= \sum_{w \in W} w \cdot 
               \Big|\left\{v \in {}^{X_0}W \cap {}^{X_0}Ww \right\}\Big| \\
                 &= \sum_{w \in W} w \cdot  \Big| {}^{X_0}W \cap {}^{X_0}Ww \Big|.
               \qedhere
           \end{align*}
      \end{proof}

    \subsubsection{Nested kernels}
      Second, there is an inclusion of kernels
      $\ker \nu_{\OOO} \subseteq \ker \nu_{\OOO'}$ 
      whenever $\OOO, \OOO'$ are $W$-orbits represented by 
      two nested subspaces $X \subseteq X'$.
      To see this, define in the general setting of hyperplane arrangements a map
      $$
        \pi^{\OOO}_{\OOO'}:
            \bigoplus_{X \in \OOO} \CCC(\AAA/X) 
               \rightarrow \bigoplus_{X' \in \OOO} \CCC(\AAA/X')
      $$
      as a direct sum of the natural maps 
      $$
        \begin{aligned}
           \pi^X_{X'}  : \CCC(\AAA/X) &  \longrightarrow & \CCC(\AAA/X') \\
                            c/X       &  \longmapsto     & c/X'
        \end{aligned}
      $$
      indexed by pairs of subspaces $(X,X') \in \OOO \times \OOO'$ for
      which $X \subseteq X'$.  Given $X' \in \OOO'$, define an integer
      $c_{\OOO,X'}$ to be the number of $X \in \OOO$ for which $X \subseteq X'$.

      \begin{Proposition}
        \label{prop:nested-kernels}
        Let $\AAA$ be an arrangement with some linear symmetries $W$,
        and let $\OOO, \OOO'$ be two $W$-orbits within $\LLL$ represented by 
        two nested subspaces.

        Then the integers $c_{\OOO,X'}$ do not depend upon the choice
        of $X'$ within $\OOO'$, and denoting this common integer $c_{\OOO,\OOO'}$ 
        one has
        \begin{equation}
          \label{eqn:nested-kernel-equation}
          c_{\OOO,\OOO'} \cdot \pi_{\OOO'} = \pi^{\OOO}_{\OOO'} \circ \pi_{\OOO}.
       \end{equation}
       Consequently,
       $$
         \begin{array}{ccc}
            \ker \pi_\OOO & \subseteq & \ker \pi_{\OOO'} \\
            \Vert         &           & \Vert \\
            \ker \nu_\OOO &           & \ker \nu_{\OOO'} \\
         \end{array}
       $$
     \end{Proposition}

     \begin{proof}
       Because $\pi_{X'} = \pi^X_{X'} \circ \pi_X$, one
       has generally that
       $$
         \pi^{\OOO}_{\OOO'} \circ \pi_{\OOO}
           = \sum_{X' \in \OOO'} c_{\OOO,X'} \,\, \pi_{X'}.
       $$
       However, whenever $\OOO, \OOO'$ are $W$-orbits,
       if $X',X''$ are subspaces in the same $W$-orbit $\OOO'$, say with 
       $w \cdot X' = X''$, then the element $w$ gives a bijection
       between the two sets counted by $c_{\OOO,X'}, c_{\OOO,X''}$.
       Thus $c_{\OOO,\OOO'} :=c_{\OOO,X'}$ satisfies
       \begin{gather*}
         \pi^{\OOO}_{\OOO'} \circ \pi_{\OOO}
         = c_{\OOO,\OOO'} \sum_{X' \in \OOO'} \pi_{X'} 
         = c_{\OOO,\OOO'} \cdot \pi_{\OOO'}. 
         \qedhere
       \end{gather*}
     \end{proof}

     \begin{Example}
      \label{ex:kernel-nesting}
      We again consider the setting of \ref{ex:first-type-A}, \ref{ex:first-type-A-action}
      and the partitions $\lambda = (k,1^{n-k})$, $1 \leq k \leq n$. Then for each
      $1 \leq k < k' \leq n$ and each subspace $X \in \OOO_{(k,1^{n-k})}$ there is
      a subspace $X' \in \OOO_{(k',1^{n-k'})}$ for which $X' \subseteq X$.
      Thus \ref{prop:nested-kernels} applies, and we will take advantage of
      the nesting $\ker \nu_{(k',(1^{n-k'})} \subset \ker \nu_{(k,1^{n-k})}$
      in \ref{subsec:kernelfiltration}.
     \end{Example}

  \subsection{The ``Fourier transform'' reduction}
    \label{subsec:Fourier-transform}

    When $W$ is a real reflection group, the fact that 
    we are considering operators which are right-multiplication
    on the group algebra $\ZZ W$ by elements of $\ZZ W$ 
    allows us to take advantage of a standard trick
    for partially block-diagonalizing $\nu_\OOO$.  This trick sometimes
    goes by the name of the ``Fourier transform''.

    For each irreducible complex $W$-character $\chi$, choose
    a representation $\rho_\chi: W \rightarrow \GL_\CC(U^\chi)$
    \nomenclature[al]{$\rho_\chi$}{representation with character $\chi$}%
    affording the character $\chi$, in some complex vector space 
    $U^\chi$ of dimension $d_\chi:=\chi(1)$.
    \nomenclature[al]{$d_\chi$}{degree of character $\chi$}%
    Then the ring map 
    $
     \CC W \longrightarrow \bigoplus_\chi \End_\CC(U^\chi)
    $
    defined $\CC$-linearly by sending $w \longmapsto \bigoplus_\chi \rho_\chi(w)$,
    is well-known to be an algebra isomorphism. Furthermore,
    the direct summand $\End_\CC(U^\chi)$ is isomorphic to
    the algebra of $d_\chi \times d_\chi$ matrices.
    Thus one can view this as a change-of-basis in $\CC W$
    that simultaneously block-diagonalizes the commuting actions of
    $\CC W$ on the left and on the right.
    Also, as a (left-) $\CC W$-module the summand 
    $\End_\CC(U^\chi)$ is $\chi$-isotypic, 
    carrying $d_\chi$ copies of the irreducible $\chi$.

    Restricting the action of $\CC W$ on the right of
    the summand $\End_\CC(U^\chi)$ to the elements $\nu_\OOO$ and $w_0$, one has 
    the commuting left-action of $\CC W$ and the right-action of
    the elements $\rho_\chi(\nu_\OOO)$ and $\rho_\chi(w_0)$
    inside $\End_\CC(U^\chi)$.  

    Now identify $\End_\CC(U^\chi)$
    with $d_\chi \times d_\chi$ matrices by choosing 
    for $U^\chi$ a basis of simultaneous 
    eigenvectors $\{v_i\}_{i=1,2,\ldots,d_\chi}$
    for the action of $\rho_\chi(\nu_\OOO)$ and
    the commuting involution $\rho_\chi(w_0)$.
    One then finds that the subspace of matrices supported only in column $i$ 
    form an irreducible $W$-module affording the character $\chi$.

    This proves the following.

    \begin{Proposition} \label{Fourier-transform-prop}
      Let $W$ be a real reflection group and assume one has a $W$-stable
      subset $\OOO \subseteq \LLL$.  Let $\chi$ be a
      complex irreducible $W$-character, $\lambda \in \RR$, and $\epsilon \in \{\pm 1\}$.

      The number of copies of $\chi$ occurring in  
      $$\ker (\nu_\OOO - \lambda \matid_V) \cap  \ker (\tau - \epsilon \matid_V)$$
      equals the dimension of 
      $$\ker (\rho_\chi(\nu_\OOO) - \lambda \matid_{U^\chi}) \cap \ker (\rho_\chi(w_0) - \epsilon \matid_{U^\chi}).$$

      In particular, if $\lambda$ is an eigenvalue of $\nu_\OOO$ and $\epsilon$
      an eigenvalue of $\tau$, then the number of copies of $\chi$
      occurring in the $\lambda$-eigenspace for $\nu_\OOO$ intersected
      with the $\epsilon$-eigenspace for $\tau$ is the same
      as the dimension of the $\lambda$-eigenspace for
      $\rho_\chi(\nu_\OOO)$ intersected
      with the $\epsilon$-eigenspace for $\rho_\chi(w_0)$.
    \end{Proposition}

    As a very special case of this, when $\chi$ is a 
    {\deffont degree one} or {\deffont linear} character of $W$,
    \index{degree of a character}%
    \index{linear character}%
    \index{character}%
    one can be much more precise.  

    \begin{Proposition} \label{degree-one-character-prop}
      For any degree one character $\chi$ of $W$  and any $W$-stable
      subset $\OOO \subseteq \LLL$, multiples of the $\chi$-idempotent 
      $$
        \idem_\chi : = \frac{1}{|W|} \sum_{w \in W} \chi(w) \cdot w
                 = \frac{1}{|W|} \sum_{w \in W} \chi(w^{-1}) \cdot w
     $$
     \nomenclature[al]{$\idem_\chi$}{idempotent for character $\chi$ in $\QQ W$}%
     in $\QQ W$ are eigenvectors for $\nu_\OOO$, with integer eigenvalue
     $$
       \begin{aligned}
          \lambda_\OOO(\chi) & := & \sum_{w \in W} \ninv_\OOO(w) \chi(w) \\
                    & =  & \sum_{X \in \OOO} 
                              \sum_{\substack{w \in W:\\ c_w/X=c_\grid/X}} \chi(w).
       \end{aligned}
     \nomenclature[al]{$\lambda_\OOO(\chi)$}{eigenvalue  Idempotent for character $\chi$ in $\QQ W$}%
     $$
     In particular, the trivial character $\trivial$ gives
     \nomenclature[al]{$\trivial$}{trivial character}%
     rise to an all positive eigenvector $\idem_\trivial= \frac{1}{|W|} \sum_{w \in W} w$, having
     eigenvalue 
     $$
       \lambda_\OOO(\trivial)= \sum_i \Big( [W:\Norma_W(X_i)] \cdot [W: \Centr_W(X_i)] \Big) 
     $$
     where $\{X_i\}$ is any set of representatives 
     for the $W$-orbits within $\OOO$.
   \end{Proposition}

   \begin{proof}
      First note that since a reflection group $W$ is generated by involutions,
      any degree one character $\chi$ takes values in $\{\pm 1\}$ and satisfies
      $\chi(w^{-1})=\chi(w)$.
      Now check the eigenvalue equation:
      \begin{eqnarray*} 
        |W| \idem_\chi \cdot \nu_\OOO & = & \left(\sum_{u \in W} \chi(u) \cdot u \right) 
                                      \left( \sum_{v \in W} \ninv_\OOO(v)\cdot v \right) \\
                                  & = & \sum_{u \in W} \sum_{v \in W} 
                                           \chi(u) \ninv_\OOO(v) \cdot u v \\
                                  & = & \sum_{w \in W} w \left( \sum_{v \in W} 
                                           \chi(w v^{-1}) \ninv_\OOO(v) \right)\\
                                  & = & \left( \sum_{w \in W} \chi(w) w \right) \left( \sum_{v \in W} 
                                           \chi(v^{-1}) \ninv_\OOO(v) \right) \\
                                  & = & \lambda_\OOO(\chi) \left( |W| \idem_\chi \right)
      \end{eqnarray*}
      One can also rewrite
      $$
        \begin{aligned}
          \lambda_\OOO(\chi) & =\sum_{w \in W} \ninv_\OOO(w) \chi(w) \\
                             & =\sum_{w \in W} 
                                  \sum_{\substack{X \in \OOO:\\ c_w/X = c_\grid/X}} \chi(w) \\
                             & =\sum_{X \in \OOO} 
                                  \sum_{\substack{w \in W:\\ c_w/X = c_\grid/X}} \chi(w) \\
        \end{aligned}
      $$
      Lastly, when $\chi=\trivial$ one has
      $$
        \begin{aligned}
          \lambda_\OOO(\chi) & =\sum_{X \in \OOO} \left|\left\{w \in W: c_w/X = c_\grid/X\right\}\right| \\
                      & =\sum_i \sum_{X'_i \in W \cdot X_i} \left|\left\{w \in W: c_w/X'_i = c_\grid/X'_i\right\}\right| \\
                      & =\sum_i [W:\Norma_W(X_i)] [W:\Centr_W(X_i)]
        \end{aligned}
      $$
     where the last equality uses both the fact that $|W \cdot X_i|=[W:\Norma_W(X_i)]$
     and that \ref{prop:parabolic-factorization} tells us that the elements 
     from ${}^XW = \{w \in W: c_w/X = c_\grid/X\}$ form a set of coset representatives
     for $W/\Centr_W(X)$. 
   \end{proof}

   \begin{Example}
     We return to the setting of \ref{ex:first-type-A} with
     $W = \symm_n$ acting on $V=\RR^n$, 
     and $\OOO=\OOO_{(k,1^n-k)}$.
     There are two
     degree one characters of $W$, namely the {\deffont trivial} character $\trivial$,
     \index{trivial character}%
     and the {\deffont sign} character $\sgn$.  Since a representative
     \index{sign character}%
     \nomenclature[al]{$\sgn$}{sign character of symmetric group}%
     subspace $x_1 = x_2 = \cdots =x_k$ in $\OOO$ has 
     $\Norma_W(X)= \symm_k \times \symm_{n-k}$
     and  $\Centr_W(X)=\symm_k$, for the trivial character $\trivial$ one finds that 
     $$
       \lambda_\OOO(\trivial) = [W:\Norma_W(X)] [W:\Centr_W(X)] 
                   = \frac{n!}{k! (n-k)!} \cdot \frac{n!}{k!}
                   = \binom{n}{k}^2 (n-k)!.
     $$
     For the sign character $\sgn$ one finds that
     $$
       \begin{aligned}
          \lambda_{\OOO}(\sgn) 
          &=
           \sum_{X \in \OOO} \quad \sum_{w \in W: c_w/X = c_\grid/X} \sgn(w)\\
          &=\sum_{1 \leq i_1 < \cdots < i_k \leq n}  \qquad 
           \sum_{\substack{w \in W:\\ \{i_1,\ldots,i_k\} \text{ appear } \\
                              \text{ left-to-right in }w}} \sgn(w)\\
          &=\begin{cases}
              1 & \text{ if }k=n \\
              1 & \text{ if }k=n-1 \text{ and }n\text{ is odd,} \\
              0 & \text{ if }k=n-1 \text{ and }n\text{ is even,} \\
              0 & \text{ if }1 \leq k\leq n-2,
            \end{cases}
       \end{aligned}
     $$
     for the following reasons.  

     When $k=n$ this is because there is
     only one term in the outer sum, and the inner sum contains only $w=\grid$.

     When $1 \leq k \leq n-2$, picking any pair $\{i,j\}$ in the complement 
     $[n]\setminus \{i_1,\ldots,i_k\}$ gives rise to a sign-reversing
     involution $w \leftrightarrow (i,j) \cdot w$, which shows that the
     inner sum vanishes.   

     When $k = n-1$, this calculation 
     appears as \cite[Proposition~5.3]{UyemuraReyes2002}.  Each term
     in the outer sum is determined by the index $i$ in the complement
     $[n]\setminus \{i_1,\ldots,i_k\}$, and each $w$ in the inner 
     sum determined by the position $j$ where $i$ appears in $w$,
     that is, $j=w^{-1}(i)$.  Hence the result is 
     $$
       \sum_{i=1}^n \sum_{j=1}^n (-1)^{i-j} = \left( \sum_{i=1}^n (-1)^i \right)^2
     $$
     which is $1$ for $n$ odd and $0$ for $n$ even.
   \end{Example}

 \subsection{Perron-Frobenius and primitivity}
   \label{subsec:Perron-Frobenius}

   Since the matrices representing the $\nu_\OOO$ have nonnegative
   entries, and since the trivial idempotent $\idem_\trivial$ gives
   an eigenvector with all positive entries, one might wish to 
   apply Perron-Frobenius theory (see e.g. \cite[Theorem 8.4.4]{HornJohnson1985})
   \index{Perron-Frobenius Theorem}%
   to conclude that the eigenspace spanned by $\idem_\trivial$ is simple.   
   This is true in the cases of most interest to us, but we must
   first deal with a degenerate case that can occur when the reflection
   group $W$ does not act irreducibly.

   Recall that for any finite reflection group $W$ acting on the real
   vector space $V$, one can always
   decompose $W=\prod_{i=1}^t W^{(i)}$ and find an orthogonal decomposition
    $V =\bigoplus_{i=1}^t V^{(i)}$
    such that each $W^{(i)}$ acts as a reflection group {\deffont irreducibly}
    \index{irreducible!real reflection group}%
    \index{real reflection group!irreducible}%
    on $V^{(i)}$.  In this situation, one has a disjoint decomposition
    of the arrangement of reflecting hyperplanes $\AAA = \bigsqcup_{i=1}^t \AAA^{(i)}$.

    \begin{Example}
       \label{ex:imprimitive-example}
       Let $W$ be of type $A_1 \times A_1$, that is, 
       the reflection group isomorphic to $\ZZ_2 \times \ZZ_2$
       acting on $V=\RR^2$ generated by two commuting reflections 
       $s_1, s_2$ through perpendicular hyperplanes $H_1, H_2$
       (lines, in this case). Thus $W = W^{(1)} \times W^{(2)}$ 
       where $W^{(i)} = \{\grid,s_i\}$. Choose $\OOO=\{H_1\}$.  Then one finds that
       $$
         \begin{tabular}{cc}
            \toprule
            $w$                 & $\ninv_{\OOO}(w)$ \\ \midrule
            $\grid$             & $1$ \\ \midrule
            $s_1$               & $0$ \\ \midrule
            $s_2$               & $1$ \\ \midrule
            $s_1s_2=s_2s_1=w_0$ & $0$ \\ \bottomrule
         \end{tabular}
       $$
       so that as an element of $\ZZ W$, one has 
       $
        \nu_\OOO = \grid + s_2
       $
       whose action on $\ZZ W$ on the right can be expressed 
       in matrix form with respect to the ordered basis $(\grid,s_1,s_2,w_0)$ as
       $$
         \left[
           \begin{matrix}
             1 & 0 & 1 & 0 \\
             0 & 1 & 0 & 1 \\
             1 & 0 & 1 & 0 \\
             0 & 1 & 0 & 1 \\
           \end{matrix}
         \right].
       $$
       Even though this matrix is nonnegative, it is imprimitive 
       in the sense that no power of it will have all strictly positive
       entries.  Thus one cannot apply the simplest version of
       the Perron-Frobenius theorem.   
       However, under the identification $\ZZ W \cong \ZZ W_1 \otimes_\ZZ \ZZ W_2$ 
       one has
       $$
         \nu_\OOO = (1\cdot \grid + 0 \cdot s_1) \otimes (1 \cdot \grid + 1 \cdot s_2).
       $$
       and correspondingly the above matrix can be rewritten as
       $$
         \left[
           \begin{matrix}
              1 & 0 \\
              0 & 1 
           \end{matrix}
         \right]
         \otimes
         \left[
           \begin{matrix}
             1 & 1 \\
             1 & 1 
           \end{matrix}
         \right].
       $$
       Note that this second tensor factor {\emphfont is} a primitive matrix, to which
       Perron-Frobenius does apply.
     \end{Example}

     The following proposition can be proven in a completely straightforward
     fashion.

     \begin{Proposition}
       \label{prop:reducibility-reduction1}
       Let $W$ be a finite real reflection group and $W = \prod_{i=1}^t W^{(i)}$ for
       irreducible reflection groups $W^{(i)}$. Let $\AAA^{(i)}$
       be the arrangements of reflecting hyperplanes of the reflections from $W^{(i)}$, 
       \index{reflecting!hyperplane}%
       $1 \leq i \leq t$. Suppose there is an $1 \leq i \leq t$ such that the $W$-stable 
       subset $\OOO \subseteq \LLL$ has every $X$ in $\OOO$ a subspace of 
       $X^{(i)}:=\bigcap_{H \in \AAA^{(i)}} H$,
       so that one can write uniquely $X=X^{(i)} \cap Y$ where
       $Y$ is an intersection of the hyperplanes in  $\bigsqcup_{j \neq i}\AAA^{(j)}$.
       Then letting $W' := \prod_{j \neq i} W_j$ and identifying 
       $\ZZ W \cong \ZZ W^{(i)} \otimes \ZZ W'$, one has
       $$
         \nu_\OOO = \grid_{\ZZ W^{(i)}} \otimes \nu_{\OOO'}
       $$
       where $\OOO':=\{ Y: X^{(i)} \cap Y \in \OOO\}$.
     \end{Proposition}

     \begin{Example}
       \ref{ex:imprimitive-example} illustrates the scenario
       of \ref{prop:reducibility-reduction1} 
       with $V=\RR^2 = V^{(1)} \oplus V^{(2)}=\RR^1 \oplus \RR^1$.
       Here $i=1$ with $X=X^{(1)}=H_1$ and $Y=V^{(2)}$ is the second copy of $\RR^1$
       considered as the empty intersection of hyperplanes from $\AAA^{(2)}$.
       In the tensor decomposition of $\nu_\OOO$,
       the first tensor factor is $\grid_{\ZZ W^{(1)}}$ and the second tensor factor is
       $\nu_{\OOO'}$.
     \end{Example}

     Let $W$ be a finite real reflection group and $\OOO \subseteq \LLL$ a $W$-invariant
     subset of $\LLL$. Assume that $W = \prod_{i=1}^t W^{(i)}$ for
     irreducible reflection groups $W^{(i)}$ and $\AAA^{(i)}$ the arrangement
     of reflecting hyperplanes of $W^{(i)}$. We call $\OOO$ 
     {\it irreducible} if there is no $1 \leq i \leq t$ such that all $X \in \OOO$
     satisfy $X \subseteq \bigcap_{H \in \AAA^{(i)}} H$.  
     As a consequence of this proposition, in analyzing the eigenvalues and
     eigenspaces of $\nu_\OOO$, it suffices for us to assume that $\OOO$ is
     irreducible. 

     \begin{Proposition}
       Let $W$ be a finite real reflection group and $\OOO \subseteq \LLL$ an
       irreducible $W$-invariant subset of $\LLL$.
       Then the nonnegative $|W| \times |W|$ matrix $\nu_\OOO$ is primitive
       in the sense that it has some positive power $\nu_\OOO^m$ with all
       strictly positive entries. In particular,
       the $\lambda(\trivial)$-eigenspace is simple, spanned by
       the trivial idempotent $\idem_\trivial$.
     \end{Proposition}
     \begin{proof}
       Recall that $\nu_\OOO = \sum_{w \in W} \ninv_\OOO(w) \cdot w$
       as an element of $\ZZ W$, and that it has nonnegative coefficients. 
       Consequently, it suffices to show that the set of $w$ in $W$ having
       positive coefficient $\ninv_\OOO(w)>0$ is a generating set for
       $W$.  We will exhibit an explicit generating set for $W$ with all
       having positive coefficients.

       Recall that for finite real reflection groups $W$, 
       the set $S$ of reflections through the hyperplanes which bound the chosen
       fundamental chamber $c_\grid$ gives rise to a {\deffont Coxeter presentation} for
       \index{Coxeter!presentation}%
       $W$, or a {\deffont Coxeter system} $(W,S)$.
       \nomenclature[al]{$(W,S)$}{Coxeter system}%
       \index{Coxeter!system}%
       In the above situation, for each
       $i=1,2,\ldots,t$, we can choose
       the fundamental chambers for each group $W^{(i)}$ independently.
       Make this choice so that for each $i=1,2,\ldots,t$, the subspace 
       $X^{(i)}$ lies in the intersection of some subset of the walls of the 
       fundamental chamber for $W^{(i)}$, say the walls indexed by
       the proper subset $J^{(i)}$ of $S^{(i)}$.  

       Because each $W^{(i)}$ acts irreducibly, the Coxeter system
       $(W^{(i)},S^{(i)})$ has connected Coxeter diagram, and one
       \index{Coxeter!diagram}%
       can number its nodes $s^{(i)}_1,s^{(i)}_2,\ldots$ in such a way
       that $s^{(i)}_1$ is not in $J^{(i)}$, and 
       each initial segment of the nodes induces a connected
       subdiagram.  

       We claim that the union over $i=1,2,\ldots,t$
       of the sets
       $$ 
         s^{(i)}_1, s^{(i)}_1 s^{(i)}_2, \ldots, s^{(i)}_1 s^{(i)}_2 \cdots s^{(i)}_{|S^{(i)}|}
       $$
       is a generating set for $W$, and that each of these elements
       has positive value of $\ninv_\OOO$.  The reason they generate $W$
       is that $S^{(i)}= \{ s^{(i)}_1,s^{(i)}_2,\ldots, s^{(i)}_{|S^{(i)}|} \}$
       generates $W^{(i)}$.  We want to show that any of the elements
       $w=s^{(i)}_1 s^{(i)}_2 \cdots s^{(i)}_j$ inside $W^{(i)}$ will have a positive value of
       $\ninv_\OOO$. For that consider the subspace $X^{(i)}$ of $\OOO$.  We claim that
       $X^{(i)}$ forms a noninversion for
       $w$.  To see this, by \ref{prop:parabolic-factorization} 
       and subsequent comments one needs to check that $w$ is one of the 
       minimal length coset representatives
       for $W_{J^{(i)}} \backslash W^{(i)}$, that is, it has no reduced
       expressions that start with an element of $J^{(i)}$ on the left.
       But by our construction of the word $w=s^{(i)}_1 s^{(i)}_2 \cdots s^{(i)}_j$,
       and by Tits' solution to the word problem for $W$ (see \cite[Theorem 2.33]{AbramenkoBrown2008}), 
       this would be impossible because
       no element of $J^{(i)}$ can be commuted
       past the $s^{(i)}_1$ on the left.

       The fact that the $\lambda(\trivial)$-eigenspace is simple and is spanned by
       the trivial idempotent $\idem_\trivial$ now follows from the Perron-Frobenius
       theorem \cite[Theorem 8.4.4]{HornJohnson1985}. 
     \end{proof}

     For future use (in \ref{subsec:rank-one-proof}), we mention another trivial  
     reduction, similar to \ref{prop:reducibility-reduction1}, that can occur when 
     the finite real reflection group $W$ acting on $V$
     does not act irreducibly.  Its proof is similarly straightforward.

     \begin{Proposition}
       \label{prop:reducibility-reduction2}
       Let $W$ be a finite real reflection group and $W = \prod_{i=1}^t W^{(i)}$ for
       irreducible reflection groups $W^{(i)}$. Let $\AAA^{(i)}$
       be the arrangements of reflecting hyperplanes of the reflections from $W^{(i)}$, 
       \index{reflecting!hyperplane}%
       $1 \leq i \leq t$. Let $\OOO \subseteq \LLL$ be a $W$-invariant subset of $\LLL$. 

       Assume that there is an $1 \leq i \leq t$ such that 
       $\OOO$ contains no subspaces $X$ lying below any 
       hyperplanes from $\AAA^{(i)}$.  
       Then we can consider $\OOO$ as a subset of the
       intersection lattice for the arrangement
       $\AAA' := \AAA \setminus \AAA^{(i)}$ of the reflection 
       group $W' := \prod_{j \neq i} W_j$. We have 
       $\ZZ W \cong \ZZ W^{(i)} \otimes \ZZ W'$
       and
       $$
         \nu_\OOO = \matone_{\ZZ W^{(i)}} \otimes \nu_{\OOO'}
       $$
       where $\matone_{\ZZ W^{(i)}}$ is represented by the $|W^{(i)}| \times |W^{(i)}|$ matrix
       having all ones as entries.
     \end{Proposition}

     Since the eigenvalues and eigenvectors of $\matone_{\ZZ W}$
     are easy to write down, by \ref{prop:reducibility-reduction2} 
     one is reduced to studying $\nu_{\OOO'}$ in this situation.

     \begin{Example}
       \ref{ex:imprimitive-example} also illustrates the scenario
       of \ref{prop:reducibility-reduction2} except now $i=2$, 
       and one should interpret the first tensor factor as 
       $\nu_{\OOO'}$ and the second tensor factor as
       $\matone_{\ZZ W^{(2)}}$.
     \end{Example}

\section{The case where $\OOO$ contains only hyperplanes}
  \label{sec:rank-one}

  \subsection{Review of twisted Gelfand pairs}
    \label{subsec:twisted-Gelfand-review}

   We review here some of the theory of (twisted) Gelfand pairs;  a good introduction is 
   Stembridge \cite{Stembridge1992}.

   \begin{Definition}
     Given a finite group $G$, a subgroup $U$, and a linear character $\chi: U \rightarrow \CC^\times$,
     say that $(G,U,\chi)$ forms a {\deffont twisted Gelfand pair} (or {\deffont triple}) if the induced representation
     \index{twisted!Gelfand pair}%
     \nomenclature[al]{$(G,U,\chi)$}{twisted Gelfand pair (or triple)}%
     $\Ind_U^G \chi$ is a multiplicity-free $\CC G$-module.  
   \end{Definition}

   One can fruitfully rephrase this is in terms of the algebra structure of $A:=\CC G$ and
   the $\chi$-idempotent for $U$
   \begin{equation}
     \label{eqn:H-chi-idempotent}
     \idem :=\frac{1}{|U|} \sum_{u \in U} \chi(u^{-1}) u.
   \end{equation}
   It is well-known and
   easy to see that the left-ideal $A \idem$ carries a left $A$-module structure isomorphic
   to $M=\Ind_U^G \chi$.  As with any finite dimensional $A$-module, $M$ can be expressed
   as $M=\bigoplus_{i} (S_i)^{\oplus m_i}$ for distinct
   simple $A$-modules $S_i$ and uniquely defined multiplicities $m_i$.
   One can detect these multiplicities by looking at
   the commutant algebra \index{commutant algebra} \index{algebra!commutant} $\End_A M$, which is isomorphic to
   the direct sum of matrix algebras $\oplus_i \Mat_{m_i \times m_i}(\CC)$.
   Thus the commutant algebra is itself a {\deffont commutative} algebra if and only if each $m_i=1$,
   \index{commutative algebra}%
   \index{algebra!commutative}%
   that is, if and only if $M$ is multiplicity-free as an $A$-module.
   Therefore, the condition for $(G,U,\chi)$ to be a twisted Gelfand pair is equivalent
   to $\End_A M$ being commutative.

   On the other hand, for any algebra with unit $A$ and idempotent $\idem$,
   taking $M=A \idem$, the map defined by
   $$
     \begin{aligned}
       \End_A M =\End_A (A \idem) &\longrightarrow & \idem A \idem  \\
       \varphi                    & \longmapsto    & \varphi( \idem )
     \end{aligned}
   $$
   is easily seen to be an algebra isomorphism.  In the case $A=\CC G$ and 
   $\idem$ is the idempotent in \ref{eqn:H-chi-idempotent}, the algebra
   $\idem A \idem$ is sometimes called the {\deffont (twisted) Hecke algebra}.
   \index{twisted!Hecke algebra}%
   \index{algebra!twisted Hecke}%
   \index{Hecke algebra!twisted Hecke}%
   \index{Hecke algebra}%
   \index{algebra!Hecke}%
   If one chooses double coset representatives $\{g_1,\ldots,g_t\}$ for $U\backslash G/U$,
   then it is easy to see that the nonzero elements in the set
   $\{\idem g_i\idem \}_{i=1,2,\ldots,t}$ form a $\CC$-basis for this Hecke algebra $\idem A \idem$.  
   This leads to the following commonly used trick for verifying that one has a twisted Gelfand pair.

   \begin{Proposition}[Twisted version of ``Gelfand's trick'']
     \label{prop:Gelfand-trick}
     Let $G$ be a finite group, $U$ a subgroup of $G$ and $\chi: U \rightarrow \CC^\times$ a
     linear character with $\chi(u^{-1})=\chi(u)$ for all $u$ in $U$, that is, $\chi$ takes values
     in $\{\pm1\}$.

     If every double coset $UgU$ within $G$ for which $\idem g \idem \neq 0$ 
     contains an involution, then 
     $(G,U,\chi)$ forms a twisted Gelfand pair.
   \end{Proposition}
   \begin{proof}
     As above let $\idem :=\frac{1}{|U|} \sum_{u \in U} \chi(u^{-1}) u$.
     Consider the algebra {\emphfont anti-}automorphism $\psi$ of $A=\CC G$ that sends
     $g \mapsto g^{-1}$.  The assumption that $\chi(u^{-1})=\chi(u)$
     implies $\psi(\idem )=\idem$.  Thus for any involution $g=g^{-1}$ in $G$, one
     has that $\psi$ also fixes the element $\idem g \idem $ in $\CC G$:
     $$
       \psi(\idem g \idem ) = \psi(\idem ) \psi(g) \psi(\idem ) = \idem g^{-1} \idem = \idem g \idem.
     $$
     The assumption that every double coset $UgU$ for which 
     $\idem g \idem \neq 0$ contains an involution
     therefore implies that $\psi$ fixes every element in a spanning set
     for the subalgebra $\idem A \idem $ within the
     group algebra $A=\CC G$.  Since $\psi$ is an anti-automorphism on
     all of $A$, this subalgebra $\idem A \idem $ must be commutative:  for any $x,y$ in $\idem A \idem $,
     one has
     $$
       x \cdot y = \psi(x) \psi(y) = \psi( y \cdot x ) = y \cdot x.
     $$  
     Thus $\End_A (A\idem )=\idem A \idem $
     is commutative. Hence $A \idem$ is a multiplicity-free left $A$-module,
     i.e. $(G,U,\chi)$ is a twisted Gelfand pair.
   \end{proof}

   \subsection{A new twisted Gelfand pair}
     \label{subsec:new-twisted-Gelfand-pair}

  \newcommand{\aaa}{\mf{a}}
  \newcommand{\bbb}{\mf{b}}

      Recall the statement of \ref{thm:Gelfand-triple} from the introduction.

      \theoremstyle{plain}
      \newtheorem*{TheoremGelfandTriple}{\ref{thm:Gelfand-triple}}
      \begin{TheoremGelfandTriple}
        Let $W \leq \GL(V)$ be any finite irreducible real reflection group and
        $H$ any of its reflecting hyperplanes with associated reflection $s$.

        Then the linear character $\chi$ of the $W$-centralizer $\Centr_W(s)$ given by
        the determinant on $V/H$ or $H^\perp$ has a multiplicity-free
        induced $W$-representation $\Ind_{\Centr_W(s)}^W \chi$.

        In other words, $(W,\Centr_W(s),\chi)$ forms a twisted Gelfand pair.
      \end{TheoremGelfandTriple}

      \noindent
      As preparation for proving this, we begin with some well-known
      general observations about group actions on cosets, and double cosets. 
      Let $Z:=\Centr_W(s)$ and $\OOO$ the orbit of $H$ under the action of $W$.
      Then $Z$ is the stabilizer of the element $H$ in the transitive action of $W$ on $\OOO$.
      In other words, $\OOO$ carries the same $W$-action as the coset action of $W$ 
      left-translating $W/Z$.  One then has inverse bijections
      between the double cosets $Z\backslash W\slash Z$ and the $W$-orbits for the 
      {\it diagonal} action of $W$ on $\OOO \times \OOO$:
      $$
        \begin{array}{rcl}
           Z\backslash W /Z & \longrightarrow      & W\backslash \left( \OOO \times \OOO \right) \\
           ZwZ              & \longmapsto          & W \cdot (H, w(H)) \\
                            &                      & \\
           W\backslash \left( \OOO \times \OOO \right) &\longrightarrow & Z\backslash W /Z \\
           W \cdot (w_1(H), w_2(H)) & \longmapsto          &Z w_1^{-1} w_2 Z              
       \end{array}
      $$

      \begin{Proposition}
        \label{Coxeter-subgraph-proposition}
        Let $(W,S)$ be a Coxeter system with $W$ finite,
        and $J \subset S$ such that the Coxeter graph
        for $(W_J, J)$ is a connected subgraph of
        the Coxeter graph for $(W,S)$.
        Then for two reflecting hyperplanes $H, H'$ whose reflections $s_H, s_{H'}$ lie 
        in $W_J$ we have:
        $s_H,s_{H'}$ lie in the same $W$-orbit if and only they lie in the same $W_J$-orbit.
      \end{Proposition}
      \begin{proof}
        Since every reflection in $W_J$ is $W_J$-conjugate to a simple 
        reflection in $J$,
        one may assume without loss of generality that $s_H, s_{H'}$ are
        simple reflections lying in the subset $J$.
        It is well-known 
        (see e.g. \cite[Chapter 1, Exercise 16, p. 23]{BjornerBrenti2007})
        that two simple reflections $s,s'$ in $S$ are $W$-conjugate if and only
        if there is a path in the Coxeter graph for $(W,S)$ having all edges
        with {\it odd} labels.  Since $W$ is finite, the Coxeter graph for
        $(W,S)$ is a tree.  Hence such a path with odd labels exists if and
        only if it exists within the Coxeter subgraph for $(W_J,J)$, that is,
        if and only if $s_H, s_{H'}$ are $W_J$-conjugate.
      \end{proof}

      \begin{proof}[Proof of \ref{thm:Gelfand-triple}]

            We will show that the twisted version of Gelfand's trick 
            (\ref{prop:Gelfand-trick}) applies.  Let $w \in W$ and 
            $H':=w(H)$. Let $s_H$ and $s_{H'}$ be the reflections
            corresponding to $H$ and $H'$.

            \noindent {\sf Case 1:} $H, H'$ are orthogonal.

            In this case we claim that $\idem w \idem=0$.
            To see this, note that in this situation, both $s_H, s_{H'}$ lie in
            $Z$, with 
            $$
              \begin{aligned}
                 \chi(s_H)&=-1\\
                 \chi(s_{H'})&=+1.
              \end{aligned}
            $$  
            Thus factoring the subgroup $Z = \Centr_W(s)$ according to cosets 
            $Z/\langle s_H \rangle$
            and cosets $\langle s_{H'} \rangle \backslash Z$ gives rise to factorizations
            $$
              \begin{aligned}
                \idem &= \aaa (e - s_H)\\
                \idem &= (e + s_{H'}) \bbb
              \end{aligned}
            $$
            for some elements $\aaa, \bbb$ in $\RR W$.  One then calculates
            $$
              \begin{aligned}
                \idem w \idem &= \aaa (e + s_{H'}) w (e - s_{H}) \bbb \\
                              &= \aaa (w - w s_H + s_{H'} w - s_{H'} w s_{H}) \bbb \\
                              &= \aaa \cdot 0 \cdot \bbb \\
                              &= 0
              \end{aligned}
            $$
            where the third line uses the following equalities:
            $$
              \begin{aligned}
                 w(H)         &= H', \text{ implying }\\
                 w s_H w^{-1} &= s_{H'}\\
                 ws_H         &= s_{H'} w\\
                 w            &=s_{H'} w s_H. 
              \end{aligned}
            $$

            \noindent {\sf Case 2:} $H, H'$ are {\it not} orthogonal.

            A trivial subcase occurs when
            $H=H'$ and then the double coset $ZwZ= Z$ contains 
            the involution $s_H$. Hence we are done by \ref{prop:Gelfand-trick}.

            Otherwise, the parabolic subgroup $W_{H \cap H'}$ is dihedral, and 
            $W$-conjugate to some standard parabolic $W_J$ for some pair 
            $J=\{s,s'\} \subset S$; without loss of generality (by conjugation), 
            $s_H, s_{H'}$ lie in $W_J$.  Since $H, H'$ are not orthogonal, one must
            have $s, s'$ non-commuting, and hence the Coxeter graph for $(W_J,J)$ is
            an edge with label $m \geq 3$, forming a connected subgraph of the 
            Coxeter graph of $(W,S)$.  Since $H, H'$ were assumed to lie in the same
            $W$-orbit, \ref{Coxeter-subgraph-proposition} implies they lie in the
            same $W_J$-orbit.  However, when $s_H, s_{H'}$ lie within a dihedral
            group $W_J$, it is easy to check that if $w$ in $W_J$ sends 
            $H$ to $H'$, then either $w$ or $ws_H$ is a reflection,
            and hence an involution, sending $H$ to $H'$.
            Again the assertion follows from \ref{prop:Gelfand-trick}.
      \end{proof}

      \begin{Remark}
         The preceding proof is perhaps more subtle than it first appears.
         When $H$ and $H'$ are {\it orthogonal} hyperplanes lying in the same
         $W$-orbit, so that $H'=w(H)$ for some $w$ in $W$, it {\it can} happen that
         $H$, $H'$ do {\it not} lie in the same $W_{H \cap H'}$-orbit,
         and that the double coset $ZwZ$ for $Z=\Centr_W(s_H)$
         contains {\it no} involutions.  

         As an example, this occurs within
         the Coxeter system $(W,S)$ of type $H_3$ with 
         Coxeter generators $S=\{s_1,s_2,s_3\}$, satisfying $s_i^2=1$ and
         $(s_1 s_2)^5=(s_1 s_3)^2=(s_2 s_3)^3 = e$.
         The hyperplanes $H, H'$ fixed by $s_1, s_3$, respectively,
         are orthogonal.  They lie in the same $W$-orbit, and
         in fact $w(H)=H'$ for $w=s_2 s_1 s_2 s_3 s_1 s_2$.
         However, $H, H'$ do not lie in the same orbit for the
         rank $2$ parabolic $W_{H \cap H'}=W_{\{s_1,s_3\}}$,
         and one finds that the double
         coset $ZwZ$ for the subgroup $Z=\Centr_W(s_1)=\langle s_1, s_3, w_0 \rangle$ 
         contains elements of orders $3$ and $6$, but contains no involutions.
      \end{Remark}

\subsection{Two proofs of \ref{thm:rank-one}}
      \label{subsec:rank-one-proof}

      We recall the statement of the theorem.

      \theoremstyle{plain}
      \newtheorem*{TheoremRankOne}{\ref{thm:rank-one}}
      \begin{TheoremRankOne}
        For any finite real reflection group $W$, and
        any $W$-orbit $\OOO$ of hyperplanes,
        the matrix $\nu_\OOO$ has all its eigenvalues within
        the ring of integers of the unique minimal splitting field for $W$.
        In particular, when $W$ is crystallographic, these
        eigenvalues lie in $\ZZ$.
      \end{TheoremRankOne}

      We will offer two proofs.  In both proofs, one first notes that
      one can immediately use \ref{prop:reducibility-reduction1}
      to reduce to the case where $W$ acts irreducibly
      on $V$. Also note that if $W = \prod_{i=1}^t W^{(i)}$ for
      irreducible reflection groups $W^{(i)}$ and $\AAA^{(i)}$
      the arrangements consisting of the reflecting hyperplanes of the reflections from $W^{(i)}$,
      \index{reflecting!hyperplane}%
      then a $W$-orbit $\OOO$ of hyperplanes in $\AAA$ contains only hyperplanes 
      from a single subarrangement $\AAA^{(i)}$ for some $1 \leq i \leq t$.

      Thus one can assume $W$ acts irreducibly on $V$, and
      both proofs will rely on \ref{thm:Gelfand-triple}.

\subsubsection{First proof of \ref{thm:rank-one}.}
        The first proof is shorter, but makes forward reference
        to the equivariant theory of \BHR random walks in \ref{sec:BHR}.
        This \BHR theory will show that when $\nu_\OOO$ acts on $\RR W$, its image
        subspace $U:=\ker(\nu_\OOO)^\perp$ affords the
        $W$-representation $\trivial_W \oplus \Ind_{\Centr_W(s)}^W \chi$,
        where $\chi=\det|_{V/H}$.  Note that this image can 
        have no multiplicity on the trivial representation $\trivial_W$,
        since the ambient space $\RR W$ contains only one copy of $\trivial_W$.
        Hence \ref{thm:Gelfand-triple}
        tell us that the $\nu_\OOO$-stable subspace $U$
        is multiplicity-free as a $W$-representation.
        Since $U=\ker(\nu_\OOO)^\perp$ is a $\QQ$-subspace (as $\nu_\OOO$ has
        $\ZZ$ entries) an application of \ref{prop:integrality-principle}
        finishes the proof.

\subsubsection{Second proof  of \ref{thm:rank-one}.}

        This proof, although longer, does not rely on results to be proven
        later, and also introduces an important idea, useful both in 
        understanding the eigenspaces of $\nu_\OOO$, and with potential
        applications to the analysis of linear ordering polytopes
        (see \ref{subsec:linear-ordering-polytopes}).
        We start by developing this idea here.

        For the moment, return to the situation where $\AAA$ is a
        central arrangement of hyperplanes in $V=\RR^d$ having
        some finite subgroup of $\GL(V)$ acting as symmetries, with
        chambers $\CCC$, intersection lattice $\LLL$, and 
        $\OOO$ any $W$-stable subset of $\LLL$.
        Recall that $\nu_\OOO = \pi_\OOO^T \circ \pi_\OOO$ where
          \begin{align*}
            \pi_\OOO:  \ZZ\CCC & \longrightarrow \bigoplus_{X \in \OOO} \ZZ\CCC(\AAA/X) \\
                             c & \longmapsto     (c/X)_{X \in \OOO}
          \end{align*}
        Note that $\pi_\OOO$ is $W$-equivariant for the
        obvious $W$-actions on the source and targets.
        It is also equivariant for the commuting $\ZZ_2$-action
        that sends $c \mapsto -c$ in the source, and sends
        $c/X \mapsto -c/X$ in the target.

        This gives us the freedom to consider instead of $\nu_\OOO= \pi_\OOO^T \circ \pi_\OOO$,
        the eigenvectors and eigenvalues of the closely related map 
        $$
          \mu_\OOO = \pi_\OOO \circ \pi_\OOO^T:
             \bigoplus_{H \in \OOO}  \ZZ^{\CCC(\AAA/H)} 
             \longrightarrow  \bigoplus_{H \in \OOO}  \ZZ^{\CCC(\AAA/H)}
        $$
        having matrix entries given by
        $$
          \left( \mu_\OOO \right)_{c_\grid/X_1, c_2/X_2}
          = \Big|\big\{ c \in \CCC : c/X_1=c_\grid/X_1 \text{ and } c/X_2 = c_2/X_2 \big\}\Big|.
       $$

       \begin{Proposition}
         For each nonzero eigenvalue $\lambda$ in $\RR$, the
         maps $\pi_\OOO$ and $\pi_\OOO^T$ give $W \times \ZZ_2$-equivariant 
         isomorphisms between the $\lambda$-eigenspaces of
         $\nu_\OOO$ and $\mu_\OOO$.
       \end{Proposition}
       \begin{proof}
         This is a general linear algebra fact.  Assume
         $A:U \rightarrow U'$ and $B:U' \rightarrow U$ are 
         $\KK$-linear maps of finite-dimensional $\KK$-vector spaces $U$ and $U'$ such that
         all eigenvalues of $AB$ and $BA$ lie in $\KK$.  We claim that
         for each potential nonzero eigenvalue $\lambda$ in $\KK$, the maps $A, B$ give isomorphisms
         between the {\deffont generalized $\lambda$-eigenspaces} defined
         \index{generalized eigenspace}%
         \index{eigenspace!generalized}%
         to be the subsets of $U$ and $U'$ on which $\lambda - BA$ and $\lambda-AB$
         act nilpotently.  To see that $A, B$ map between these generalized
         eigenspaces, note that given a vector $v$ in $V$
         with $(\lambda \matid_U -BA)^N v=0$, the fact that
         $$
           (\lambda \matid_{U'} -AB)A = A(\lambda \matid_U -BA)
         $$ 
         implies 
         $$
           (\lambda \matid_{U'} -AB)^N Av=A^N (\lambda \matid_U -BA)v=0.
         $$  
         To see that $A, B$ are injective, note that if $Av=0$ then $(\lambda-BA)v=\lambda v$
         and hence
         $$
           0=(\lambda \matid_U -BA)^N v = \lambda^N v
         $$
         would imply that $v=0$.

         When applying this with $A=\pi_\OOO$ and $B=\pi^T_\OOO$ and $\KK=\RR$,
         self-adjointness implies not only that all the eigenvalues $\lambda$
         all lie in $\RR$, but also semisimplicity, so that generalized
         $\lambda$-eigenspaces are just $\lambda$-eigenspaces.
       \end{proof}

       Now we specialize to the situation where $\AAA$ is the reflection
       arrangement for a finite real reflection group $W$, and the
       $W$-stable subset $\OOO$ contains only hyperplanes $H$ (but we
       do not assume yet that $\OOO$ is a single $W$-orbit).

       In this case, each of the localized subarrangements 
       $\AAA/H$ has only one hyperplane $H$, and only two chambers/half-spaces
       in $\CCC(\AAA/H)$, which one can identify with the two 
       unit normals $\pm \alpha$ (or {\deffont roots}) to the hyperplane $H$.
       \index{root}%
       Letting $\Phi_\OOO$ denote the union of all such pairs of roots
       \nomenclature[al]{$\Phi_\OOO$}{roots corresponding to hyperplanes in $\OOO$}%
       \nomenclature[al]{$\Phi$}{root system}%
       $\pm \alpha$ normal to the hyperplanes $H$ in $\OOO$,
       one can identify $\bigoplus_{H \in \OOO} \ZZ^{\CCC(\AAA/H)}$ with $\ZZ^{\Phi_\OOO}$,
       having a basis element $\unitbase_\alpha$ 
       \nomenclature{$\unitbase_i$}{unit basis vector indexed by $i$}%
       for each $\alpha$ in the orbit of roots  $\Phi_\OOO$.
       Under this identification, the map $\ZZ W  \overset{\pi_\OOO}{\rightarrow} \ZZ^{\Phi_\OOO}$
       has 
       $$
         \left( \pi_\OOO \right)_{w,\unitbase_\alpha} =
           \begin{cases}
             1 & \text{ if }w(\alpha) \in \Phi_+\\
             0 & \text{ otherwise.}
           \end{cases}
       $$
       That is, $\pi_\OOO$ sends a basis element $w$ in $\ZZ W$
       to the sum of basis elements $\unitbase_\alpha$ for which $w^{-1}(\alpha)$ is an
       element of the {\deffont positive roots} $\Phi_+$,
       \index{positive root}%
       \nomenclature[al]{$\Phi_{\pm}$}{positive/negative roots in root system $\Phi$}%
       i.e. $c_\grid$ and $c_w$ lie on the same side of the hyperplane $H_\alpha$.
       Therefore the map $\ZZ^{\Phi_\OOO} \overset{\mu_\OOO}{\longrightarrow} \ZZ^{\Phi_\OOO}$
       has entry 
       \begin{equation}
       \label{angle-measure-formula-for-mu}
        (\mu_\OOO)_{\unitbase_\alpha,\unitbase_\beta} 
          =\# \{ w \in W: w^{-1}(\alpha), 
                   w^{-1}(\beta) \text{ both lie in }\Phi_+ \} 
             =|W| \cdot \frac{\theta_{\alpha,\beta}}{2\pi}
       \end{equation}
       where $\theta_{\alpha,\beta}$ is the angular measure in radians of
       the sector which is the intersection of the half-spaces $H_\alpha^+ \cap H_\beta^+$.

       Note that the $\ZZ_2$-action now sends $\unitbase_\alpha$ to $\unitbase_{-\alpha}$.
       We use this $\ZZ_2$-action to decompose
       \nomenclature[al]{$\Phi_{\OOO,\pm}$}{positive/negative roots in $\Phi_\OOO$}%
       $$
         \RR^{\Phi_\OOO} = \RR^{\Phi_\OOO,+} \oplus  \RR^{\Phi_\OOO,-}
       $$
       in which 
       $$
         \begin{aligned}
            \RR^{\Phi_\OOO,+}&\text{ has }\RR\text{-basis }
            \{ f^+_\alpha:= 
            \unitbase_\alpha + \unitbase_{-\alpha} 
              \}_{\alpha \in \Phi_\OOO \cap \Phi_+},\\
            \RR^{\Phi_\OOO,-}&\text{ has }\RR\text{-basis }
            \{ f^-_\alpha:= 
           \unitbase_\alpha - \unitbase_{-\alpha} 
              \}_{\alpha \in \Phi_\OOO \cap \Phi_+}.
         \end{aligned}
       $$
       For the formulation of the following proposition, recall $\lambda_\OOO(\chi)$ 
       defined in \ref{degree-one-character-prop}.

       \begin{Proposition}
         \label{prop:rank-one-plus-minus}
         Acting on $\RR^{\Phi_\OOO,+}$, the map $\mu_\OOO$
         has a one-dimensional eigenspace with eigenvalue
         $\lambda_\OOO(\trivial_W)$ carrying the trivial
         $W$-representation $\trivial_W$, and whose orthogonal
         complement within $\RR^{\Phi_\OOO,+}$ lies in the kernel.

         If $\OOO$ decomposes into $W$-orbits as
         $\OOO =\bigsqcup_{i=1}^t \OOO_i$ in which
         $\OOO_i$ is the orbit of a hyperplane $H_i$ having
         associated reflection $s_i$, then 
         $\RR^{\Phi_\OOO,-}$ carries the $W$-representation
         $
           \bigoplus_{i=1}^t \Ind_{\Centr_W(s_i)}^W \chi_i,
         $
         where $\chi_i$ is the one-dimensional character
         $\det{}_{V/H_i}$.
       \end{Proposition}

       \begin{proof}
         Using the fact that for any $w$ in $W$, exactly one out 
         of $w^{-1}(\alpha)$ and $w^{-1}(-\alpha)$ will be a positive
         root, one checks using \ref{angle-measure-formula-for-mu} that  
         $$
           \mu_\OOO(f^+_\beta) = \sum_{\alpha \in  \Phi_\OOO \cap \Phi_+} f^+_\alpha
         $$
         for {\emphfont any} $\beta$ in $\Phi_\OOO \cap \Phi_+$.
         This implies that $\mu_\OOO$ acts on $\RR^{\Phi_\OOO,+}$ as an operator
         of rank one, whose only nonzero eigenspace is the line spanned by  
         $\sum_{\Phi_\OOO \cap \Phi_+} f^+_\alpha$, affording the trivial $W$-representation
         $\trivial_W$, and with eigenvalue $\lambda_\OOO(\trivial_W)=\frac{|W|}{2} |\OOO|$.
         Because $\mu_\OOO$ is self-adjoint, the subspace of $\RR^{\Phi_\OOO,+}$ 
         perpendicular to this eigenspace will be preserved, and must lie
         entirely in the kernel.  

         The assertion about the $W$-representation carried by
         $\RR^{\Phi_\OOO,-}$ follows because one has
         $$
           s_i(f^-_{\alpha_i})=-f^-_{\alpha_i}=\det{}_{V/H_i}(s_i) f^-_{\alpha_i}
         $$ 
         and $\Centr_W(s_i)$ is the stabilizer of the line spanned by $f^-_\alpha$.  
       \end{proof}

       \vskip.1in
       \noindent
       \begin{proof}[Second proof of \ref{thm:rank-one}]
         Assuming $\OOO$ is a transitive $W$-orbit of some hyperplane $H$ with
         associated reflection $s$, \ref{prop:rank-one-plus-minus} says that 
         the $\RR$-subspace $U:=\RR^{\Phi_\OOO,-}$, which is a rational subspace
         in the sense that $\RR^{\Phi_\OOO,-}=\RR \otimes_\QQ \QQ^{\Phi_\OOO,-}$, 
         affords the $W$-representation
         $\Ind_{\Centr_W(s)}^W \chi$.  Then \ref{thm:Gelfand-triple} 
         and \ref{prop:integrality-principle}
         imply that the operator $\mu_\OOO$ has all eigenvalues on $U$
         lying within the algebraic integers of any splitting field for $W$.
         Its remaining eigenvalues on the complementary subspace $\RR^{\Phi_\OOO,+}$
         are either zero or $\lambda_\OOO(\trivial_W)=\frac{|W|}{2} |\OOO|$ 
         by \ref{prop:rank-one-plus-minus}.
       \end{proof}

       \begin{Remark}
         \label{first-Renteln-remark}
         After posting this work on the arXiv, the authors discovered that,
         independently, P. Renteln \cite[\S 4]{Renteln2011} recently studied the
         spectrum of the operator $\nu_\OOO$ for a real finite reflection group $W$,
         taking $\OOO$ to be the set of {\it all} reflecting hyperplanes for $W$.
         Note that irreducible finite reflection groups can have at most two
         $W$-orbits of hyperplanes, and whenever $W$ has only one orbit of 
         hyperplanes (that is, outside of types $B_n (=C_n), F_4$ and the 
         dihedral types $I_2(m)$ with $m$ even), Renteln's object of study 
         is the same as our operator $\nu_\OOO$.

         In particular, he also uses the technique from our
         second proof of \ref{thm:rank-one}, introducing the
         maps $\pi_\OOO$ and $\mu_\OOO$ in his context.  We will point
         out in \ref{second-Renteln-remark} and \ref{third-Renteln-remark} 
         below the places where we borrow from and/or extend his work.
       \end{Remark}

    \subsection{The eigenvalues and eigenspace representations}

      We return again to the situation where $\OOO$ is a single $W$-orbit of hyperplanes.
      Having proven \ref{thm:rank-one} on the integrality of eigenvalues of $\nu_{\OOO}$
      or $\mu_\OOO$, one can still ask for the eigenvalues of $\mu_\OOO$ and the $W$-irreducible
      decomposition of its eigenspaces.  It turns out that one can be surprisingly explicit here.

      Note that \ref{prop:rank-one-plus-minus} reduces this to the analysis of $\mu_\OOO$
      acting on $U:=\RR^{\Phi_\OOO,-}$, which affords the $W$-representation
      $\Ind_{Z}^W \chi$, where $Z=Z_W(s)$ for a reflection $s$ whose hyperplane $H$ represents
      the orbit $\OOO$, and $\chi: Z \rightarrow \{\pm 1\}$ is the character of $Z$ acting
      on the line $H^\perp$.  
      We analyze this representation more fully.

      We know the $W$-irreducible decomposition
      of $\Ind^W_{Z} \chi$ is multiplicity-free
      from \ref{thm:Gelfand-triple}.  Recall this
      is controlled by the double cosets $ZwZ$, or diagonal $W$-orbits
      $W \cdot (H,H')$ in $\OOO \times \OOO$, 
      giving rise to nonzero elements $\idem w \idem$ in
      the twisted Hecke algebra $\idem \RR W \idem$
      (see \ref{subsec:twisted-Gelfand-review}).
      We next explain how dihedral angles between hyperplanes
      play a crucial role here.

      \begin{Definition}
        Given two hyperplanes $H, H'$ within $V$, define their {\it dihedral angle}
        $\angle\{H,H'\}$ to be the unique angle in the interval $[0,\frac{\pi}{2}]$ 
        separating them.
      \end{Definition}

      \begin{Proposition}
        \label{angles-equals-orbits-proposition}
        Let $W$ be a finite real reflection group, and $H$,$H'$,$H''$ 
        hyperplanes in the same $W$-orbit $\OOO$, but with neither
        $H'$ nor $H''$ orthogonal to $H$.
        Then $(H,H'), (H,H'')$ lie in the same diagonal $W$-orbit
        on $\OOO \times \OOO$ if and only if 
        $
         \angle \{H,H'\} = \angle \{H,H''\}.
        $
      \end{Proposition}
      \begin{proof}
        The forward implication is clear.  For the reverse,
        assume $\angle \{H,H'\} = \angle \{H,H''\}$, and consider three
        cases based on the codimension of $X := H \cap H' \cap H''$.

        \noindent {\sf Case 1:} $X$ has codimension $1$.

          This case is trivial, since then $H=H'=H''$.  

        \noindent {\sf Case 2:} $X$ has codimension $2$.

          This case is also straightforward.  One checks inside the
          dihedral reflection subgroup  $W_X$ containing $s_H, s_{H'}, s_{H''}$
          that whenever $\angle \{H,H'\} = \angle \{H,H''\}$, either
          one is in the trivial case $H'=H''$, or else
          $s_H$ sends $(H, H')$ to $(H,H'')$.

        \noindent {\sf Case 3:} $X$ has codimension $3$.
          Then by conjugation, one may assume that the rank $3$ reflection subgroup $W_X$ 
          containing $s_H, s_{H'}, s_{H''}$ is a standard parabolic subgroup $W_J$
          for some triple $J=\{s_1,s_2,s_3\} \subset S$ among
          the Coxeter generators $S$ of $W$.  In fact,
          $(W_J,J)$ must be a connected subgraph of the
          Coxeter graph of $(W,S)$, else $W_J$ contains no 
          three reflections $s_H, s_{H'}, s_{H''}$ 
          with $H \cap H' \cap H''$ of
          codimension $3$ having 
          $\angle \{H,H'\}=\angle\{H,H''\} \neq \frac{\pi}{2}$.
          Thus \ref{Coxeter-subgraph-proposition} implies
          that $H, H', H''$ lie in the same $W_X$-orbit, since
          they lie in the same $W$-orbit.  Finiteness of $W$
          further forces $W_X$ to be one of
          the rank three irreducible types 
          $A_3 (\cong D_3)$ or $B_3 (\cong C_3)$ or $H_3$.
          Now it is not hard to check by brute force in any
          of these three types that a triple $H, H', H''$ in the
          same $W_X$-orbit having $\angle \{H,H'\} = \angle \{H,H''\} \neq \frac{\pi}{2}$
          will have $(H,H')$ and $(H,H'')$ in the same diagonal $W_X$-orbit.  This then implies that $(H,H')$ and $(H,H'')$ lie 
          in the same diagonal $W$-orbit.
      \end{proof}

      The following example shows that the non-orthogonality assumption in
      \ref{angles-equals-orbits-proposition} is perhaps more subtle than
      it first appears.
      Indeed, if $\angle\{H, H'\}=\angle\{H,H''\}=\frac{\pi}{2}$,
      it is possible that $(H,H'), (H,H'')$ lie in different
      $W$-orbits of $\OOO \times \OOO$.

      \begin{Example}
        Let $W$ be of type $D_n$ for $n \geq 4$, and 
        $$
          \begin{aligned}
            H & = \{x_1=x_2\} \\ 
            H'& = \{x_1=-x_2\} \\
            H''& = \{x_3 = x_4\}.
          \end{aligned}
        $$
        Then it is easily checked that $(H,H'), (H,H'')$ lie in different
        $W$-orbits of $\OOO \times \OOO$.
        The problem here is that $X=H \cap H' \cap H''$ has 
        $W_X$ of the reducible type $A_1 \times A_1 \times A_1$,
        so that \ref{Coxeter-subgraph-proposition} does not apply.
      \end{Example}

      \ref{angles-equals-orbits-proposition}
      has very strong consequences in the {\it crystallographic} case,
      that is, where $W$ is a finite {\it Weyl} group.
      For this we distinguish two cases for a given reflecting hyperplane $H$ 
      for a finite reflection group $W$ and its $W$-orbit $\OOO$:
      \begin{itemize}
        \item[($\spi$)] There is a hyperplane $H' \in \OOO$ for which
            for which $\angle\{H,H'\} = \frac{\pi}{3}$.
        \item[($\npi$)] There is no hyperplane $H' \in \OOO$ for which
            for which $\angle\{H,H'\} = \frac{\pi}{3}$.
      \end{itemize}
      
      Note that ($\npi$) occurs only in the situation when $W$ is of 
      type $B_n(\cong C_n)$, and the reflection $s_H$ along $H$ is the special
      ``non-simply-laced'' node, corresponding to a sign change
      in a coordinate of $V=\RR^n$.

      \begin{Corollary}
        \label{at-most-two-constituents-corollary}
        Let $W$ be a finite Weyl group,
        $\OOO$ the $W$-orbit of a reflecting
        hyperplane $H$ with reflection $s$, and $Z=\Centr_W(s)$.
        Let $\chi: Z \rightarrow \{\pm 1\}$ be the character
        of $Z$ on $H^\perp$.

        Then:
        \begin{enumerate}
          \item[(i)]
            In situation {\rm ($\spi$)} we have 
            $$
              \Ind_Z^W \chi = V \oplus V'
            $$
            for a unique $W$-irreducible $V'$ of dimension $|\OOO|-|V|$.
          \item[(ii)]
            In situation {\rm ($\npi$)} we have 
            $$
              \Ind_Z^W \chi = V.
            $$
        \end{enumerate}

        Moreover, in \emph{(i)}, one can realize
        the $W$-irreducible $V'$ as the subspace $\RR^{\Phi_\OOO,-}$ of $\RR^{\Phi_\OOO}$ that is
        $\RR$-linearly spanned by the vectors
        $$
          \psi_{\alpha,\beta,\gamma}:= \unitbase_\alpha + \unitbase_\beta + \unitbase_\gamma
                              - (\unitbase_{-\alpha} + \unitbase_{-\beta} + \unitbase_{-\gamma})
        $$
        as $\{\alpha,\beta,\gamma\}$ run through all triples of
        roots in the $W$-orbit $\OOO$ having $\alpha+\beta+\gamma=0$
        and having normal hyperplanes $H_\alpha, H_\beta, H_\gamma$
        with pairwise dihedral angles of $\frac{\pi}{3}$.
      \end{Corollary}
      \begin{proof}
        By \ref{angles-equals-orbits-proposition}, the number of $W$-irreducible constituents in 
        $\Ind_Z^W \chi$
        is the number of dihedral angles $\angle\{H,H'\}$ other than $\frac{\pi}{2}$
        which occur among pairs $\{H, H'\}$ in the $W$-orbit $\OOO$.
        By conjugation, one may assume $W_{H \cap H'}$ is a standard
        parabolic subgroup $W_J$, of some dihedral type $I_2(m)$
        with $m \geq 3$.  Since $W$ is a Weyl group, this limits
        $m$ to be $3,4,6$, and then one can check that $\{H, H'\}$ lying
        in the same $W$-orbit $\OOO$ forces either $H=H'$ or
        $\angle\{H,H'\} = \frac{\pi}{3}$.  The irreducible decompositions
        in situations {\rm ($\spi$)} and {\rm ($\npi$)} then follow.

        To prove the last assertion, let $Y \subset \RR^{\Phi_\OOO,-} \subset \RR^{\Phi_\OOO}$
        be the subspace spanned by the vectors $\psi_{\alpha,\beta,\gamma}$ described above.
        Consider the $\RR$-linear map
        $
         \RR^{\Phi_\OOO}  \overset{g}{\longrightarrow}  V
        $
        that sends $\unitbase_\alpha \longmapsto  \alpha.$
        It is easy to see that $g$ is $W$-equivariant, and also $\ZZ_2$-equivariant
        for the $\ZZ_2$-action on $\RR^{\Phi_\OOO}$ that swaps
        $\unitbase_\alpha \leftrightarrow \unitbase_{-\alpha}$ and the $\ZZ_2$-action
        on $V$ by the scalar $-1$.
        The calculations
        $$
          \begin{array}{rccl}
            \unitbase_\alpha - \unitbase_{-\alpha} & \overset{g}{\longmapsto} & \alpha - (-\alpha) &= 2\alpha\\
            \unitbase_\alpha + \unitbase_{-\alpha} & \overset{g}{\longmapsto} &\alpha + (-\alpha) &= 0\\
            \psi_{\alpha,\beta,\gamma} & \overset{g}{\longmapsto} & 2( \alpha + \beta + \gamma) &=  0\\
          \end{array}
        $$
        then show that 
        \begin{enumerate}
          \item[$\bullet$] the kernel $\ker(g)$ contains $\RR^{\Phi_\OOO,+}$, and hence $g$
            induces a map $\RR^{\Phi_\OOO,-} \overset{\bar{g}}{\rightarrow} V$, 
          \item[$\bullet$] the map $g$, and hence also $\bar{g}$, surjects onto $V$,
            since $V$ is irreducible, and
          \item[$\bullet$] the subspace $Y$ lies in the kernel of $g$, so also
            $Y \subset \ker\left( \RR^{\Phi_\OOO,-} \overset{\bar{g}}{\twoheadrightarrow} V \right)$.
        \end{enumerate}
        Since in situation {\rm ($\spi$)}, one has the
        $W$-irreducible decomposition
        $\RR^{\Phi_\OOO,-}\cong \Ind_Z^W \cong V \oplus V'$, 
        the $W$-equivariance of $g$ then implies $Y \cong V'$.
      \end{proof}

      \begin{Remark}
        \label{second-Renteln-remark}
        Here we have borrowed from Renteln's paper \cite[\S 4.8.1]{Renteln2011} the
        explicit realization of $V'$ by the vectors $\psi_{\alpha,\beta,\gamma}$,
        and its proof via the map $g$, although we substitute
        our argument via irreducibility for his dimension-counting argument.
      \end{Remark}

      \begin{Example}
        \label{ex:rank-one-constituents}
        In type $A_{n-1}$, when $W=\symm_n$ and 
        $\OOO$ is the unique $W$-orbit of hyperplanes, one can
        check that 
        $$
        \begin{aligned}
          \Ind^W_{Z} \chi 
          &= \Ind^{\symm_n}_{\symm_2 \times \symm_{n-2}} \sgn \otimes \trivial\\
          &=\chi^{(n-1,1)} + \chi^{(n-2,1,1)}\\
          &= V \oplus \wedge^2 V
        \end{aligned}
        $$
        using standard calculations with the $\symm_n$-irreducible characters 
        $\chi^{\lambda}$ indexed by integer partitions $\lambda$ of $n$.
        \nomenclature[al]{$\chi^{\lambda}$}{irreducible character of the symmetric group corresponding to the number partition $\lambda$}%
        Thus the irreducible $V' \cong \wedge^2 V \cong  \chi^{(n-2,1,1)}$ 
        in this case.
      \end{Example}

      \noindent
      Based on the $W$-irreducible description for $\RR^{\Phi_\OOO,-}\cong \Ind_Z^W \chi$
      given in \ref{at-most-two-constituents-corollary},
      one can now be more precise about the
      eigenspaces of $\nu_\OOO$ or $\mu_\OOO$.

      \begin{Theorem}
        \label{Weyl-group-rank-one-eigenspaces}
        Let $W$ be a finite Weyl group,
        $\OOO$ the $W$-orbit of a reflecting
        hyperplane $H$ with reflection $s$, and $Z=\Centr_W(s)$.
        Let $\chi: Z \rightarrow \{\pm 1\}$ be the character
        of $Z$ on $H^\perp$.

        Then either of $\nu_\OOO$ or $\mu_\OOO$ have
        nonzero eigenvalues and accompanying $W$-irreducible eigenspaces described
        as follows:
        \begin{enumerate}
          \item[(i)]
            There is a $1$-dimensional eigenspace carrying the trivial $W$-representation
            with eigenvalue $\lambda_\OOO = \frac{|\OOO| |W|}{2}$.
          \item[(ii)]
            In the case of situation {\rm ($\spi$)}
            there is an $|\OOO|-\ell$-dimensional eigenspace carrying the $W$-representation $V'$
            with eigenvalue $\frac{|W|}{6}$.
          \item[(iii)]
            In either situation {\rm ($\spi$)} and {\rm ($\npi$)}, there is an $\ell$-dimensional eigenspace 
            carrying the $W$-representation $V$
            with eigenvalue 
            $$
              \begin{cases}
                \frac{(2|\OOO|+\ell)|W|}{6\ell} &\text{ in situation {\rm ($\spi$)}},\\
                \frac{2^{n-1}}{n!} &\text{ in situation {\rm ($\npi$)}.}
              \end{cases}
            $$
            Furthermore, in the subcase of situation {\rm ($\spi$)} where $W$ is simply-laced (type $A_\ell, D_\ell,$ or $E_6, E_7, E_8$), 
            one can rewrite this eigenvalue as $\frac{(h+1)|W|}{6}$, where $h$ is the Coxeter number.
        \end{enumerate}
      \end{Theorem}
      \begin{proof}
        \ref{prop:rank-one-plus-minus} already shows assertion (i),
        and the fact that $\RR^{\Phi_\OOO,-}$ gives the remaining non-kernel 
        eigenspaces of $\mu_\OOO$.  Calculating traces, one sees 
        from \ref{angle-measure-formula-for-mu} that
        the diagonal entry $\left( \mu_\OOO \right)_{\alpha, \alpha} = \frac{|W|}{2}$
        for each root $\alpha$ in $\Phi_\OOO$, so that $\mu_\OOO$ has trace 
        $
         \frac{|W||\Phi_\OOO|}{2} = |W| |\OOO|
        $
        when acting on $\RR^{\Phi_\OOO}$.  Since the eigenvalues of $\mu_\OOO$ on
        $\RR^{\Phi_\OOO,+}$ are all zero except for the 
        eigenvalue $\lambda_\OOO = \frac{|W||\OOO|}{2}$
        with multiplicity one, one concludes that 
        $\mu_\OOO$ has trace 
        $
         |W| |\OOO| - \frac{|W||\OOO|}{2} = \frac{|W||\OOO|}{2}
        $
        when restricted to $\RR^{\Phi_\OOO,-}$.

        Thus in situation {\rm ($\npi$)}, where $\RR^{\Phi_\OOO,-} \cong V \cong \RR^\ell = \RR^n$,
        it acts with eigenvalue $\frac{|W||\OOO|}{2\ell} = 2^{n-1} n!$.

        In situation {\rm ($\spi$)}, Schur's Lemma implies that the $W$-irreducible constituent
        $V'$ of $\RR^{\Phi_\OOO,-}$ will lie in a single eigenspace 
        for $\mu_\OOO$.  Since this copy of $V'$ is realized as 
        the span of the elements $\{\psi_{\alpha,\beta,\gamma}\}$, one can, for example,
        determine this eigenvalue by using 
        \ref{angle-measure-formula-for-mu} to compute
        that the coefficient of $\unitbase_\alpha$ in 
        $\mu_\OOO( \psi_{\alpha,\beta,\gamma} )$
        is 
        $$
          \frac{|W|}{2\pi}\left( \pi + \frac{\pi}{3} + \frac{\pi}{3} 
          - 0 - \frac{2\pi}{3} -\frac{2\pi}{3} \right) =\frac{|W|}{6}.
        $$
        Thus $V'$ is an eigenspace for $\mu_\OOO$ with
        eigenvalue $\frac{|W|}{6}$, having dimension $|\OOO|-\ell$.
        Since the only other constituent $V$ of $\RR^{\Phi_\OOO,-}$ has 
        dimension $\ell$, it must lie in a single eigenspace, whose eigenvalue
        $\lambda$ satisfies
        $
         \lambda \cdot \ell = \frac{|W||\OOO|}{2} - \frac{|W|}{6} \left( |\OOO|-\ell \right)
         = \frac{(2|\OOO|+\ell)|W|}{6}
        $,
        and hence $\lambda = \frac{(2 |\OOO| + \ell) |W|}{6\ell}$.

        For the last assertion, in the simply-laced case, one has that $\OOO$ is
        the set of all hyperplanes, whose cardinality is well-known 
        \cite[\S 3.18]{Humphreys1992} to be $\frac{\ell h}{2}$.
        The formula for the eigenvalue follows.
      \end{proof}

      \begin{Remark}
        \label{third-Renteln-remark}
        The above assertion about the structure of the eigenspaces of $\mu_\OOO$
        in the simply-laced subcase of situation {\rm ($\spi$)} was a conjecture in 
        the previous version of our paper,
        and turned out to be Renteln's \cite[Theorem 39]{Renteln2011}.
        We have adapted his method of proof to give the more
        general statement above.
       \end{Remark}

       We have implemented in {\tt Mathematica} \cite{Wolfram2008} the calculation of 
       this matrix for $\mu_\OOO$ acting on $\RR^{\Phi_\OOO,-}$, and produced the 
       characteristic polynomials shown in \ref{table:rank-one-table}.  
       \ref{Weyl-group-rank-one-eigenspaces} predicts the answers 
       for all rows of the figure corresponding to Weyl groups, 
       but makes no prediction for the 
       non-crystallographic groups $H_3, H_4$.  Note that we have
       omitted any data on the dihedral types $I_2(m)$, as here
       the matrices for $\mu_\OOO$ are easily-analyzed circulant matrices, 
       discussed thoroughly in \cite[\S4.1 and \S4.6]{Renteln2011}.

       \begin{figure}
         \renewcommand{\arraystretch}{1.25}
         \begin{tabular}{crl}\toprule
           type              & \multicolumn{2}{c}{factored characteristic polynomial} \\ \midrule
           $A_{n-1}=\symm_n$ & $\left(x - \frac{(n+1)!}{6}\right)^{n-1}$ &
                    $\left(x - \frac{n!}{6}\right)^{\binom{n-1}{2}}$  \\ \midrule
           $B_n, s=\text{sign change}$ & $(x - 2^{n-1} n!)^{n}$ & \\ \midrule[0.5\lightrulewidth]
           $B_n, s=\text{transposition}$ & $\left(x - \frac{2^{n-1} 
                                   \cdot n! \cdot (2n-1)}{3}\right)^{n}$& 
                   $\left( x - \frac{2^{n-1} n!}{3} \right)^{n(n-2)}$\\\midrule
           $D_n$ & $\left( x - \frac{2^{n-2} \cdot n! \cdot (2n-1)}{3} \right)^{n}$ 
                              &$\left( x - \frac{2^{n-2} n!}{3} \right)^{n(n-2)}$\\\midrule
           $E_6$ & $(x-112320)^6$ & $(x-8640)^{30}$ \\ \midrule[0.5\lightrulewidth]
           $E_7$ & $(x-9192960)^7$ & $(x-483840)^{56}$ \\ \midrule[0.5\lightrulewidth]
           $E_8$ & $(x-3599769600)^8$ & $(x-116121600)^{112}$ \\ \midrule
           $F_4$ & $(x - 1344)^{4}$ & $(x - 192)^{8}$\\\midrule
           \multirow{2}{*}{$H_3$} & $(x^{2} - 248x + 3856)^{3}$ & $(x - 24)^{4} \cdot (x - 12)^{5} $\\
              & $= (x - 124 \pm 48 \sqrt{5})^{3}$ & $(x - 24)^{4} \cdot (x - 12)^{5}$ \\\midrule[0.5\lightrulewidth]
           \multirow{2}{*}{$H_4$} & $ (x^2 - 79680x + 94233600)^4$ &$(x-3840)^{16} \cdot (x-1440)^5$\\
	      & = $(x - 39840 \pm 17280\sqrt5)^4$ & $(x-3840)^{16} \cdot (x-1440)^5$\\
           \bottomrule
         \end{tabular}
         \caption{{Factored characteristic polynomials for $\nu_\OOO$ or $\mu_\OOO$ on their
           eigenspaces affording $\Ind_{\Centr_W(s)}^W \chi$, where $\chi=\det|_{V/H}$
           if $s=s_H$.}} 
           \label{table:rank-one-table}
       \end{figure}

%

     \begin{Remark}
       \ref{thm:rank-one} can fail without the
       hypothesis that $\OOO$ is a single $W$-orbit of hyperplanes.
       For example, when $W=B_2=I_2(4)$ and $\OOO$ is the set of
       {\emphfont all four} hyperplanes, one finds that
       $$
         \det(t\matid_{\RR^{8}} -\nu_\OOO) = t^3 (t-16)(t^2-8t+8)^2,
       $$
       which contains quadratic factors irreducible over $\QQ$,
       the unique minimal splitting field of $W$ in characteristic $0$.  
       The issue here is that $\OOO$ contains
       two different $W$-orbits of hyperplanes, so that \ref{thm:Gelfand-triple}
       does not apply.  It turns out that the irreducible
       quadratic factors $(t^2-8t+8)^2$ are the characteristic polynomial
       for $\nu_\OOO$ acting on two eigenspaces that both afford 
       the reflection representation $V$ for $W$.
     \end{Remark}

   \subsection{Relation to linear ordering polytopes} 
     \label{subsec:linear-ordering-polytopes}

     We pause here to discuss a topic from discrete geometry and polytopes
     that motivated some of these explorations.
     We refer to Ziegler's book \cite{Ziegler1995}
     for basic facts and unexplained terminology from  polytope theory.

     Given the hyperplane arrangement $\AAA$, with some possible subset
     of linear symmetries $W$, and ($W$-stable) subset $\OOO$ of $\LLL$,
     note that the map from \ref{defn:chamber-localization-maps}
     $$
       \begin{array}{rcl}
         \pi_\OOO: \ZZ \CCC & \longrightarrow &\oplus_{X \in \OOO} \ZZ\CCC(\AAA/X) \\
                          c & \longmapsto & \bigoplus_{X \in \OOO} c/X
       \end{array}
     $$
     is $W$-equivariant for the natural $W$-permutation actions in the source
     and target. Clearly, $\pi_\OOO$ extends to a mapping from 
     $\RR \CCC$ to $\oplus_{X \in \OOO} \RR\CCC(\AAA/X)$, which allows a 
     definition of a new class of polytopes. Recall for the definition that the
     image of a convex polytope under a linear map is again a convex polytope. 

     \begin{Definition}
       Let $\AAA$ be an arrangement of hyperplanes and $\CCC$ its set of
       chambers.  
       Denote by $\Delta_{|\CCC|-1}$ the standard $(|\CCC|-1)$-dimensional
       simplex $\Delta_{|\CCC|-1}$ which is the convex hull of the 
       standard basis vectors  within $\RR\CCC$. The convex polytope $\Lin_\OOO$ 
       is defined to be 
       $$
         \Lin_\OOO = \pi_\OOO(\Delta_{|\CCC|-1}),
       $$
       the image of the polytope $\Delta_{|\CCC|-1}$ under the linear map $\pi_\OOO$.
       \index{standard simplex}%
       \index{simplex!standard}%
       \index{linear ordering polytope}%
       \index{polytope!linear ordering}%
       \nomenclature{$\Delta_n$}{$n$-dimensional standard simplex}%
       \nomenclature{$\Lin_\OOO$}{analog of the linear ordering polytope for $\OOO \subseteq \LLL$}%
     \end{Definition}

     Since the map $\pi_\OOO$ has all entries in $\{0,1\}$ when expressed
     with respect to the standard basis, $\Lin_\OOO$ is a {\deffont $0/1$-polytope},
     \index{$0/1$-polytope}%
     \index{polytope!$0/1$}%
     and its vertex set will simply be the distinct images (after eliminating
     duplicates) $\pi_\OOO(c)$ of the chambers $c$ in $\CCC$.  Letting
     $\AAA(\OOO)$ denote the subset of hyperplanes $H$ in $\AAA$ that contain
     at least one subspace $X$ in $\OOO$, it is easy to see that two chambers
     in $\CCC$ have distinct images under $\pi_\OOO$ if and only if they
     lie in the same chamber of the arrangement $\AAA(\OOO)$.  Thus $\Lin_\OOO$ has vertex set
     in bijection with the chambers $\CCC(\AAA(\OOO))$.

     \begin{Proposition}
       The polytope $\Lin_\OOO$ has dimension $r-1$ where
       $$
         r:=\rank \pi_\OOO = \rank \nu_\OOO = \rank \mu_\OOO.
       $$

       In particular, when $\AAA$ is a reflection arrangement and 
       $\OOO$ is a $W$-stable subset of hyperplanes $H$,
       the dimension of $\Lin_\OOO$ is the cardinality $|\OOO|$.
     \end{Proposition}
     \begin{proof}
       Consider the vector $v_\trivial:=\sum_{c \in \CCC} c$ inside $\RR\CCC$
       that has all coordinates equal to $1$, and note that its 
       image $\pi_\OOO(v_\trivial)$ within 
       $\oplus_{X \in \OOO} \RR\CCC(\AAA/X)$ is nonzero.
       On the other hand, the perpendicular space $v_\trivial^\perp$, which is
       spanned by the elements $c-c'$ for $c, c' \in \CCC$, is sent by
       $\pi_\OOO$ into the codimension one subspace of $\oplus_{X \in \OOO} \RR\CCC(\AAA/X)$
       where the sum of the coordinates is zero. This is easily checked on the 
       above spanning set for $v_\trivial^\perp$.

       This shows that $\pi_\OOO$ restricts to a linear map out of $v_\trivial^\perp$
       that has rank $r-1$, where $r$ is the rank of $\pi_\OOO$.  Since the simplex
       $\Delta_{|\CCC|-1}$ contains an open neighborhood within the affine translate of 
       $v_\trivial^\perp$
       where the sum of coordinates is $1$, the image of the simplex
       under $\pi_\OOO$ will also have dimension
       $r-1$.

       When $\AAA$ is a reflection arrangement and $\OOO$ is a $W$-stable subset of 
       hyperplanes $H$, the \BHR theory (see \ref{cor:general-equivariant-kernel} and
       \ref{ex:rank-one-kernel-image}) shows that the space perpendicular to the kernel of 
       $\pi_\OOO$ carries the $W$-representation 
       $$
         \trivial_W \oplus \left( \bigoplus_{i=1}^t \Ind_{\Centr_W(s_i)}^W \chi_i \right).
       $$
       Since the dimension of the representation $ \Ind_{\Centr_W(s_i)}^W \chi_i$
       is $[W : \Centr_W(s_i)]=|\OOO_i|$, this shows that the rank of $\pi_\OOO$ is 
       $1+\sum_{i=1}^t |\OOO_i|=1+|\OOO|$.
     \end{proof}

     \begin{Example} 
       Let $W=\symm_n$ and $\AAA$ its reflection arrangement.

       \begin{figure}[ht]
         \setlength{\unitlength}{4100sp}%
         \begin{picture}(1000,0)%
           \includegraphics{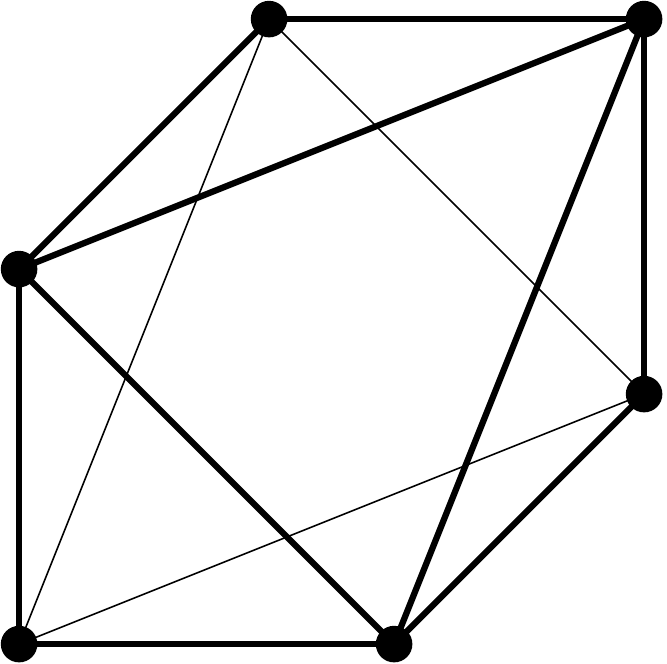}%
         \end{picture}%
         \begin{picture}(2000,3000)(5000,-2000)
           \put(7100, 0900){\makebox(0,0)[lb]{\color[rgb]{0,0,0}$123 \leftrightarrow (1,1,1,0,0,0)$}}
           \put(7100,-0829){\makebox(0,0)[lb]{\color[rgb]{0,0,0}$132 \leftrightarrow (1,1,0,0,0,1)$}}
           \put(3400, 0900){\makebox(0,0)[lb]{\color[rgb]{0,0,0}$(0,1,1,1,0,0) \leftrightarrow 213$}}
           \put(6000,-2050){\makebox(0,0)[lb]{\color[rgb]{0,0,0}$231 \leftrightarrow (1,0,0,0,1,1)$}}
           \put(2200,-0250){\makebox(0,0)[lb]{\color[rgb]{0,0,0}$(0,0,1,1,1,0) \leftrightarrow 312$}}
           \put(2200,-2050){\makebox(0,0)[lb]{\color[rgb]{0,0,0}$(0,0,0,1,1,1) \leftrightarrow 321$}}
         \end{picture}%
         \caption{Linear ordering polytope for $\symm_3$}
           \label{fig:linearorderingpolytope}
       \end{figure}
       Consider the case when $\OOO$ is the set of all hyperplanes 
       $\AAA$.  The polytope $\Lin_\OOO$ lives in a space
       isomorphic to $\RR^{n(n-1)}$ whose coordinates are indexed by
       ordered pairs $(i,j)$ with $1 \leq i \neq j \leq n$.  The vertices of $\Lin_\OOO$ 
       are labelled by the $n!$ elements of $\symm_n$ or, equivalently, the 
       different linear orders $\preceq$ on $[n]$. If we consider the vertex
       labelled by $w \in \symm_n$, then its coordinate indexed by $(i,j)$ is $1$ if 
       $w(i) < w(j)$ and $0$ otherwise. If we choose the labeling by
       linear orders, then the vertex labelled by $\preceq$ has a $1$ in 
       coordinate $(i,j)$ whenever $i \preceq j$, and $0$ otherwise.
       \ref{fig:linearorderingpolytope} shows the linear ordering polytope for $S_n$ with
       coordinates indexed by $(1,2)$, $(1,3)$, $(2,3)$, $(2,1)$, $(3,1)$, $(3,2)$.        

       Note that $\Lin_\OOO$ lies in an affine
       subspace where the sum of the $(i,j)$ and $(j,i)$ 
       coordinates is $1$.  Therefore $\Lin_\OOO$ is affinely 
       isomorphic to its projection onto the space 
       $\RR^{\binom{n}{2}}$ via the map $p$ preserving the coordinates $(i,j)$ with $i < j$, and forgetting the rest of the coordinates.

       This projection of $\Lin_\OOO$ 
       onto  $\RR^{\binom{n}{2}}$ is called the {\deffont linear ordering polytope},
       \index{linear ordering polytope}%
       and has a rich history, having appeared in several 
       guises (see \cite{Fishburn1992}), with great importance in 
       combinatorial optimization;  see e.g. \cite{GroetschelJuengerReinelt1985},
       \cite{Fiorini2006}.  Its possible first appearance was in
       mathematical psychology, where the question---phrased in our terms---was the 
       following. Consider $\Delta_{n!-1}$ as the set of all probability distributions on
       $\symm_n$ or equivalently on the set of linear orders on $[n]$. 

       \begin{Question}
         \label{qu:HuellermeierQuestion}
         Describe the set of vectors $(u_{ij})_{1 \leq i < j \leq n}$ in
         $\RR^{n \choose 2}$ for which 
         $$u_{ij} = \sum_{{{\pi \in \symm_n} \atop {\pi(i) < \pi(j)}}} \Prob (\pi)$$
         as $\Prob$ ranges over all probability distributions $\Prob \in \Delta_{n!-1}$.
       \end{Question}
       Note that in the terminology used above the set described in 
       \ref{qu:HuellermeierQuestion} is given as $p \circ \pi_\OOO (\Delta_{n!-1})$ and
       hence is the linear ordering polytope. 
       The description asked for in mathematical psychology is the same crucial question
       asked in optimization:  find a list of facet inequalities. Since it is
       known (see \cite{GroetschelJuengerReinelt1985}) that optimization of a general linear cost function 
       over the linear ordering polytope is
       NP-hard, providing a polynomial size description of its facets 
       would prove P=NP. 
       However, this suggests the following problem.

       \begin{Problem}
         \label{prob:linear-ordering}
         Let $W=\symm_n$ and $\OOO = \AAA$ the set of all
         reflecting hyperplanes so that $\Phi_\OOO=\Phi$ is
         \index{reflecting!hyperplane}%
         the set of all roots.
         Can one make use of the explicit
         $\RR W$-module orthogonal decomposition of 
         $\RR^{\Phi} = V \oplus \wedge^2 V$,
         coming from \ref{at-most-two-constituents-corollary} worked out in this special case in 
         \ref{ex:rank-one-constituents},
         as a good coordinate system in which to study the 
         polytope $\Lin_\OOO$, which is isomorphic to the linear ordering polytope?
       \end{Problem}
     \end{Example}

     \begin{Example} 
       Let $W$ be the hyperoctahedral group of all signed permutations, that is,
       the Weyl group of type $B_n$, and let $\OOO$ be the set of all reflecting
       hyperplanes. Then $\OOO$ is a union of two $W$-orbits, namely the coordinate hyperplanes
       $x_i=0$, and the hyperplanes of the form $x_i \pm x_j=0$ for $1 \leq i < j \leq n$.
       Then the polytope $\Lin_\OOO$ is affinely isomorphic to one considered by
       Fiorini and Fishburn \cite{FioriniFishburn2003}, having the linear ordering polytope 
       as one of its faces.
     \end{Example}

\section{Equivariant theory of \BHR random walks}
   \label{sec:BHR}

  It will turn out to be useful to exploit a relation
  between the operators $\nu_\OOO$ which we have been
  considering and certain operators studied by
  Bidigare, Hanlon and Rockmore \cite{BidigareHanlonRockmore1999}.  
  We begin by defining these operators, and then exhibit the special case which
  is relevant for us when considering $\nu_\OOO$ for a
  reflection arrangement $\AAA$ and $\OOO$ a single $W$-orbit.

  After this we review the (non-equivariant) aspects of the theory,
  followed by the equivariant versions that we will need,
  which are in some cases stronger than what we find in the literature,
  that is, 
  \cite{BidigareHanlonRockmore1999, BrownDiaconis1998, Brown2000, Saliola2008}.
  However, we generally borrow some of the proofs from the literature 
  directly, or in other cases, simply beef-up the techniques.
  One new feature here is the consideration of the
  extra $\ZZ_2$-action that comes from the antipodal action
  on chambers and faces of a central arrangement.

  \subsection{The face semigroup}

    Given a real, central arrangement of hyperplanes $\AAA$ in a $d$-dimensional
    real vector space $V$, we have already discussed the dissection of the
    complement $V \setminus \bigcup_{H \in \AAA} H$ into the chambers $\CCC$.
    More generally, $\AAA$ dissects $V$ into relatively open polyhedral cones
    which we will call the {\deffont faces} $\FFF$, that are the equivalence classes
    \index{faces of an arrangement}%
    \nomenclature[ar]{$\FFF$}{faces of an arrangement $\AAA$}%
    for the relation $\equiv$ having $v \equiv v'$ whenever $v$ and $v'$ lie within exactly the
    same subset of the closed half-spaces defined by all the hyperplanes $H$ in $\AAA$.

    There is a natural semigroup structure on $\FFF$ defined as follows.
    Given two faces $x, y$, define a new face $x \circ y$ 
    \nomenclature[ar]{$x \circ y$}{$x$ pulled by $y$}%
    ({\deffont $x$ pulled by $y$}) to be the unique face that one enters first
    \index{pulled face}%
    (possibly $x$ itself) when following a straight line from a point in the
    relative interior of the cone $x$ toward a point in the relative interior 
    of the cone $y$.  More formally, the face $x \circ y$ is uniquely defined
    by the properties that for each hyperplane $H$ of $\AAA$ 
    the points of $x \circ y$ lie 
    \begin{itemize}
      \item on the same side of $H$ as $x$ if $x \not\subset H$,
      \item on the same side of $H$ as $y$ if $x \subset H$, but
        $y \not\subseteq H$, and
      \item inside $H$ if $x,y \subset H$.
    \end{itemize}

    It is not hard to see that if $c$ is a chamber then $x \circ c$ is
    always a chamber, and hence $\KK \CCC$ becomes a left-ideal within
    \nomenclature[ar]{$\KK \FFF$}{left-ideal generated by $\CCC$ within $\KK \FFF$}%
    the semigroup algebra $\KK \FFF$ of the semigroup $\FFF$ with coefficients in $\KK$.
    \nomenclature[ar]{$\KK \FFF$}{semigroupalgebra of $\FFF$ with coefficients in $\KK$}%

    \begin{Definition}
      \label{defn:bhr-operators}
      We will define a {\deffont \BHR random walk} or {\deffont \BHR operator}
      \index{random!walk}%
      \index{BHR operator}%
      to be any of the family of $\KK$-linear operators on the left-ideal
      $\KK \CCC$ within $\KK \FFF$ that comes from multiplication 
      on the left by an element 
      \begin{equation}
        \label{eqn:bhr-operator}
        \sum_{x \in \FFF} p_x x
      \end{equation}
      for some $p_x$ in $\KK$.
    \end{Definition}

    Note that we are not assuming that the $p_x$ are real,
    nor even nonnegative, nor that they sum to $1$ as in the case
    of a probability distribution on the faces $\FFF$;
    for the moment, they lie in an arbitrary field $\KK$.

  \subsection{The case relevant for $\nu_\OOO$}

    When $W$ is a finite subgroup of $\GL(V)$ acting as symmetries of $\AAA$,
    it permutes the faces in $\FFF$, and it is easily seen that the 
    $W$-action respects the semigroup structure,
    that is, $w(x) \circ w(y) = w(x \circ y)$.  Thus $\KK\FFF$ becomes
    a $W \times \KK\FFF$-bimodule, as does the left-ideal $\KK\CCC$ within
    $\KK\FFF$.


    For the remainder of this subsection, assume that $\AAA$ is the
    reflection arrangement for a finite real reflection group $W$ acting on
    $V$.  As discussed in   
    \ref{subsec:reflection-group-setting}, having picked a fundamental base 
    chamber $c_\grid$, the simply transitive $W$-action on the chambers $\CCC$ leads 
    to a $W$-equivariant identification $\KK W \rightarrow \KK \CCC$ that sends 
    $w \mapsto c_w:=w(c_\grid)$.  

    Let $S$ denote the set of Coxeter generators for $W$ that
    \index{Coxeter!generators}%
    come from the reflections through the walls of $c_\grid$.
    It is well-known that every face $x$ in $\FFF$ lies
    in the $W$-orbit of a unique subface $x(J)$ of $c_\grid$, stabilized
    by the parabolic subgroup $W_J := \langle J \rangle$ for some unique
    subset $J \subseteq S$.

    Consequently, the $W$-invariant subalgebra $(\KK\FFF)^W$ of $\KK\FFF$
    will have $\KK$-basis given by the $2^{|S|}$ elements
    $$
      \left\{ \sum_{y \in x(J)^W} y \right\}_{J \subseteq S}
    $$
    where as usual $x(J)^W$ denotes the $W$-orbit of the face $x(J)$.

    The following observation, due originally to Bidigare \cite[\S3.8.3]{Bidigare1997} (see also \cite[Theorem 8]{Brown2000}), is crucial.
    For the formulation, we use the notation ${}^JW$, $W^J$ and ${}^JR$, $R^J$ from \ref{prop:parabolic-factorization}, subsequent comments and 
    \ref{subsec:second-square-root}.

    \begin{Proposition}
      \label{prop:Bidigare}
     Under the $W$-equivariant isomorphism $\KK W \rightarrow \KK\CCC$,
     multiplication on the right of $\KK W$ by
     the element $R^J:=\sum_{ u \in W^J} u$ of $\KK W$ 
     corresponds to the action on $\KK \CCC$ coming from
     multiplication on the left by the element $\sum_{y \in x(J)^W} y$ of $(\KK \FFF)^W$.
    \end{Proposition}
    \begin{proof}
      We wish to show that for each $w$ in $W$,
      $$
        \left( \sum_{y \in x(J)^W} y \right) \circ c_w= \sum_{ u \in W^J} c_{w u}.
      $$
      Acting by $w^{-1}$ on the left, and using the 
      $W$-equivariance, this is equivalent to showing
      $$
        w^{-1}\left( \sum_{y \in x(J)^W} y \right) \circ c_\grid = \sum_{ u \in W^J} c_{u}.
      $$
      Since $\sum_{y \in x(J)^W} y$ lies in $\KK \FFF^W$, this means showing
      $$
        \sum_{y \in x(J)^W} y \circ c_\grid = \sum_{ u \in W^J} c_{u}.
      $$
      On the other hand, as $W_J$ is the $W$-stabilizer subgroup for
      the face $x(J)$, and $W^J$ are coset representatives for $W/W_J$, one has
      $x(J)^W = \{ u \cdot x(J) : u \in W^J \}$.  Thus it suffices to show that
      for $u \in W^J$ one has $u \cdot x(J) \circ c_\grid = c_u$.
      This follows because 
      $$
        \begin{aligned}
           u  \in W^J & \Leftrightarrow u^{-1} \in {}^JW\\
              & \Leftrightarrow c_{u^{-1}}/X = c_\grid/X\\
              & \Leftrightarrow x(J) \circ c_{u^{-1}} = c_\grid\\
              & \Leftrightarrow u \cdot x(J) \circ c_\grid = c_u\\
        \end{aligned}
      $$
      where $X$ is the subspace fixed pointwise by $W_J$ and
      where the last step comes from applying the left-action of $u$.
    \end{proof}

    This has the following consequence.  Denote by $b_J$ 
    the linear operator on $\KK \CCC$ given by multiplication
    on the left by $\sum_{y \in x(J)^W} y$.  Let $b_J^T$
    denote its adjoint operator with respect to
    the standard inner product on $\KK \CCC$ in which the elements
    of $\CCC$ form an orthonormal basis.

    \begin{Corollary}
      \label{cor:second-square-root-is-bhr}
      Let $W$ be a finite real reflection group and $\OOO \subseteq \LLL$ a single
      $W$-orbit of intersection subspaces.  Choose a representative
      subspace $X_0$ for the $W$-orbit $\OOO$ that contains a face
      $x(J)$ of the fundamental chamber $c_\grid$, for some $J \subseteq S$.

     Then the action of $\nu_\OOO$ multiplying $\KK W$ on the right
     corresponds under the $W$-equivariant isomorphism $\KK W \rightarrow \KK \CCC$
     to the operator $\frac{1}{n_{X_0}}b_J^T b_J$.
     In particular, $\nu_\OOO$ and $b_J$ share the same kernel.
   \end{Corollary}

   \begin{proof}
     \ref{prop:second-square-root}
      asserts that as an element of $\KK W$ one has
      $\nu_\OOO = \frac{1}{n_{X_0}} R^{X_0} \cdot {}^{X_0}R$. Here
      we choose minimal length coset representatives $W^{X_0}$ and ${}^{X_0}W$ 
      in the definition of $R^{X_0}$ (see \ref{subsec:second-square-root}). 
      Note that multiplication on the right by $R^{X_0}$ and ${}^{X_0}R$
      are adjoint with respect to the standard inner product on 
      $\KK W$ (since multiplying on the right by $w$
      and by $w^{-1}$ are adjoint).  Thus one must show that
      multiplication on the right by $R^{X_0}=R^J$ in $kW$ corresponds
      to multiplication on the left by $b_J$.  
      This is exactly \ref{prop:Bidigare}. 
    \end{proof}

  \subsection{Some nonequivariant \BHR theory}

    It is an interesting and nontrivial fact that,
    when working over $\KK =\RR$ and assuming that the
    coefficients $p_x$ are nonnegative, the 
    \BHR operators as in \ref{eqn:bhr-operator} act
    semisimply.
    We recapitulate in this section a beautiful argument for this due to
    Brown \cite{Brown2000}.  

    Brown begins with an interesting way to capture
    the {\deffont minimal polynomial} of an element $a$ in a finite-dimensional
    \index{minimal polynomial}%
    {\deffont $\KK$-algebra} $A$, via generating functions.  Recall that this minimal polynomial
    \index{algebra!$\KK$-}%
    is the unique monic polynomial $m_a(T)$ in the univariate polynomial
    ring $\KK [T]$ that generates the principal ideal which is the kernel of the
    map defined by 
    $$
      \begin{aligned}
        \KK[T] &\longrightarrow A \\
        T &\longmapsto a.
      \end{aligned}
    $$
    For the sake of factoring
    $m_a(T)$, extend coefficients to the algebraic closure $\overline{\KK}$ of $\KK$,
    that is, replace $\KK$ with $\overline{\KK}$, and replace $A$ with $\overline{\KK} \otimes_\KK A$.
    Then one can uniquely express
    $$
      m_a(T) = \prod_{i} (T-\lambda_i)^{m_i}
    $$
    for some distinct $\lambda_i$ in $\KK$ and positive integers $m_i$.

    It turns out that the roots $\lambda_i$ and multiplicities $m_i$ can be read off from a certain
    generating function
    $$
      f_a(z) := \sum_{\ell \geq 0} a^\ell z^\ell = \frac{1}{1-a \cdot z}.
    $$
    We claim that $f_a(z)$ makes sense an element of $A \otimes_\KK \KK\llbracket z \rrbracket$:
    if we choose for $A$ some $\KK$-basis $\{a_j\}_{j=1,2,\ldots,t}$, then expressing each
    power $a^{\ell}$ uniquely as $a^{\ell} = \sum_{j=1}^t c_{\ell,j} a_j$
    one has
    $$
      f_a(z) = \sum_{j=1}^t a_j \otimes f_{a,j}(z)
    $$
    where $f_{a,j}(z)=\sum_{\ell} c_{\ell,j} z^\ell$ lies in $\KK \llbracket z\rrbracket$.

    \begin{Proposition}
      \label{Browns-criterion}
      In the above setting, each coefficient 
      $f_{a,i}(z)$ in $\KK\llbracket z\rrbracket$ is a rational function in $z$,
      that is, it lies in $\KK(z)$.  Furthermore,
      one can recover the roots $\lambda_i$ and multiplicities $m_i$ 
      in the minimal polynomial of $a$ from the location and orders
      of poles in the partial fraction expansion of the related function
      $$
        \begin{aligned}
          g_a(z)& := \frac{1}{z} f_a\left(\frac{1}{z}\right)= \frac{1}{z-a} \\
                &=\sum_{i} \left( \frac{b_i^0}{z-\lambda_i}
                    + \frac{b_i^1}{(z-\lambda_i)^2}
                    + \cdots
                    + \frac{b_i^{m_i-1}}{(z-\lambda_i)^{m_i}}
                      \right)
        \end{aligned}
      $$
      where the $b_i$ are some elements of $A$ satisfying
      $b_i^{m_i} = 0$ but $b_i^{m_i-1} \neq 0$.
    \end{Proposition}
    \begin{proof}
      The Chinese Remainder Theorem
      says that the subalgebra $R$ of $A$ generated by $a$ is isomorphic 
      as an algebra to the product $\prod_{i} \KK[T]/(T-\lambda_i)^{m_i}$.
      From this one can immediately reduce to the case where the minimal polynomial
      has only a single factor $(T-\lambda)^m$.
      Then one can express $a=\lambda +b$ where $b^m=0$ but $b^{m-1} \neq 0$, and
      compute directly that
        \begin{align*}
          g_\phi(z) &= \frac{1}{z-a} = \frac{1}{(z-\lambda)- b} 
              = \frac{1}{(z-\lambda)} \cdot \frac{1}{1- \frac{b}{z-\lambda}} \\
          &= \frac{1}{(z-\lambda)} \sum_{\ell \geq 0} \frac{b^\ell}{(z-\lambda)^\ell} \\
          &= \frac{b^0}{z-\lambda}
               + \frac{b^1}{(z-\lambda)^2}
               + \cdots +
               + \frac{b^{m-1}}{(z-\lambda)^{m}}.
          \qedhere
        \end{align*}
    \end{proof}

    Brown then applies this criterion to elements $a$ of the face semigroup algebra $A=\RR \FFF$, and
    more generally for semigroup algebras $\KK \FFF$ of semigroups $\FFF$ that 
    satisfy the {\deffont left-regular band} axioms:
    \index{left-regular band}%
    $$
      x^2=x  \qquad \text{ and } \qquad xyx=xy.
    $$
    He shows that for any left-regular band $\FFF$, one recovers a
    semilattice\footnote{Brown orders the semilattice $\LLL$ using the opposite
    order that we have chosen here. In particular, he orders intersection
    subspaces of a hyperplane arrangement by inclusion rather than
    reverse-inclusion.} $\LLL$,
    playing the role of the intersection lattice for the arrangement
    when $\FFF$ is the face semigroup, in the following fashion.
    Consider the preposet (reflexive, transitive, but not
    antisymmetric) structure on $\FFF$ defined by $y \preceq x$ if 
    $xy=x$, and then let $\LLL$ be the associated poset structure on the equivalence
    classes.  One obtains in this way a semilattice $\LLL$, endowed naturally 
    with a surjection $\supp: \FFF \twoheadrightarrow \LLL$, and having the
    \nomenclature[ar]{$\supp(x)$}{support of $x \in \FFF$}%
    following properties:
    $$
      \begin{aligned}
        \supp(xy) &=\supp(x) \vee \supp(y), \\
        xy &=x \text{ if }\supp(y) \leq \supp(x).
      \end{aligned}
    $$
    If the semigroup $\FFF$ has an identity element $\semigrid$, {\emphfont which we assume
    from now on}, then $\supp(\semigrid)=\hat{0}$ is a minimum element of $\LLL$.

    This leads to the following considerations for factorizations
    of elements of $\FFF$, which will help expand the generating function $f_a(z)$.

    \begin{Definition}
      Given a word $\xx :=(x_1,\ldots,x_\ell)$ in $\FFF^\ell$,
      let $\ell(\xx):=\ell$ denote its length, and let $\prod \xx :=x_1 \cdots x_\ell$
      \nomenclature[ar]{$\ell(\xx)$}{length of $\xx \in \FFF^\ell$}%
      denote its product as an element of the semigroup $\FFF$.
      Define for $i=1,2,\ldots,\ell$ the 
      elements $X_i(\xx):=\supp(x_1 x_2 \cdots x_i)$ in $\LLL$, 
      with convention $X_0(\xx)=\hat{0}$, so that
      $$
        \hat{0}=X_0(\xx) \leq X_1(\xx) \leq \cdots\leq  X_\ell(\xx)
      $$
      is a multichain in the semilattice $\LLL$.  Say that $\xx$ is {\deffont reduced} if
      \index{reduced word}%
      \index{word!reduced}%
      this multichain is actually a chain, that is, the $\{ X_i(\xx) \}_{i=0}^{\ell}$ 
      are distinct.

      Given the word $\xx$, uniquely define a reduced subword $\redword(\xx)$ of $\xx$ by repeatedly
      \nomenclature[ar]{$\redword(\xx)$}{reduced subword of $\xx$}%
      removing any letter $x_i$ for which $\supp(x_i) \leq \supp(x_1 x_2 \cdots x_{i-1})$.
      Note that $\prod \redword(\xx) = \prod \xx$ in $\FFF$.
    \end{Definition}

    From this we can now calculate the generating function $f_a(z)$ that
    determines the minimal polynomial of any element $a=\sum_{x \in \FFF} p_x x$ in
    the semigroup algebra $\KK\FFF$.
    Having fixed $a$, define
    $$
      \lambda_X:=\sum_{\substack{x \in \FFF:\\ \supp(x) \subseteq X}} p_x
    $$
    for $X$ in $\LLL$, and define
    $\pp_\xx:=p_{x_1} \cdots p_{x_\ell}$  for words $\xx=(x_1,\ldots,x_\ell)$.

    \begin{Proposition}
      Given a left-regular band $\FFF$ with identity, 
      and $a=\sum_{x \in \FFF} p_x x$ in $\KK  \FFF$, as above, one has the rational expansion
      \begin{equation}
        \label{expansion-for-g}
        g_a(z) = \sum_{\text{ reduced words }\yy} \left( \prod \yy \right) \cdot
        \frac{\pp_{\yy}}{(z - \lambda_{X_0(\yy)})(z- \lambda_{X_1(\yy)}) \cdots
               (z - \lambda_{X_{\ell(\yy)}(\yy)})}.
      \end{equation}
    \end{Proposition}

    \begin{proof}
      $$
        \begin{aligned}
          f_a(z) &= \sum_{\ell \geq 0} a^\ell z^\ell 
          = \sum_{\text{ words }\xx} z^{\ell(\xx)} \left( \prod \xx \right) \pp_{\xx} \\
          &= \sum_{\text{ reduced words }\yy} \qquad 
          \sum_{\substack{ \text{ words }\xx: \\ \redword(\xx) = \yy}} 
          z^{\ell(\xx)} \left( \prod \xx \right) \pp_{\xx} 
        \end{aligned}
      $$
      For a given reduced word $\yy=(\yy_1,\ldots,\yy_\ell)$, the
      set of all words $\xx$ having $\redword(\xx)=\yy$ is obtained by
      inserting between $y_i$ and $y_{i+1}$ an arbitrary collection of
      elements of $\FFF$ having support contained in $X_i(\yy)$;  this means
      elements of support $\hat{0}=X_0(\yy)$ can be inserted before
      $y_1$, and elements of support $X_\ell(\yy)$ after $\yy_\ell$.
      From this one concludes that
      $$
        \begin{aligned}
          f_a(z) &= \sum_{\text{ reduced words }\yy} 
            \left( \prod \yy \right) \pp_{\yy} 
              \frac{1}{1-z \cdot \lambda_{X_0(\yy)}}
              \frac{1}{1-z \cdot \lambda_{X_1(\yy)}} \cdots
              \frac{1}{1-z \cdot \lambda_{X_{\ell(\yy)}(\yy)}} \\
          &= \sum_{\text{ reduced words }\yy} \left( \prod \yy \right) \cdot
            \frac{\pp_{\yy}}{(1-z \cdot \lambda_{X_0(\yy)})(1-z \cdot \lambda_{X_1(\yy)}) \cdots
               (1-z \cdot \lambda_{X_{\ell(\yy)}(\yy)})}.
        \end{aligned}
      $$
      The formula claimed for $g_a(z):=\frac{1}{z} f\left( \frac{1}{z} \right)$ then follows.
    \end{proof}

    \begin{Corollary}
      \label{cor:bhr-semisimplicity}
      Assume $a=\sum_{x\in\FFF} p_x x$ lies in $\RR \FFF$ for a left-regular band $\FFF$,
      and that the $p_x$ are nonnegative.  Then the minimal polynomial of
      $a$ has only simple roots, contained in the set
      $\{\lambda_X\}_{X \in \LLL}$.

      In particular, $a$ generates a semisimple subalgebra of $A$, and
      $a$ acts semisimply on any finite-dimensional $A$-module $U$,
      with eigenvalue support contained in $\{\lambda_X\}_{X \in \LLL}$.
    \end{Corollary}
    \begin{proof}
      Under the above hypotheses, the only terms in the sum \eqref{expansion-for-g} for
      $g_a(z)$ that contribute with $\pp_{\yy} \neq 0$ will be
      indexed by reduced words $\yy=(y_1,\ldots,y_\ell)$ for which
      $$
        \lambda_{X_0(\yy)} < \lambda_{X_1(\yy)} < \cdots < \lambda_{X_{\ell}(\yy)}
      $$
      since $y_i$ is an element of $X_i(\yy) \setminus X_{i-1}(\yy)$ with
      $p_{y_i} > 0$ for each $i=1,2,\ldots,\ell(\yy)$.  Hence the corresponding
      product term in the summation
      $$
        \frac{1}{(z- \lambda_{X_0(\yy)})(z-\lambda_{X_1(\yy)})\cdots (z-\lambda_{X_{\ell(\yy)}(\yy)})}
      $$
      for $g_z(z)$ has only simple poles at each of these
      $\lambda_{X_i(\yy)}$.  Thus $g_a(z)$ itself has only simple
      poles, all of which are contained in the set $\{\lambda_X\}_{X \in \LLL}$.
      Now apply \ref{Browns-criterion} to conclude the asserted
      form for the minimal polynomial of $a$.  The remaining assertions are
      immediate from this.
    \end{proof}

  \subsection{Equivariant structure of eigenspaces}

    We now return to the setting of a central, essential hyperplane
    arrangement $\AAA$ in $V=\RR^d$, having $\FFF$ as its face semigroup (with
    identity, since $\AAA$ is central).  Recall that the \BHR operator may
    be thought of as the action by left multiplication of an element 
    $a=\sum_{x \in \FFF} p_x x$
    inside $\RR \FFF$ on the left-ideal $\KK \CCC$ spanned 
    $\KK$-linearly by the chambers of $\AAA$.

    Bidigare, Hanlon, and Rockmore computed the eigenvalue multiplicities.
    In their re-proof of this result, Brown and Diaconis 
    \cite{BrownDiaconis1998}
    introduced an important exact sequence\footnote{Later observed
    in \cite{Saliola2009}
    to be a projective resolution of $\KK$ as $\KK \FFF$-module.} 
    of $\KK \FFF$-modules, allowing them to compute the eigenvalue
    multiplicities for any \BHR operator inductively, using the recurrence for
    the M\"obius function of the intersection lattice $\LLL$.

    In this section, we will recall their exact sequence, and then use
    it in the {\deffont equivariant setting}, where $W$ is some 
    finite subgroup
    of $\GL(V) \cong \GL_n(\RR)$ that preserves the arrangement $\AAA$,
    to identify the $\RR W$-module structure on the \BHR-eigenspaces.

    To this end, recall that in \ref{subsec:arrangements} we defined 
    for each subspace $X$ in $\LLL$ the {\deffont localized arrangement}
    \index{localized arrangement}%
    \index{arrangement!localized}%
    $$
      \AAA/X:=\{ H / X: H \in \AAA, H \supset X \}
    $$
    inside the quotient space $V/X$, 
    having intersection lattice $\LLL(\AAA/X) \cong [V,X]_{\LLL}.$  
    Accompanying this is the 
    {\deffont restriction arrangement} of hyperplanes
    \index{restriction arrangement}%
    \index{arrangement!restriction}%
    $$
      \AAA|_X:=\{ H \cap X: H \in \AAA, H \not\supset X\}
    $$ 
    inside the subspace $X$, having intersection lattice 
    $\LLL(\AAA|_X) \cong [X,\{0\}]_{\LLL}$.
    We will use $\CCC_X$ to denote the subset of faces in $\FFF$ that 
    represent chambers of $\AAA|_X$.

    The exact sequence used by Brown and Diaconis then takes the form
    \begin{equation}
      \label{BrownDiaconisSequence}
      0 \longrightarrow \KK \FFF_d \overset{\partial_d}{\longrightarrow} \cdots 
        \longrightarrow \KK\FFF_i \overset{\partial_i}{\longrightarrow} \cdots 
        \longrightarrow \KK\FFF_1 \overset{\partial_1}{\longrightarrow} 
        \KK \FFF_0 \overset{\partial_0}{\longrightarrow} \KK \longrightarrow 0
    \end{equation}
    in which $\FFF_i$ is the set of faces $x$ in $\FFF$ for which
    $\supp(x)$ has codimension $i$.  Thus
    $$
      \KK \FFF_i =\bigoplus_{\substack{X \in \LLL:\\ \dim V/X=i}} \KK \CCC_X 
    $$
    so, for example, 
    $\KK \FFF_0 =\KK \CCC$
    and
    $\KK \FFF_1 =\bigoplus_{ H \in \AAA} \KK \CCC_H$.
    The boundary map $\partial_0$ sends each chamber $c$ of $\AAA$ to the
    same element $\grid$ in $\KK$.
    The boundary map $\partial_i$ for $i\geq 1$ sends a face $x$ to the sum 
    $$
      \sum_y[x:y] y
    $$
    where $y$ ranges over all faces containing $x$ as a codimension one subface,
    and where  $[x:y]$ are certain {\deffont incidence coefficients} 
    \index{incidence coefficients}%
    taking values $\pm 1$ defined in the following way.  First choose an
    arbitrary orientation on each subspace $X$ in $\LLL$, and then
    decree $[x:y]$ to be the sign with respect to the orientation in $\supp(y)$
    of any basis for $y$ that is obtained by appending to a positively 
    oriented basis for $\supp(x)$ any vector that points from $x$ into $y$.  
    Exactness of \eqref{BrownDiaconisSequence} follows because it is 
    essentially the complex of cellular chains
    for the regular CW-decomposition into faces of the {\deffont zonotope}
    \index{zonotope}%
    \index{polytope!zonotope}%
    having $\AAA$ as its normal fan. 

    Each $\KK  \CCC_X$ carries the structure of a (left-)$\KK \FFF$-module
    by deforming the product in $\KK \FFF$ as follows:
    for $x \in \FFF$ and $y \in \CCC_X$, set
    $$
      x \mathbin{\Diamond} y:=\begin{cases}
             xy, & \text{ if }x \subseteq X \text{ (so that }xy \in \CCC_X),\\
             0, & \text{ otherwise.}
           \end{cases}
    $$
    One can check that this makes the exact sequence \eqref{BrownDiaconisSequence}
    a complex of (left-)$\KK \FFF$-modules.  Consequently for any choice of $a = \sum_{x \in \FFF} p_x x$
    in $\KK \FFF$, it becomes an exact sequence of $\KK [T]$-modules by letting $T$ act as
    the element $a$.

    An important feature to note about this structure is that for each subspace $X$ in $\LLL$
    having $\dim V/X=i$, the subspace $\KK  \CCC_X$ inside $\KK \FFF_i$ is again a $\KK \FFF$-module
    and $\KK[T]$-module of the same type as $\KK \CCC$.  Combined with the semisimplicity
    of the $\KK [T]$-structure when $\KK =\RR$ and $p_x \geq 0$, this will allow for arguments
    about the $T$-eigenspaces by induction on $\dim V$.  For example,
    Brown and Diaconis use such an argument, along with the defining recurrence for
    the M\"obius function of $\LLL$, to show in this setting 
    that the \BHR eigenvalue $\lambda_X$ occurrs in $\KK \CCC$ with multiplicity
    $|\mu(V,X)|$, where $\mu$ denotes the M\"obius function of $\LLL$.  
    \nomenclature[ar]{$\mu(\cdot, \cdot)$}{M\"obius function}%
    \index{M\"obius function}%

    There are two preliminary observations we need before proving
    the $W$-equivariant version of this assertion.  First, note that
    in order to place the desired $\KK W$-module structure on $\KK  \CCC$,
    and to have \eqref{BrownDiaconisSequence} be a complex of $\KK W$-modules,
    each summand $\KK  \CCC_X$ inside the term $\KK \FFF_i$ with $i=\dim_\RR (V/X)$
    has to be twisted by $\det{}_{V/X}$.  This means that, as a $\KK W$-module,
    $\KK \FFF_i$ has the following description:
    $$
      \KK \FFF_i = \bigoplus_{
             \substack{ X^W \in \LLL/W:\\ 
             \dim_{\RR}(V/X)=i} }  \Ind_{W_X}^W \left( \RR \CCC_X \otimes \det{}_{V/X} \right).
    $$

    Secondly, we will need a $W$-equivariant version 
    of the M\"obius function recurrence. It can be deduced from 
    \cite[Lemma 1.1]{Sundaram1994} and \cite[Proposition 2.2]{SundaramWelker1997}. 
    However, since the proof of Proposition 2.2 in  \cite{SundaramWelker1997} 
    only invites the reader to verify that the non-equivariant proof generalizes,
    we give an explicit proof here for completeness.
    The proof proceeds via the equivariant
    generalization of a standard sign-reversing-involution proof
    for P. Hall's M\"obius function formula.

    As preliminary notation, when a group $W$ acts on a set $M$, let
    $M/W$ denote the set of $W$-orbits, with the $W$-orbit containing some element
    \nomenclature[al]{$M/W$}{set of $W$-orbits of $W$ acting on $M$}%
    $m$ of $M$ denoted by $m^W$.  Let $W_m:=\{w \in W: w(m)=m\}$ denote the
    \nomenclature[al]{$m^W$}{orbit of the element $m \in M$ under the group $W$ acting on $M$}%
    \nomenclature[al]{$\Stab_W(m)$}{stabilizer of $m \in M$ within the group $W$ acting on $M$}%
    $W$-stabilizer of $m$, so that one can identify the permutation $W$-action on 
    $m^W$ with the action of $W$ on the left cosets $W/\Stab_W(m)$ by left multiplication.  
    Let $\Gamma(\KK W)$ denote the Grothendieck
    \nomenclature[al]{$\Gamma(\KK W)$}{Grothendieck group of all virtual $\KK W$-modules}%
    group of virtual $\KK W$-modules (see \cite[\S 5.1]{Benson1998}). 
    Finally, for a poset $P$ and $X \leq Y$ in $P$ we denote by $(X,Y)$ the {\deffont open interval}
    \index{open interval}%
    \index{interval!open}%
    \nomenclature[co]{$(X,Y)$}{open interval between $X$ and $Y$ in a poset}%
    $\{ Z \in P : X < Z < Y \}$. By $\ReducedHomology^i(P;\KK)$, respectively $\ReducedHomology^i((X,Y);\KK)$, we denote the
    $i$th {\deffont reduced cohomology group} of the {\deffont order complex} of $P$, respectively $(X,Y)$.
    \nomenclature[co]{$\ReducedHomology^i( \bullet ;\KK)$}{$i$th reduced cohomology group with $\KK$ coefficients}%
    \index{cohomology!group}%
    \index{order complex}%
    \index{simplicial complex}%
    Recall that the order complex of a poset is the {\deffont simplicial complex} of all
    chains in the poset.

    \begin{Proposition}
      \label{equivariant-Mobius-recurrence}
      Let $P$ be a finite poset with bottom and top elements $\hat{0}$ and $\hat{1}$, respectively,
      and $W$ a finite group acting as a group of poset automorphisms on $P$.  Then in $\Gamma(\KK W)$
      one has
      $$
        \sum_{ \substack{ X^W \in P/W \\ i \geq -1}} (-1)^i \Ind_{\Stab_W(X)}^W \ReducedHomology^i((X,\hat{1});\KK) = 0.
      $$
      In particular, if $P$ is a poset which is Cohen-Macaulay over $\KK$, with
      \index{Cohen-Macaulay poset}%
      \index{poset!Cohen-Macaulay}%
      rank function $\rank_P(-)$, then one has
      $$
        \sum_{ X^W \in P/W } (-1)^{t(X)} \Ind_{\Stab_W(X)}^W \ReducedHomology^{t(X)}((X,\hat{1});\KK) = 0.
      $$
      where $t(X):=\rank_P(\hat{1})-\rank_P(X)-2$.
    \end{Proposition}
    \begin{proof}
      By the Hopf trace formula, it is equivalent to show that
      \index{Hopf trace formula}%
      $$
        \sum_{ \substack{ X^W \in P/W \\ i \geq -1}} (-1)^i \Ind_{\Stab_W(X)}^W \ReducedChainGroup^i((X,\hat{1});\KK) = 0,
      $$
      where $\ReducedChainGroup^i((X,\hat{1});\KK)$ is the $i$th reduced cochain group of the order complex of $(X,\hat{1})$ with
      coefficients in $\KK$%
      \nomenclature[al]{$\ReducedChainGroup^i(\bullet;\KK)$}{$i$th reduced cochain group with coefficients in $\KK$}.
      \index{cochain group}%
      \index{group!cochain}%
       Because $W$ acts by poset automorphisms, it acts on the cochain group $\ReducedChainGroup^i((X,\hat{1});\KK)$
      as a permutation representation:  the usual $\KK$-bases dual to oriented simplicial chains
      $[X_1,\ldots,X_{i+1}]$, listed in their $P$-order $X_1 < \cdots < X_{i+1}$,
      will be {\emphfont permuted without any $\pm$ sign}.

      Consequently, if we let $M$ be the set of all pairs $(X,\chain)$ where $X$ is an element
      of $P$ and $\chain=\{X_1,\ldots,X_{i+1}\}$ satisfies $X < X_1 < \ldots < X_{i+1} < \hat{1}$ 
      in $P$, then it is enough to show 
      $$
        \sum_{ {(X,\chain)^W} \in M/W } (-1)^{|\chain|} \Ind_{\Stab_W(X,\chain)}^W \trivial = 0.
      $$
      To show this, note that every $X \neq \hat{0}$ in $P$ has 
      $
       \Stab_W(X,\chain) = \Stab_W(\hat{0},\{X\} \cup \chain).
      $
      Hence the two terms 
      $\Ind_{\Stab_W(X,\chain)}^W \trivial$ and $\Ind_{\Stab_W(\hat{0},\{X\} \cup \chain)}^W \trivial$
      cancel in the sum.
    \end{proof}

    We can now state and prove our $W$-equvariant description
    of the \BHR eigenspaces when $\KK=\RR$ and the coefficients $p_x$
    in $a=\sum_{x \in \FFF} p_x x$ are chosen not only {\emphfont nonnegative},
    but also {\emphfont $W$-invariant}: 
    $$
      p_{gx} = p_x \text{ for all }g \in W, x \in \FFF.
    $$
    Of course, when $W$ is the trivial group, one recovers the
    usual theory.  One further bit of notation:  for an $\RR[T]$-module $U$,
    and an eigenvalue $\lambda$, let $U_\lambda$ denote the $\lambda$-eigenspace
    of $T$ on $U$, that is, $U_\lambda:=\ker(T-\lambda \matid_U)$.
    \nomenclature{$U_\lambda$}{$\lambda$-eigenspace of $R[T]$-module $U$}%

    \begin{Theorem}
      \label{Equivariant-BHR-theorem}
      For any choice of coefficients $\{p_x\}_{x \in \FFF}$
      which are nonnegative and $W$-invariant,
      the $\RR[T][W]$-module structure
      on $\RR \CCC$ is semisimple.  The $T$-eigenvalues are
      contained in the set $\{\lambda_X\}_{X \in \LLL}$, and 
      the $\lambda$-eigenspace has the following description
      as an element of the Grothendieck group $\Gamma(\RR W)$:
      $$
        \left( \RR \CCC \right)_\lambda
        = \sum_{\substack{X^W \in \LLL/W: \\ \lambda_X=\lambda}}
         \Ind_{\Stab_W(X)}^W \left( \ReducedHomology^*((V,X);\RR) \otimes \det{}_{V/X} \right).
      $$
    \end{Theorem}
    \begin{proof}
      \ref{cor:bhr-semisimplicity} tells us that $\RR \CCC$ is a 
      semisimple $\RR[T][W]$-module and that its $T$-eigenvalues are
      contained in the set $\{\lambda_X\}_{X \in \LLL}$.
      We claim that it suffices to show the assertion of the theorem only for
      those choices of $p_x$ which make $\lambda_V > \lambda_X$ for $X \subsetneq V$.
      First we explain why this is a valid reduction.
      Note that such choices of $p_x$ form a {\emphfont dense subset} of 
      all the relevant choices of $p_x$ in the theorem.  Also note that the theorem can
      be viewed as asserting for each $W$-irreducible $\chi$, 
      that the operator $T_\chi(p_x)$ acting on the
      $\chi$-isotypic component $\RR \CCC^\chi$ of $\RR \CCC$ has a certain
      factorization for its characteristic polynomial
      $$
        \det(t\matid_{\RR \CCC^\chi} -T_{\chi}(p_x)) = \prod_{X \in \LLL} (t - \lambda_X(p_x))^{m_X}
      $$
      with $m_X$ independent of $\{p_x\}$, but with the operators
      $T_{\chi}(p_x)$ and the eigenvalues $\lambda_X(p_x)$ depending
      {\emphfont polynomially} on the $\{p_x\}$.  If this identity holds on
      a dense set of $\{p_x\}$, it holds for all of them.

      So assume  $\lambda_V > \lambda_X$ for $X \subsetneq V$, and
      we will prove the assertion of the theorem by induction on 
      $d:=\dim_\RR(V)$.  The base case $d=0$ is easily verified.

      In the inductive step, one obtains an exact sequence of $\RR W$-modules
      by restricting the terms $\FFF_i$ in \eqref{BrownDiaconisSequence}
      to their eigenspaces $(\FFF_i)_\lambda=\bigoplus_{X} (\RR \CCC_X)_\lambda$.
      By induction, and because $\lambda_V > \lambda_X$ for $X \subsetneq V$, only the
      last two terms $\RR \FFF_0 = \RR \CCC$ and $\RR$ have a nonzero
      $\lambda_V$-eigenspace, and so they are isomorphic, proving the assertion
      for $\lambda=\lambda_V$.  For $\lambda < \lambda_V$, one finds that
      $$
        \begin{aligned}
          (\RR \CCC)_\lambda 
             &= - \sum_{i \geq 1} (-1)^i (\RR \FFF_i)_\lambda \\
             &= - \sum_{Y^W \neq \{V\} \in \LLL/W} (-1)^{\dim_\RR V/Y} 
               (\RR \CCC_Y)_\lambda \otimes \det{}_{V/Y} \\
             &= - \sum_{( Y^W , X^{\Stab_W(Y)})} 
               (-1)^{\dim_\RR V/Y}
               \left( \Ind_{\Stab_W(X,Y)}^W \ReducedHomology^*((Y,X);\RR) \otimes \det{}_{Y/X} \right)
                \otimes \det{}_{V/Y}          
        \end{aligned}
      $$
      where the last step applies the induction hypothesis to $\RR \CCC_Y$,
      and where the sum runs over pairs $( Y^W , X^{\Stab_W(Y)} )$
      in which $Y^W$ is a $W$-orbit in $\LLL$ not equal to $\{V\}$,
      and  $X^{\Stab_W(Y)}$ is a $\Stab_W(Y)$-orbit on the set
      $\{ X \in \LLL: X \subseteq Y, \lambda_X = \lambda\}$.
      Note that a set of representatives $(Y,X)$ for such pairs
      is the same as for the pairs $( X^W, Y^{\Stab_W(X)} )$
      in which $X^W$ is a $W$-orbit not equal to $\{V\}$
      on the set $\{ X \in \LLL:  \lambda_X = \lambda\}$
      and  $Y^{\Stab_W(X)}$ is a $\Stab_W(X)$-orbit not equal to $\{V\}$
      on the set $\{ Y \in \LLL: X \subseteq Y \}$.
      Consequently, using the fact that 
      $\det{}_{Y/X} \det{}_{V/X} = \det{}_{V/Y}$
      and transitivity of induction, one obtains
      $$
        \begin{aligned}
          (\RR \CCC)_\lambda 
            &= - \sum_{( X^W , Y^{\Stab_W(X)} )} 
           (-1)^{\dim_\RR V/Y} 
             \Ind_{\Stab_W(X)}^W 
                \left( \Ind_{\Stab_W(X,Y)}^{\Stab_W(X)} \ReducedHomology^*((Y,X);\RR) \otimes
                  \det{}_{V/X} \right) \\
           &= \sum_{ \substack{ X^W \in \LLL/W: \\ \lambda_X=\lambda} }
           \Ind_{\Stab_W(X)}^W \left( \ReducedHomology^*((V,X);\RR) \otimes \det{}_{V/X} \right)
        \end{aligned}
      $$
      where the last step applies \ref{equivariant-Mobius-recurrence}
      to the localized arrangement $\AAA/X$.
    \end{proof}

    Note for future use the following consequence
    of \ref{Equivariant-BHR-theorem}
    which simply ignores the $\RR[T]$-structure.

    \begin{Corollary}
      \label{cor:chamber-representation}
      For any finite subgroup $W \subset \GL(V)$ that preserves $\AAA$, one has
      in the Grothendieck group $\Gamma(\RR W)$
      $$
        \RR \CCC 
          = \sum_{X^W \in \LLL/W}
         \Ind_{\Stab_W(X)}^W \left( \ReducedHomology^*((V,X);\RR) \otimes \det{}_{V/X} \right).
      $$
    \end{Corollary}

    \ref{cor:chamber-representation} can also be seen as a special case of the equivariant version
    \cite[Theorem 2.5 (ii)]{SundaramWelker1997} of the Goresky-MacPherson formula 
    for the cohomology of the complement of a {\deffont subspace arrangement}.
    \index{subspace arrangement}%
    \index{arrangement!subspace}%
    For a hyperplane arrangement $\AAA$ invariant under $W$, the complement is the union of the
    (open) chambers in $\CCC$. Its non-reduced
    cohomology with coefficients in $\RR$ is $\RR^{|\CCC|}$ carrying 
    the representation induced by the action of $W$ on $\CCC$. 
    This reduces \cite[Theorem 2.5 (ii)]{SundaramWelker1997} to \ref{cor:chamber-representation}
    once one observes that the representation of non-reduced cohomology differs from reduced 
    cohomology by a copy of the trivial representation and the fact that the representation of 
    $W$ on the unique non-vanishing homology of $X^\perp$ intersected with a $W$-invariant sphere  
    is $\det{}_{V/X}$.

  \subsection{$(W \times \ZZ_2)$-equivariant eigenvalue filtration}

    Because we have been working with a {\deffont central} hyperplane arrangement $\AAA$,
    \index{central hyperplane arrangement}%
    \index{hyperplane arrangement!central}%
    the map on the set $\FFF$ of faces that sends $x \mapsto -x$ gives a $\ZZ_2$-action
    on $\FFF$, on $\KK\FFF$, and on the complex \eqref{BrownDiaconisSequence}.
    Furthermore, it commutes with the action of any group of symmetries $W \subset GL(V)$ of $\AAA$.

    If we only assume that $p_{gx}=p_x$ for $g \in W$ and $x \in X$, but
    make no assumption that $p_{-x}=p_x$, then in general this $\ZZ_2$-action does {\emphfont not}
    commute with the $T$-action coming from the element $a = \sum_{x \in \FFF} p_x x$.
    However, the $\ZZ_2$-action will preserve a certain natural filtration
    that comes from the $T$-eigenspaces, as we now show.  Given a semisimple
    $\KK[T]$-module $U$ having only real $T$-eigenvalues, and a real number $\lambda$,
    let 
    $$
      U_{\leq\lambda}:=\bigoplus_{\mu \leq \lambda} U_\mu.
    $$
    Note that since there are only finitely many different eigenvalues $\mu$,
    these subspaces $U_{\leq\lambda}$ as $\lambda$ increases through all
    real numbers form a finite filtration of $U$.
    For the formulation of the following result we denote by $\chi^{+} := \trivial_{\ZZ_2}$ the trivial character of $\ZZ_2$ and by $\chi^{-}$ the 
    \nomenclature[al]{$\chi^{-}$}{nontrivial character of $\ZZ_2$}%
    \nomenclature[al]{$\chi^{+}$}{trivial character of $\ZZ_2$}%
    unique nontrivial character of $\ZZ_2$.

    \begin{Theorem}
      \label{thm:Z2-equivariant-BHR-theorem}
      Let $\{p_x\}_{x \in \FFF}$ be real numbers such that
      $p_{wx} = p_x$ for all $w \in W$ and $x \in \FFF$.
      Then the $\ZZ_2$-action on $\RR \CCC$ preserves the filtration 
      of $\RR \CCC$ by $\{ (\RR \CCC)_{\leq \lambda} \}$. 
      Furthermore, in the Grothendieck group
      $\Gamma(\RR[W \times \ZZ_2])$, one has
      $$
        (\RR \CCC )_{\leq \lambda}
        = \sum_{\substack{X^W \in \LLL/X:\\\lambda_X \leq \lambda}}
        \left( \Ind_{\Stab_W(X)}^W \ReducedHomology^*((V,X);\RR) \otimes \det{}_{V/X} \right)
        \otimes \left(\chi^{-} \right)^{\otimes\dim_\RR V/X}.
      $$
    \end{Theorem}
    \begin{proof}
      We first show $(\RR \CCC )_{\leq \lambda}$ is $\ZZ_2$-stable.
      Since all eigenvalues of $T$ on $\RR \CCC$ lie in
      $\{\lambda_X\}_{X \in \LLL}$ by \ref{cor:bhr-semisimplicity}, 
      when  $\lambda \geq \lambda_V$ one has $(\RR \CCC)_{\leq \lambda}=\RR \CCC$
      and the stability is trivial.
      When $\lambda < \lambda_V$, one can induct on the dimension $d$ of the ambient space 
      (with the base case $d=0$
      trivial as before) using the exact sequence \eqref{BrownDiaconisSequence}
      restricted to the spaces $(\RR \FFF_i)_{\leq \lambda}$.
      From this restricted exact sequence one concludes that
      $$
        (\RR \CCC)_{\leq \lambda} =
        (\im \partial_1)_{\leq \lambda} =
        \partial_1 \left( \bigoplus_{H \in \AAA} (\RR \CCC_H)_{\leq \lambda} \right).
      $$
      Since $(\RR \CCC_H)_{\leq \lambda}$ is $\ZZ_2$-stable by induction,
      and since the $\ZZ_2$-action does commute with the $\partial_i$,
      this shows that $(\RR \CCC)_{\leq \lambda}$ is $\ZZ_2$-stable.

      Regarding the description of the $\RR[W \times \ZZ_2]$-module
      structure of $(\RR \CCC)_{\leq \lambda}$, one also checks this
      in two cases.  When $\lambda \geq \lambda_V$ so that 
      $(\RR \CCC)_{\leq \lambda}=\RR \CCC$, it follows by
      applying \ref{Equivariant-BHR-theorem} to the finite subgroup 
      $$
        \hat{W} := W \times \ZZ_2 = W \times \langle -\matid_V \rangle \subset GL(V)
      $$
      and noting that 
      \begin{enumerate}
        \item[$\bullet$] $\ZZ_2$ acts trivially on the lattice $\LLL$, 
        \item[$\bullet$] so in particular, $W$ and $\hat{W}$ have the same orbits on $\LLL$, and
        \item[$\bullet$] the $\ZZ_2$-characters $\left(\chi^{-} \right)^{\otimes\dim_\RR V/X}$
          and $\det_{V/X}$ are the same when the generator of $\ZZ_2$ acts by $-\matid_V$
          on $V$.
      \end{enumerate}
      When $\lambda < \lambda_V$, one proceeds by induction on $d$ using
      the exact sequence \eqref{BrownDiaconisSequence}
      restricted to the $(\RR \FFF_i)_{\leq \lambda}$,
      proceeding exactly as in the proof of \ref{Equivariant-BHR-theorem}.
    \end{proof}

    \begin{Example}
      It is worth examining the $d=1$ case of the preceding results in detail.
      Here the central arrangement $\AAA$ inside the real line $V=\RR^1$
      has two chambers $c,c'$, separated by the unique hyperplane $H=\{0\}$.
      So the face semigroup $\FFF$ is $\{H, c, c'\}$, and $H$ is the identity
      element of $\FFF$.
      The sequence \eqref{BrownDiaconisSequence} is 
      $$
        \begin{array}{ccccccc}
          0 \rightarrow & \RR \FFF_1& \overset{\partial_1}{\rightarrow} &\RR\FFF_0&
            \overset{\partial_0}{\rightarrow} & \RR & \rightarrow 0 \\
                        &\Vert& &\Vert& &\Vert&  \\
                        &\RR \{H\}& &\RR \{c,c'\}& &\RR\{\varnothing\}& 
        \end{array}
      $$
      with  $\partial_1(H)=c -c'$ and $\partial_0(c) = \partial_0(c') =\varnothing$.
      Let $T$ act by the element 
      $$
        a=p_0 H + p c + p' c'
      $$
      in $\RR \FFF$.  Then the $\RR[T]$-module 
      structure on $\RR \FFF_1$ has $T$ scaling $H$ by $p_0$,
      and on $\RR \FFF_{-1}$ has $T$ scaling $\varnothing$ by $p_0+p+p'$,
      while on $\RR \FFF_0$, the element $T$ acts in the ordered basis $(c,c')$ by
      $$
        \left[
          \begin{matrix}
            p_0+p & p \\
            p'    & p_0 + p'.
          \end{matrix}
        \right]
      $$
      Changing to an ordered basis of $T$-eigenvectors $(c-c',\,\, p c + p' c')$,
      will diagonalize the action of $T$ on $\RR \CCC$:
      $$
        \left[
          \begin{matrix}
            p_0 & 0 \\
            0   & p_0 + p + p'
          \end{matrix}
        \right].
      $$
      Note that 
      $$
        (\RR \CCC)_{\lambda_V} = (\RR \CCC)_{p_0 + p+ p'} 
                               = \RR \{p c + p' c'\}
      $$
      is not $\ZZ_2$-stable unless $p=p'$. However
      $$
        (\RR \CCC)_{\lambda_{\{0\}}} = (\RR \CCC)_{p_0}
        = \RR\{c-c'\}
      $$
      is always $\ZZ_2$-stable.
    \end{Example}

    \begin{Remark}
      \ref{thm:Z2-equivariant-BHR-theorem} suggests a conjectural
      stronger statement in the case of a reflection arrangement $\AAA$
      corresponding to reflection group $W$,
      tying in with the work of
      Hanlon and Hersh \cite{HanlonHersh2004} in type $A$.  We discuss this
      briefly here. 

      For a reflection arrangement, one can identify the $W \times \ZZ_2$-action on
      $\KK \CCC$ discussed above with the $W \times \ZZ_2$-action on $\KK W$ 
      where $W$ acts via left-translation
      and the generator of $\ZZ_2$ acts via right-translation by $w_0$.
      Note that here the face semigroup $\KK \FFF$ also acts on the left on the ideal
      $\KK \CCC$ inside $\KK \FFF$.

      Since $w_0$ is the unique element of $W$ having descent set all of $S$,
      it not only lies in the group algebra $\KK W$, but also inside
      the {\deffont descent algebra}, which is spanned by the sums over $w$ in $W$ having a 
      \index{descent!algebra}%
      \index{algebra!descent}%
      {\deffont fixed descent set}.  
      \index{descent!set}%
      Furthermore in type $A_{n-1}$,
      it lies in a subalgebra called the {\deffont Eulerian subalgebra},
      \index{Eulerian!subalgebra}%
      \index{subalgebra!Eulerian}%
      \index{algebra!Eulerian sub-}%
      spanned by the sums over $w$ in $W$ having%
      \footnote{Note that Hanlon and Hersh \cite{HanlonHersh2004} and other authors
       put a coefficient of $\det(w)$ in front of each $w$ in the sum. Thus for our purposes we need to 
       twist by the automorphism $w \mapsto \det(w) \cdot w$ in order to compare our \BHR
       operators with the {\emphfont signed} random-to-top shuffle operator they are using.} 
      a {\deffont fixed number of descents.}  There is a complete system of orthogonal
      \index{descent}%
      \index{descent!number}%
      \index{number of descents}%
      idempotents for this Eulerian subalgebra known as the {\deffont Eulerian idempotents}
      \index{Eulerian!idempotent}%
      $\{\idem^{(j)}_n\}_{j=1,2,\ldots,n}$ defined by the generating function
        \begin{align}
        \label{eulerianidempotents}
        \sum_{j=1}^n \idem^{(j)}_n t^j = \frac{1}{n!}\sum_{\sigma\in\symm_n}
        (t - \operatorname{des}(\sigma))^{\uparrow{n}} \sigma,
        \end{align}
      where $\operatorname{des}(\sigma)$ is the number of descents of the
      permutation $\sigma \in \symm_n$ and $t^{\uparrow{n}}$ denotes the
      {\deffont increasing factorial} $t^{\uparrow{n}} = t(t+1)(t+2) \cdots (t+n-1)$.
      These idempotents decompose spaces $U$ on which the Eulerian subalgebra
      acts into their {\deffont Hodge decomposition} $U = \oplus_j U\idem^{(j)}_n$,
      \index{Hodge decomposition}%
      and have the property that 
      $$
        \begin{aligned}
          1  &= \sum_{j=1}^n \idem^{(j)}_n,\\
          (-1)^n w_0&= \sum_{j=1}^n (-1)^j \idem^{(j)}_n.
        \end{aligned}
      $$
      (These identifies can be proved by taking $t=1$ and $t=-1$, respectively,
      in \ref{eulerianidempotents}.) Consequently, the two projectors onto the
      $\chi^+$ and $\chi^-$-isotypic components for the group $\ZZ_2=\{1,w_0\}$
      can be expressed as
      $$
        \begin{aligned}
          \frac{1}{2}\left(1+(-1)^nw_0\right)&= \sum_{j \text{ even}} \idem^{(j)}_n\\
          \frac{1}{2}\left(1-(-1)^nw_0\right)&= \sum_{j \text{ odd}} \idem^{(j)}_n.\\
        \end{aligned}
      $$
      In light of this, the following result generalizes \ref{Equivariant-BHR-theorem}
      as well as the results of \cite[Section 2]{HanlonHersh2004}.

      \begin{Theorem}[Saliola]
        Let $\AAA$ be a hyperplane arrangement, $\LLL$ its intersection
        lattice (ordered by reverse-inclusion, as usual), and $\FFF$ its
        semigroup of faces.
        \begin{enumerate}
          \item[(i)]
            There is a natural filtration of $\KK\CCC$ by $\KK\FFF$-modules
            indexed by the order ideals of $\LLL$. Explicitly, there is an
            inclusion-reversing map $\III \mapsto U_\III$ where $\III$ is any
            order ideal of $\LLL$ and
            \begin{align*}
            U_\III := \{ a \in \KK\CCC : xa = 0 \text{ for every } x \in \FFF \text{ with } \supp(x) \in \III \}.
            \end{align*}
          \item[(ii)]
            For any choice of $a \in \RR\FFF$ giving a $\RR[T]$-module
            structure on $\RR\CCC$, this poset-filtration refines
            the $T$-eigenvalue filtration $(\RR \CCC)_{\leq \lambda}$
            in the following fashion:
            one has $(\RR \CCC)_{\leq \lambda}=U_\III$ for the
            order ideal $\III=\{X \in \LLL: \lambda_X \leq \lambda\}$.
          \item[(iii)]
            Now assume $\AAA$ is the type $A_{n-1}$ reflection arrangement and the
            $p_x$ are chosen $W$-invariant, so that the $W$-invariant subalgebra
            of $\RR\FFF$ acts on $\RR\CCC \cong \RR W$, and
            can be identified with the descent algebra acting on the right
            \index{descent!algebra}%
            \index{algebra!descent}%
            within $\RR W$.  Then for any $W$-stable order ideal $\III$ of $\LLL$,
            the $j^{th}$ Hodge decomposition component
            $U_\III \idem^{(j)}_n$ of the poset-filtration space $U_\III$
            carries $\RR W$-module structure isomorphic to
            $$
              U_\III \, \idem^{(j)}_n \cong \sum
              \Ind_{\Stab_W(X)}^W \left(\ReducedHomology^*((V,X);\RR) \otimes \det{}_{V/X}\right),
            $$
            where the sum ranges over all $W$-orbits $X^W$ in $\LLL/W$ with $X\in\III$
            and $\dim_\RR(X)=j$.
        \end{enumerate}
      \end{Theorem}
      A proof of this theorem would lead us too far afield; it will be
      published separately \cite{Saliola-EulerianIdempotents}.
    \end{Remark}

  \subsection{Consequences for the kernels}

    For $\AAA$ a real hyperplane arrangement and $W$ a finite group
    of linear symmetries, introduce a notation for the following $\RR W$-modules
    that recur in the $W$-equivariant \BHR theory:
    \begin{equation}
      \label{eqn:Whitney-homology}
         \WH_{\OOO_X} = \Ind_{\Norma_W(X)}^W \left(\ReducedHomology^*((V,X);\RR) \otimes \det{}_{V/X}\right)
    \end{equation}
    where $\OOO_X := X^W$ is the $W$-orbit on $\LLL$ represented by
    \nomenclature[ar]{$\OOO_X$}{orbit of subspace $X \in \LLL$ under group $W$}%
    \nomenclature[ar]{$\WH_{\OOO_X}$}{$\Ind_{\Norma_W(X)}^W \ReducedHomology^*((V,X);\RR) \otimes \det{}_{V/X}$}%
    some subspace $X$.
    The module $\WH_{\OOO_X}$ is {\emphfont almost} a submodule of the {\deffont Whitney cohomology} $\WH^*(P;\RR)$ with
    \index{Whitney cohomology}%
    \index{cohomology!Whitney}%
    \nomenclature[ar]{$\WH^*(P;\RR)$}{Whitney cohomology of a poset $P$ with real coefficients}%
    real coefficients of a poset $P$ with unique minimal element $\hat{0}$. The latter was introduced by
    Baclawski \cite{Baclawski1975} by truncating the usual differential of the simplicial cochain complex.
    It follows that $\WH^* (P;\RR) := \bigoplus_{p \in P} \ReducedHomology^*((\hat{0},p);\RR)$. 
    From the definition it is obvious that if a finite group $W$ acts as a group of 
    poset automorphisms on $P$ then $\WH^*(P;\RR)$ becomes a $W$-module with submodule
    $\bigoplus_{p \in \OOO} \ReducedHomology^*((\hat{0},p);\RR) \cong \Ind_{\Stab_W(q)}^W \ReducedHomology^*((\hat{0},q);\RR)$ 
    for any $W$ orbit $\OOO$ of $P$ and $q \in \OOO$.
    Clearly, if $W$ is a finite subgroup of $\GL(V)$ acting on $\AAA$ then $W$ acts on $\LLL$ and except for the
    twist with $\det{}_{V/X}$ our module $\WH_{\OOO_X}$ coincides with a submodule of the Whitney cohomology of $\LLL$.
    We have chosen this twist since if facilitates the formulation of our applications of Whitney cohomology.

    Also, define a partial order on the $W$-orbits $\OOO$ in $\LLL/W$
    by setting $\OOO \leq \OOO'$ if there exist representatives $X, X'$
    in $\OOO,\OOO'$ with $X \leq X'$ in $\LLL$, that is $X' \subseteq X$.

    \begin{Corollary}
      \label{cor:general-equivariant-kernel}
      For $\OOO \subseteq \LLL$ a single $W$-orbit, one has
      $$
        \begin{aligned}
          \ker(\nu_\OOO) & \cong 
          \bigoplus_{\substack{\OOO' \in \LLL/W:\\ \OOO' \not\leq \OOO}} 
          \WH_{\OOO'} \otimes (\chi^-)^{\otimes \dim_\RR(V/X)} \\
          \im(\nu_\OOO) & \cong 
          \bigoplus_{\substack{\OOO' \in \LLL/W: \\  \OOO' \leq \OOO}} 
          \WH_{\OOO'} \otimes (\chi^-)^{\otimes \dim_\RR(V/X)} \\
        \end{aligned}
      $$
      as $\RR[W \times \ZZ_2]$-modules.
    \end{Corollary}
    \begin{proof}
      By the semisimplicity of the self-adjoint operator
      $\nu_\OOO$ acting on $\RR \CCC$, together
      with \ref{cor:chamber-representation},
      it suffices to prove the assertion about $\ker(\nu_\OOO)$.

      \ref{cor:second-square-root-is-bhr} tells us that
      $\ker(\nu_\OOO) = \ker(b_J)$ where $b_J$ is a \BHR-operator
      that has $p_x > 0$ if and only if $x$ is in the $W$-orbit of
      some particular face $x(J)$ whose support subspace $X_0$ 
      lies in the $W$-orbit $\OOO$.  Since $\ker(b_J) = (\KK\CCC)_{\leq 0}$
      for this \BHR-operator $b_J$, one deduces from 
      \ref{thm:Z2-equivariant-BHR-theorem} that
      its kernel carries the  $\RR[W \times \ZZ_2]$-module
      structure which is the sum of 
      $\WH_{\OOO'} \otimes (\chi^-)^{\otimes \dim_\RR(V/X)}$
      over those $W$-orbits $\OOO'$ for which each representative
      subspace $X$ has $\lambda_X=\sum_{x \subset X}p_x = 0$.
      This occurs if and only if each subspace $X$ in $\OOO'$ contains
      no face in the $W$-orbit of $x(J)$, which occurs if and only if
      $X$ contains no subspace in the $W$-orbit $\OOO$ of $X_0$,
      which occurs if and only if $ \OOO' \not\leq \OOO$.
    \end{proof}

    \begin{Example}
      \label{ex:rank-one-kernel-image}
      When $\AAA$ is the reflection arrangement for a finite real reflection
      group $W$, and $\OOO$ is the $W$-orbit of an intersection subspace
      having low codimension, \ref{cor:general-equivariant-kernel}
      says that $\ker \nu_\OOO$ will be large, and $\im \nu_\OOO$ small.

      In particular, if $\OOO$ is the $W$-orbit of some hyperplane $H$
      corresponding to a reflection $s$, then $\im \nu_\OOO$ is the following sum
      over two $W$-orbits: $\OOO$ itself and
      the singleton orbit $\{V\}$.
      $$
        \im \nu_\OOO \cong 
        \left( \Ind_{\Centr_W(s)}^W \det{}_{V/H} \otimes \chi^- \right)
        \oplus
        \left( \trivial_W \otimes \chi^+ \right).
      $$ 
    \end{Example}

    \begin{Example}
      \label{ex:two-type-A-examples}
      For future use in \ref{sec:second-family} and \ref{sec:original-family}, 
      we wish to discuss two further
      examples in which $\AAA$ is the reflection arrangement of
      type $A_{n-1}$ with $W=\symm_n$.
      Recall from \ref{ex:first-type-A} that an intersection
      subspace $X$ here corresponds to the 
      set partition $[n] = \bigsqcup_i B_i$ whose
      blocks $B_i$ tell us which coordinates $x_j$ are equal on the subspace $X$.
      The $W$-orbit $\OOO_X$ is then determined by the number partition $\lambda$ of
      $n$ whose parts $\lambda_i$ are the weakly decreasing reorderings of the block sizes $|B_i|$.
      Let $X_\lambda$ be any representative of this $W$-orbit indexed by $\lambda$.
      Note that $\dim_\RR(V/X_\lambda)=n-\ell(\lambda)$
      where $\ell(\lambda)$ is the number of parts of $\lambda$.

      If $\OOO, \OOO'$ are the orbits of $X_\lambda, X_\mu$, one finds
      that $\OOO \leq \OOO'$ if and only $\mu$ refines $\lambda$, that is,
      if one can combine some of the parts of $\mu$ to obtain $\lambda$.

      Therefore, \ref{cor:general-equivariant-kernel}
      implies that if $\OOO$ is the orbit of $X_{\lambda}$ then
      $\ker \nu_\OOO$ and $\im \nu_\OOO$, respectively, are the sums of 
      $\WH_{\OOO_\mu} \otimes (\chi^-)^{\otimes n-\ell(\mu)}$
      over all $\mu$ which do or do not, respectively, refine $\lambda$.

      An interesting instance is when $\lambda=(n-k,1^k)$, and
      the set of $\mu$ which {\emphfont do not} refine $\lambda$ are those
      $\mu$ that have at most $k-1$ parts of size $1$.  When $k=1$,
      this is the set of all $\mu$ having no parts of size $1$.

      Another interesting instance is when $\lambda=(2^k,1^{n-2k})$, and
      the set of $\mu$ which {\emphfont do} refine $\lambda$ are those
      of the form $\mu=(2^j,1^{n-2j})$ for $j \leq k$.
    \end{Example}

  \subsection{Reformulation of $\WH_{\OOO_X}$}

    When $\AAA$ is the reflection arrangement for a finite real reflection
    group $W$, the representation  $\WH_{\OOO_X}$ in \eqref{eqn:Whitney-homology}
    has some well-known extra features
    and reformulations, which we discuss below.  When $W=\symm_n$, 
    there are even more reformulations, also discussed below, some of
    which will be used in \ref{sec:second-family} and \ref{sec:original-family}.

    \subsubsection{Reformulations for any reflection group}

      Our first reformulation originates in topology.  Let $\AAA$ be
      an arrangement of (real) hyperplanes in $V=\RR^d$. The chambers $\CCC$ of $\AAA$
      are the connected components of the complement. Since each of them is easily seen
      to be contractible the complement is not an {\emphfont interesting} topological space.
      One gains interesting topology when one extends scalars to $\CC$ and considers the arrangement
      $\AAA \otimes \CC$ of complex hyperplanes $H \otimes \CC$, $H \in \AAA$, in $\CC^d$ defined by the
      same linear forms as $H$.
      We call $\AAA \otimes \CC$ the {\deffont complexification} of $\AAA$ in $V_\CC = \CC^d$.
      The
      \index{complexification of an arrangement}%
      \index{arrangement!complexification}%
      {\deffont complexified complement}
      \index{complexified complement}%
      \index{arrangement!complexified complement}%
      $$
        M_\AAA:= \CC^d \setminus \bigcup_{H \in \AAA} H \otimes \CC
      $$
      is a rich and complicated mathematical object (see for example \cite{OrlikTerao1992}).
      It has cohomology algebra $\OS(\AAA):=\Homology^*(M_\AAA;\RR)$ described by
      \nomenclature[ar]{$\OS(\AAA)$}{Orlik-Solomon of $\AAA$}%
      the {\deffont Orlik-Solomon presentation} \cite[Theorem 5.2]{OrlikSolomon1980a},
      which we now recall.
      \index{Orlik-Solomon!presentation}%
      \index{cohomology!algebra}%
      \index{algebra!cohomology}%
      
      Choose for each hyperplane $H \in \AAA$ a linear form $\ell_H : \CC^d \rightarrow \CC$ with 
      kernel $H \otimes \CC$. Then there is an $\RR$-algebra surjection from the exterior
      algebra $\bigwedge \AAA$ over $\RR$ on generators $\{\unitbase_H\}_{H \in \AAA}$
      onto the cohomology algebra $\Homology^*(M_\AAA;\RR)$
      \index{exterior algebra}%
      \index{algebra!exterior}%
      $$
        \begin{aligned}
          \bigwedge(\AAA) & \longrightarrow \Homology^*(M_\AAA;\RR) \\
          \unitbase_H & \longmapsto \frac{d \ell_H}{\ell_H}
        \end{aligned}
      $$
      whose kernel is generated by the elements
      $$
        \sum_{s=0}^t (-1)^s \unitbase_{H_1} \wedge \cdots \wedge 
         \widehat{\unitbase_{H_s}} \wedge \cdots \wedge \unitbase_{H_t}
      $$
      as $\{H_1,\ldots,H_t\}$ runs through all (minimal) 
      subsets of hyperplanes in $\AAA$ that are dependent (in the sense
      that $\bigcap_{i=1}^t H_i$ has codimension strictly less than $t$).
      The algebra $\OS(\AAA)$ is called the {\deffont Orlik-Solomon algebra} of
      $\AAA$.
      \index{algebra!Orlik-Solomon}%
      \index{Orlik-Solomon!algebra}%
      Note that the result by Orlik and Solomon holds even for
      integer coefficients. We use coefficients in the real numbers 
      since we will consider the Orlik-Solomon algebra as a module.

      The above presentation of the Orlik-Solomon algebra leads to a direct sum decomposition 
      of $\OS(\AAA)$ that comes
      from the subspaces $\OS(\AAA)_X$ which are the images of the
      decomposable wedges $\unitbase_{H_1} \wedge \cdots \wedge \unitbase_{H_t}$ having
      $\bigcap_{i=1}^t H_i=X$ for some fixed $X$ in $\LLL$:
      $$
        \OS(\AAA) = \bigoplus_{X \in \LLL} \OS(\AAA)_X.
      $$

      \begin{Proposition}[Theorem 5.2 \cite{OrlikSolomon1980a} and Lemma 2.5 \cite{LehrerSolomon1986}]
        For any arrangement $\AAA$ in $V \cong \RR^d$ and subspace $X$ in $\LLL$, there is a
        natural isomorphism
        $$
          \ReducedHomology^*((V,X);\RR) \cong \OS(\AAA)_X.
        $$
        Consequently, when a finite subgroup $W$ of $\GL(V)$ acts on $\AAA$, one has
        $$
          \Ind_{\Norma_W(X)}^W \ReducedHomology^*((V,X);\RR) \cong \Ind_{\Norma_W(X)}^W \OS(\AAA)_X.
        $$
        In particular,
        $$
          \WH_{\OOO_X} \cong \Ind_{\Norma_W(X)}^W \OS(\AAA)_X \otimes \det{}_{V/X}.
        $$
      \end{Proposition}

      In the case of reflection arrangements, the dimension of 
      of  $\WH_{\OOO_X}$ has a well-known reformulation.

      \begin{Proposition}[Lemma 4.7 \cite{OrlikSolomon1980b}]
        \label{moebius-proposition}
        For a finite real reflection group $W$ acting on the arrangement   
        $\LLL$ in $V = \RR^d$, and 
        for any intersection subspace $X$ in $\LLL$, one has 
        $$
          \mu(V,X) = (-1)^{\dim V/X} \Big|\big\{ w \in W: \Fix_w(V)=X \big\}\Big|,
        $$
        where $\Fix_w(V)$ is the set of elements in $V$ fixed by the action of $w$.
        \nomenclature[vw]{$\Fix_w(M)$}{set of elements in $M$ fixed by the action of $w$}%
        Consequently, 
        $$
          \dim \WH_{\OOO_X} = \Big|\big\{w \in W: \Fix_w(V) \in \OOO_X \big\}\Big|.
        $$
      \end{Proposition}

      This reformulation suggested an interesting conjecture
      of Lehrer and Solomon, which they verified in \cite{LehrerSolomon1986}
      for $W=\symm_n$ of type $A_{n-1}$ (see also \ref{pr:A-case-ls-conjecture}) 
      and also for dihedral groups $W=I_2(m)$.
      Note that the subset $\{w \in W: \Fix_w(V) \in \OOO_X \}$ of $W$ is stable under
      conjugation, and hence a union of $W$-conjugacy classes.

      \begin{Conjecture}[Conjecture 1.6 \cite{LehrerSolomon1986}]
        \label{cj:lehrersolomon}
        There is an isomorphism of $W$-modules
        $$
          \WH_{\OOO_X} \cong \bigoplus_{v} \Ind_{\Centr_W(v)}^W \xi_v
        $$
        where $v$ runs over a system of representatives of the
        $W$-conjugacy classes which comprise $\{w \in W: \Fix_w(V) \in \OOO_X \}$,
        and $\xi_v: \Centr_W(v) \rightarrow \CC^\times$ is a degree one
        character of $\Centr_W(v)$.
      \end{Conjecture}

    \subsubsection{Reformulations in type $A$}

      When $W=\symm_n$, one can both
      \begin{enumerate}
        \item[$\bullet$] 
          be much more explicit in the Lehrer-Solomon reformulation, and
        \item[$\bullet$]
          tie this in with other interesting reformulations, involving
          Lyndon words, free Lie algebras, etc.
          \index{word!Lyndon}%
          \index{Lie!algebra!free}%
          \index{Lyndon!word}%
      \end{enumerate}

      As explained in \ref{ex:first-type-A}, an intersection
      subspace $X$ for $W=\symm_n$ will correspond
      to a set partition $[n] = \bigsqcup_i B_i$ of $[n]$,
      and its $W$-orbit $\OOO_X$ is determined by the number partition $\lambda$ of
      $n$ whose parts give the block sizes $|B_i|$.
      Let $X_\lambda$ be any representative of this $W$-orbit indexed by $\lambda$,
      and say that $\lambda$ contains the part of size $j$ with multiplicity $m_j$ for each $j$.
      A typical element $v_\lambda \in W$ having $V^{v_\lambda}=X_\lambda$ will be 
      a product of disjoint cycles of sizes $\lambda_i$ supported on the blocks $B_i$.
      One then has that
      $$
        \begin{aligned}
          \Norma_W(X_\lambda) & \cong & \prod_i \symm_{m_i}[ \symm_i ] \\
          \Centr_W(v_\lambda) &\cong & \prod_i \symm_{m_i}[ \ZZ_i ].
        \end{aligned}
      $$
      where $\symm_m[G]$ denotes the {\deffont wreath product} of $G$ with
      \index{wreath product}%
      the symmetric group $\symm_m$, and $\ZZ_i$ denotes the cyclic group of order $i$.

      In \cite{LehrerSolomon1986} Lehrer and Solomon describe a degree one character 
      $\xi_\lambda$ of $\Centr_W(v_\lambda)$ that sends each of the
      disjoint cycles of size $j$ in $v_\lambda$ to the same primitive $j^{th}$ root of
      unity. This character fits the motivating type $A$ case of their \ref{cj:lehrersolomon}
      which they prove in their paper.

      \begin{Proposition}[Theorem 4.5 \cite{LehrerSolomon1986}] 
        \label{pr:A-case-ls-conjecture}
        Let $W=\symm_n$ and $\lambda$ a partition of $n$.
        For an element $v_\lambda \in W$ such that $V^{v_\lambda}=X_\lambda$ we have
        $$
          \WH_{\OOO_{X_\lambda}} \cong \Ind_{\Centr_W(v_\lambda)}^W \xi_\lambda.
        $$
      \end{Proposition}

      One has a reformulation of the previous proposition in the
      language of symmetric functions; see \cite{Macdonald1998} and
      \index{symmetric function}%
      \cite[Chapter 7 (see in particular Exercise 7.89)]{Stanley1999}
      for the basic facts used here.  

      Recall Frobenius's
      \index{Frobenius!characteristic map}%
      {\deffont characteristic map} $\ch$ giving an isomorphism between
      \index{characteristic map}%
      virtual $\symm_n$-characters and symmetric functions 
      of degree $n$.
      Under this isomorphism, one has
      $$
        \ch\left( \Ind_{ \symm_{n_1} \times \symm_{n_2} }^{\symm_{n_1+n_2}} 
             \chi_1 \otimes \chi_2
           \right)
          = \ch(\chi_1) \cdot \ch(\chi_2)
      $$
      where the product on the right is in the ring of symmetric functions.
      Also, given representations $U,V$ of $\symm_m,\symm_n$, one can
      construct a representation $U[V]$ of $\symm_m[\symm_n]$ on $V^{\otimes m} \otimes U$
      by having $(\symm_n)^m$ act on $V^{\otimes m}$ in the usual way, while
      $\symm_m$ permutes the tensor factors of $V^{\otimes m}$ and
      simultaneously acts on $U$; this construction is well-known 
      (see for example \cite[I.8 Remark 2, p. 136]{Macdonald1998}) 
      to have 
      $$
        \ch U[V] = \ch U [ \ch V ]
      $$
      where $f[g]$ denotes the {\deffont plethysm}
      \index{plethysm}%
      operation on symmetric functions $f, g$.
      We denote by $h_m := \ch (\trivial_{\symm_m})$ 
      the $m$-th {\deffont homogeneous symmetric function}.
      \index{symmetric function!homogeneous}%
      \index{homogeneous symmetric function}%
      \nomenclature[co]{$h_m$}{$m$-th homogeneous symmetric function}%

      For $W=\symm_n$, let $\omega_n$ denote the $W$-representation 
      $\WH_{\OOO_{X_{(n)}}}$, corresponding to the subspace $X_{(n)}$ which
      is the intersection of all the hyperplanes.  Thus 
      by definition of $\WH_{\OOO_X}$ the representation $\omega_n$ is the 
      $\symm_n$-representation on the top homology of the order complex
      \index{order complex}%
      of the 
      proper part of the lattice
      of all set partitions of $n$, twisted by the character $\det_V=\sgn$.
      \index{lattice!set partitions}%
      Based on work of Hanlon \cite{Hanlon1981} it was shown by Stanley \cite[Theorem 7.3]{Stanley1982} that
      $\omega_n = \Ind_{\ZZ_n}^{\symm_n} \exp(\frac{2\pi i }{n})$ is induced from a degree one character of the
      cyclic group $\ZZ_n$ generated in $S_n$ by a fixed $n$-cycle with character value $\exp(\frac{2\pi i }{n})$.
      From this the case $W = \symm_n$ and $\OOO_X = \OOO_{X_{(n)}}$ of \ref{pr:A-case-ls-conjecture} follows.
      This fact implies that Proposition 4.4 (iii) from \cite{LehrerSolomon1986} proved along the way to
      \ref{pr:A-case-ls-conjecture} actually can be stated as follows.

      \begin{Proposition}
        \label{prop:plethysm-reformulation}
        For $W=\symm_n$ and $\lambda$ any partition of $n$,
        $$
          \WH_{\OOO_{X_\lambda}} \cong 
          \Ind_{\symm_{m_1}[\symm_1] \times \symm_{m_2}[\symm_2] \times \cdots}^{\symm_n}
            \left( 
                   \trivial_{\symm_{m_1}}[ \omega_1 ] \otimes 
                   \trivial_{\symm_{m_2}}[ \omega_2 ] \otimes \cdots 
            \right),
        $$
        or in other words,
        \begin{equation}
          \label{eq:liecharacteristic}
          \ch \WH_{\OOO_{X_\lambda}} = \prod_i h_{m_i}[ \ch(\omega_i) ].
        \end{equation}
      \end{Proposition}

      A proof of this result based solely on the description of $\WH_{\OOO_{X_\lambda}}$ in terms of simplicial
      homology was first presented by Sundaram \cite[Theorem 1.7]{Sundaram1994}.
      Note that the result by Sundaram is stated for $\WH_{\OOO_{X_\lambda}}$ tensored with the sign representation.
      Tensoring once more by the sign and some standard transformations of plethysms show the equivalence of her formula with the above formula.

      \ref{prop:plethysm-reformulation} can be further reformulated in terms of
      the combinatorics of {\deffont Lyndon words} and {\deffont quasisymmetric functions}.
      \index{Lyndon!word}%
      \index{quasisymmetric function}%
      Let $A = \{ a_1 < a_2 < a_3 < \cdots \}$ be a linearly ordered alphabet. 
      Recall that a word $\xx = \xx_1 \cdots \xx_n$ with letters $\xx_i \in A$, $1 \leq i \leq n$,
      is called a {\deffont Lyndon word} if it is lexicographically strictly smaller than
      any of its cyclic rearrangements (see \cite[\S 5]{Reutenauer1993} for more details).  
      It is also well known \cite[Theorem 5.1]{Reutenauer1993} that every word $\xx$ has a
      unique {\deffont Lyndon factorization} $\xx=\xx^{(1)} \xx^{(2)} \cdots \xx^{(g)}$,
      \index{Lyndon!factorization}%
      meaning that the $\xx^{(i)}$, $1 \leq i \leq g$, are Lyndon words satisfying
      $$
        \xx^{(1)} \geq_{lex} \xx^{(2)} \geq_{lex} \cdots \geq_{lex} \xx^{(g)}.
      $$
      The {\deffont Lyndon type} of $\xx$ is the weakly decreasing rearrangement of
      \index{Lyndon!type}%
      the lengths of the $\xx^{(i)}$. We use this to define a power series in the
      ring of formal power series $\CC\llbracket t_a~|~a \in A\rrbracket$ by
      $$
        \frL_{\lambda}(\nicet):=\sum_{\substack{\text{words } \xx \\ \text{ of Lyndon type }\lambda}} t_\xx
      $$
      \nomenclature[co]{$\frL_{\lambda}(\nicet)$}{characteristic of $L(V)_\lambda$}%
      where $t_\xx:=t_{\xx_1} t_{\xx_2} \cdots t_{\xx_n}$ for $\xx = \xx_1 \cdots \xx_n$.
      In \cite[Theorem 3.6]{GesselReutenauer1993} it is stated that $\frL_\lambda(\nicet)$ coincides with
      the symmetric function on the right hand side of \eqref{eq:liecharacteristic} from 
      \ref{prop:plethysm-reformulation}. This then yields the following description
      of $\ch \WH_{\OOO_{X_\lambda}}$.

      \begin{Proposition}
        \label{pr:characteristic}
        For $W=\symm_n$ and $\lambda$ any partition of $n$,
        $$
          \ch \WH_{\OOO_{X_\lambda}} =  \frL_\lambda(\nicet).
        $$
      \end{Proposition}

      This reformulation is proved by Gessel and Reutenauer in \cite{GesselReutenauer1993} in parallel
      with a reformulation in terms of quasisymmetric functions.
      It is derived from a bijection attributed to Gessel (see e.g. \cite[p. 175]{Reutenauer1993}, 
      \cite{DiaconisMcGrathPitman1993}, \cite{UyemuraReyes2002}) and described in \cite{GesselReutenauer1993}.
      The bijection allows one to expand $\frL_\lambda(\nicet)$ in terms of the 
      {\deffont fundamental quasisymmetric functions} associated with
      \index{fundamental!quasisymmetric functions}%
      subsets $D \subseteq [n-1]$:
      $$
        F_D:= 
            \sum_{\substack{1 \leq i_1 \leq \cdots \leq i_n:\\ 
               i_j < i_{j+1} \text{ if }j \in D}}
               t_{i_1} \cdots t_{i_n}.
      $$
      \nomenclature[co]{$F_D$}{fundamental quasisymmetric function}%
      Given a permutation $w=(w_1,\ldots,w_n)$ in $\symm_n$,
      define its {\deffont descent set} by
      \index{descent!set}%
      $$
        \Des(w):=\{j \in [n-1] : w_j > w_{j+1} \}.
      $$
      \nomenclature[co]{$\Des(w)$}{descent set of $w$}%
      Theorem 3.6 in \cite{GesselReutenauer1993} states that 
      $\sum_{w \text{ of cycle type }\lambda} F_{\Des(w)}$
      equals $\frL_\lambda$. Thus we arrive at the desired reformulation.
      
      \begin{Proposition}
        For $W=\symm_n$ and $\lambda$ any partition $\lambda$ of $n$,
        $$
          \ch \WH_{\OOO_{X_\lambda}} = \sum_{w \text{ of cycle type }\lambda} F_{\Des(w)}.
        $$
      \end{Proposition}

      In parallel to the study of $\WH_{\OOO_{X_\lambda}}$ motivated by the action of the
      symmetric group on the Orlik-Solomon algebra, the module $\WH_{\OOO_{X_\lambda}}$ 
      appeared in a different guise already in the 1940's, in the context of 
      the {\deffont free Lie algebra}
      (see \cite{Thrall1942}, \cite{Brandt1944}).
      \index{free Lie algebra}%
      \index{Lie!algebra!free}%
      \index{algebra!Lie}%
      Before we can explain this connection we introduce the free Lie algebra and related notation
      (see \cite{Reutenauer1993} for more background information).
      The tensor algebra
      \index{tensor algebra}%
      \index{algebra!tensor}%
      $$
        T(V):=\bigoplus_{d \geq 0} V^{\otimes d}
      $$
      \nomenclature[al]{$T(V)$}{tensor algebra of $V$}%
      is an associative algebra, and hence also a Lie algebra for
      the usual bracket operation $[x,y]=xy-yx$.  The Lie subalgebra 
      $\Lie(V)$ of $T(V)$ generated by its degree one part $T^1(V)=V$ is called the {\deffont free Lie algebra}.
      \index{free Lie algebra}%
      \index{algebra!free}%
      \index{Lie algebra!free}%
      \nomenclature[al]{$\Lie(V)$}{free Lie algebra over $V$}%
      It inherits a grading 
      $$
        \Lie(V)=\oplus_{d \geq 0} {\Lie(v)_d}
      $$
      where $\Lie(v)_d=\Lie(V) \cap V^{\otimes d}$.
      Because (see \cite[Theorem 0.5]{Reutenauer1993}) $T(V)$ turns out to be the universal enveloping algebra
      for $\Lie(V)$, the Poincar\'e-Birkhoff-Witt theorem (see \cite[Theorem 0.2]{Reutenauer1993}) gives a
      \index{Poincar\'e-Birkhoff-Witt theorem}%
      $\symm_n$-equivariant vector space isomorphism
      $$
        T^n(V) = V^{\otimes n} \cong \sum_{\lambda \vdash n} \Lie_{\lambda}(V)
      $$
      where for a partition $\lambda=(1^{m_1},2^{m_2},\ldots)$ one defines
      $$
        \Lie_{\lambda}(V):=\Sym^{m_1}(\Lie(V)_1) \otimes \Sym^{m_2}(\Lie(V)_2) \otimes \cdots.
      $$
      Here $\Sym^m(U)$ denotes the $m$th graded component of the {\deffont symmetric algebra} over a vector space $U$.
      \nomenclature[al]{$\Sym^m(U)$}{$m$th graded component of symmetric algebra over vectorspace $U$}%
      \index{algebra!symmetric}%
      \index{symmetric algebra}%
      For a partition $\lambda$ of $n$ each $\Lie_{\lambda}(V)$ is itself an $\symm_n$-module. 
      Assume $\dim V = n$, fix a basis of $T^1(V) = V$ and denote by $E_n$ the space spanned
      in $T^n(V) = V^{\otimes n}$ by the tensors of the 
      $n!$ permutations of the basis elements. For $\lambda$ a partition of $n$ the multilinear part of 
      $\Lie_{\lambda}(V)$ is its 
      \index{multilinear part}%
      intersection with $E_n$. Since both $E_n$ and $\Lie_{\lambda}(V)$ are $\symm_n$-modules
      it follows that the multilinear part is also an $\symm_n$-module.
      For $\lambda = (n)$ it was shown by Klyachko \cite{Klyachko1974} that $\Lie_{(n)} (V) \cap E_n$ is isomorphic as an $\symm_n$-module to
      $\Ind_{\ZZ_n}^{\symm_n} \exp(\frac{2\pi i }{n})$ and hence by the above mentioned result
      of Stanley \cite[Theorem 7.3]{Stanley1982}, it is isomorphic to $\WH_{\OOO_{(n)}}$.
      More generally a result by F. Bergeron, N. Bergeron and 
      A. Garsia \cite{BergeronBergeronGarsia1988} (see also \cite[Theorem 8.23]{Reutenauer1993})
      shows that the characteristic of the $\symm_n$-representation on the multilinear part of $\Lie_{\lambda}(V)$ is 
      the symmetric function on the right hand side of \eqref{eq:liecharacteristic} from 
      \ref{prop:plethysm-reformulation}. 

      \begin{Proposition}
        \label{pr:multilinear} 
        Let $V \cong \RR^n$ and $\lambda$ any partition of $n$.
        Then there is an isomorphism of $\symm_n$-modules
        $$
         \WH_{\OOO_{X_\lambda}} \cong E_n \cap \Lie_{\lambda}(V).
        $$
      \end{Proposition}
      Indeed, using this reformulation, Theorem 8.24 of \cite{Reutenauer1993}, which was first proved in \cite{BergeronBergeronGarsia1988},
      gives an alternative proof of \ref{pr:A-case-ls-conjecture}.
      
      \medskip
      The action of the group $\GL(V)$ on $T^1(V)$ induces a diagonal action of $\GL(V)$ on each $T^d(V)$.
      By standard facts about the representation theory 
      of $\GL(V)$ and from \ref{pr:multilinear} and \ref{pr:characteristic},
      the character of $\Lie_\lambda(V)$ as a $\GL(V)$-module (that is, the trace of a diagonal element
      of $\GL(V)$ having eigenvalues $x_1,\ldots,x_n$) is $\frL_\lambda(\xx)$.
      In other words, one has the following.

      \begin{Proposition}
        For $W=\symm_n$ and $\lambda$ any partition of $n$, the symmetric function
        $
         \ch \WH_{\OOO_{X_\lambda}}
        $
        is the $\GL(V)$-character of $\Lie_{\lambda}(V)$.
      \end{Proposition}
      
      Thus the following problems are equivalent, and attributed by Schocker \cite{Schocker2003} 
      to Thrall \cite{Thrall1942}.

      \begin{Problem}
        \label{prob:Thrall}
        Find any of
        \begin{itemize}
          \item the $\symm_n$-irreducible decomposition of $\ch \WH_{\OOO_{X_\lambda}}$,
          \item the $\GL(V)$-irreducible decomposition of $\Lie_{\lambda}(V)$, or
          \item the Schur function expansion of $\frL_\lambda(\xx)$.
        \end{itemize}
      \end{Problem}
      Only partial results are known in this regard.  
      For example 
      for $\lambda=(n)$ it was shown by by Kraskiewi\'cz and Weyman in a preprint
      from 1987, published as \cite{KraskiewiczWeyman2001}, that 
      $$
        \frL_{(n)}(\xx) = \sum_{\substack{T\text{ in }SYT_n:\\ \maj(T) \equiv 1 \mod n}} s_{\lambda(T)}.
      $$
      This result can also be seen as a reformulation of a result by Springer from \cite{Springer1974}.

\section{The family $\nu_{(2^k,1^{n-2k})}$}
           \label{sec:second-family}

  Our goal here is to prove \ref{thm:second-family}, whose
  statement we recall.

  \theoremstyle{plain}
  \newtheorem*{TheoremSecondFamily}{\ref{thm:second-family}}
  \begin{TheoremSecondFamily}
      The operators from the family 
      $\{\nu_{(2^k,1^{n-2k})}\}_{k=1,2\ldots,\lfloor \frac{n}{2}\rfloor}$
      pairwise commute, and
      have only integer eigenvalues.
  \end{TheoremSecondFamily}
    
  Recall that here one is considering the reflection arrangement for
  the reflection group $W=\symm_n$, and the $W$-orbit on $\LLL$,
  denoted here $\OOO_{(2^k,1^{n-2k})}$, consisting of all
  intersection subspaces of the form
  $$
    \{ x_{i_1}=x_{i_2} \} \cap \{ x_{i_3}=x_{i_4} \} \cap  \cdots 
    \cap \{ x_{i_{2k-1}}=x_{i_{2k}} \}.
  $$
  Then $\nu_{(2^k,1^{n-2k})}=\nu_{\OOO_{(2^k,1^{n-2k})}}$.
  Define also $\pi_{(2^k,1^{n-2k})}=\pi_{\OOO_{(2^k,1^{n-2k})}}$.
  \nomenclature[co]{$\pi_{(2^k,1^{n-2k})}$}{$\pi_{\OOO_{(2^k,1^{n-2k})}}$}%

  \subsection{A Gelfand model for $\symm_n$}

    \ref{thm:second-family}
    will follow from applying our eigenvalue integrality principle
    \ref{prop:integrality-principle} to the following 
    representation-theoretic fact, which will identify the
    nonzero eigenspaces of the operators $\{ \nu_{(2^k,1^{n-2k})} \}$ 
    with a certain {\deffont Gelfand model} for $W=\symm_n$, whose construction
    is related to the construction of the Gelfand model of $\symm_n$ in 
    \cite{AdinPostnikovRoichman2008}.
    \index{Gelfand model}%
    Recall that a Gelfand model for $W=\symm_n$ is an $\symm_n$-module that carries exactly one copy of each
    $\symm_n$-irreducible $\chi^{\lambda}$; see \cite{AdinPostnikovRoichman2008} for further 
    references.
    Given a number partition $\lambda$,
    let $\oddcols(\lambda)$ denote the number of columns of odd length in the Ferrers
    diagram for $\lambda$, or the number of parts of odd length in the conjugate
    partition $\lambda'$.

    \begin{Proposition}
      \label{prop:second-family-essence}
      For $W=\symm_n$ and nonnegative integers $a,b$ with $2a+b=n$ one has
      $$
        \WH_{\OOO_{(2^a,1^b)}} = \bigoplus_{\substack{\lambda: \\ \oddcols(\lambda)=b}} \chi^\lambda.
      $$
      Consequently, we obtain a Gelfand model for $\symm_n$ by combining the modules $\WH_{\OOO_{(2^a,1^{n-2a})}}$:
      $$
        \bigoplus_{a=0}^{\lfloor\frac{n}{2}\rfloor} \WH_{\OOO_{(2^a,1^{n-2a})}} = \bigoplus_{\lambda\vdash n} \chi^\lambda.
      $$
    \end{Proposition}
    \begin{proof}
      By \ref{prop:plethysm-reformulation}, one has
      $$
        \begin{aligned}
          \sum_{a,b \geq 0} \ch \left(\WH_{\OOO_{(2^a,1^b)}}\right) t^b 
            & = \sum_{a,b \geq 0} h_a[ \ch \omega_2 ] \,\, h_b[ \ch \omega_1 ] \,\, t^b \\
            & = \sum_{a,b \geq 0} h_a[ e_2(\xx) ] \,\, h_b(\xx) \,\, t^b \\
            & = \prod_{i<j} (1-x_i x_j)^{-1} \prod_{i} (1 - t x_i)^{-1} \\
            & = \sum_{\lambda} t^{\oddcols(\lambda)} s_\lambda(\xx)
        \end{aligned}
      $$
      where the last equality uses a well-known identity (see 
      \cite[Exercise 7.28(b)]{Stanley1999}, \cite[Chap. I, \S 5,  Ex. 7]{Macdonald1998}).
    \end{proof}

  \subsection{Proof of \ref{thm:second-family}}

We will actually prove the following precision of Theorem~\ref{thm:second-family}, which tells us a bit more about the eigenvalues of the operators
$\nu_{(2^k,1^{n-2k})}$.

\begin{Theorem}
\label{thm:second-family-precision}
There exists an orthogonal direct sum decomposition of $\RR[\symm_n \times \ZZ_2]$-modules
\begin{equation}
\label{eqn:second-family-eigenspaces}
        \RR \symm_n = K \oplus \left( \bigoplus_{\lambda \vdash n} U^{\lambda} \right)
\end{equation}
with these properties.
\begin{enumerate}
\item[(i)]
The subspace $K$ is annihilated by each of the
operators $\{ \nu_{(2^k,1^{n-2k})} \}_{k=0,1,2,\ldots,\lfloor \frac{n}{2} \rfloor}$.
\item[(ii)]
The subspace $U^\lambda$ 
lies inside the eigenspace for $\nu_{(2^k,1^{n-2k})}$ having eigenvalue
\begin{equation}
\label{explicit-eigenvalue-for-second-family}
\gamma_{(2^k,1^{n-2k}),\lambda} = 
\sum_{ w \in \symm_n } \ninv_{(2^k,1^{n-2k})}(w) \cdot \chi^\lambda(w).
\end{equation}
\item[(iii)] 
The subspace $U^\lambda$ affords the irreducible $\RR[\symm_n \times \ZZ_2]$-module
$
\chi^\lambda \otimes \left(\chi^{-}\right)^{\otimes \frac{n-\oddcols(\lambda)}{2} }.
$
\end{enumerate}
\end{Theorem}

    \begin{proof}[Proof of \ref{thm:second-family-precision}]
      We set $k_{\max} := \lfloor \frac{n}{2} \rfloor$, and define
      $$
       K:= \ker \pi_{(2^{k_{\max}},1^{n-2k_{\max}})} 
         = \ker \nu_{(2^{k_{\max}},1^{n-2k_{\max}})}.
      $$
      Since one can find a nested chain of representative subspaces
      for the $W$-orbits $\OOO_{(2^k,1^{n-2k})}$ as $k$ varies,
      \ref{prop:nested-kernels} implies the following inclusions of kernels:
      \begin{equation}
        \label{eqn:second-family-filtration}
        \begin{array}{cccccccc}
          K=& \ker \pi_{(2^{k_{\max}},1^{n-2k_{\max}})} & 
               \subset \cdots \subset 
             & \ker \pi_{(2^2,1^{n-4})} & 
               \subset 
             & \ker \pi_{(2^1,1^{n-2})} &
               \subset 
             & \ker \pi_{(1^n)} \\
             & \Vert & & \Vert & & \Vert & & \Vert \\
             & \ker \nu_{(2^{k_{\max}},1^{n-2k_{\max}})} & 
               & \ker \nu_{(2^2,1^{n-4})} & 
               & \ker \nu_{(2^1,1^{n-2})} &
               & \ker \nu_{(1^n)} \\
        \end{array}
      \end{equation}
      In particular, $K$ is annihilated, and hence preserved, 
      by every one of the self-adjoint
      operators $\{ \nu_{(2^k,1^{n-2k})} \}_{k=0,1,2,\ldots,k_{\max}}$.  
      
      Hence they also preserve the perpendicular
      space $U:=K^\perp$, a $\QQ$-rational subspace of $\RR \symm_n$.  Note that
      \ref{cor:second-square-root-is-bhr} implies that the $\symm_n$-representation 
      afforded by $U$ is the same as that 
      afforded by the nonzero eigenspaces of a certain \BHR operator $b_J$.
      Meanwhile, \ref{ex:two-type-A-examples} shows that this $\symm_n$-representation
      is the sum of $\bigoplus_{a=0}^{k_{\max}} \WH_{\OOO_{(2^a,1^{n-2a})}}$.  Note that
      this sum is isomorphic to the multiplicity-free Gelfand model described in
      \ref{prop:second-family-essence}.

      This multiplicity-freeness has two consequences.  
      First, it shows that by combining the $\symm_n$-isotypic 
      decomposition $U = \bigoplus_\lambda U^{\lambda}$
      together with the complementary space $K$, one obtains a direct sum decomposition
      as in \eqref{eqn:second-family-eigenspaces}
      that simultaneously diagonalizes all of the operators 
      $\{ \nu_{(2^k,1^{n-2k})} \}_{k=0,1,2,\ldots,k_{\max}}$.  

      Secondly, the eigenvalue integrality principle, \ref{prop:integrality-principle},
      implies that each operator $\nu_{(2^k,1^{n-2k})}$ 
      acts on $U$ with integer eigenvalues.  
      Since $\nu_{(2^k,1^{n-2k})}$ also annihilates the subspace $K=U^\perp$
      complementary to $U$, it has only integer eigenvalues on all
      of $\RR \symm_n$.  

      However, we know more about the eigenvalue\footnote{The first author thanks C.E. Csar for discussions leading to this expression for $\gamma_{(2^k,1^{n-2k}),\lambda}$.} 
      $\gamma_{(2^k,1^{n-2k}),\lambda}$  with which $\nu_{(2^k,1^{n-2k})}$ acts on $U^\lambda$.
      Picking any realization $\symm_n \overset{\rho_\lambda}{\rightarrow} GL_\CC(V)$
      of the irreducible $\symm_n$-representation with character $\chi^\lambda$,
      the Fourier transform reduction, \ref{Fourier-transform-prop}, tells us that
      $\rho_\lambda( \nu_{(2^k,1^{n-2k})} )$ has $\gamma_{(2^k,1^{n-2k}),\lambda}$ as its only
      potential nonzero eigenvalue, and hence 
$$
\begin{aligned}
\gamma_{(2^k,1^{n-2k}),\lambda}&= \Trace \rho_\lambda( \nu_{(2^k,1^{n-2k})} )\\  
&= \Trace\left( \sum_{w \in \symm_n} \ninv_{(2^k,1^{n-2k})}(w) \cdot \rho_\lambda(w) \right)\\
&= \sum_{w \in \symm_n} \ninv_{(2^k,1^{n-2k})}(w) \cdot \Trace \rho_\lambda(w)   \\
&= \sum_{w \in \symm_n} \ninv_{(2^k,1^{n-2k})}(w) \cdot \chi^\lambda(w)   
\end{aligned}
$$

      Lastly, to see how $\ZZ_2$ acts on $U^\lambda$, note that
      \ref{prop:second-family-essence} implies that $U^\lambda$ lies in 
\begin{equation}
\label{eqn:second-family-layer}
\im( \nu_{(2^a,1^{n-2a})} ) \cap \im( \nu_{(2^{a-1},1^{n-2a+2})} )^\perp
\end{equation}
      where $a:=\frac{n-\oddcols(\lambda)}{2}$.   Since $\nu_{(2^a,1^{n-2a})},\pi_{(2^a,1^{n-2a})}$
      share the same kernels, one has an isomorphism of
      $\RR[\symm_n \times \ZZ_2]$-modules
$$
\im( \nu_{(2^a,1^{n-2a})} ) \cong \im( \pi_{(2^a,1^{n-2a})} ).
$$
      Consequently the space \eqref{eqn:second-family-layer} carries
      $\RR[\symm_n \times \ZZ_2]$-module structure isomorphic to that of
$$
\im( \pi_{(2^a,1^{n-2a})} ) / \im( \pi_{(2^{a-1},1^{n-2a+2})} )
$$ 
      which is $\WH_{\OOO_{(2^a,1^{n-2a})}} \otimes (\chi^-)^{\otimes a}$ by 
      \ref{ex:two-type-A-examples}.  Thus $\ZZ_2$ acts by $(\chi^-)^{\otimes a}$ on $U^\lambda$.
    \end{proof}

    \begin{Remark}
      In contrast with the situation for the original family 
      $\{\nu_{(k,1^{n-k})} \}_{k=1,2\ldots,n}$
      one does {\emphfont not} have that the
      associated \BHR-operators $b_J$ pairwise commute.
    \end{Remark}

\begin{Remark}
The formula for the eigenvalue $\gamma_{(2^k,1^{n-2k}),\lambda}$
given in \eqref{explicit-eigenvalue-for-second-family} is somewhat explicit,
but still leaves something to be desired.
For example, the character values $\chi^\lambda(w)$ for $w$ 
in $\symm_n$ are integers, but they can be negative.  
Thus \eqref{explicit-eigenvalue-for-second-family} does not manifestly show
the fact that $\gamma_{(2^k,1^{n-2k}),\lambda}$ is {\it nonnegative},
nor does it show the fact that  
$\gamma_{(2^k,1^{n-2k}),\lambda}$ vanishes unless $\oddcols(\lambda) \geq n-2k$.
This suggests the following problem.

      \begin{Problem}
        \label{prob:second-family-eigenvalues}
        For each partition $\lambda$ of $n$, and each $k$ with 
        $\oddcols(\lambda) \geq n-2k$, find a more explicit 
        formula for the nonzero eigenvalue $\gamma_{(2^k,1^{n-2k}),\lambda}$ 
        of $\nu_{(2^k,1^{n-2k})}$ acting on its (non-kernel) eigenspace $U^\lambda$
        affording $\chi^\lambda$.
      \end{Problem}

      We have computed some of these eigenvalues using {\tt Sage} \cite{Sage},
      and we present this data for $3\leq n \leq 6$ in the tables below.
      The data is presented as follows: 
      \begin{itemize}
      \item 
          each row of the table corresponds to the subspace $U^\lambda$ affording $\chi^\lambda$;
      \item
          the entry in the column indexed by $\nu_{(2^k,1^{n-2k})}$ is the
          eigenvalue $\gamma_{(2^k,1^{n-2k}),\lambda}$;
      \item 
          the entry in the column indexed by $w_0$ is the eigenvalue for the
          $\ZZ_2$-action on $U^\lambda$.
      \end{itemize}
      To enhance the presentation of the data, every zero eigenvalue has been replaced by
      a dot.
    \end{Remark}

    \begin{center}
    \begin{tabular}{cc}

      \begin{tabular}{lrrr}
      \toprule
      & $\nu_{(1^3)}$ & $\nu_{(2^1,1^1)}$ & $\ZZ_2$ \\
      \midrule[0.1em]
       $\chi^{3}$ & 6 & 9 & 1 \\
       $\chi^{21}$ & $\cdot$ & 4 & $-1$ \\
       $\chi^{111}$ & $\cdot$ & 1 & $-1$ \\
      \bottomrule
      \end{tabular}

      &

      \qquad
      \begin{tabular}{lrrrr}
      \toprule
      & $\nu_{(1^4)}$ & $\nu_{(2^1,1^2)}$ & $\nu_{(2^2)}$ & $\ZZ_2$ \\
      \midrule[0.1em]
       $\chi^{4}$ & 24 & 72 & 18 & 1 \\
       $\chi^{31}$ & $\cdot$ & 20 & 10 & $-1$ \\
       $\chi^{211}$ & $\cdot$ & 4 & 2 & $-1$ \\
       $\chi^{22}$ & $\cdot$ & $\cdot$ & 8 & 1 \\
       $\chi^{1111}$ & $\cdot$ & $\cdot$ & 2 & 1 \\
      \bottomrule
      \end{tabular}

      \\
      \vspace{1em}
      \\

      \begin{tabular}{lrrrr}
      \toprule
      & $\nu_{(1^5)}$ & $\nu_{(2^1,1^3)}$ & $\nu_{(2^2,1^1)}$ & $\ZZ_2$ \\
      \midrule[0.1em]
       $\chi^{5}$ & 120 & 600 & 450 & 1 \\
       $\chi^{41}$ & $\cdot$ & 120 & 180 & $-1$ \\
       $\chi^{311}$ & $\cdot$ & 20 & 30 & $-1$ \\
       $\chi^{32}$ & $\cdot$ & $\cdot$ & 68 & 1 \\
       $\chi^{221}$  & $\cdot$ & $\cdot$ & 12 & 1 \\
       $\chi^{2111}$  & $\cdot$ & $\cdot$ & 12 & 1 \\
       $\chi^{11111}$ & $\cdot$ & $\cdot$ & 2 & 1 \\
      \bottomrule
      \end{tabular}

      &

      \qquad
      \begin{tabular}{lrrrrr}
      \toprule
      & $\nu_{(1^6)}$ & $\nu_{(2^1,1^4)}$ & $\nu_{(2^2,1^2)}$ & $\nu_{(2^3)}$ & $\ZZ_2$ \\
      \midrule[0.1em]
       $\chi^{6}$ & 720 & 5400 & 8100 & 1350 & 1 \\
       $\chi^{51}$ & $\cdot$ & 840 & 2520 & 630 & $-1$ \\
       $\chi^{411}$ & $\cdot$ & 120 & 360 & 90 & $-1$ \\
       $\chi^{42}$ & $\cdot$ & $\cdot$ & 616 & 308 & 1 \\
       $\chi^{321}$ & $\cdot$ & $\cdot$ & 96 & 48 & 1 \\
       $\chi^{3111}$ & $\cdot$ & $\cdot$ & 84 & 42 & 1 \\
       $\chi^{222}$ & $\cdot$ & $\cdot$ & 24 & 12 & 1 \\
       $\chi^{21111}$ & $\cdot$ & $\cdot$ & 12 & 6 & 1 \\
       $\chi^{33}$ & $\cdot$ & $\cdot$ & $\cdot$ & 204 & $-1$ \\
       $\chi^{2211}$ & $\cdot$ & $\cdot$ & $\cdot$ & 36 & $-1$ \\
       $\chi^{111111}$ & $\cdot$ & $\cdot$ & $\cdot$ & 6 & $-1$ \\
      \bottomrule
      \end{tabular}
    \end{tabular}

    \medskip
    {Eigenvalues of $\nu_{(2^k,1^{n-2k})}$ acting on the non-kernel eigenspace afforded by $\chi^\lambda$.}
    \end{center}
      \bigskip

    \begin{center}
    \end{center}

      \begin{center}
        \centering
        \begin{tabular}{lrrrrr}
        \toprule
         & $\nu_{(1^7)}$&$\nu_{(2^1,1^5)}$&$\nu_{(2^2,1^3)}$&$\nu_{(2^3,1^1)}$&$\ZZ_2$ \\
         \midrule
         $\chi^{7}$ & 5040 & 52920 & 132300 & 66150 & 1 \\
         $\chi^{61}$ & $\cdot$ &  6720 &  33600 &  25200 & -1 \\
         $\chi^{511}$ & $\cdot$ &  840 &  4200 &  3150 & -1 \\
         $\chi^{52}$ & $\cdot$ & $\cdot$ & 6048 & 9072 & 1 \\
         $\chi^{421}$ & $\cdot$ & $\cdot$ & 840 & 1260 & 1 \\
         $\chi^{4111}$ & $\cdot$ & $\cdot$ & 672 & 1008 & 1 \\
         $\chi^{322}$ & $\cdot$ & $\cdot$ & 168 & 252 & 1 \\
         $\chi^{31111}$ & $\cdot$ & $\cdot$ & 84 & 126 & 1 \\
         $\chi^{43}$ & $\cdot$ & $\cdot$ & $\cdot$ &  2976 & -1 \\
         $\chi^{331}$ & $\cdot$ & $\cdot$ & $\cdot$ &  396 & -1 \\
         $\chi^{3211}$ & $\cdot$ & $\cdot$ & $\cdot$ &  396 & -1 \\
         $\chi^{2221}$ & $\cdot$ & $\cdot$ & $\cdot$ &  96 & -1 \\
         $\chi^{22111}$ & $\cdot$ & $\cdot$ & $\cdot$ &  48 & -1\\
         $\chi^{211111}$ & $\cdot$ & $\cdot$ & $\cdot$ &  48 & -1\\
         $\chi^{1111111}$ & $\cdot$ & $\cdot$ & $\cdot$ &  6 & -1\\
        \bottomrule
        \end{tabular}

        \bigskip
        \begin{tabular}{lrrrrrr}
        \toprule
         & $\nu_{(1^8)}$&$\nu_{(2^1,1^6)}$&$\nu_{(2^2,1^4)}$&$\nu_{(2^3,1^2)}$&$\nu_{(2^4)}$&$\ZZ_2$\\
        \midrule
         $\chi^{8}$ & 40320 & 564480 & 2116800 & 2116800 & 264600 & 1 \\
         $\chi^{71}$ & $\cdot$ & 60480 & 453600 & 680400 & 113400 & -1 \\
         $\chi^{611}$ & $\cdot$ & 6720 & 50400 & 75600 & 12600 & -1 \\
         $\chi^{62}$ & $\cdot$ & $\cdot$ & 64512 & 193536 & 48384 & 1 \\
         $\chi^{521}$ & $\cdot$ & $\cdot$ & 8064 & 24192 & 6048 & 1 \\
         $\chi^{5111}$ & $\cdot$ & $\cdot$ & 6048 & 18144 & 4536 & 1 \\
         $\chi^{422}$ & $\cdot$ & $\cdot$ & 1344 & 4032 & 1008 & 1 \\
         $\chi^{41111}$ & $\cdot$ & $\cdot$ & 672 & 2016 & 504 & 1 \\
         $\chi^{53}$ & $\cdot$ & $\cdot$ & $\cdot$ & 42240 & 21120 & -1 \\
         $\chi^{431}$ & $\cdot$ & $\cdot$ & $\cdot$ & 5376 & 2688 & -1 \\
         $\chi^{4211}$ & $\cdot$ & $\cdot$ & $\cdot$ & 4544 & 2272 & -1 \\
         $\chi^{332}$ & $\cdot$ & $\cdot$ & $\cdot$ & 960 & 480 & -1 \\
         $\chi^{3221}$ & $\cdot$ & $\cdot$ & $\cdot$ & 896 & 448 & -1 \\
         $\chi^{32111}$ & $\cdot$ & $\cdot$ & $\cdot$ & 512 & 256 & -1 \\
         $\chi^{311111}$ & $\cdot$ & $\cdot$ & $\cdot$ & 432 & 216 & -1 \\
         $\chi^{22211}$ & $\cdot$ & $\cdot$ & $\cdot$ & 128 & 64 & -1 \\
         $\chi^{2111111}$ & $\cdot$ & $\cdot$ & $\cdot$ & 48 & 24 & -1 \\
         $\chi^{44}$ & $\cdot$ & $\cdot$ & $\cdot$ & $\cdot$ & 11904 & 1 \\
         $\chi^{3311}$ & $\cdot$ & $\cdot$ & $\cdot$ & $\cdot$ & 1584 & 1 \\
         $\chi^{2222}$ & $\cdot$ & $\cdot$ & $\cdot$ & $\cdot$ & 384 & 1 \\
         $\chi^{221111}$ & $\cdot$ & $\cdot$ & $\cdot$ & $\cdot$ & 192 & 1 \\
         $\chi^{11111111}$ & $\cdot$ & $\cdot$ & $\cdot$ & $\cdot$ & 24 & 1 \\
        \bottomrule
        \end{tabular}

        \smallskip
        {Eigenvalues of $\nu_{(2^k,1^{n-2k})}$ acting on the non-kernel eigenspace afforded by $\chi^\lambda$.}
      \end{center}

\section{The original family $\nu_{(k,1^{n-k})}$}
           \label{sec:original-family}

  Let us return to the context of \ref{thm:original-family-commutativity}.
  Here $W=\symm_n$ and $\OOO=\OOO_{(k,1^{n-k})}$ is the $W$-orbit containing
  the subspace $X_{(k,1^{n-k})}$, so we wish to analyze the elements 
  $$
    \nu_\OOO= \nu_{(k,1^{n-k})} := \sum_{w \in W} \ninv_k(w) \cdot w
  $$
  where $\ninv_k(w)$ is the number of increasing sequences of
  length $k$ contained in $w$.

   \subsection{Proof of \ref{thm:original-family-commutativity}.}

     Recall the statement of \ref{thm:original-family-commutativity}.

     \theoremstyle{plain}
     \newtheorem*{TheoremFirstFamily}{\ref{thm:original-family-commutativity}}
     \begin{TheoremFirstFamily}
         The operators from the family $\{\nu_{(k,1^{n-k})} \}_{k=1,2,\ldots,n}$
         pairwise commute.
     \end{TheoremFirstFamily}
     \begin{proof}
       Fix $k$ and $\ell$. One  has $\nu_{(k,1^{n-k})} \nu_{(\ell,1^{n-\ell})} 
       = \sum_{w \in \symm_n} d^{k,\ell}_w \cdot w$, where
       \begin{equation} \label{unreformulated-coefficient-equation}
         d^{k,\ell}_w 
           = \sum_{\substack{1 \leq i_1 < \cdots <i_k \leq n \\ 1 \leq j_1 < \cdots <j_\ell \leq n}}
           \left| \left\{(u,v) \in \symm_n \times \symm_n :
           \begin{matrix} u v= w, & \\
                     u(i_1) < \cdots < u(i_k), \\
                     v(j_1) < \cdots < v(j_\ell) 
           \end{matrix} \right\} \right|.
       \end{equation}
       We want to show that $d^{k,\ell}_w = d^{\ell,k}_w$ for any permutation $w$ in $\symm_n$.

       \medskip
       Let us reformulate this coefficient $d^{k,\ell}_w$ a bit.  First get
       rid of $v$ using $v=u^{-1}w$.  Second, if one names the sequences of lengths $k$ and 
       $\ell$
       $$
         \begin{aligned}
           K&:=(u(i_1),\ldots,u(i_k)) \\
           L&:=(j_1,\ldots,j_\ell)\\
           \text{ so }w(L)&:=(w(j_1),\ldots,w(j_\ell)),
         \end{aligned}
       $$
       then the condition on $u$ in \eqref{unreformulated-coefficient-equation}
       is that both sequences $K$ and $w(L)$ appear from left-to-right
       as subsequences inside $(u_1,u_2,\ldots,u_n)$.  In other words, $u$ lies in the set 
       $\LLL(P_{K,w(L)})$ of all linear extensions of the poset $P_{K,w(L)}$ on 
       $[n]=\{1,2,\ldots,n\}$
       defined as the transitive closure of the order relations 
       $$
         \begin{array}{rcl}
           u(i_1)&< \cdots <&u(i_k)\\
           w(j_1)&< \cdots <&w(j_\ell).
         \end{array}
       $$
       \begin{Example}
       If $n=10$, $k=6$, $\ell=4$ and 
       $$
         \begin{aligned}
           w&=\left[ \begin{matrix} 1&2&3&4&5&6&7&8&9&10 \\
                       9&1&7&3&5&2&6&8&10&4 \\ \end{matrix} \right] \\
           K&=(1,3,5,6,8,10) \\
           w(L)&=(7,3,2,8) \\
         \end{aligned}
       $$
       then the poset $P_{K,w(L)}$ for this example is shown in \ref{fig:exampleposet}.

       \begin{figure}[ht]
         \begin{picture}(0,0)%
           \includegraphics{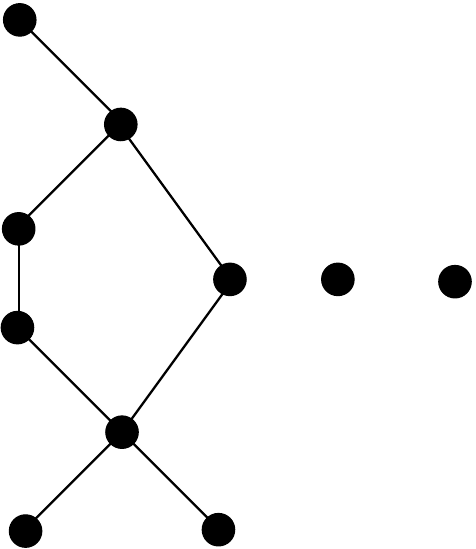}%
         \end{picture}%
         \setlength{\unitlength}{4100sp}%
         \begin{picture}(2000,3000)(4000,-2000)
           \put(4250, 0500){\makebox(0,0)[lb]{\color[rgb]{0,0,0}10}}
           \put(4300,-0100){\makebox(0,0)[lb]{\color[rgb]{0,0,0}8}}
           \put(3850,-0550){\makebox(0,0)[lb]{\color[rgb]{0,0,0}6}}
           \put(4850,-0820){\makebox(0,0)[lb]{\color[rgb]{0,0,0}2}}
           \put(5350,-0820){\makebox(0,0)[lb]{\color[rgb]{0,0,0}4}}
           \put(6220,-0820){\makebox(0,0)[lb]{\color[rgb]{0,0,0}9}}
           \put(3850,-1100){\makebox(0,0)[lb]{\color[rgb]{0,0,0}5}}
           \put(4360,-1520){\makebox(0,0)[lb]{\color[rgb]{0,0,0}3}}
           \put(3900,-2050){\makebox(0,0)[lb]{\color[rgb]{0,0,0}1}}
           \put(4800,-2050){\makebox(0,0)[lb]{\color[rgb]{0,0,0}7}}
         \end{picture}
         \caption{The poset $P_{K,w(L)}$ with 
            $K=(1,3,5,6,8,10)$ and $w(L)=(7,3,2,8)$.}
         \label{fig:exampleposet}
       \end{figure}
       \end{Example}

       Lastly note that when $w$ is written
       in two-line notation, $K$ is a $k$-subsequence of the top line, while $w(L)$ is
       an $\ell$-subsequence of the bottom line.
       Thus one has
       $$
         d^{k,\ell}_w = \sum_{\text{ such }K,L} \#\LLL(P_{K,w(L)})
       $$
       and we wish to show that $d^{k,\ell}_w = d^{\ell,k}_w$ for all permutations
       $w$ in $\symm_n$.

       \medskip
       First we fix the intersection and union of the underlying sets of $K$ and $w(L)$
       $$
         \begin{aligned}
           I &:=K \cap w(L)\\
           J &:=K \cup w(L)
         \end{aligned}
       $$
       and define a new coefficient
       \begin{equation}
       \label{refined-d-coefficients}
         d^{k,\ell}_{(w,I,J)} := 
         \sum_{\substack{K,L:\\ K\cap w(L)=I\\ K \cup w(L)=J}} \#\LLL(P_{K,w(L)}).
       \end{equation}
       Thus it suffices to show that for each fixed pair $I \subseteq J \subseteq [n]$,
       one has $d^{k,\ell}_{(w,I,J)} = d^{\ell,k}_{(w,I,J)}$.

       \medskip
       One can reduce to the case where $J=K \cup w(L)=[n]$ as follows.
       If $m$ lies in the complement $[n] \setminus J$, then
       $m$ is incomparable to all other elements in $P_{K \cup w(L)}$ (such as
       $m=4$ or $m=9$ in the previous example), and one finds
       $$
         \Big|\LLL(P_{K,w(L)})\Big| = n \cdot 
         \Big|\LLL(\hat{P}_{K,w(L)})\Big| 
       $$
       where $\hat{P}_{K,w(L)}$ is the poset on $[n] \setminus \{m\}$
       in which the element $m$ has been removed.
       Thus we will assume without loss of generality that $J:=K \cup w(L)=[n]$.

       \medskip
       We reformulate further.  Think of the fixed elements in $I:=K \cap w(L)$ as a set of $i:=|I|$
       vertical {\emphfont dividers} that partition the remaining elements $[n] \setminus I$ in the
       \index{dividers}%
       top and bottom of $w$ into $i+1$ divisions:
       $$
         w=\left[ 
             \begin{matrix} 
                t^{(1)} &|& t^{(2)} &|& \cdots &|& t^{(i+1)} \\
                b^{(1)} &|& b^{(2)} &|& \cdots &|& b^{(i+1)} 
             \end{matrix} 
           \right]
       $$
       Note that the sequences $t^{(j)}$ and $b^{(j)}$ need not have the same length,
       and any of them are allowed to be empty sequences.  
       
       \begin{Example}
       If one has $k=7$, $\ell=5$, 
       \begin{gather*}
        \begin{array}{@{}l@{}l@{}l@{}l@{}l@{}l@{}l@{}l@{}l@{}l@{}l@{}}
            w = \left[\begin{array}{c}  1 \\  9 \end{array}\right.
            & \begin{array}{c}  2 \\  1 \end{array}
            & \begin{array}{c}  3 \\  7 \end{array}
            & \begin{array}{c}  4 \\  3 \end{array}
            & \begin{array}{c}  5 \\  5 \end{array}
            & \begin{array}{c}  6 \\  2 \end{array}
            & \begin{array}{c}  7 \\  6 \end{array}
            & \begin{array}{c}  8 \\  8 \end{array}
            & \begin{array}{c}  9 \\  10 \end{array}
            & \left.\begin{array}{c} 10 \\ 6 \end{array}\right].
          \end{array}
       \end{gather*}
       and $I=\{3,6\}$, then the divided $w$ looks like
       \begin{gather}
       \label{divided-w}
        \begin{array}{@{}l@{}l@{}l@{}l@{}l@{}l@{}l@{}l@{}l@{}}
            & & & \,\text{\scriptsize3} & & \,\text{\scriptsize6} \\
            w = \left[\begin{array}{c}  1 \\  9 \end{array}\right.
            & \begin{array}{c}  2 \\  1 \end{array}
            & \begin{array}{c}  \phantom{3} \\  7 \end{array}
            & \,\left|\,\begin{array}{c}  4 \\  5 \end{array}\right.
            & \begin{array}{c}  5 \\  2 \end{array}
            & \,\left|\,\begin{array}{c}  7 \\  8 \end{array}\right.
            & \begin{array}{c}  8 \\ 10 \end{array}
            & \begin{array}{c}  9 \\  4 \end{array}
            & \left.\begin{array}{c} 10 \\  \phantom{6} \end{array}\right].
          \end{array}
       \end{gather}
       \end{Example}

       One may as well relabel \ref{divided-w}
       to look like the following {\emphfont divided permutation}
       \index{divided permutation}%
       $w'$ of $[n']:=[8]$ where $n':=n-|I|$:
       \begin{gather*}
        \begin{array}{@{}l@{}l@{}l@{}l@{}l@{}l@{}l@{}l@{}l@{}}
            w' = \left[\begin{array}{c}  1 \\  7 \end{array}\right.
            & \begin{array}{c}             2 \\  1 \end{array}
            & \begin{array}{c}    \phantom{3} \\ 5 \end{array}
            & \,\left|\,\begin{array}{c}       3 \\  4 \end{array}\right.
            & \begin{array}{c}             4 \\  2 \end{array}
            & \,\left|\,\begin{array}{c}       5 \\  6 \end{array}\right.
            & \begin{array}{c}             6 \\  8 \end{array}
            & \begin{array}{c}             7 \\  3 \end{array}
            & \left.\begin{array}{c}       8 \\  \phantom{6} \end{array}\right].
          \end{array}
       \end{gather*}

       Note that the remaining choice of 
       $$
         \begin{aligned}
           K'&:=K \setminus I\\
           L'&:=w(L) \setminus I
         \end{aligned}
       $$
       gives a {\emphfont disjoint decomposition}
       $
        [n] \setminus I = K' \sqcup L'.
       $
       So the number of linear extensions of $P_{K,w(L)}$ becomes the product,
       running over
       each of the $i+1$ divisions $(t^{(j)},b^{(j)})$ in $w$, 
       of the number of shuffles of the two sequences $K' \cap t^{(j)}$ and
       $L' \cap b^{(j)}$.  
       Therefore, to count the linear extensions that make up $d^{k,\ell}_{(w,I,J)}$
       in \ref{refined-d-coefficients}, it is equivalent to sum over 
       the decompositions of $[n']:=K' \sqcup L'$ that have
       $$
         \begin{aligned}
           k'&:=|K'|=k-i\\
           l'&:=|L'|=\ell-i
         \end{aligned}
       $$
       and for each such decomposition, sum up the product of the number of shuffles of
       $K' \cap t^{(j)}$ with $L' \cap b^{(j)}$, that is, the product
       \begin{equation}
         \label{product-of-shuffles}
           \prod_{j=1}^{i+1} 
             \binom{|K' \cap t^{(j)}|+|L' \cap b^{(j)}|}{|K' \cap t^{(j)}|,|L' \cap b^{(j)}|}.
       \end{equation}

       Call such a choice of decomposition and
       the shuffles, a {\emphfont decomposition-shuffle} of the divided permutation $w'$ of $[n']$,
       \index{decomposition-shuffle}%
       and call the total number of them $d^{k',\ell'}_{w'}$.
       Thus we wish to show that $d^{k',\ell'}_{w'}=d^{\ell',k'}_{w'}$ for
       every divided permutation $w'$ of $[n']$ and every pair $(k',\ell')$ with 
       $k'+\ell'=n'$.  This will be done by induction on $n'$.

       \medskip
       First, note that one can reorder the elements in any of the $t^{(j)}, b^{(j)}$
       arbitrarily; this does not affect the possible choices of a decomposition
       $[n']=K' \sqcup L'$ nor does it affect the product
       \eqref{product-of-shuffles}.

       So without loss of generality, assume that the numbers appear in integer
       order in each $t^{(j)}$ and $b^{(j)}$;  in particular, the largest
       number $n'$ will appear {\emphfont last}
       within both its division on the top of $w'$ and its division on the bottom.
       For example, the divided permutation $w'$ from \ref{divided-w} would be reordered to
       \begin{gather*}
        \begin{array}{@{}l@{}l@{}l@{}l@{}l@{}l@{}l@{}l@{}l@{}}
            w' = \left[\begin{array}{c}  1 \\  1 \end{array}\right.
            & \begin{array}{c}             2 \\  5 \end{array}
            & \begin{array}{c}    \phantom{3} \\ 7 \end{array}
            & \,\left|\,\begin{array}{c}       3 \\  2 \end{array}\right.
            & \begin{array}{c}             4 \\  4 \end{array}
            & \,\left|\,\begin{array}{c}       5 \\  3 \end{array}\right.
            & \begin{array}{c}             6 \\  6 \end{array}
            & \begin{array}{c}             7 \\  8 \end{array}
            & \left.\begin{array}{c}       8 \\  \phantom{6} \end{array}\right].
          \end{array}
       \end{gather*}
       We now count $d^{k',\ell'}_{w'}$ by classifying the
       decompositions-shuffles according to the entry $m$ (if any) that
       appears directly after the entry $n'$, within the shuffle that contains $n'$.

       \vskip.1in
       \noindent
       {\sf Case 1.}
         The decomposition-shuffle has no such value of $m$, that is, $n'$ appears last in the shuffle
         for its division. 

         Let $w''$ be obtained from $w'$ by removing $n'$ from both the top and bottom.

         \vskip.05in
         \noindent
         {\sf Subcase 1a.}
           The decomposition $[n']:=K' \sqcup L'$ has $n'$ in $K'$.
           It is straightforward  to check that these decomposition-shuffles are counted by 
           $d^{k'-1,\ell'}_{w''}$.

         \vskip.05in
         \noindent
         {\sf Subcase 1b.}
           The decomposition $[n']:=K' \sqcup L'$ has $n'$ in $L'$.
           Similarly, it is straightforward to check that these decomposition-shuffles are counted by 
           $d^{k',\ell'-1}_{w''}$.

         Putting together the two subcases,
         the decomposition/shuffles in this Case 1 are counted by the
         sum $d^{k'-1,\ell'}_{w''}+d^{k',\ell'-1}_{w''}$.

         Note that when considering the corresponding 
         decomposition/shuffles counted by $d^{\ell',k'}_{w'}$ (where the roles of
         $k', \ell'$ have been reversed, but $w'$ is the same), 
         those that fall in this Case 1 
         will analogously be counted by the sum $d^{\ell'-1,k'}_{w''}+d^{\ell',k'-1}_{w''}$,
         where $w''$ is the {\emphfont same} permutation derived from $w'$.
         By induction, 
         $$
           \begin{aligned}
             d^{\ell'-1,k'}_{w''} &= d^{k',\ell'-1}_{w''} \\
             d^{\ell',k'-1}_{w''} &=d^{k'-1,\ell'}_{w''} 
           \end{aligned}
         $$
         and hence these two sums are the same.

       \vskip.1in
       \noindent
       {\sf Case 2.} 
         The decomposition/shuffle has such a value $m$ (i.e. something appearing
         directly after $n'$ within the shuffle that contains $n'$).

         Then $n', m$ must appear in opposite sets within
         the decomposition $[n']:=K' \sqcup L'$, due to the fact
         fact that $n'$ appears last in its division.  

         \vskip.1in
         \noindent
         {\sf Subcase 2a.} 
           The decomposition puts $m \in K'$ and $n' \in L'$.

           Since $m$ appears directly after $n'$ in its shuffle,  both $m,n'$ must
           appear in the same division, i.e. $m \in K' \cap t^{(j)}$ and
           $n' \in L' \cap b^{(j)}$ for some $j$.

           This time let $w''$ be obtained from $w'$ by removing all occurrences of $n',m$ and
           replacing the division 
           $
            \left[
              \begin{smallmatrix}
                t^{(j)} \\
                b^{(j)}
              \end{smallmatrix}
            \right]
           $ 
           with two divisions separated by a divider labelled $(n',m)$, 
           \begin{gather*}
            \begin{array}{@{}r@{}l}
                \left[\begin{array}{c}  t' \\  b^{(j)} \end{array}\right.
                & \left|
                    \begin{array}{c}  t'' \\  - \end{array}\right]
              \end{array}
           \end{gather*}
           in which $t', t''$ are the elements that appeared before and
           after $m$ within $t^{(j)}$.  
           
           \begin{Example} If $n'=8$ and $m=6$ in the above
           example, with $n' \in K$ and $m \in L$, one would replace the third division
           $$
             \left[
               \begin{matrix}
                 t^{(3)} \\
                 b^{(3)}
               \end{matrix}
             \right]
             =
             \left[ 
               \begin{matrix} 
                 5&6&7&8 \\
                 3&6&8&   
               \end{matrix} 
             \right]
           $$
           obtaining
           \begin{gather*}
            \begin{array}{@{}l@{}l@{}l@{}l@{}l@{}l@{}l@{}l@{}l@{}}
                & & & & & & \text{\scriptsize{(8,6)}} \\
                w'' = \left[\begin{array}{c}   1 \\  1 \end{array}\right.
                & \begin{array}{c}             2 \\  5 \end{array}
                & \begin{array}{c}    \phantom{3}\\ 7 \end{array}
                & \,\left|\,\begin{array}{c}   3 \\  2 \end{array}\right.
                & \begin{array}{c}             4 \\  4 \end{array}
                & \,\left|\,\begin{array}{c}   5 \\  3 \end{array}\right.
                & \phantom{\text{\scriptsize 8}}\left|\phantom{\text{\scriptsize 6}}
                    \begin{array}{c}   7 \\  - \end{array}\right].
              \end{array}
           \end{gather*}
           \end{Example}

           It is not hard to check that the
           decomposition-shuffles in Subcase 2a are then counted by 
           $d^{k'-1,\ell'-1}_{w''}$.
           Note that by induction on $n'$ one has
           $$
             d^{k'-1,\ell'-1}_{w''} =d^{\ell'-1,k'-1}_{w''}.
           $$
           Hence the decomposition-shuffles in the same Subcase 2a (with the
           same value of $m$) when
           the roles of $k',\ell'$ are reversed will have the same cardinality.

         \vskip.05in
         \noindent
         {\sf Subcase 2b.} The decomposition puts $m \in L'$ and $n' \in K'$.

           Same as Subcase 2a, with an analogous construction of $w''$ from $w'$
           by introducing one new divider labelled $(n',m)$.

         \vskip.1in
         Thus in each case, reversing the roles of $k',\ell'$ leads to
         cases with the same cardinalities.  Hence $d^{k',\ell'}_{w'}=d^{\ell',k'}_{w'}$,
         completing the proof.
    \end{proof}

    \begin{Problem}
      \label{prob:better-commutativity-proof}
      Find a more enlightening (noninductive?) proof of 
      \ref{thm:original-family-commutativity}!
    \end{Problem}

    It turns out at that one also has pairwise commutativity for the family of
    \BHR operators $\{ b_{(k,1^{n-k})} \}_{k\in[n]}$ (this follows by combining
    \ref{prop:Bidigare} and \cite[Main Theorem 2.1]{Schocker2003derangement}),
    which are closely related to the operators $\nu_{(k,1^{n-k})}$ by
    \ref{cor:second-square-root-is-bhr}. Perhaps this fact can be used
    as a starting point to prove \ref{thm:original-family-commutativity}?

  \subsection{The kernel filtration and block-diagonalization}
    \label{subsec:kernelfiltration}

    There is a way to get a good start on simultaneously diagonalizing
    the commuting family $\{ \nu_{(k,1^{n-k})} \}$, by looking
    at a filtration that comes from their kernels.

    As in the proof of \ref{thm:second-family},
    since one can find a nested chain of representative subspaces
    for the $W$-orbits $\OOO_{(k,1^{n-k})}$ as $k$ varies, 
    \ref{prop:nested-kernels} implies the following inclusions of kernels:
    \begin{equation}
      \label{eqn:first-family-filtration}
      \begin{array}{ccccccccccc}
        0=& \ker \pi_{(n)} & \subset & \ker \pi_{(n-1,1)} & \subset &\ker \pi_{(n-2,1^2)} &
            \subset \cdots \subset 
          & \ker \pi_{(2,1^{n-2})}& \subset & \ker \pi_{(1^n)} & \subset \RR\symm_n \\
            & \Vert &         & \Vert &         & \Vert &
            & \Vert &         & \Vert  & \\
          & \ker \nu_{(n)} &         & \ker \nu_{(n-1,1)} &         &\ker \nu_{(n-2,1^2)} &
            & \ker \nu_{(2,1^{n-2})}&         & \ker \nu_{(1^n)} & \\
      \end{array}
    \end{equation}

    Since \ref{thm:original-family-commutativity}
    says that the $\nu_{(k,1^{n-k})}$ pairwise commute, they
    preserve each others kernels, and hence \eqref{eqn:first-family-filtration}
    gives an $\RR[W \times \ZZ_2]$-module filtration of $\RR \symm_n$ which
    is preserved by each of the $\nu_{(k,1^{n-k})}$.  Denote the filtration
    factors for $j=1,2,\ldots,n$ by
    $$
      \begin{aligned}
        F_{n,j}&= \ker \nu_{(n-j-1,1^{j+1})} / \ker \nu_{(n-j,1^j)} \\
         &= \ker \pi_{(n-j-1,1^{j+1})} / \ker \pi_{(n-j,1^j)} \\
      \end{aligned}
    $$
    with the convention that $F_{n,n}=\RR\symm_n/\ker \pi_{(1^n)}$.
    One knows from the self-adjointness of each $\nu_{(k,1^{n-k})}$
    that there exists an (orthogonal) direct sum decomposition 
    \begin{equation}
      \label{eqn:block-diagonalization}
      \RR \symm_n = \bigoplus_{j=1}^n V_{n,j}
    \end{equation}
    of $\RR[W \times \ZZ_2]$-modules in which
    $$
      V_{n,j} = \ker \nu_{(n-j-1,1^{j+1})}  \cap \ker \nu_{(n-j,1^j)}^\perp \cong F_{n,j}
    $$
    and hence \eqref{eqn:block-diagonalization} gives a simultaneous block diagonalization of the
    operators $\{ \nu_{(k,1^{n-k})} \}_{k\in[n]}$.

    At this point, we can use some of the equivariant \BHR theory to analyze
    the $\RR[W \times \ZZ_2]$-module structure of each $V_{n,j}$ or $F_{n,j}$:
    one has $\ker \nu_{(n-j,1^j)}=\ker b_{(n-j,1^j)}$ for a certain
    \BHR operator $b_{(n-j,1^j)}$, whose kernel was analyzed in 
    \ref{ex:two-type-A-examples}.  This shows that
    \begin{equation}
      \label{eqn:filtration-factor-inexplicitly}
      V_{n,j} \cong F_{n,j} \cong
      \bigoplus_{\substack{\lambda 
      \text{ has exactly }\\j\text{ parts of size }1} } 
      \WH_{\OOO_{X_\lambda}} \otimes (\chi^-)^{\otimes n-\ell(\lambda)}.
    \end{equation}
    This shows that the dimension of $V_{n,j}$
    is the number of permutations $w$ in $\symm_n$ having
    $j$ fixed points, in light of \ref{moebius-proposition}.
    However it is not very explicit as decomposition into
    $\RR[W \times \ZZ_2]$-modules since we do not have solution for 
    \ref{prob:Thrall} in general.

    It turns outs that with a little work, we {\emphfont can} provide a much 
    more explicit description of $V_{n,j}$.  The representation theory of 
    $W=\symm_n$ asserts an $\RR W$-module decomposition into irreducibles
    \begin{equation}
      \label{eqn:ordinary-tableau-decomposition}
      \RR \symm_n = \bigoplus_{Q} \chi^{\shape(Q)}
    \end{equation}
    where $Q$ runs over all {\deffont standard Young tableaux} of size $n$, and
    \index{standard Young tableau}%
    where $\shape(Q)$ is the partition whose {\deffont Ferrers diagram} gives the
    \index{standard Young tableau!shape}%
    \index{Ferrers diagram}%
    \nomenclature[co]{$\shape(Q)$}{shape of the standard Young tableau $Q$}%
    shape of $Q$.  Although we will not need it here,
    one can also incorporate the $\ZZ_2$-action in
    \eqref{eqn:ordinary-tableau-decomposition} and give an
    explicit $\RR[W \times \ZZ_2]$-module decomposition
    \begin{equation}
      \label{eqn:Z2-ordinary-tableau-decomposition}
      \RR \symm_n = \bigoplus_{Q} \chi^{\shape(Q)} \otimes (\chi^{-})^{\otimes \maj(Q)}
    \end{equation}
    where $\maj(Q)$ is the {\deffont major index} statistic on standard Young tableaux;  this
    \nomenclature{$\maj(Q)$}{major index of the standard Young tableau $Q$}%
    \index{major index}%
    \index{standard Young tableau!major index}%
    follows from Springer's theory of regular elements \cite{Springer1974}, the fact that
    $w_0$ is a regular element of $\symm_n$ \cite[Lemma 8.4]{RStantonWhite}, 
    and the formula     for the fake degree polynomials in type $A_{n-1}$ in terms of major 
    indices \cite{KraskiewiczWeyman2001}.

    Instead our goal in the next few subsections, culminating in 
    \ref{thm:Z2-equivariant-filtration-factor},  will be to provide a similar
    decomposition of $V_{n,j}$, as a sum of irreducible $\RR[W \times \ZZ_2]$-modules
    of the form 
    \begin{equation}
      \label{eqn:desired-tableau-form}
      \sum_Q \chi^{\shape(Q)} \otimes \chi^{\epsilon(Q)}
    \end{equation}
    where $Q$ runs over a certain class of standard Young tableaux that depends on $j$,
    and $\epsilon(Q)$ is a $\pm$ sign that depends upon $Q$.
    Here is an outline of how this goal is achieved.
    \begin{enumerate}
      \item[{\sf Step 1.}]
        Relate the bottom kernel $F_{n,0}=\ker \pi_{(n-1,1)}$ in the
        filtration to the homology of the {\deffont complex of injective words},
        \index{complex of injective words}%
        by showing that $\pi_{(n-1,1)}$ is a sign-twisted version of
        the top boundary map in this complex.  This is achieved 
        in \ref{prop:kernels-sign-twisted}.
      \item[{\sf Step 2.}]
        Use homological techniques to describe this
        homology as an $\RR[W \times \ZZ_2]$-module {\emphfont recursively}.
        This is achieved in \ref{Z2-plus-kernel-recurrence}.
      \item[{\sf Step 3.}]
        Show that a description of $F_{n,0}$ automatically leads
        to one for $F_{n,j}$.  This is achieved in \ref{Z2-filtration-analysis}.
      \item[{\sf Step 4.}]
        Solve this recursion for $F_{n,0}$ and $F_{n,j}$
        in the form of \eqref{eqn:desired-tableau-form}.  This
        is achieved in \ref{thm:Z2-equivariant-filtration-factor}.
    \end{enumerate}

  \subsection{The (unsigned) maps on injective words}

    \begin{Definition}
      Given a finite alphabet $A$, and an integer $i$ in the range $0 \leq i \leq |A|$
      let $A^{\langle i \rangle}$ denote the set of {\deffont injective words} of length $i$ with
      \index{injective word}%
      \index{word!injective}%
      \nomenclature[co]{$A^{\langle i \rangle}$}{set of injective words of length $i$ over $A$}%
      letters taken from the alphabet $A$, that is, those words 
      \index{alphabet}%
      which use each letter at most once.  

      For a set $M$ let $\RR^M$ denote an $\RR$-vector space with basis indexed by $M$.
      Given integers $i,j$ with $0 \leq i \leq j \leq |A|$, define a map
      $$ 
        \delplus(A,j,i): \RR^{A^{\langle j \rangle}} \longrightarrow \RR^{A^{\langle i \rangle}}
      $$
      that sends an injective word $a=(a_1,\ldots,a_j)$ of length $j$ to the sum $\sum b$
      of its subwords 
      $b=(a_{k_1},\ldots,a_{k_i})$, $1 \leq k_1 < \cdots < k_i \leq j$, of length $i$.
    \end{Definition}

    Note that the $\RR$-linear maps $\delplus(A,j,i)$ are actually maps of
    $\RR[\symm_A \times \ZZ_2]$-modules, when we 
    consider $\RR^{A^{\langle i \rangle}}$ as an $\RR[\symm_A \times \ZZ_2]$-module
    in the following fashion:
    \begin{enumerate}
      \item[$\bullet$] $\symm_A$ permutes the letters $A$, and 
      \item[$\bullet$] the nonidentity element
         of $\ZZ_2$ sends a word $a=(a_1,a_2,\ldots,a_i)$ to its {\deffont reversed word} 
         \index{reversed word}%
         \index{word!reversed}%
         $a^{\rev}:=(a_i,\ldots,a_2,a_1)$.
         \nomenclature[co]{$a^{\rev}$}{word $a$ reversed}%
    \end{enumerate}

    Our goal in the next few subsections is to begin by understanding the kernel of the
    {\emphfont first} of the maps $\delplus(A,j,i)$, for which we use an abbreviated
    notation:
    $$
      \pi_A := \delplus(A,|A|,|A|-1).
    $$ 
    The kernel of this map will turn out be closely related to the
    homology of the complex of injective words on $A$; see \ref{su:injectivewords}.

    \begin{Remark}
      In fact, the maps $\delplus(A,|A|,i)$ are simply instances
      of the maps $\pi_\OOO$ where $W=\symm_A$ and $\OOO$ is the
      $W$-orbit of intersection subspaces where $i$ of the coordinates are
      set equal.
    \end{Remark}

  \subsection{The complex of injective words}
    \label{su:injectivewords}

    \begin{Definition}
      The {\deffont complex of injective words on $A$}
      \index{complex of injective words}%
      \index{injective word!complex}%
      is the chain complex $(K_A, \delminus(A,\cdot))$ having 
      $i^{th}$ chain group 
      $K_{A,i}:=\RR^{A^{\langle i+1 \rangle}}$ and whose 
      $i^{th}$ boundary map
      $$
        \delminus(A,i): K_{A,i}  \longrightarrow K_{A,i-1}
      $$
      is a {\deffont signed} version of the map $\delplus(A,i+1,i)$:
      $$
        \delminus(A,i)(a_0,a_1,\ldots,a_i) := 
               \sum_{m=0}^i (-1)^m (a_0,\ldots,\widehat{a_m},\ldots,a_i).
      $$
    \end{Definition}

    One can check that 
    the complex $(K_A, \delminus(A,\cdot))$ becomes a complex
    of $\RR[\symm_A \times \ZZ_2]$-modules
    only after we slightly twist our previously-defined $\ZZ_2$-action:  
    one must now have the nonidentity
    element of $\ZZ_2$ send an injective word $a$ of length $\ell$ to 
    $(-1)^{\lfloor \frac{\ell}{2} \rfloor} \cdot a^{\rev}$.

    There is a very simple relation between
    the maps $\pi_A$ and $\delminus(A,|A|-1)$, once one identifies their
    source and targets with the group algebra $\RR \symm_n$ appropriately.
    Define an $\RR$-linear map 
    $i_A: \RR^{A^{\langle |A|-1 \rangle}} \longrightarrow \RR \symm_A$
    that sends an injective word $u$ of length $|A|-1$
    that is missing exactly one letter $a$ from $A$
    to the permutation of the set $A$ which starts with the letter $a$ and continues
    with the word $u$.  The following proposition is straightforward.

    \begin{Proposition}
      \label{prop:kernels-sign-twisted}
      The map $i_A: \RR^{A^{\langle |A|-1 \rangle}} \longrightarrow \RR \symm_A$
      is an $\RR$-linear isomorphism that makes the following diagram commute:
      $$
      \xymatrix@C+=4em{
        \RR\symm_A
            \ar[r]^{\pi_A}
            \ar[d]_{\sgn}
      & \RR^{A^{\langle |A|-1\rangle}}
            \ar[r]^{i_A}
      & \RR\symm_A
            \ar[d]^{\sgn}
      \\
        \RR\symm_A
            \ar[r]^{\delminus(A,|A|-1)\ }
      & \RR^{A^{\langle |A|-1\rangle}}
            \ar[r]^{i_A}
      & \RR\symm_A
      }
      $$
      where $\RR \symm_n \overset{\sgn}{\longrightarrow} \RR \symm_n$ is the 
      involutive map that scales the basis element corresponding to a permutation $w$ in 
      $\symm_A$ by the sign of $w$.

      In particular, as subspaces of $\RR \symm_A$, the 
      kernels of the two maps $\pi_A$ and $\delminus(A,|A|-1)$
      are sent to each other by the map $\sgn$.
    \end{Proposition}

  \subsection{Pieri formulae for $\symm_n$ and $\symm_n \times \ZZ_2$}

    We quickly review here the Pieri rules from 
    the representation theory of $\symm_n$ and $\symm_n \times \ZZ_2$
    that we will need, and introduce a more compact notation for
    induction products of characters.

    Recall that for a finite group $G$, the irreducible complex representations
    $\Irr(G)$ are determined by their characters $\chi$. Therefore, we will
    often speak of irreducible characters when we speak of elements of $\Irr(G)$. 
    \nomenclature[al]{$\Irr(G)$}{irreducible representations of $G$ over the complex numbers}%

    The irreducible characters $\Irr(\symm_n)$ are
    indexed $\chi^{\lambda}$ by partitions $\lambda$ of $n$, with $\chi^{(n)} = \trivial$
    and $\chi^{(1^n)} = \sgn$.  Since $\ZZ_2$ is abelian, its
    irreducible characters are both of degree $1$:
    $$
      \Irr(\ZZ_2)=\{\chi^+=\trivial,\chi^-\}.
    $$
    Therefore, the product group $\symm_n \times \ZZ_2$ has irreducible characters
    $$
      \Irr(\symm_n \times \ZZ_2) = \{ 
          \chi^{\lambda,+}: = \chi^\lambda \otimes \chi^+, \quad 
          \chi^{\lambda,-}: = \chi^\lambda \otimes \chi^- : \lambda \text{ a partition of }n\}.
    $$

    Given this setup, the following is a corollary to \ref{prop:kernels-sign-twisted}.
    Recall that $\lambda^T$ denotes the conjugate partition of $\lambda$.
    \nomenclature[co]{$\lambda^T$}{conjugate partition of number partition $\lambda$}
    \index{conjugate partition}
    \index{partition!conjugate}

    \begin{Corollary} 
      \label{sign-twisting-corollary}
      For $A=[n]$, as $\RR[\symm_n \times \ZZ_2]$-modules, one has
      $$
        \ker \pi_A \cong \bigoplus_\alpha \ \chi^{\lambda_\alpha,\epsilon_\alpha}
      $$
      if and only if
      $$
        \ker \delminus(A,|A|-1) \cong \bigoplus_\alpha \chi^{\lambda_\alpha^T,\epsilon_\alpha}.
      $$
    \end{Corollary}
    \begin{proof}
      The map $\RR \symm_n \overset{\sgn}{\longrightarrow} \RR \symm_n$
      has these effects:
      \begin{enumerate}
        \item[$\bullet$] For the $\RR \symm_n$-module structure, it tensors with
          the $\sgn$-character, which on irreducibles does the following:
          $$
            \chi^\lambda \mapsto \sgn \otimes \chi^\lambda = \chi^{\lambda^T}.
          $$
        \item[$\bullet$] For the $\RR \ZZ_2$-module structure it is equivariant, since
          the nonidentity element of $\ZZ_2$ acts by $w \mapsto ww_0$,
          when we are thinking of $\RR \symm_n$ as the domain of $\pi_A$,
          but this nonidentity element introduces an extra sign in front when we are 
          thinking of $\RR \symm_n$ as the domain of $\delminus(A,|A|-1)$ within 
          the complex of injective words:
          $$
          \xymatrix@C+=4em@R+=5ex{
            w \ar@{|->}[r]^{i_A \circ \pi_A}
              \ar@{|->}[d]_{\sgn}
          & ww_0 \ar@{|->}[d]_{\sgn}
          \\
            \sgn(w) \cdot w
                \ar@{|->}[r]^(0.4){i_A\circ\delminus(A,|A|-1)\ }
          & {\ \begin{array}{c}
            \sgn(ww_0) \cdot ww_0 \\
            = (-1)^{\lfloor\frac{n}{2}\rfloor}\sgn(w)\cdot w w_0
            \end{array}}}
          $$
      \end{enumerate}
    \end{proof}

    The Young subgroup embedding $\symm_{n_1}\times \symm_{n_2} \hookrightarrow \symm_{n_1+n_2}$
    leads to the usual {\deffont induction product} of characters $\chi_i$
    \index{induction product}%
    in $\Irr(\symm_{n_i})$, for $i=1,2$, defined by
    $$
      \chi_1 * \chi_2 
         := \Ind_{\symm_{n_1} \times \symm_{n_2}}^{\symm_{n_1+n_2}} \chi_1 \otimes \chi_2.
    $$
    The {\deffont Pieri formulae} give two important special cases of the
    \index{Pieri formulae}%
    the irreducible expansion for the induction product of two $\symm_n$-irreducibles:
    \begin{equation}
      \label{usual-Pieris}
      \begin{aligned}
        \chi^\mu * \chi^{(j)} 
          &= \sum_{ \substack{ \lambda: 
                      \\ \lambda/\mu \text{ is a horizontal}
                      \\\text{strip of size }j }} \chi^\lambda \\
           \chi^\mu * \chi^{(1^j)} 
          &= \sum_{ \substack{ \lambda: 
                      \\ \lambda/\mu \text{ is a vertical}
                      \\\text{strip of size }j }} \chi^\lambda 
      \end{aligned}
    \end{equation}

    One can define an induction product of characters
    $\chi_i$ in $\Irr(\symm_{n_i} \times \ZZ_2)$ for $i=1,2$ by
    $$
      \chi_1 * \chi_2 
        := \Res_{\symm_{n_1+n_2} \times \ZZ_2}^{\symm_{n_1+n_2} \times (\ZZ_2)^2}
          \Ind_{\symm_{n_1}\times \ZZ_2 \times \symm_{n_2} \times \ZZ_2}^{\symm_{n_1+n_2} \times (\ZZ_2)^2} 
             \left( \chi_1 \otimes \chi_2 \right).
    $$
    where the restriction map above comes from the {\deffont diagonal} embedding
    $ 
     \ZZ_2 \hookrightarrow (\ZZ_2)^2 
    $
    that sends $x \mapsto (x,x)$.

    \begin{Proposition}
      For any partition $\mu$ and signs $\epsilon_1, \epsilon_2$ in $\{+,-\}$, one has
      \begin{equation}
        \label{Z2-Pieris}
        \begin{aligned}
          \chi^{\mu,\epsilon_1} * \chi^{(j),\epsilon_2} 
            &= \sum_{ \substack{ \lambda: 
                      \\ \lambda/\mu \text{ is a horizontal}
                      \\\text{strip of size }j }} \chi^{\lambda,\epsilon_1 \epsilon_2} \\
          \chi^{\mu,\epsilon_1} * \chi^{(1^j),\epsilon_2} 
            &= \sum_{ \substack{ \lambda: 
                      \\ \lambda/\mu \text{ is a vertical}
                      \\\text{strip of size }j }} \chi^{\lambda,\epsilon_1 \epsilon_2} \\
        \end{aligned}
      \end{equation}
    \end{Proposition}
    \begin{proof}
      More generally, for any embedding of finite groups 
      $G_1 \times G_2 \hookrightarrow G$, and an abelian group $A$, along
      with characters $\chi_i$ in $\Irr(G_i)$ for $i=1,2$, and characters 
      $\epsilon_1,\epsilon_2$ in $\Irr(A)$, we claim
      $$
        \begin{aligned}
          &\Res_{G \times A}^{G \times A^2}
           \Ind_{G_1 \times A \times G_2 \times A}^{G \times A^2} 
           \left( \chi_1 \otimes \epsilon_1 \otimes \chi_2 \otimes \epsilon_2 \right)  \\
          &\qquad = \bigoplus_{\chi \in \Irr(G)} 
           \left\langle \Ind_{G_1 \times G_2}^G \chi_1 \otimes \chi_2 , \chi  \right\rangle_{G} 
           \,\, \cdot \,\, \chi \otimes \epsilon_1 \epsilon_2. 
        \end{aligned}
      $$
      This comes, for example, using Frobenius reciprocity to calculate the
      \index{Frobenius!reciprocity}%
      inner product with an irreducible $\chi \otimes \epsilon$ in $\Irr(G \times A)$:
        \begin{align*} 
          & \left\langle \quad 
            \Res_{G \times A}^{G \times A^2}
            \Ind_{G_1 \times A \times G_2 \times A}^{G \times A^2} 
            \left( \chi_1 \otimes \epsilon_1 \otimes \chi_2 \otimes \epsilon_2 \right) \,\,  
            ,\,\, \chi \otimes \epsilon 
            \quad \right\rangle_{G \times A} \\
          & \qquad 
            =\left\langle \quad 
            \chi_1 \otimes \epsilon_1 \otimes \chi_2 \otimes \epsilon_2  \,\,  , \,\, 
            \Res_{G_1 \times A \times G_2 \times A}^{G \times A^2} 
            \Ind_{G \times A}^{G \times A^2}
            \chi \otimes \epsilon 
            \quad \right\rangle_{G_1 \times A \times G_2 \times A} \\
          & \qquad 
            =\left\langle \,\,  \chi_1 \otimes \chi_2 \,\, , \,\, 
                 \Res_{G_1 \times G_2}^G \chi \,\, \right\rangle_{G_1 \times G_2} 
            \,\, \cdot \,\, 
            \left\langle \,\, \epsilon_1 \otimes \epsilon_2 \,\, , \,\, \Ind_A^{A^2}\epsilon  \,\, \right\rangle_A \\
          & \qquad =\left\langle 
            \,\, \Ind_{G_1 \times G_2}^G \chi_1 \otimes \chi_2 \,\, , \,\, \chi \,\,
            \right\rangle_G
            \,\, \cdot \,\, 
            \left\langle \,\, \Res^{A^2}_A \epsilon_1 \otimes \epsilon_2 \,\, , \,\, \epsilon \,\, \right\rangle_A \\
          & \qquad 
            =\begin{cases}
               \left\langle
               \,\,  
               \Ind_{G_1 \times G_2}^G \chi_1 \otimes \chi_2 \,\, , \,\, \chi \,\, 
               \right\rangle_G & \text{ if } \epsilon = \epsilon_1 \epsilon_2 \\
               0 & \text{ otherwise.}
             \end{cases}
             \qedhere
        \end{align*}
    \end{proof}

  \subsection{Some derangement numerology}

    The nullity of either map $\pi_A$ or $\delminus(A,|A|-1)$
    turns out to be the number of {\deffont derangements} (that is, permutations
    \index{derangement}%
    with no fixed points) in $\symm_n$.  We review here some
    easy, classical, enumerative results about derangements, along with
    a few somewhat more recent results about even and odd derangements, relevant
    for the $\RR[\symm_n \times \ZZ_2]$-module structures on the kernels;
    see also Chapman \cite{Chapman2001}, 
    Mantaci and Rakotandrajao \cite{MantaciRakotandrajao2003},
    Gordon and McMahon \cite[\S 4]{GordonMcMahon2009}.

    \begin{Definition}
    \label{def:derangements}
      For $n \geq 1$, let $d_n, d_n^{+}, d_n^{-}$ denote, respectively, 
      the total number of derangements in $\symm_n$,
      the number whose sign is positive, and the number whose sign is negative.  
      (The table in \ref{fig:derangement-numbers} lists the first few values.)
    \end{Definition}

    \begin{figure}[ht]
      \begin{tabular}{rrrr}
        $n$ & $d_n$ & $d_n^+$ & $d_n^-$ \\ \toprule
          0 & 1     &   1     &  0    \\
          1 & 0     &   0     &  0    \\
          2 & 1     &   0     &  1    \\
          3 & 2     &   2     &  0    \\
          4 & 9     &   3     &  6    \\
          5 & 44    &   24    &  20   \\
          6 & 265   &   130   &  135
      \end{tabular}
      \caption{The first few values of $d_n$, $d_n^+$ and $d_n^-$: the total
      number of derangements, even derangements and odd derangements in
      $\symm_n$, respectively.}
      \label{fig:derangement-numbers}
    \end{figure}

    \begin{Proposition}
      \label{derangement-numerology}
      The numbers $d_n, d_n^{+}, d_n^{-}$ satisfies the
      initial conditions
      $$
        \begin{aligned}
          d_0 & =d_0^+=1, d_0^-=0\\
          d_1 & =d_1^+=d_1^-=0
        \end{aligned}
      $$
      as well as the following recurrences and identities:
      \begin{subequations}
      \begin{align}
        \label{(i)}
            d_n =& (n-1)(d_{n-1} + d_{n-2})  \text{ for } n \geq 2; \\[1ex]
        \label{(ii)}
            d_n^{+} =& (n-1)(d_{n-1}^- + d_{n-2}^-)  \text{ for } n \geq 2; \\[1ex]
        \label{(iii)}
            d_n^{-} =& (n-1)(d_{n-1}^+ + d_{n-2}^+)  \text{ for } n \geq 2; \\[1ex]
        \label{(iv)}
            d_n =& n d_{n-1} + (-1)^n  \text{ for } n \geq 1; \\[1ex]
        \label{(v)}
            d_n^{+} - d_n^{-} =& (-1)^{n-1}(n-1)  \text{ for } n \geq 0; \\[1ex]
        \label{(vi)}
            d_n =& \binom{n}{2} 2d_{n-2} + (-1)^n(n-1)  \text{ for } n \geq 2; \\[1ex]
        \label{(vii)}
            n! =& \sum_{j = 0}^n \binom{n}{j} d_{n-j},
      \end{align}
        \begin{quote}
        as $\binom{n}{j} d_{n-j}$ (resp. $\binom{n}{j} d^+_{n-j},
        \binom{n}{j} d^-_{n-j}$) counts the number of permutations (resp. even,
        odd permutations) having exactly $j$ fixed points.
        \end{quote}
      \end{subequations}
    \end{Proposition}

    \begin{proof}
      Recurrences \ref{(i)}, \ref{(ii)}, \ref{(iii)} follow from the fact that given
      a derangement $w$ in $\symm_n$, erasing $n$ from the cycle structure of $w$
      results in one of two possibilities.
      \begin{enumerate}
        \item[$\bullet$]
          A derangement $\hat{w}$ in $\symm_{n-1}$ having opposite sign to $w$.
          From $\hat{w}$ one can uniquely recover $w$ by specifying the value $w(n)$ in $[n-1]$.
        \item[$\bullet$]
          A permutation in $\symm_{n-1}$ with exactly one fixed point.  After removing
          this fixed point $w(n)$, one obtains a derangement $\hat{w}$ in $\symm_{n-2}$ 
          having opposite sign to $w$.  And again, 
          from $\hat{w}$ one can uniquely recover $w$ by specifying the 
          value $w(n)$ in $[n-1]$.
      \end{enumerate}

      Recurrence \ref{(iv)} follows from \ref{(i)} by induction on $n$.  The base cases where $n=0,1$
      are easily checked.  In the inductive step where $n \geq 2$, recurrence \ref{(i)} 
      implies 
      $$
        d_n-nd_{n-1} = -(d_{n-1}-(n-1)d_{n-2}) = -(-1)^{n-1}
      $$
      where the second equality uses induction.

      Recurrence \ref{(v)} follows from \ref{(ii)} and \ref{(iii)} by induction on $n$.  The base cases where $n=0,1$
      are easily checked.  In the inductive step where $n \geq 2$, recurrences \ref{(ii)} and \ref{(iii)}
      imply
      $$
        \begin{aligned}
          d_n^+-d_n^- & = (n-1)\left( (d_{n-1}^--d_{n-2}^-) - (d_{n-1}^+-d_{n-2}^+) \right) \\
            & = (n-1)\left( (d_{n-2}^+-d_{n-2}^-) - (d_{n-1}^+-d_{n-1}^-) \right) \\
            & = (n-1)\left( (-1)^{n-3}(n-3) - (-1)^{n-2}(n-2) \right) \\
            & = (n-1) (-1)^{n-1}
        \end{aligned}
      $$
      where the third equality uses induction.

      Recurrence \ref{(vi)} is a rewriting of the first iterate of recurrence \ref{(iv)}:
      $$
        \begin{aligned}
          d_n &= n d_{n-1} + (-1)^n\\
            & = n \left( (n-1) d_{n-2} + (-1)^{n-1} \right) + (-1)^n \\
            & = \binom{n}{2} \cdot 2 d_{n-2} + (-1)^{n-1} (n-1) 
        \end{aligned}
      $$

      The assertions in \ref{(vii)} all come from the fact that every permutation $w$ in $\symm_n$
      having $j$ fixed points gives rise to a derangement $\hat{w}$ on its 
      set of $n-j$ nonfixed points;  this $\hat{w}$ has the same sign as $w$.
    \end{proof}

  \subsection{$(\symm_n \times \ZZ_2)$-structure of the first kernel}

    We begin with a proposition showing that the kernel of the top boundary map $\delminus([n],n-1)$ 
    in the complex of injective words satisfies the representation-theoretic analogues of the
    derangement number recurrences in \ref{(iv)} and \ref{(vi)}. 
    For the $\RR\symm_n$-module structure, this was observed in \cite[\S 2]{ReinerWebb2004};
    for the $\RR[\symm_n \times \ZZ_2]$-module structure it appears to be new.

    \begin{Proposition}
      Considered as a virtual character of $\symm_n$,
      \begin{equation}
        \label{kernel-recurrence}
        \ker \delminus([n],n-1) 
        = \ker \delminus([n-1],n-2) * \chi^{(1)} + (-1)^n \chi^{(n)}.
      \end{equation}

      Considered as a virtual character of $\symm_n \times \ZZ_2$,
      \begin{equation}
        \label{Z2-kernel-recurrence}
        \ker \delminus([n],n-1) 
        = \ker \delminus([n-2],n-3) * \left( \chi^{(2),-} + \chi^{(1^2),+} \right) 
         + (-1)^{n-1} \chi^{(n-1,1),+}.
      \end{equation}
    \end{Proposition}

    \begin{proof}
      (cf. Proof of \cite[Propositions 2.1, 2.2]{ReinerWebb2004})
      The complex of injective words $(K_{[n]}, \delminus([n],\cdot))$ 
      is known to be the augmented cellular chain complex corresponding 
      to a regular $CW$-complex of dimension $n-1$, homotopy equivalent to
      a bouquet of spheres of dimension $n-1$; see Farmer \cite{Farmer1978}, 
      Bj\"orner and Wachs \cite{BjornerWachs1983}.
      Consequently, its homology $\ReducedHomology_\bullet(K_{[n]})$ 
      is concentrated in dimension $n-1$, and coincides with 
      $\ker \delminus([n],n-1)$.

      On the other hand, the Hopf trace formula gives the following identity of
      \index{Hopf trace formula}%
      virtual characters for any finite group acting on $K_{[n]}$:
      $$
        \sum_{i \geq -1} (-1)^i \ReducedHomology_i(K_{[n]}) = \sum_{i \geq -1} (-1)^i K_{[n],i}.
      $$
      From this we conclude that
      \begin{equation}
        \label{kernel-as-alternating-sum}
        \ker \delminus([n],n-1) = \sum_{i \geq -1} (-1)^{n-i-1} K_{[n],i}.
      \end{equation}
      Using this expression \eqref{kernel-as-alternating-sum}, the
      two recurrences in the proposition will follow after deriving
      recurrences for $\symm_n$ and $\symm_n \times \ZZ_2$-structures
      on the chain groups $K_{[n],i}$.

      The recurrence as characters of $\symm_n$ takes the form
      $$
        K_{[n],i} = 
        \begin{cases}
          \chi^{(n)} &\text{ if }i = -1\\
          K_{[n-1],i-1} * \chi^{(1)} &\text{ if }i \geq 0. \\
        \end{cases}
      $$
      This is because as $\RR \symm_n$-modules one has the general description
      $$
        \begin{aligned}
          K_{[n],i} &= \RR^{[n]^{\langle i+1 \rangle}} \\ 
            &\cong \chi^{(n-i-1)} *  
            \underbrace{\chi^{(1)} * \cdots  * \chi^{(1)}}_{i+1\text{ factors}}.
        \end{aligned}
      $$

      The recurrence as characters of $\symm_n \times \ZZ_2$ takes the form
      $$
        K_{[n],i} = 
        \begin{cases}
          \chi^{(n),+} &\text{ if }i = -1\\
          \chi^{(n-1),+} * \chi^{(1),+} &\text{ if }i = 0\\
          K_{[n-2],i-2} * \left( \chi^{(2),-} + \chi^{(1^2),+} \right) &\text{ if }i \geq 1. \\
        \end{cases}
      $$
      To understand this, note that the $\ZZ_2$-action reversing the positions in injective
      words of length $i+1$ decomposes according to the cycle structure of the 
      reversing permutation $w_0$ in $\symm_{i+1}$.
      This allows one to describe the chain groups via the induction product as follows:  
      as $\RR[\symm_n \times \ZZ_2]$-modules,
      $$
        \begin{aligned}
          K_{[n],i} &= \RR^{[n]^{\langle i+1 \rangle}} \\
          &\cong \chi^{(n-i-1),+} *
            \begin{cases}
                \underbrace{\RR \symm_2 * \cdots 
                               * \RR \symm_2}_{\frac{i+1}{2}\text{ factors}}
                                             & \text{ if }i \text{ is odd,}\\
                \underbrace{\RR \symm_2 * \cdots 
                               * \RR \symm_2}_{\frac{i}{2}\text{ factors}}
                               * \RR \symm_1 & \text{ if }i \text{ is even}.
            \end{cases}
        \end{aligned}
      $$
      Also note that 
      $$
        \begin{aligned}
          \RR \symm_1 &\cong \chi^{(1),+} \\ 
          \RR \symm_2 &\cong \chi^{(2),-} + \chi^{(1^2),+}.
        \end{aligned}
      $$

      It only remains to explain the last term on the right 
      in \eqref{Z2-kernel-recurrence}, arising from the
      following computation:
        \begin{align*}
          (-1)^{n-1} K_{[n],0} + (-1)^{n} K_{[n],-1} 
          & = (-1)^{n-1} \left( \chi^{(n-1),+} * \chi^{(1),+} - \chi^{(n),+} \right)  \\
          & = (-1)^{n-1} \chi^{(n-1,1),+}
          \qedhere
        \end{align*}
    \end{proof}

    Combining this with \ref{sign-twisting-corollary} immediately gives the following
    version of the same recurrences, which are again analogues of
    the derangement recurrences \ref{(iv)} and \ref{(vi)}.

    \begin{Corollary}
      Considered as a virtual character of $\symm_n$,
      \begin{equation}
        \label{plus-kernel-recurrence}
        \ker \pi_{[n]} 
        = \ker \pi_{[n-1]} * \chi^{(1)} + (-1)^n \chi^{(1^n)}.
      \end{equation}

      Considered as a virtual character of $\symm_n \times \ZZ_2$,
      \begin{equation}
        \label{Z2-plus-kernel-recurrence}
        \ker \pi_{[n]} 
        = \ker \pi_{[n-2]} * \left( \chi^{(2),+} + \chi^{(1^2),-} \right) 
         + (-1)^{n-1} \chi^{(2,1^{n-2}),+}. \qed
      \end{equation}
    \end{Corollary}

    \begin{Corollary}
      \label{kernel-dimension-corollary}
      The kernels of the two maps $\pi_{[n]}$ and $\delminus([n],n-1)$
      both have dimension $d_n$, the number of derangements in $\symm_n$.
      Furthermore, both have the dimension of their $\ZZ_2$-isotypic components
      equal to $d_n^+$, $d_n^-$, the number of even, odd derangements
      in $\symm_n$, respectively.
    \end{Corollary}
    \begin{proof}
      The first assertion follows upon comparison of the recurrence
      \eqref{kernel-recurrence} with \ref{(iv)}.

      For the second assertion, note that
      \eqref{Z2-kernel-recurrence} implies that the dimensions $\hat{d}_n^+,\hat{d}_n^-$
      of the $\ZZ_2$-isotypic components of the kernel of $\pi_{[n]}$
      satisfy for $n \geq 2$ the recurrences
      $$
        \begin{aligned}
          \hat{d}_n^+ &= \binom{n}{2} \hat{d}_{n-2}^- 
               + \binom{n}{2} \hat{d}_{n-2}^+ 
               + (-1)^{n-1} (n-1) \\
          \hat{d}_n^- &= \binom{n}{2} \hat{d}_{n-2}^+ 
               + \binom{n}{2} \hat{d}_{n-2}^-.
        \end{aligned}
      $$
      Subtracting these gives for $n \geq 2$,
      $$
        \hat{d}_n^+ - \hat{d}_n^- = (-1)^{n-1}(n-1) = d_n^+ - d_n^-
      $$
      where the last equality is \ref{(v)}.
      One can directly verify that this holds also for $n=0,1$.
      On the other hand, by the first assertion of the corollary, one has 
      $$
        \hat{d}_n^+ + \hat{d}_n^- = d_n = d_n^+ + d_n^-
      $$
      and hence one concludes that $\hat{d}_n^{\epsilon} = d_n^{\epsilon}$ for $\epsilon = +, -$.
    \end{proof}

  \subsection{$(\symm_n \times \ZZ_2)$-structure of the kernel filtration}

    We continue our study of the $(\symm_n \times \ZZ_2)$-structure 
    of the filtration factors $F_{n,j}$ in \eqref{eqn:first-family-filtration}.
    The following proposition is a straightforward special case
    of \ref{eqn:nested-kernel-equation}, and will be used in the proof
    of \ref{image-technical-lemma} and \ref{Z2-filtration-analysis}.

    \begin{Proposition}
      \label{delplus-composites}
      For any finite set $A$ and $0 \leq i \leq j \leq k \leq |A|$, one has
      $$
        \delplus(A,j,i) \circ \delplus(A,k,j) = \binom{k-i}{j-i} \delplus(A,k,i). 
      $$
    \end{Proposition}

    We first use this proposition in the proof of a technical lemma.

    \begin{Definition}
      Given a set $A$, let $\binom{A}{j}$ denote the collection of all $j$-element
      subsets $J$ of $A$.  Given $J$ in $\binom{A}{j}$,  
      define an $\RR$-bilinear  concatenation product
      $$
        \RR \symm_J \times \RR \symm_{A \setminus J} \longrightarrow \RR \symm_A
      $$
      by sending $(u,v)$ to $u\concat v:=(u_1,\ldots,u_j,v_1,\ldots,v_{n-j})$; that is
      the permutation sending $i$ to $u_i$ if $1 \leq i \leq j$ and $i$ to $v_{i-j}$
      if $j+1 \leq i \leq n$.
    \end{Definition}

    \begin{Lemma}
      \label{image-technical-lemma}
      For any $J$ in $\binom{[n]}{j}$,
      $$
        \ker \pi_{[n] \setminus J} \subset \im \left( \delplus([n],n,n-j) \right).
      $$
    \end{Lemma}
    \begin{proof}
      Let $u, v$ be permutations in $\symm_J, \symm_{[n] \setminus J}$, respectively.
      Then their concatenation product $u \concat v$ in $\RR \symm_n$ has the following image
      under $\pi_{[n],n,n-j}$:
      $$
        \begin{aligned}
          \delplus([n],n,n-j)(u \concat v) 
            &= \sum_{\substack{\text{ subwords }\hat{u},\hat{v} \text{ of }u,v:
                                           \\ \ell(u)+\ell(v)=n-j}} \hat{u} \concat \hat{v} \\
            &= \sum_{\substack{\text{ subwords }\hat{u}\text{ of }u}} 
                \hat{u} \concat \delplus( {[n]} ,{n} ,{n-j-\ell(\hat{u})} )(v).
        \end{aligned}
      $$

      Now assume $x$ lies in $\ker \pi_{[n] \setminus J}$.
      Since $\pi_{[n] \setminus J}:=\pi_{[n]\setminus J, n-j, n-j-1}$,
      \ref{delplus-composites} shows that $x$ also lies in $\ker \pi_{[n],n,n-j-\ell}$
      for every $\ell \geq 1$.  Therefore,
      $$
        \delplus([n],n,n-j)(u \concat x) 
        = \sum_{\substack{\text{ subwords }\hat{u}\text{ of }u}} 
          \hat{u} \concat \delplus([n],n,{n-j-\ell(\hat{u})})(x)
        = x
      $$
      as only the {\emphfont empty} subword $\hat{u} = \varnothing$ can contribute in the sum
      above.  Thus $x$ lies in $\im \delplus([n],n,n-j)$.
    \end{proof}

    We will also need one simple general linear algebra fact.

    \begin{Proposition}
      \label{linear-algebra-composite-fact}
      Given any linear maps $A \overset{f}{\longrightarrow} B \overset{g}{\longrightarrow} C$,
      the map $f$ induces an isomorphism 
      $$
        A/\ker(f) \cong \im(f)
      $$
      which restricts to an isomorphism
      $$
        \ker(g \circ f) / \ker(f) \cong f(\ker(g \circ f)) = \im(f) \cap \ker(g). \qed
      $$
    \end{Proposition}

    \begin{Theorem}
      \label{Z2-filtration-analysis}
      For each $j=0,1,\ldots,n$, the map 
      $$
        \delplus([n],n,n-j):  
        \RR \symm_n \longrightarrow \RR^{[n]^{\langle n-j \rangle}} 
                               = \bigoplus_{J \in \binom{[n]}{j}} \RR \symm_{[n] \setminus J}
      $$
      induces an $\RR[\symm_n \times \ZZ_2]$-module isomorphism 
      \begin{equation}
        \label{Z2-plus-eigenspace-recurrence}
        F_{n,j} \overset{\sim}{\longrightarrow} 
        \bigoplus_{J \in \binom{[n]}{j}} \ker( \pi_{[n] \setminus J} )
          \quad \cong \quad 
          \ker \pi_{[n-j]} * \chi^{(j),+}.
      \end{equation}
    \end{Theorem}
    \begin{proof}
      Since $\delplus([n],n,n-j)$ is a map of $\RR[\symm_n \times \ZZ_2]$-modules,
      one needs only show it is an $\RR$-linear isomorphism.
      We prove this by a dimension-counting argument, beginning with a chain of
      equalities and inequalities justified below:
        \begin{eqnarray*}
          n! &\overset{(1)}{=}& \sum_{j=0}^n \binom{n}{j} d_{n-j} \\
             &\overset{(2)}{=}& \sum_{j=0}^n \sum_{J \in \binom{[n]}{j}}  
                     \dim_\RR \ker( \pi_{[n] \setminus J} )\\
             &\overset{(3)}{\leq}& \sum_{j=0}^n \dim_\RR \im( \delplus([n],n,n-j) ) 
                  \cap \ker( \delplus([n],n,n-j) )\\
             &\overset{(4)}{=}& \sum_{j=0}^n 
                   \dim_\RR \ker( \delplus([n],n-j,n-j-1) \circ \delplus([n],n,n-j) ) 
                              / \ker( \delplus([n],n,n-j) )\\
             &\overset{(5)}{=}& \sum_{j=0}^n \dim_\RR \underbrace{\ker( \delplus([n],n,n-j-1) ) 
                                 / \ker( \delplus([n],n,n-j) )}_{F_{n,j:=}}\\
             &\overset{(6)}{=}& n!.
        \end{eqnarray*}

      Equality (1) is \ref{(vii)}.
      Equality (2) comes from  \ref{kernel-dimension-corollary}.
      Inequality (3) comes from  the inclusion 
      \begin{equation}
        \label{key-inclusion}
        \bigoplus_{J \in \binom{n}{j}} \ker( \pi_{[n] \setminus J} ) 
          \quad \subseteq \quad \im( \delplus([n],n,n-j) ) \cap \ker( \delplus([n],n,n-j) )
      \end{equation}
      implied by \ref{image-technical-lemma}.
      Equality (4) comes from  \ref{linear-algebra-composite-fact}
      applied to the composition
      $$
      \begin{CD}
        \RR \symm_n 
          @>{f:=\delplus([n],n,n-j)}>>
          \displaystyle \bigoplus_{J \in \binom{[n]}{j}} \RR \symm_{[n] \setminus J}
          @>{g:=\delplus([n],n-j,n-j-1)}>>
          \displaystyle \bigoplus_{K \in \binom{[n]}{j-1}} \RR \symm_{[n] \setminus K}.
      \end{CD}
      $$
      Equality (5) comes from  \ref{delplus-composites}.
      Equality (6) comes from telescoping the dimensions of the factors
      $F_{n,j}$ in the filtration \ref{eqn:first-family-filtration} of
      $\RR \symm_n$.

      One concludes that the inequality (3) is actually an equality.  Hence
      the set inclusion \ref{key-inclusion} must actually be an equality of sets. 
      Since Equality (4) was mediated by the map $f:=\delplus([n],n,n-j)$,
      the desired conclusion follows.
    \end{proof}

    Combining \ref{Z2-filtration-analysis} with \ref{kernel-dimension-corollary}
    and \ref{(vii)} immediately implies the following.

    \begin{Corollary}
      The factor $F_{n,j}$ in the filtration \eqref{eqn:first-family-filtration} 
      has dimension equal to
      the number $\binom{n}{j} d_{n-j}$ of permutations with exactly $j$ fixed
      points.  Furthermore, its $\ZZ_2$-isotypic components have dimensions
      $\binom{n}{j} d^+_{n-j}, \binom{n}{j} d^-_{n-j}$ equal to the number of
      even, odd permutations with exactly $j$ fixed points.
    \end{Corollary}

  \subsection{Desarrangements and the random-to-top eigenvalue of a tableaux}

    There is a well-known $\RR \symm_n$-module decomposition of the group algebra
    $$
      \RR \symm_n \cong \bigoplus_{Q} \chi^{\shape(Q)}
    $$
    where $Q$ runs over all {\deffont standard Young tableaux} of size $n$.
    \index{standard Young tableau}%
    The next two sections refine this in two ways.  
    The current section first reviews D\'esarm\'enien and Wachs \cite{DesarmenienWachs1988}
    notion of desarrangements, as well as some of the unpublished work \cite{ReinerWachs2002}.
    In particular, it is shown how to assign to each tableau $Q$ an integer $\eig(Q)$ such that 
    \nomenclature[co]{$\eig(Q)$}{index of filtration component to which $Q$ contributes its irreducible}%
    $Q$ contributes the $\symm_n$-irreducible $\chi^{\shape(Q)}$ to 
    the kernel filtration factor $F_{n,\eig(Q)}$.  
    In the next section, we refine this further to give
    the $\RR[\symm_n \times \ZZ_2]$-structure, 
    defining a sign $\epsilon(Q)$ such that $Q$ contributes 
    the $\symm_n \times \ZZ_2$-irreducible $ \chi^{\shape(Q),\epsilon(Q)}$ to 
    $F_{n,\eig(Q)}$.

    We begin by recalling some well-known definitions about
    ascents/descents in permutations and tableaux.

    \begin{Definition}
      For a permutation $w$ in $\symm_n$, say that $i$ in $\{1,2,\ldots,n-1\}$
      is an {\deffont ascent} (resp. {\deffont descent}) of $w$ if
      \index{ascent!permutation}%
      \index{descent!permutation}%
      \index{permutation!ascent}%
      \index{permutation!descent}%
      $w(i) < w(i+1)$ (resp. $w(i) > w(i+1)$).  
      We will furthermore artificially decree that $n$ is always an ascent of any $w$ in $\symm_n$.

      For a standard Young tableau $Q$ of size $n$, say that $i$ in $\{1,2,\ldots,n-1\}$
      is an {\deffont ascent} (resp. {\deffont descent}) of $Q$ if
      \index{standard Young tableau!ascent}%
      \index{standard Young tableau!descent}%
      \index{ascent!standard Young tableau}%
      \index{descent!standard Young tableau}%
      $i+1$ appears weakly to the north and east (resp. south) of $i$, 
      using English notation for tableaux.
      Again we artificially decree that
      $n$ is always an ascent of $Q$ for any $Q$ of size $n$.

      Let $\SYT_n$ denote the set of all standard Young tableaux of size $n$,
      \nomenclature[co]{$\SYT_n$}{set of standard Young tableaux of size $n$}%
      and say that such a tableau $Q$ has  $\size(Q):=n$.
    \end{Definition}

    Recall that the {\deffont Robinson-Schensted algorithm} is a bijection
    \index{Robinson-Schensted algorithm}%
    $$
      \symm_n \longrightarrow \{(P,Q) \in \SYT_n^2: \shape(P)=\shape(Q) \}.
    $$
    This algorithm has many wonderful properties, and relations to
    Sch\"utzenberger's {\deffont jeu-de-taquin}.  We refer
    \index{jeu-de-taquin}%
    the reader to Sagan's book \cite[Chapter 3]{Sagan} for background
    on some of these.  For a $w \in \symm_n$ we denote by $P(w)$ and $Q(w)$ the standard
    Young tableaux such that under the Robinson-Schensted algorithm we have
    $w \mapsto (P(w),Q(w))$. Among the wonderful properties mentioned above is the fact that
    when $w \mapsto (P(w),Q(w))$, then $w$ shares the same set of ascents and descents as $Q(w)$.

    \begin{Proposition}
    For $Q$ in $\SYT_n$ (resp. $w$ in $\symm_n$), there exists a unique value
    $j$ lying in $\{0,1,2,\ldots,n-2,n\}$
    such that 
    \begin{itemize}
    \item $1,2,\ldots,j-1$ are ascents in $Q$, and
    \item if $Q$ has at least one descent then the first ascent among $j+1,j+2,\ldots,n$ in $Q$
          occurs at a value $j+k$ with $k$ even.
    \end{itemize}
    The value of $k$ is unique provided that $Q$ has at least one descent.
    \end{Proposition}

    We denote these unique values by $\eig(Q):=j$ and $k(Q):=k$;
    and we set $k(Q):=0$ if $Q$ has no descents.

    \begin{proof}
      We give the proof in the case of tableaux;  the case for permutations
      is similar, and also follows from the property of the Robinson-Schensted
      algorithm mentioned above.

      If $Q$ is empty, so that $n=0$, then one is forced to take $(j,k)=(0,0)$.

      If $Q$ is non-empty then it contains a unique {\emphfont maximal} subtableaux
      of the form
      $$
      \begin{tikzpicture}[scale=0.9]
      \begin{scope}
        \node at (0.5,4.5) {\small1};
        \node at (1.5,4.5) {\small2};
        \node at (2.55,4.5) {\small$\cdots$};
        \node at (3.5,4.5) {\small$\ell\mathord{-}1$};
        \node at (4.5,4.5) {\small$\ell$};
        \draw (0,4) grid (2,5);
        \draw (3,4) grid (5,5);
        \node at (0.5,3.5) {\small$\ell\mathord{+}1$};
        \draw (0,3) grid (1,4);
        \node at (0.5,2.5) {\small$\ell\mathord{+}2$};
        \draw (0,2) grid (1,3);
        \node at (0.5,1.6) {\small$\vdots$};
        \node at (0.5,0.5) {\small$\ell\mathord{+}m$};
        \draw (0,0) grid (1,1);
      \end{scope}
      \end{tikzpicture}
      $$
      for which $\ell+m$ is an ascent.  Here $\ell \geq 1$
      and one allows the possibility that $m=0$ or that
      $\ell+m=n$, the size of $Q$.  Then one is forced to choose
      \begin{gather*}
        (j,k) = 
        \begin{cases}
        (\ell,m) & \text{ if }m\text{ is even,} \\
        (\ell-1,m+1) & \text{ if }m\text{ is odd}.
        \end{cases}
         \qedhere
      \end{gather*}
    \end{proof}

    \begin{Definition}
      Say $w$ is a {\deffont desarrangement} if $\eig(w)=0$.
      \index{desarrangement}%

      Say $Q$ is a {\deffont desarrangement tableau} if $\eig(Q)=0$.
      \index{desarrangement!tableau}%
    \end{Definition}

    We will use the notion of {\deffont jeu-de-taquin slides}
    \index{jeu-de-taquin!slides}%
    on skew tableaux;  again see \cite[Chapter 3]{Sagan}.  Given a standard Young tableau $Q$,
    \index{skew tableau}%
    \index{standard Young tableau!skew}%
    its {\deffont (Sch\"utzenberger) demotion} will be the tableau $\demote(Q)$
    \index{Sch\"utzenberger demotion}%
    obtained by replacing the entry $1$ in its northwest corner
    with a jeu-de-taquin hole, doing jeu-de-taquin to slide the
    hole out, and subtracting $1$ from all of the entries in
    the resulting tableau.

    \begin{Proposition}
      \label{demotion-relates-the-eigenspaces}
      For $1 \leq j \leq n-1$ the map $Q  \longmapsto ( \demote^j(Q), \shape(Q) )$
      gives a bijection
      $$
        \{Q \in \SYT_n: \eig(Q)=j \} \longrightarrow  \{ (\hat{Q}, \mu) \} \\
      $$
      in which on the right side, $\hat{Q}$ is a desarrangement tableaux
      of size $n-j$, and $\mu$ is a partition of $n$, such that
      the skew shape $\mu/\shape(Q)$ is a horizontal $j$-strip.
    \end{Proposition}
    \begin{proof}
      We describe the inverse map.  Start with $(\hat{Q},\mu)$ and do 
      outward jeu-de-taquin slides
      on $\hat{Q}$ into the cells of  $\mu/\shape(Q)$, from left-to-right.
      Then add $j$ to all of the entries in the result.
      Properties of jeu-de-taquin \cite[Exercise 3.12.6]{Sagan} imply
      that the sliding will have created $j$ empty cells in the first row,
      which one now fills with the values $1,2,\ldots,j$.  The
      resulting tableau $Q$ will have $\eig(Q)=j$.
    \end{proof}

    The assertion of the next theorem appears for $j=0$ in \cite[Proposition 2.3]{ReinerWebb2004}, 
    and for $j > 0$ in the unpublished work \cite{ReinerWachs2002}.

    \begin{Theorem}
      \label{tableaux-model-for-layers}
      As $\RR \symm_n$-modules, the 
      $j^{th}$ filtration factor $F_{n,j}$ from \eqref{eqn:first-family-filtration}
      has irreducible decomposition
      $$
        \bigoplus_{Q \in \SYT_n: \eig(Q)=j} \chi^{\shape(Q)}.
      $$
    \end{Theorem}
    \begin{proof}
      Temporarily denote by $U_{n,j}$ the direct sum appearing above.
      We first prove it is isomorphic to $F_{n,j}$ for $j=0$, so that $Q$ runs
      over desarrangement tableaux in the direct sum, by checking that
      $U_{n,0}$ satisfies recurrence \eqref{plus-kernel-recurrence};
      cf. \cite[proof of Proposition 2.3]{ReinerWebb2004}.  

      When $n$ is even, we must show that
      $$
        U_{n,0}  = U_{n-1,0} * \chi^{(1)} + \chi^{(1^n)}.
      $$
      The usual Pieri formula \ref{usual-Pieris} shows that
      the term $U_{n-1,0} * \chi^{(1)}$ on the right give rises to
      all desarrangement tableau of size $n$ which are obtained by
      adding one cell labelled $n$ from a desarrangement tableau of size $n-1$.
      The only desarrangement tableau of size $n$ it will not produce 
      is the desarrangement tableau 
      $$
      \begin{tikzpicture}[scale=0.8]
      \begin{scope}
        \node at (0.5,4.5) {\small1};
        \node at (0.5,3.5) {\small2};
        \node at (0.5,2.6) {\small$\vdots$};
        \node at (0.5,1.5) {\small$n\mathord{-}1$};
        \node at (0.5,0.5) {\small$n$};
        \draw (0,0) grid (1,2);
        \draw (0,3) grid (1,5);
      \end{scope}
      \end{tikzpicture}
      $$
      that appears on the left, accounted for by the term $\chi^{(1^n)}$ on the right.

      When $n$ is odd, we must show that
      $$
        U_{n,0} + \chi^{(1^n)} = U_{n-1,0} * \chi^{(1)}.
      $$
      Again the term $U_{n-1,0} * \chi^{(1)}$ on the right give rises to
      all desarrangement tableaux of size $n$ which are obtained by
      adding one cell labelled $n$ from a desarrangement tableau of size $n-1$.
      Because $n$ is odd, this accounts for {\emphfont all} of the terms in
      $U_{n,0}$ on the left.  However, it also produces one extra
      {\deffont non-desarrangement tableau}, namely
      \index{non-desarrangement standard Young tableau}%
      \index{standard Young tableau!non-desarrangement}%
      $$
      \begin{tikzpicture}[scale=0.8]
      \begin{scope}
        \node at (0.5,4.5) {\small1};
        \node at (0.5,3.5) {\small2};
        \node at (0.5,2.6) {\small$\vdots$};
        \node at (0.5,1.5) {\small$n\mathord{-}1$};
        \node at (0.5,0.5) {\small$n$};
        \draw (0,0) grid (1,2);
        \draw (0,3) grid (1,5);
      \end{scope}
      \end{tikzpicture}
      $$
      accounted for by $\chi^{(1^n)}$ on the left.

      For $j \geq 1$, it suffices to check that $U_{n,j}$ satisfies the
      relation 
      $$
        U_{n,j} = U_{n-j,0} * \chi^{(n-j)}
      $$
      which one knows is satisfied by $F_{n,j}$ by
      forgetting the $\ZZ_2$-action in \eqref{Z2-plus-eigenspace-recurrence}.
      This follows from the Pieri formula \ref{usual-Pieris}
      and \ref{demotion-relates-the-eigenspaces}.
    \end{proof}

  \subsection{Shaving tableaux}
  \label{ss:shaving}

      The goal here is to define a sign $\epsilon(Q)=\pm 1$ so
      that a standard Young tableau $Q$ having $\eig(Q)=j$ contributes the irreducible
          $\RR[\symm_n \times \ZZ_2]$-module $\chi^{\shape(Q),\epsilon(Q)}$ to the filtration
          factor $F_{n,j}$.
      The idea is to define the sign first for a very special class
      of desarrangement tableaux, which will form the base case when extending the sign
      inductively to all desarrangement tableaux, and finally extend the sign
      using the demotion operator to all tableaux.
      
      \begin{Definition}[Shaven desarrangement tableaux]
      \label{shaven-desarrangement-definition}
      
      Define the following three kinds of {\it shaven} desarrangement
      tableaux in $Q$ in $\SYT_n$.
      
      \begin{enumerate}
              \item When $n=0$, define the empty tableau $\varnothing$ to be shaven.
              \item When $n$ is even and at least $4$, define the following
      desarrangement tableau $Q^{(n)}_-$ to be shaven:
      $$
      Q^{(n)}_-:=
      \begin{array}{c}
      \begin{tikzpicture}[scale=0.8]
        \node at (0.5,5.5) {\small1};
        \node at (1.5,5.5) {\small$n\mathord{-}1$};
        \node at (0.5,4.5) {\small2};
        \node at (0.5,3.5) {\small3};
        \node at (0.5,2.6) {\small$\vdots$};
        \node at (0.5,1.5) {\small$n\mathord{-}2$};
        \node at (0.5,0.5) {\small$n$};
        \draw (0,0) grid (1,2);
        \draw (1,5) grid (2,6);
        \draw (0,3) grid (1,6);
      \end{tikzpicture}
      \end{array}
      $$
              \item When $n$ is odd and at least $3$, define the following 
      desarrangement tableau $Q^{(n)}_+$ to be shaven:
      $$
      Q^{(n)}_+:=
      \begin{array}{c}
      \begin{tikzpicture}[scale=0.8]
        \node at (0.5,5.5) {\small1};
        \node at (1.5,5.5) {\small$n$};
        \node at (0.5,4.5) {\small2};
        \node at (0.5,3.5) {\small3};
        \node at (0.5,2.6) {\small$\vdots$};
        \node at (0.5,1.5) {\small$n\mathord{-}2$};
        \node at (0.5,0.5) {\small$n\mathord{-}1$};
        \draw (0,0) grid (1,2);
        \draw (1,5) grid (2,6);
        \draw (0,3) grid (1,6);
      \end{tikzpicture}
      \end{array}
      $$
      \end{enumerate}
      Call any other desarrangement tableau not in one of these
      three special forms $\varnothing, Q^{(n)}_+, Q^{(n)}_-$ an
      {\it unshaven} desarrangement tableau.
      \end{Definition}
      
      Unshaven desarrangements can be ``shaved'' down to shaven
      desarrangements due to the following proposition.
      
      \begin{Proposition}
      \label{unshaven-prop}
      Any unshaven desarrangement tableau $Q$ has size at least $2$,
      and the tableau $\hat{Q}$ obtained from $Q$ by removing the largest two entries $\{n-1,n\}$ 
      is again a desarrangement tableau.
      \end{Proposition}
      \begin{proof}
      An unshaven desarrangement tableau $Q$ in $\SYT_n$ must be nonempty, so that
      $n \geq 1$.  Since there are no desarrangements of size $1$, one must have $n \geq 2$.
      If the tableau $\hat{Q}$ obtained by removing its two largest
      entries $\{n-1, n\}$ is {\it not} a desarrangement tableau, then
      $n-1$ must be the first ascent of $Q$, and be even.  This forces
      $Q$ to be the shaven desarrangement tableau $Q^{(n)}_+$ for some odd $n \geq 3$,
      a contradiction.
      \end{proof}
      
      \begin{Definition}
      Define the sign $\epsilon(Q)$ for $Q$ in $\SYT_n$ inductively as follows. 
      \begin{itemize}
      \item In the base case, $Q$ is one of the
      three kinds of shaven desarrangement tableaux from 
      \ref{shaven-desarrangement-definition}, for which we decree
      $$
      \epsilon(Q):=
      \begin{cases}
      +1 & \text{ if }Q= \varnothing \\
      +1 & \text{ if }Q=Q^{(n)}_+ \\
      -1 & \text{ if }Q=Q^{(n)}_-
      \end{cases}
      $$
      \item If $Q$ is an unshaven desarrangement tableaux, let $\hat{Q}$
      be the tableau obtained from $Q$ by removing the largest two entries
      $\{n-1,n\}$ (so that $\hat{Q}$ is again a desarrangement tableau
      by \ref{unshaven-prop}) and define inductively
      $$
      \epsilon(Q):= 
      \begin{cases}
      +\epsilon(\hat{Q}) & \text{ if }\{n-1,n\}\text{ form an ascent in }Q \\
      -\epsilon(\hat{Q}) & \text{ if }\{n-1,n\}\text{ form a descent in }Q
      \end{cases}
      $$
      \item If $Q$ is not a desarrangement tableaux, so $j:=\eig(Q) > 0$,
      define inductively 
      $$ 
      \epsilon(Q):=\epsilon(\demote^{j}(Q)).
      $$
      \end{itemize}
      \end{Definition}
      
      \begin{Example}
      We compute the sign $\epsilon(Q)$ for this
      tableau $Q$ in $\SYT_{15}$:
      One can check that $j:=\eig(Q)=3$, so $Q$ has the same
      sign as the desarrangement tableaux obtained by applying the 
      demotion operator $3$ times
      From this unshaven desarrangement tableau $\demote^{j}(Q)$, one can 
      \begin{itemize}
      \item[] first ``shave'' the descent pair $\{11,12\}$,
      \item[] then the descent pair $\{9,10\}$,
      \item[] then the ascent pair $\{7,8\}$,
      \end{itemize}
      leaving as a result the shaven desarrangement tableau
      Since there were two descent
      pairs shaved, the original tableau $Q$ has sign 
      $$
      \epsilon(Q) = (-1)^2 \cdot \epsilon\left(Q^{(6)}_-\right) = -1.
      $$
      \end{Example}

      \begin{Theorem}
        \label{thm:Z2-equivariant-filtration-factor}
        As an $\RR[\symm_n \times \ZZ_2]$-module
        the $j^{th}$ filtration factor $F_{n,j}$ from \eqref{eqn:first-family-filtration}
        has irreducible decomposition
        $$
          \bigoplus_{\substack{Q \in \SYT_n:\\ \eig(Q)=j}} \chi^{\shape(Q),\epsilon(Q)}.
        $$
      \end{Theorem}
      \begin{proof}
        We follow roughly the same plan as in the proof of \ref{tableaux-model-for-layers}.
        Temporarily denote by $U_{n,j}$ the direct sum in the theorem.  Assume for the
        moment that we have shown $F_{n,0}$ is isomorphic to $U_{n,0}$.  Then
        for $j > 0$, it suffices to check that $U_{n,j}$ satisfies the relation 
        $
          U_{n,j} = U_{n-j,0} * \chi^{(j)}
        $
        from \eqref{Z2-plus-eigenspace-recurrence}.
        This follows from the $\ZZ_2$-Pieri formula \eqref{Z2-Pieris} and 
        the fact that demotion respects signs.

        Thus it only remains to show  $U_{n,0} \cong F_{n,0}$ for $j=0$.
        In other words, we wish to show that the sum of $\chi^{\shape(Q),\epsilon(Q)}$ over all 
        desarrangement tableaux $Q$ satisfies the recurrence
        \eqref{Z2-plus-kernel-recurrence}.

        When $n$ is odd and at least $3$, we must show that
        $$
          U_{n,0}  = U_{n-2,0} * \left( \chi^{(2),+} + \chi^{(1^2),-} \right) + \chi^{(2,1^{n-2}),+}.
        $$
        This follows because most of the desarrangements $Q$ in $\SYT_n$ which appear
        on the left are unshaven, with $\{n-1,n\}$ forming either an ascent or descent.
        The $\ZZ_2$-Pieri formula \eqref{Z2-Pieris} 
        shows that these terms are counted with appropriate sign $\epsilon(Q)$ by
        a term of  $U_{n-2,0} * \chi^{(2),+}$ or $U_{n-2,0} * \chi^{(1^2),-}$
        on the right.  The only term on the left which is shaven is $Q^{(n)}_+$,
        and is accounted for by the extra summand $\chi^{(2,1^{n-2}),+}$ on the right.

        When $n$ is even, we must show that
        $$
          U_{n,0} + \chi^{(2,1^{n-2}),+} = 
             U_{n-2,0} * \left( 
                           \chi^{(2),+} + \chi^{(1^2),-} \right) 
        $$
        This again follows because most of the desarrangements $Q$ in $\SYT_n$ which appear
        on the left are unshaven, with $\{n-1,n\}$ forming either an ascent or descent,
        in which case the $\ZZ_2$-Pieri formula \eqref{Z2-Pieris} 
        shows that they are counted with appropriate sign $\epsilon(Q)$ by
        a term of  $U_{n-2,0} * \chi^{(2),+}$ or $U_{n-2,0} * \chi^{(1^2),-}$ on the right.  
        But there are two other terms $\chi^{(2,1^{n-2}),+} + \chi^{(2,1^{n-2}),-}$ on the
        right, which will be generated from the term inside $U_{n-2,0}$ for the 
        desarrangement tableaux having a single column of length $n-2$.
        Correspondingly on the left, there are two other terms 
        $\chi^{(2,1^{n-2}),-} + \chi^{(2,1^{n-2}),+}$, the first coming from
        the unique {\it shaven} desarrangement of size $n$, namely $Q^{(n)}_-$,
        and the second coming from the extra summand on the left.
      \end{proof}

\newpage

  \subsection{Fixing a small value of $k$ and letting $n$ grow.}

      With the help of {\tt Sage} \cite{Sage} we computed the decomposition of the
      $\symm_n$-modules afforded by the eigenspaces of the operators
      $\nu_{(k,1^{n-k})}$. We present this data in \ref{f:decomposition2-4}
      through \ref{f:decomposition8c}, as follows:
      \begin{itemize}
      \item 
          to enhance the presentation, every zero has been replaced by a dot;
      \item
          each row of the table corresponds to a subspace $E$ in a 
          decomposition of $\RR\symm_n$ into $\symm_n$-modules,
      \item the horizontal lines partition the rows into blocks of rows
          whose corresponding subspaces $E$ contribute to $F_{n,j}$ for 
          a fixed $j$, for $j = n,\ldots, 1$ reading from top to bottom, 
      \item 
          the entry in the column indexed by $\nu_{(k,1^{n-k})}$ is the
          eigenvalue of $\nu_{(k,1^{n-k})}$ on $E$;
      \item
          the entry in the column indexed by $w_0$ is the eigenvalue for the
          $\ZZ_2$-action on $E$;
      \item
          for $n\leq5$,
          the entry in the column indexed by the $\symm_n$-irreducible
          $\chi^\lambda$ is the tableau from
          \ref{thm:Z2-equivariant-filtration-factor} that contributes
          $\chi^\lambda$ to the $\symm_n$-module afforded by $E$,
          whereas for $n>5$ the corresponding entry is 
          just the multiplicity of $\chi^\lambda$ in $E$.
      \end{itemize}
      For $n\leq5$ the quantities $\eig(Q)$, $\epsilon(Q)$ and $\shape(Q)$
      determine the placements of the tableaux in the tables,
      with the exception of the two tableaux marked by $\dagger$ 
      in \ref{f:decomposition5} (they share the same $\eig$- and
      $\epsilon$-statistic).
      For $n>5$ there is much more ambiguity.

    We now highlight some patterns that jump out from this data.

    \medskip
    For a fixed value of $k$, as $n$ grows large, most of $\RR \symm_n$ will
    be swallowed up in the $0$-eigenspace (kernel) of $\nu_{(k,1^{n-k})}$
    according to \ref{ex:two-type-A-examples}.
    For example, it shows that the nonzero eigenspaces 
    $\im \nu_{(k,1^{n-k})}$ comprise a representation of the
    form $\psi * \trivial_{n-k}$ for some $\symm_k$-representation $\psi$.
    Hence the Pieri formula  shows that any irreducible
    $\chi^\lambda$ that occurs within it must have
    have $n-\lambda_1 \leq 2k$, that is, most of its cells will 
    live in the first part $\lambda_1$
    when $n$ grows large.

    For $k=1,2,3$, one can certainly easily write down exactly which $\symm_n$-irreducibles occur
    outside the kernel of $\nu_{(k,1^{n-k})}$, 
    segregated by the subspaces $V_{n,j}$ from \ref{eqn:block-diagonalization}
    in which they will occur.  However, even for $k=2,3$ it
    is already not immediately obvious how they will segregate further into simultaneous eigenspaces,
    nor is it obvious what will be their corresponding eigenvalues as a function of $n$.
    The data suggests the following conjectural table summarizing the story for $k=1,2,3$.  
    It is correct for $k=1$, and probably not so hard to prove for $k=2,3$ by
    brute force (i.e. write down the eigenvectors explicitly), but we have not tried.  

    \begin{Conjecture}
      \label{stability-for-first-three-conj}
      For $\nu_{(1^n)}$, $\nu_{(2,1^{n-2})}$ and $\nu_{(3,1^{n-3})}$, 
      all of the nonzero eigenspaces can be simultaneously
      described by subspaces carrying irreducible $\RR \symm_n$-modules described
      in the second column of \ref{fig:conjectural-table}, and having
      eigenvalues as shown in the remaining columns.

      \begin{figure}
      \centering
        \renewcommand{\arraystretch}{2}
        \begin{tabular}{ccccccccc}
          \toprule
          $\symm_n$-module $V_{n,j}$ & $\symm_n$-irreducibles $\chi^\lambda$ & \multicolumn{3}{c}{eigenvalue of $\nu_{(k,1^{n-k})}$ on $\chi^\lambda$} \\[-2ex]
                                     &                                       & $\nu_{(1^n)}$ & $\nu_{(2,1^{n-2})}$ & $\nu_{(3,1^{n-3})}$ \\\midrule
          \multirow{1}{*}{$V_{n,n}$}
               &  $\chi^{(n)}$       & $\binom{n}{1}(n-1)!$  & $\binom{n}{2}(n-2)!$  & $\binom{n}{3}(n-3)!$ \\\midrule
          \multirow{2}{*}{$V_{n,n-2}$}
               &  $\chi^{(n-1,1)}$   & $0$ & $\frac{(n+1)!}{3!}$ & $\frac{(n+1)!}{4!}$        \\
               &  $\chi^{(n-2,1,1)}$ & $0$ & $\frac{n!}{3!}$     & $\binom{n}{2}\frac{(n-1)!}{3!}$ \\\midrule
          \multirow{4}{*}{$V_{n,n-3}$}
               &  $\chi^{(n-1,1)}$   & $0$ & $0$ & $\frac{(n+2)!}{5!}$ \\
               &  $\chi^{(n-2,2)}$   & $0$ & $0$ & $\frac{(n+1)!}{30}$ \\
               &  $\chi^{(n-2,1,1)}$ & $0$ & $0$ & $\frac{(n+1)!}{60}$ \\
               &  $\chi^{(n-3,2,1)}$ & $0$ & $0$ & $\frac{n!}{15}$     \\\bottomrule
        \end{tabular}
        \caption{(Conjectural) decomposition of the nonzero eigenspaces of
        $\nu_{(1^n)}$, $\nu_{(2,1^{n-2})}$ and $\nu_{(3,1^{n-3})}$ into
        irreducible $\RR\symm_n$-modules together with their eigenvalues; 
        c.f. \ref{stability-for-first-three-conj}.}
        \label{fig:conjectural-table}
      \end{figure}
    \end{Conjecture}

    The form of the eigenvalues in this last table
    suggests the following somewhat vague {\emphfont stability} conjecture,
    in the spirit of the {\it representation stability} recently discussed
    by Church and Farb \cite{ChurchFarb2010}.

    \begin{Conjecture}
      \label{stability-conjecture}
      There exists an infinite sequence of partitions 
      $\lambda^{1}$, $\lambda^{2}$, \ldots and positive integers
      $j_0^1$, $j_0^2$, \ldots with the following property.
      For each positive integer $n$, there exist a positive integer $\tau(n)$
      and subspaces
      $E^{(n)}_1$, $E^{(n)}_2$, \ldots $E^{(n)}_{\tau(n)} \subseteq \RR \symm_n$
      such that
      \begin{itemize}
      \item
      $E^{(n)}_i$ carries the $\symm_n$-irreducible indexed by the partition
            $$\lambda^i + \big( n-|\lambda^i|, 0, 0, \ldots, 0 \big)$$
      \item
      $E^{(n)}_i$ is a simultaneous eigenspace for the operators 
      $\nu_{(1^n)}$, $\nu_{(2,1^{n-2})}$, \dots, $\nu_{(n-1,1)}$
      with eigenvalue for $\nu_{(j,1^{n-j})}$ described by
      \begin{gather*}
      \begin{cases}
        0 & \text{if } 1 \leq j < j_0^i \\
        f(E^{(n)}_i, j) \neq 0 & \text{if } j_0^i \leq j \leq n
      \end{cases}
      \end{gather*}
      where $f(E^{(n)}_i, j)$ are functions for which
      $$
      \frac{f(E^{(n)}_i, j)}{f(E^{(n-1)}_i, j)}
      $$
      is a rational function of $n$ of total degree $1$
      \item
      $\left(\bigoplus_{i=1}^{\tau(n)} E^{(n)}_i\right)^\perp \subseteq \RR \symm_n$
      lies in the common kernel of
      $\nu_{(1^n)}$, $\nu_{(2,1^{n-2})}$, \dots, $\nu_{(n-1,1)}$.
      \end{itemize}
    \end{Conjecture}

      \begin{figure}[p]
      \vspace{1em}
      \subfigure
      {
            \begin{tabular}{cr|cc}
            \toprule
            $\nu_{(1^2)}$&$w_0$&$\chi^{2}$&$\chi^{11}$\\
            \midrule
             2 & 1 & $\tikztableausmall{{1,2}}$ & $\cdot$ \\
            \midrule
             $\cdot$ & $-1$ & $\cdot$ & $\tikztableausmall{{1},{2}}$ \\
            \bottomrule
            \end{tabular}
      }
      \hspace{2em}
      \subfigure
      {
            \begin{tabular}{ccr|ccc}
            \toprule
            $\nu_{(1^3)}$&$\nu_{(2,1)}$&$w_0$&$\chi^{3}$&$\chi^{21}$&$\chi^{111}$\\
            \midrule
             6 & 9 & 1 & $\tikztableausmall{{1,2,3}}$ & $\cdot$ & $\cdot$ \\
            \midrule
             $\cdot$ & 4 & $-1$ & $\cdot$ & $\tikztableausmall{ {1,2},{3} }$ & $\cdot$ \\ 
             $\cdot$ & 1 & $-1$ & $\cdot$ & $\cdot$ & $\tikztableausmall{ {1},{2},{3} }$ \\
            \midrule
             $\cdot$ & $\cdot$ & 1 & $\cdot$ & $\tikztableausmall{ {1,3},{2} }$ & $\cdot$ \\
            \bottomrule
            \end{tabular}
      }
      
      \vspace{2em}
      \subfigure
      {
            \begin{tabular}{cccr|ccccc}
            \toprule
            $\nu_{(1^4)}$&$\nu_{(2,1^2)}$&$\nu_{(3,1)}$&$w_0$&$\chi^{4}$&$\chi^{31}$&$\chi^{211}$&$\chi^{22}$&$\chi^{1111}$\\
            \midrule
             24 & 72 & 16 & 1 & $\tikztableausmall{{1,2,3,4}}$ & $\cdot$ & $\cdot$ & $\cdot$ & $\cdot$ \\
            \midrule
             $\cdot$ & 20 & 10 & $-1$ & $\cdot$ & $\tikztableausmall{{1,2,3},{4}}$ & $\cdot$ & $\cdot$ & $\cdot$ \\ 
             $\cdot$ & 4 & 6 & $-1$ & $\cdot$ & $\cdot$ & $\tikztableausmall{{1,2},{3},{4}}$ & $\cdot$ & $\cdot$ \\
            \midrule
             $\cdot$ & $\cdot$ & 6 & 1 & $\cdot$ & $\tikztableausmall{{1,2,4},{3}}$ & $\cdot$ & $\cdot$ & $\cdot$ \\ 
             $\cdot$ & $\cdot$ & 4 & 1 & $\cdot$ & $\cdot$ & $\cdot$ & $\tikztableausmall{{1,2},{3,4}}$ & $\cdot$ \\ 
             $\cdot$ & $\cdot$ & 2 & 1 & $\cdot$ & $\cdot$ & $\tikztableausmall{{1,4},{2},{3}}$ & $\cdot$ & $\cdot$ \\
            \midrule
             $\cdot$ & $\cdot$ & $\cdot$ & 1 & $\cdot$ & $\cdot$ & $\cdot$ & $\tikztableausmall{{1,3},{2,4}}$ & $\tikztableausmall{{1},{2},{3},{4}}$ \\ 
             $\cdot$ & $\cdot$ & $\cdot$ & $-1$ & $\cdot$ & $\tikztableausmall{{1,3,4},{2}}$ & $\tikztableausmall{{1,3},{2},{4}}$ & $\cdot$ & $\cdot$ \\
            \bottomrule
            \end{tabular}
      }
      \caption{$\symm_n$-module decomposition, for $2 \leq n \leq 4$, of the eigenspaces of $\nu_{(k,1^{n-k})}$.}
      \label{f:decomposition2-4}
      \end{figure}

      \begin{figure}[p]
      \resizebox{0.92\textwidth}{!}{
      \begin{tabular}{ccccr|ccccccc}
      \toprule
      $\nu_{(1^5)}$ & $\nu_{(2,1^3)}$ & $\nu_{(3,1^2)}$ & $\nu_{(4,1)}$ & $w_0$ & $\chi^{5}$                       & $\chi^{41}$                        & $\chi^{311}$                                 & $\chi^{32}$                        & $\chi^{2111}$                          & $\chi^{221}$                         & $\chi^{11111}$                           \\ \midrule
       $120$        &    $600$        & $200$           & $25$          & $1$   & $\tikztableausmall{{1,2,3,4,5}}$ & $\cdot$                            & $\cdot$                                      & $\cdot$                            & $\cdot$                                & $\cdot$                              & $\cdot$                                  \\ \midrule
       $\cdot$      &    $120$        & $90 $           & $18$          & $-1$  & $\cdot$                          & $\tikztableausmall{{1,2,3,4},{5}}$ & $\cdot$                                      & $\cdot$                            & $\cdot$                                & $\cdot$                              & $\cdot$                                  \\[-1ex]
       $\cdot$      &    $ 20$        & $40 $           & $13$          & $-1$  & $\cdot$                          & $\cdot$                            & $\tikztableausmall{{1,2,3},{4},{5}}$         & $\cdot$                            & $\cdot$                                & $\cdot$                              & $\cdot$                                  \\ \midrule
       $\cdot$      & $\cdot$         & $42 $           & $14$          & $1$   & $\cdot$                          & $\tikztableausmall{{1,2,3,5},{4}}$ & $\cdot$                                      & $\cdot$                            & $\cdot$                                & $\cdot$                              & $\cdot$                                  \\[-1ex]
       $\cdot$      & $\cdot$         & $24 $           & $11$          & $1$   & $\cdot$                          & $\cdot$                            & $\cdot$                                      & $\tikztableausmall{{1,2,3},{4,5}}$ & $\cdot$                                & $\cdot$                              & $\cdot$                                  \\[-1ex]
       $\cdot$      & $\cdot$         & $12 $           & $ 9$          & $1$   & $\cdot$                          & $\cdot$                            & $\tikztableausmall{{1,2,5},{3},{4}}$         & $\cdot$                            & $\cdot$                                & $\cdot$                              & $\cdot$                                  \\[-1ex]
       $\cdot$      & $\cdot$         & $8  $           & $ 7$          & $1$   & $\cdot$                          & $\cdot$                            & $\cdot$                                      & $\cdot$                            & $\cdot$                                & $\tikztableausmall{{1,2},{3,5},{4}}$ & $\cdot$                                  \\ \midrule
       $\cdot$      & $\cdot$         & $\cdot$         & $ 7$          & $1$   & $\cdot$                          & $\cdot$                            & $\cdot$                                      & $\tikztableausmall{{1,2,4},{3,5}}$ & $\cdot$                                & $\cdot$                              & $\cdot$                                  \\[-1ex]
       $\cdot$      & $\cdot$         & $\cdot$         & $ 6$          & $1$   & $\cdot$                          & $\cdot$                            & $\cdot$                                      & $\cdot$                            & $\tikztableausmall{{1,2},{3},{4},{5}}$ & $\cdot$                              & $\cdot$                                  \\[-1ex]
       $\cdot$      & $\cdot$         & $\cdot$         & $ 3$          & $1$   & $\cdot$                          & $\cdot$                            & $\cdot$                                      & $\cdot$                            & $\cdot$                                & $\tikztableausmall{{1,4},{2,5},{3}}$ & $\cdot$                                  \\[-1ex]
       $\cdot$      & $\cdot$         & $\cdot$         & $ 1$          & $1$   & $\cdot$                          & $\cdot$                            & $\cdot$                                      & $\cdot$                            & $\cdot$                                & $\cdot$                              & $\tikztableausmall{{1},{2},{3},{4},{5}}$ \\[-1ex]
       $\cdot$      & $\cdot$         & $\cdot$         & $ 8$          & $-1$  & $\cdot$                          & $\tikztableausmall{{1,2,4,5},{3}}$ & $\cdot$                                      & $\cdot$                            & $\cdot$                                & $\cdot$                              & $\cdot$                                  \\[-1ex]
       $\cdot$      & $\cdot$         & $\cdot$         & $ 7$          & $-1$  & $\cdot$                          & $\cdot$                            & $\tikztableausmall{{1,4,5},{2},{3}}^\dagger$ & $\cdot$                            & $\cdot$                                & $\cdot$                              & $\cdot$                                  \\[-1ex]
       $\cdot$      & $\cdot$         & $\cdot$         & $ 5$          & $-1$  & $\cdot$                          & $\cdot$                            & $\cdot$                                      & $\tikztableausmall{{1,2,5},{3,4}}$ & $\cdot$                                & $\tikztableausmall{{1,2},{3,4},{5}}$ & $\cdot$                                  \\[-1ex]
       $\cdot$      & $\cdot$         & $\cdot$         & $ 3$          & $-1$  & $\cdot$                          & $\cdot$                            & $\tikztableausmall{{1,2,4},{3},{5}}^\dagger$ & $\cdot$                            & $\cdot$                                & $\cdot$                              & $\cdot$                                  \\[-1ex]
       $\cdot$      & $\cdot$         & $\cdot$         & $ 2$          & $-1$  & $\cdot$                          & $\cdot$                            & $\cdot$                                      & $\cdot$                            & $\tikztableausmall{{1,4},{2},{3},{5}}$ & $\cdot$                              & $\cdot$                                  \\ \midrule
       $\cdot$      & $\cdot$         & $\cdot$         & $\cdot$       & 1     & $\cdot$                          & $\tikztableausmall{{1,3,4,5},{2}}$ & $\tikztableausmall{{1,3,5},{2},{4}}$         & $\tikztableausmall{{1,3,5},{2,4}}$ & $\tikztableausmall{{1,5},{2},{3},{4}}$ & $\tikztableausmall{{1,3},{2,5},{4}}$ & $\cdot$                                  \\
       $\cdot$      & $\cdot$         & $\cdot$         & $\cdot$       & $-1$  & $\cdot$                          & $\cdot$                            & $\tikztableausmall{{1,3,4},{2},{5}}$         & $\tikztableausmall{{1,3,4},{2,5}}$ & $\tikztableausmall{{1,3},{2},{4},{5}}$ & $\tikztableausmall{{1,3},{2,4},{5}}$ & $\cdot$                                  \\ \bottomrule
      \end{tabular}
      }
      \caption{The $\symm_5$-module decomposition for the operators $\nu_{(k,1^{n-k})}$.}
      \label{f:decomposition5}
      \end{figure}

  \subsection{The representation $\chi^{(n-1,1)}$}
    \label{subsec:defining-representation}

    We next focus on the $\chi^{(n-1,1)}$-isotypic component for the
    eigenspaces of $\nu_{(k,1^{n-k})}$, reasoning using our block-diagonalization
    \ref{eqn:block-diagonalization}.
    This allows us to piggyback on computations of Uyemura-Reyes for the case
    $k=n-1$.

    \begin{Proposition}
      For $j=0,1,2,\ldots,n-2$,
      the $\chi^{(n-1,1)}$-isotypic component 
      of $\RR \symm_n$ intersects the summand
      $V_{n,j}$ in \ref{eqn:block-diagonalization} 
      in a single copy $V_{n,j}^{(n-1,1)}$
      of the irreducible $\chi^{(n-1,1)}$.

     Consequently, each such intersection $V_{n,j}^{(n-1,1)}$
     for $j=0,1,2,\ldots,n-2$ lies within a 
     single eigenspace for any operator $\nu_{(k,1^{n-k})}$,
     and carries an integer eigenvalue for any of these operators.
   \end{Proposition}
   \begin{proof}
     There are exactly $n-1$ standard Young tableaux $Q$ of shape
     \index{standard Young tableau}%
     $(n-1,1)$, determined completely by their unique entry $m$
     in $\{2,3,\ldots,n\}$ lying in the second row of the tableau.  
     One can check that such a $Q$ has $j:=\eig(Q)=m-2$, and hence this accounts
     for exactly one copy of $\chi^{(n-1,1)}$ within $V_{n,j}$
     for each $j=0,1,2,\ldots,n-2$, proving the first assertion.

      For the second assertion, note that this multiplicity-freeness 
      allows one to apply \ref{prop:integrality-principle}
      to each of the subspaces $U=V_{n,j}^{(n-1,1)}$.
    \end{proof}

    Uyemura-Reyes provided a complete set of 
    $\chi^{(n-1,1)}$-isotypic eigenspaces for $\nu_{(n-1,1)}$
    using evaluations of {\deffont discrete Chebyshev polynomials},
    \index{discrete Chebyshev polynomials}%
    and computed their eigenvalues for  $\nu_{(n-1,1)}$
    explicitly using the Fourier-transform approach from \ref{subsec:Fourier-transform}
    \cite[\S 5.2.1]{UyemuraReyes2002}.
    Because these eigenvalues turned out to all
    be distinct, this implies that the entire family of operators
    $\nu_{(k,1^{n-k})}$, for $k=1,2,\ldots,n-1$, when restricted
    to the $\chi^{(n-1,1)}$-isotypic component of $\RR \symm_n$,
    become polynomials in the single operator
    $\nu_{(n-1,1)}$.  Hence they all 
    share these same $\chi^{(n-1,1)}$-isotypic eigenspaces
    which he constructed.

    The following conjecture about their common eigenvalues
    on these spaces is consistent
    with Uyemura-Reyes's eigenvalue calculation
    for $k=n-1$, and with our data up through $n=9$, but
    we have not tried to prove it.

    \begin{Conjecture}
      The eigenvalues of $\nu_{(k,1^{n-k})}$ on
      the $\chi^{(n-1,1)}$-isotypic component of $\RR \symm_n$
      are
      $$
        (n-k)! \binom{n-r-1}{k-r-1} \binom{n+r}{k+r} 
      $$
      for $r=1,2,\ldots,n-1$.
    \end{Conjecture}

\section{Acknowledgements}
  The first author thanks Michelle Wachs for several enlightening e-mail conversations in
  2002 regarding the random-to-top, random-to-random shuffling operators,
  and for her permission to include the results of some of these conversations
  here. He also thanks C.E. Csar for helpful conversations regarding 
  equation \ref{explicit-eigenvalue-for-second-family}.
  
  We thank Alain Lascoux for explaining how to efficiently construct
  irreducible matrix representations of the symmetric group. The resulting
  computations were very helpful throughout the course of this project. The
  computations were performed using the open-source mathematical software
  \texttt{Sage}~\cite{Sage}.

  We thank Paul Renteln for pointing out to us the relevance of his results 
  from \cite{Renteln2011} after this paper appeared on the arxiv. 
  Building on ideas from his paper we were able to add substantial 
  quantitative results in \ref{sec:rank-one}.

  We thank Persi Diaconis for informing us about the content of the
  computational experiments mentioned on pages 152-153 of 
  \cite{UyemuraReyes2002}.

  Part of this research was facilitated by long-term visits of the first
  and second authors to the
  {\emphfont Laboratoire d'informatique Gaspard-Monge} at the {\emphfont
  Universit\'e Paris-Est Marne-la-Vall\'ee}. We sincerely thank them for their
  hospitality.

  We are grateful to Eyke H\"ullermeier for posing a version of
  \ref{qu:HuellermeierQuestion} which arose in Computer Science that initiated
  the research. 

\newpage

\section*{Appendix: $\symm_n$-module decomposition of $\nu_{(k,1^{n-k})}$}

We include here the $\symm_n$-module decomposition of the simultaneous
eigenspaces for the operators $\nu_{(1^{n})}$, $\nu_{(2,1^{n-2})}$, \ldots
$\nu_{(n-1,1)}$ for $6 \leq n \leq 8$. 
See \ref{ss:shaving} for an explanation of the presentation of this data;
and \ref{f:decomposition2-4} and \ref{f:decomposition5} for the decomposition for $2\leq n \leq5$.

\begin{figure}[ht]
\resizebox{0.97\textwidth}{!}{
\begin{tabular}{cccccr|ccccccccccc}
\toprule
$\nu_{(1^6)}$ & $\nu_{(2,1^4)}$ & $\nu_{(3,1^3)}$ & $\nu_{(4,1^2)}$ & $\nu_{(5,1)}$ & $w_0$ & $\chi^{6}$ & $\chi^{51}$ & $\chi^{411}$ & $\chi^{42}$ & $\chi^{3111}$ & $\chi^{321}$ & $\chi^{21111}$ & $\chi^{33}$ & $\chi^{2211}$ & $\chi^{222}$ & $\chi^{111111}$ \\ \midrule
4320 & 5400 & 2400 & 450 & 36 & 1 & 1 & $\cdot$ & $\cdot$ & $\cdot$ & $\cdot$ & $\cdot$ & $\cdot$ & $\cdot$ & $\cdot$ & $\cdot$ & $\cdot$ \\
\midrule
$\cdot$ &  840 &  840 &  252 &  28 & -1 & $\cdot$ & 1 & $\cdot$ & $\cdot$ & $\cdot$ & $\cdot$ & $\cdot$ & $\cdot$ & $\cdot$ & $\cdot$ & $\cdot$ \\ 
$\cdot$ &  120 &  300 &  144 &  22 & -1 & $\cdot$ & $\cdot$ & 1 & $\cdot$ & $\cdot$ & $\cdot$ & $\cdot$ & $\cdot$ & $\cdot$ & $\cdot$ & $\cdot$ \\
\midrule
$\cdot$ & $\cdot$ & 336 & 168 & 24 & 1 & $\cdot$ & 1 & $\cdot$ & $\cdot$ & $\cdot$ & $\cdot$ & $\cdot$ & $\cdot$ & $\cdot$ & $\cdot$ & $\cdot$ \\ 
$\cdot$ & $\cdot$ & 168 & 112 & 20 & 1 & $\cdot$ & $\cdot$ & $\cdot$ & 1 & $\cdot$ & $\cdot$ & $\cdot$ & $\cdot$ & $\cdot$ & $\cdot$ & $\cdot$ \\ 
$\cdot$ & $\cdot$ & 84 & 84 & 18 & 1 & $\cdot$ & $\cdot$ & 1 & $\cdot$ & $\cdot$ & $\cdot$ & $\cdot$ & $\cdot$ & $\cdot$ & $\cdot$ & $\cdot$ \\ 
$\cdot$ & $\cdot$ & 48 & 57 & 15 & 1 & $\cdot$ & $\cdot$ & $\cdot$ & $\cdot$ & $\cdot$ & 1 & $\cdot$ & $\cdot$ & $\cdot$ & $\cdot$ & $\cdot$ \\
\bottomrule
\end{tabular}
}
\caption{The $\symm_6$-module decomposition for the operators $\nu_{(k,1^{n-k})}$.}
\label{f:decomposition6}
\end{figure}

\begin{landscape}
\begin{figure}
\resizebox{1.25\textwidth}{.45\textwidth}{
\begin{tabular}{ccccccr|ccccccccccccccc}
\toprule
$\nu_{(1^7)}$ & $\nu_{(2,1^5)}$ & $\nu_{(3,1^4)}$ & $\nu_{(4,1^3)}$ & $\nu_{(5,1^2)}$ & $\nu_{(6,1)}$ & $w_0$ & $\chi^{7}$ & $\chi^{61}$ & $\chi^{511}$ & $\chi^{52}$ & $\chi^{4111}$ & $\chi^{421}$ & $\chi^{31111}$ & $\chi^{43}$ & $\chi^{3211}$ & $\chi^{322}$ & $\chi^{211111}$ & $\chi^{331}$ & $\chi^{22111}$ & $\chi^{2221}$ & $\chi^{1111111}$ \\ \midrule
35280 & 52920 & 29400 & 7350 & 882 & 49 & 1 & 1 & $\cdot$ & $\cdot$ & $\cdot$ & $\cdot$ & $\cdot$ & $\cdot$ & $\cdot$ & $\cdot$ & $\cdot$ & $\cdot$ & $\cdot$ & $\cdot$ & $\cdot$ & $\cdot$ \\
\midrule
$\cdot$ &  6720 &  8400 &  3360 &  560 &  40 & -1 & $\cdot$ & 1 & $\cdot$ & $\cdot$ & $\cdot$ & $\cdot$ & $\cdot$ & $\cdot$ & $\cdot$ & $\cdot$ & $\cdot$ & $\cdot$ & $\cdot$ & $\cdot$ & $\cdot$ \\
$\cdot$ &   840 &  2520 &  1596 &  364 &  33 & -1 & $\cdot$ & $\cdot$ & 1 & $\cdot$ & $\cdot$ & $\cdot$ & $\cdot$ & $\cdot$ & $\cdot$ & $\cdot$ & $\cdot$ & $\cdot$ & $\cdot$ & $\cdot$ & $\cdot$ \\
\midrule
$\cdot$ & $\cdot$ & 3024 & 2016 & 432 & 36 & 1 & $\cdot$ & 1 & $\cdot$ & $\cdot$ & $\cdot$ & $\cdot$ & $\cdot$ & $\cdot$ & $\cdot$ & $\cdot$ & $\cdot$ & $\cdot$ & $\cdot$ & $\cdot$ & $\cdot$ \\
$\cdot$ & $\cdot$ & 1344 & 1176 & 312 & 31 & 1 & $\cdot$ & $\cdot$ & $\cdot$ & 1 & $\cdot$ & $\cdot$ & $\cdot$ & $\cdot$ & $\cdot$ & $\cdot$ & $\cdot$ & $\cdot$ & $\cdot$ & $\cdot$ & $\cdot$ \\
$\cdot$ & $\cdot$ & 672 & 840 & 264 & 29 & 1 & $\cdot$ & $\cdot$ & 1 & $\cdot$ & $\cdot$ & $\cdot$ & $\cdot$ & $\cdot$ & $\cdot$ & $\cdot$ & $\cdot$ & $\cdot$ & $\cdot$ & $\cdot$ & $\cdot$ \\
$\cdot$ & $\cdot$ & 336 & 504 & 192 & 25 & 1 & $\cdot$ & $\cdot$ & $\cdot$ & $\cdot$ & $\cdot$ & 1 & $\cdot$ & $\cdot$ & $\cdot$ & $\cdot$ & $\cdot$ & $\cdot$ & $\cdot$ & $\cdot$ & $\cdot$ \\
\end{tabular}
}
\caption{$\symm_7$-module decomposition for the operators $\nu_{(k,1^{n-k})}$ (continues in \ref{f:decomposition7b})}
\label{f:decomposition7a}
\end{figure}

\begin{figure}
\resizebox{1.25\textwidth}{.5\textwidth}{
\begin{tabular}{ccccccr|ccccccccccccccc}
\toprule
$\nu_{(1^7)}$ & $\nu_{(2,1^5)}$ & $\nu_{(3,1^4)}$ & $\nu_{(4,1^3)}$ & $\nu_{(5,1^2)}$ & $\nu_{(6,1)}$ & $w_0$ & $\chi^{7}$ & $\chi^{61}$ & $\chi^{511}$ & $\chi^{52}$ & $\chi^{4111}$ & $\chi^{421}$ & $\chi^{31111}$ & $\chi^{43}$ & $\chi^{3211}$ & $\chi^{322}$ & $\chi^{211111}$ & $\chi^{331}$ & $\chi^{22111}$ & $\chi^{2221}$ & $\chi^{1111111}$ \\ \midrule
$\cdot$ & $\cdot$ & $\cdot$ & $\cdot$ &  90 &  20 & -1 & $\cdot$ & $\cdot$ & 1 & 1 & $\cdot$ & $\cdot$ & $\cdot$ & $\cdot$ & $\cdot$ & $\cdot$ & $\cdot$ & $\cdot$ & $\cdot$ & $\cdot$ & $\cdot$ \\
$\cdot$ & $\cdot$ & $\cdot$ & $\cdot$ &  72 &  18 & -1 & $\cdot$ & $\cdot$ & $\cdot$ & $\cdot$ & 1 & 1 & $\cdot$ & $\cdot$ & $\cdot$ & $\cdot$ & $\cdot$ & $\cdot$ & $\cdot$ & $\cdot$ & $\cdot$ \\
$\cdot$ & $\cdot$ & $\cdot$ & $\cdot$ &  63 &  17 & -1 & $\cdot$ & $\cdot$ & $\cdot$ & $\cdot$ & $\cdot$ & $\cdot$ & $\cdot$ & 1 & $\cdot$ & $\cdot$ & $\cdot$ & $\cdot$ & $\cdot$ & $\cdot$ & $\cdot$ \\
$\cdot$ & $\cdot$ & $\cdot$ & $\cdot$ &  54 &  16 & -1 & $\cdot$ & $\cdot$ & $\cdot$ & $\cdot$ & $\cdot$ & 1 & $\cdot$ & $\cdot$ & $\cdot$ & $\cdot$ & $\cdot$ & $\cdot$ & $\cdot$ & $\cdot$ & $\cdot$ \\
$\cdot$ & $\cdot$ & $\cdot$ & $\cdot$ &  48 &  15 & -1 & $\cdot$ & $\cdot$ & $\cdot$ & $\cdot$ & $\cdot$ & $\cdot$ & $\cdot$ & $\cdot$ & 1 & $\cdot$ & $\cdot$ & $\cdot$ & $\cdot$ & $\cdot$ & $\cdot$ \\
$\cdot$ & $\cdot$ & $\cdot$ & $\cdot$ &  42 &  14 & -1 & $\cdot$ & $\cdot$ & $\cdot$ & $\cdot$ & $\cdot$ & $\cdot$ & $\cdot$ & $\cdot$ & $\cdot$ & $\cdot$ & $\cdot$ & 1 & $\cdot$ & $\cdot$ & $\cdot$ \\
$\cdot$ & $\cdot$ & $\cdot$ & $\cdot$ &  40 &  14 & -1 & $\cdot$ & $\cdot$ & $\cdot$ & $\cdot$ & $\cdot$ & $\cdot$ & $\cdot$ & $\cdot$ & $\cdot$ & 1 & $\cdot$ & $\cdot$ & $\cdot$ & $\cdot$ & $\cdot$ \\
\bottomrule
\end{tabular}
}
\caption{$\symm_7$-module decomposition for the operators $\nu_{(k,1^{n-k})}$ (continued from \ref{f:decomposition7a})}
\label{f:decomposition7b}
\end{figure}
\end{landscape}

\begin{landscape}
\begin{figure}
\resizebox{1.25\textwidth}{!}{
\begin{tabular}{cccccccr|cccccccccccccccccccccc}
\toprule
$\nu_{(1^8)}$ & $\nu_{(2,1^6)}$ & $\nu_{(3,1^5)}$ & $\nu_{(4,1^4)}$ & $\nu_{(5,1^3)}$ & $\nu_{(6,1^2)}$ & $\nu_{(7,1)}$ & $w_0$ & $\chi^{8}$ & $\chi^{71}$ & $\chi^{611}$ & $\chi^{62}$ & $\chi^{5111}$ & $\chi^{521}$ & $\chi^{41111}$ & $\chi^{53}$ & $\chi^{4211}$ & $\chi^{422}$ & $\chi^{311111}$ & $\chi^{431}$ & $\chi^{32111}$ & $\chi^{3221}$ & $\chi^{2111111}$ & $\chi^{44}$ & $\chi^{3311}$ & $\chi^{332}$ & $\chi^{221111}$ & $\chi^{22211}$ & $\chi^{2222}$ & $\chi^{11111111}$ \\ \midrule
 322560 & 564480 & 376320 & 117600 & 18816 & 1568 & 64 & 1 & 1 & $\cdot$ & $\cdot$ & $\cdot$ & $\cdot$ & $\cdot$ & $\cdot$ & $\cdot$ & $\cdot$ & $\cdot$ & $\cdot$ & $\cdot$ & $\cdot$ & $\cdot$ & $\cdot$ & $\cdot$ & $\cdot$ & $\cdot$ & $\cdot$ & $\cdot$ & $\cdot$ & $\cdot$ \\
\midrule
 $\cdot$ &  60480 &  90720 &  45360 &  10080 &  1080 &  54 & -1 & $\cdot$ & 1 & $\cdot$ & $\cdot$ & $\cdot$ & $\cdot$ & $\cdot$ & $\cdot$ & $\cdot$ & $\cdot$ & $\cdot$ & $\cdot$ & $\cdot$ & $\cdot$ & $\cdot$ & $\cdot$ & $\cdot$ & $\cdot$ & $\cdot$ & $\cdot$ & $\cdot$ & $\cdot$ \\
 $\cdot$ &  6720 &  23520 &  18480 &  5600 &  760 &  46 & -1 & $\cdot$ & $\cdot$ & 1 & $\cdot$ & $\cdot$ & $\cdot$ & $\cdot$ & $\cdot$ & $\cdot$ & $\cdot$ & $\cdot$ & $\cdot$ & $\cdot$ & $\cdot$ & $\cdot$ & $\cdot$ & $\cdot$ & $\cdot$ & $\cdot$ & $\cdot$ & $\cdot$ & $\cdot$ \\
\midrule
 $\cdot$ & $\cdot$ & 30240 & 25200 & 7200 & 900 & 50 & 1 & $\cdot$ & 1 & $\cdot$ & $\cdot$ & $\cdot$ & $\cdot$ & $\cdot$ & $\cdot$ & $\cdot$ & $\cdot$ & $\cdot$ & $\cdot$ & $\cdot$ & $\cdot$ & $\cdot$ & $\cdot$ & $\cdot$ & $\cdot$ & $\cdot$ & $\cdot$ & $\cdot$ & $\cdot$ \\
 $\cdot$ & $\cdot$ & 12096 & 13104 & 4608 & 684 & 44 & 1 & $\cdot$ & $\cdot$ & $\cdot$ & 1 & $\cdot$ & $\cdot$ & $\cdot$ & $\cdot$ & $\cdot$ & $\cdot$ & $\cdot$ & $\cdot$ & $\cdot$ & $\cdot$ & $\cdot$ & $\cdot$ & $\cdot$ & $\cdot$ & $\cdot$ & $\cdot$ & $\cdot$ & $\cdot$ \\
\end{tabular}
}
\caption{$\symm_8$-module decomposition for the operators $\nu_{(k,1^{n-k})}$ (continues in \ref{f:decomposition8b} and \ref{f:decomposition8c})}
\label{f:decomposition8a}
\end{figure}

\begin{figure}
\resizebox{1.25\textwidth}{!}{
\begin{tabular}{cccccccr|cccccccccccccccccccccc}
\toprule
$\nu_{(1^8)}$ & $\nu_{(2,1^6)}$ & $\nu_{(3,1^5)}$ & $\nu_{(4,1^4)}$ & $\nu_{(5,1^3)}$ & $\nu_{(6,1^2)}$ & $\nu_{(7,1)}$ & $w_0$ & $\chi^{8}$ & $\chi^{71}$ & $\chi^{611}$ & $\chi^{62}$ & $\chi^{5111}$ & $\chi^{521}$ & $\chi^{41111}$ & $\chi^{53}$ & $\chi^{4211}$ & $\chi^{422}$ & $\chi^{311111}$ & $\chi^{431}$ & $\chi^{32111}$ & $\chi^{3221}$ & $\chi^{2111111}$ & $\chi^{44}$ & $\chi^{3311}$ & $\chi^{332}$ & $\chi^{221111}$ & $\chi^{22211}$ & $\chi^{2222}$ & $\chi^{11111111}$ \\ \midrule
 $\cdot$ & $\cdot$ & $\cdot$ & $\cdot$ & 990 & 330 & 33 & -1 & $\cdot$ & $\cdot$ & 1 & 1 & $\cdot$ & $\cdot$ & $\cdot$ & $\cdot$ & $\cdot$ & $\cdot$ & $\cdot$ & $\cdot$ & $\cdot$ & $\cdot$ & $\cdot$ & $\cdot$ & $\cdot$ & $\cdot$ & $\cdot$ & $\cdot$ & $\cdot$ & $\cdot$ \\
 $\cdot$ & $\cdot$ & $\cdot$ & $\cdot$ & 720 & 270 & 30 & -1 & $\cdot$ & $\cdot$ & $\cdot$ & $\cdot$ & 1 & 1 & $\cdot$ & $\cdot$ & $\cdot$ & $\cdot$ & $\cdot$ & $\cdot$ & $\cdot$ & $\cdot$ & $\cdot$ & $\cdot$ & $\cdot$ & $\cdot$ & $\cdot$ & $\cdot$ & $\cdot$ & $\cdot$ \\
 $\cdot$ & $\cdot$ & $\cdot$ & $\cdot$ & 630 & 250 & 29 & -1 & $\cdot$ & $\cdot$ & $\cdot$ & $\cdot$ & $\cdot$ & $\cdot$ & $\cdot$ & 1 & $\cdot$ & $\cdot$ & $\cdot$ & $\cdot$ & $\cdot$ & $\cdot$ & $\cdot$ & $\cdot$ & $\cdot$ & $\cdot$ & $\cdot$ & $\cdot$ & $\cdot$ & $\cdot$ \\
 $\cdot$ & $\cdot$ & $\cdot$ & $\cdot$ & 540 & 230 & 28 & -1 & $\cdot$ & $\cdot$ & $\cdot$ & $\cdot$ & $\cdot$ & 1 & $\cdot$ & $\cdot$ & $\cdot$ & $\cdot$ & $\cdot$ & $\cdot$ & $\cdot$ & $\cdot$ & $\cdot$ & $\cdot$ & $\cdot$ & $\cdot$ & $\cdot$ & $\cdot$ & $\cdot$ & $\cdot$ \\
 $\cdot$ & $\cdot$ & $\cdot$ & $\cdot$ & 432 & 198 & 26 & -1 & $\cdot$ & $\cdot$ & $\cdot$ & $\cdot$ & $\cdot$ & $\cdot$ & $\cdot$ & $\cdot$ & 1 & $\cdot$ & $\cdot$ & $\cdot$ & $\cdot$ & $\cdot$ & $\cdot$ & $\cdot$ & $\cdot$ & $\cdot$ & $\cdot$ & $\cdot$ & $\cdot$ & $\cdot$ \\
\end{tabular}
}
\caption{$\symm_8$-module decomposition for the operators $\nu_{(k,1^{n-k})}$ (continued from \ref{f:decomposition8a}; continues in \ref{f:decomposition8c})}
\label{f:decomposition8b}
\end{figure}

\begin{figure}
\resizebox{1.25\textwidth}{!}{
\begin{tabular}{cccccccr|cccccccccccccccccccccc}
\toprule
$\nu_{(1^8)}$ & $\nu_{(2,1^6)}$ & $\nu_{(3,1^5)}$ & $\nu_{(4,1^4)}$ & $\nu_{(5,1^3)}$ & $\nu_{(6,1^2)}$ & $\nu_{(7,1)}$ & $w_0$ & $\chi^{8}$ & $\chi^{71}$ & $\chi^{611}$ & $\chi^{62}$ & $\chi^{5111}$ & $\chi^{521}$ & $\chi^{41111}$ & $\chi^{53}$ & $\chi^{4211}$ & $\chi^{422}$ & $\chi^{311111}$ & $\chi^{431}$ & $\chi^{32111}$ & $\chi^{3221}$ & $\chi^{2111111}$ & $\chi^{44}$ & $\chi^{3311}$ & $\chi^{332}$ & $\chi^{221111}$ & $\chi^{22211}$ & $\chi^{2222}$ & $\chi^{11111111}$ \\ \midrule
 $\cdot$ & $\cdot$ & $\cdot$ & $\cdot$ & $\cdot$ & 156 & 26 & -1 & $\cdot$ & 1 & $\cdot$ & $\cdot$ & $\cdot$ & $\cdot$ & $\cdot$ & $\cdot$ & $\cdot$ & $\cdot$ & $\cdot$ & $\cdot$ & $\cdot$ & $\cdot$ & $\cdot$ & $\cdot$ & $\cdot$ & $\cdot$ & $\cdot$ & $\cdot$ & $\cdot$ & $\cdot$ \\
 $\cdot$ & $\cdot$ & $\cdot$ & $\cdot$ & $\cdot$ & 132 & 24 & -1 & $\cdot$ & $\cdot$ & 2 & 1 & $\cdot$ & $\cdot$ & $\cdot$ & $\cdot$ & $\cdot$ & $\cdot$ & $\cdot$ & $\cdot$ & $\cdot$ & $\cdot$ & $\cdot$ & $\cdot$ & $\cdot$ & $\cdot$ & $\cdot$ & $\cdot$ & $\cdot$ & $\cdot$ \\
 $\cdot$ & $\cdot$ & $\cdot$ & $\cdot$ & $\cdot$ & 110 & 22 & -1 & $\cdot$ & $\cdot$ & $\cdot$ & $\cdot$ & 1 & 3 & $\cdot$ & 2 & $\cdot$ & $\cdot$ & $\cdot$ & $\cdot$ & $\cdot$ & $\cdot$ & $\cdot$ & $\cdot$ & $\cdot$ & $\cdot$ & $\cdot$ & $\cdot$ & $\cdot$ & $\cdot$ \\
 $\cdot$ & $\cdot$ & $\cdot$ & $\cdot$ & $\cdot$ & 90 & 20 & -1 & $\cdot$ & $\cdot$ & $\cdot$ & $\cdot$ & $\cdot$ & $\cdot$ & 1 & $\cdot$ & 3 & $\cdot$ & $\cdot$ & $\cdot$ & $\cdot$ & $\cdot$ & $\cdot$ & $\cdot$ & $\cdot$ & $\cdot$ & $\cdot$ & $\cdot$ & $\cdot$ & $\cdot$ \\
 $\cdot$ & $\cdot$ & $\cdot$ & $\cdot$ & $\cdot$ & 88 & 20 & -1 & $\cdot$ & $\cdot$ & $\cdot$ & $\cdot$ & $\cdot$ & $\cdot$ & $\cdot$ & 1 & $\cdot$ & $\cdot$ & $\cdot$ & $\cdot$ & $\cdot$ & $\cdot$ & $\cdot$ & $\cdot$ & $\cdot$ & $\cdot$ & $\cdot$ & $\cdot$ & $\cdot$ & $\cdot$ \\
\bottomrule
\end{tabular}
}
\caption{$\symm_8$-module decomposition for the operators $\nu_{(k,1^{n-k})}$ (continued from \ref{f:decomposition8a} and \ref{f:decomposition8b})}
\label{f:decomposition8c}
\end{figure}

\end{landscape}

\newpage
\printnomenclature[2.8cm]

\printindex

\end{document}